\documentclass{preprint}

\usepackage[full]{textcomp}
\usepackage[osf]{newtxtext}
\usepackage[title,titletoc]{appendix}

\usepackage{ifthen} 
\usepackage{hyperref}
\usepackage{breakurl}
\usepackage{amsmath}
\usepackage{amssymb}
\usepackage{mhequ}
\usepackage{mhenvs}
\usepackage{mhsymb}
\usepackage{mathrsfs}
\usepackage{microtype}
\usepackage{wasysym}
\usepackage{centernot}
\usepackage{booktabs}
\usepackage{tikz}
\usetikzlibrary{snakes}
\usepackage{colortbl}

\usepackage{enumitem}
\usepackage{scalerel}
\usepackage{longtable}

\usepackage{comment}
\usepackage{stmaryrd}

\let\comm\comment

\DeclareSymbolFont{timesoperators}{T1}{ptm}{m}{n}
\SetSymbolFont{timesoperators}{bold}{T1}{ptm}{b}{n}
\makeatletter
\renewcommand{\operator@font}{\mathgroup\symtimesoperators}
\makeatother

\DeclareFontFamily{U}{BOONDOX-calo}{\skewchar\font=45 }
\DeclareFontShape{U}{BOONDOX-calo}{m}{n}{
  <-> s*[1.05] BOONDOX-r-calo}{}
\DeclareFontShape{U}{BOONDOX-calo}{b}{n}{
  <-> s*[1.05] BOONDOX-b-calo}{}
\DeclareMathAlphabet{\mcb}{U}{BOONDOX-calo}{m}{n}
\SetMathAlphabet{\mcb}{bold}{U}{BOONDOX-calo}{b}{n}

\setlist[enumerate]{itemsep=-1ex}

\renewcommand\vec[1]{\overset{{}_\shortrightarrow}{#1}}

\newtheorem{example}[lemma]{Example}
\newtheorem{assumption}{Assumption}

\def\BPHZ{\text{\upshape\tiny{BPHZ}}}
\def\KPZ{\text{\upshape\tiny{KPZ}}}
\def\sol{\text{\upshape{sol}}}
\def\qEW{\text{\upshape\tiny{qEW}}}
\def\ren{\ell_{\text{\upshape\tiny{BPHZ}}}}

\def\bXi{\boldsymbol{\Xi}}
\def\hXi{\hat{\Xi}}
\def\bxi{\boldsymbol{\xi}}

\def\bp{\mathbf{p}}

\def\SC{\mathscr{C}}
\def\SW{\mathscr{W}}
\def\SL{\mathscr{L}}
\def\SF{\mathscr{F}}
\def\SU{\mathscr{U}}
\def\SS{\mathscr{S}}
\def\Eps{\mcb{E}}

\def\Nodes{\mcb{N}}
\def\Rule{\mathrm{R}}
\def\CE{\mathcal{E}}
\def\CO{\mcb{O}}
\def\Proj{\mathrm{Proj}}
\def\ex{{\mathrm{ex}}}

\def\btau{\bar{\tau}}
\def\one{\mathbf{1}}
\def\1{\mathbf{1}}
\def\X{\mathbf{X}}

\renewcommand\fX{\mcb{X}}

\def\Wick{\scriptscriptstyle\mathrm{Wick}}
\def\CWick{\CWick}

\def\CWick{C^{\Wick}}

\def\${|\!|\!|}

\def\mcbP{\mcb{P}}

\def\ZZ{\mathscr{Z}}
\def\M{\mathscr{M}}

\def\w{\mathbf{w}}
\def\bu{\mathbf{u}}
\def\ba{\mathbf{a}}
\def\bv{\mathbf{v}}
\def\bh{\mathbf{h}}
\def\bUpsilon{\boldsymbol{\Upsilon}}
\def\graft_#1^#2{\mathbin{\curvearrowright_{#1}^{#2}}}
\def\up{\uparrow}

\def\0{\mathbf{0}}
\def\bo{\bar{o}}
\def\reg{\mathop{\mathrm{reg}}}

\def\init{{\mathrm{init}}}
\def\func{{\mathrm{func}}}
\def\sol{\mathrm{sol}}
\def\fo{\mathfrak{o}}

\def\fd{\mathfrak{d}}
\def\bfd{\bar{\mathfrak{d}}}
\def\fl{\mathfrak{l}}

\def\fD{\mathfrak{D}}
\def\ft{\mathfrak{t}}
\def\fb{\mathfrak{b}}
\def\ff{\mathfrak{f}}

\def\fm{\mathfrak{m}}

\def\J{\mathscr{I}}

\def\TT{\mcb{T}}
\def\T{\symb T}
\def\BF{\mathbf{F}}
\def\BG{\mathbf{G}}

\def\Deltap{\Delta^{\!+}}

\def\Deltam{\Delta^{\!-}}

\def\PPi{\boldsymbol{\Pi}}
\def\hPPi{\rlap{$\boldsymbol{\hat\Pi}$}\phantom{\boldsymbol{\Pi}}}

\def\bbUpsilon{\rlap{$\boldsymbol{\bar\Upsilon}$}\phantom{\boldsymbol{\Upsilon}}}

\def\id{\mathrm{id}}
\def\Poly{{\mathrm{Poly}}}
\def\loc{{\mathrm{loc}}}

\def\SG{\mathscr{G}}

\def\SB{\mathscr{B}}

\def\SM{\mathscr{M}}
\def\Ev{\mathrm{Ev}}

\def\SP{\mathscr{P}}
\def\SQ{\mcb{Q}}
\def\SH{\mathscr{H}}
\def\SD{\mathscr{D}}
\def\s{{\mathfrak s}}

\def\TStruc{\texttt{TStruc}}
\def\Vec{\texttt{Vec}}

\def\restr{\mathord{\upharpoonright}}

\def\emptyset{\varnothing}

\makeatletter
\newcommand{\globalcolor}[1]{%
	\color{#1}\global\let\default@color\current@color
}
\makeatother

\usetikzlibrary{calc}
\usetikzlibrary{decorations}
\usetikzlibrary{positioning}
\usetikzlibrary{shapes}
\usetikzlibrary{external}

\newif\ifdark
\IfFileExists{dark}{\darktrue}{\darkfalse}
\darkfalse

\ifdark

\definecolor{darkred}{rgb}{0.9,0.2,0.2}
\definecolor{darkblue}{rgb}{0.7,0.3,1}
\definecolor{darkgreen}{rgb}{0.1,0.9,0.1}
\definecolor{pagebackground}{rgb}{.15,.21,.18}
\definecolor{pageforeground}{rgb}{.84,.84,.85}
\pagecolor{pagebackground}
\AtBeginDocument{\globalcolor{pageforeground}}

\definecolor{symbols}{rgb}{0,0.7,1}
\colorlet{connection}{red!80!black}
\colorlet{boxcolor}{blue!50}

\else

\definecolor{darkred}{rgb}{0.7,0.1,0.1}
\definecolor{darkblue}{rgb}{0.4,0.1,0.8}
\definecolor{darkgreen}{rgb}{0,0.5,0}
\definecolor{pagebackground}{rgb}{1,1,1}
\definecolor{pageforeground}{rgb}{0,0,0}

\colorlet{symbols}{blue!90!black}
\colorlet{connection}{red!30!black}
\colorlet{boxcolor}{blue!50!black}

\fi

\def\hbar{\overline{h}}

\def\Image{\mathop{\mathrm{Image}}}

\makeatletter 
\newcommand*{\bigcdot}{}
\DeclareRobustCommand*{\bigcdot}{%
  \mathbin{\mathpalette\bigcdot@{}}%
}
\newcommand*{\bigcdot@scalefactor}{.5}
\newcommand*{\bigcdot@widthfactor}{1.15}
\newcommand*{\bigcdot@}[2]{%
  \sbox0{$#1\vcenter{}$}
  \sbox2{$#1\cdot\m@th$}%
  \hbox to \bigcdot@widthfactor\wd2{%
    \hfil
    \raise\ht0\hbox{%
      \scalebox{\bigcdot@scalefactor}{%
        \lower\ht0\hbox{$#1\bullet\m@th$}%
      }%
    }%
    \hfil
  }%
}
\makeatother

\def\slash{\leavevmode\unskip\kern0.18em/\penalty\exhyphenpenalty\kern0.18em}

\makeatletter
\pgfdeclareshape{crosscircle}
{
	\inheritsavedanchors[from=circle] 
	\inheritanchorborder[from=circle]
	\inheritanchor[from=circle]{north}
	\inheritanchor[from=circle]{north west}
	\inheritanchor[from=circle]{north east}
	\inheritanchor[from=circle]{center}
	\inheritanchor[from=circle]{west}
	\inheritanchor[from=circle]{east}
	\inheritanchor[from=circle]{mid}
	\inheritanchor[from=circle]{mid west}
	\inheritanchor[from=circle]{mid east}
	\inheritanchor[from=circle]{base}
	\inheritanchor[from=circle]{base west}
	\inheritanchor[from=circle]{base east}
	\inheritanchor[from=circle]{south}
	\inheritanchor[from=circle]{south west}
	\inheritanchor[from=circle]{south east}
	\inheritbackgroundpath[from=circle]
	\foregroundpath{
		\centerpoint%
		\pgf@xc=\pgf@x%
		\pgf@yc=\pgf@y%
		\pgfutil@tempdima=\radius%
		\pgfmathsetlength{\pgf@xb}{\pgfkeysvalueof{/pgf/outer xsep}}%
		\pgfmathsetlength{\pgf@yb}{\pgfkeysvalueof{/pgf/outer ysep}}%
		\ifdim\pgf@xb<\pgf@yb%
		\advance\pgfutil@tempdima by-\pgf@yb%
		\else%
		\advance\pgfutil@tempdima by-\pgf@xb%
		\fi%
		\pgfpathmoveto{\pgfpointadd{\pgfqpoint{\pgf@xc}{\pgf@yc}}{\pgfqpoint{-0.707107\pgfutil@tempdima}{-0.707107\pgfutil@tempdima}}}
		\pgfpathlineto{\pgfpointadd{\pgfqpoint{\pgf@xc}{\pgf@yc}}{\pgfqpoint{0.707107\pgfutil@tempdima}{0.707107\pgfutil@tempdima}}}
		\pgfpathmoveto{\pgfpointadd{\pgfqpoint{\pgf@xc}{\pgf@yc}}{\pgfqpoint{-0.707107\pgfutil@tempdima}{0.707107\pgfutil@tempdima}}}
		\pgfpathlineto{\pgfpointadd{\pgfqpoint{\pgf@xc}{\pgf@yc}}{\pgfqpoint{0.707107\pgfutil@tempdima}{-0.707107\pgfutil@tempdima}}}
	}
}
\makeatother

\tikzstyle{densely dotted}= [dash pattern=on \pgflinewidth off \pgflinewidth]

\tikzset{
	dot/.style={circle,fill=pageforeground,inner sep=0pt, minimum size=1mm},
	root/.style={circle,fill=black,inner sep=0pt, minimum size=1.2mm},
	kernel/.style={semithick,shorten >=2pt,shorten <=2pt},
	kernels/.style={snake=zigzag,shorten >=2pt,shorten <=2pt,segment amplitude=1pt,segment length=4pt,line before snake=2pt,line after snake=5pt,},
	xi/.style={very thin,solid,circle,draw=symbols,fill=symbols!10!pagebackground,inner sep=0pt,minimum size=1.2mm},
	xi1/.style={very thin,solid,regular polygon,regular polygon sides=5, draw = darkgreen ,fill=darkgreen!10!pagebackground,inner sep=0pt,minimum size=1.2mm},
	xi2/.style={very thin,solid,star, star points=6, draw = red ,fill=red!10!pagebackground,inner sep=0pt,minimum size=1.4mm},
	xin/.style={very thin,solid,circle, draw = darkgreen ,fill=darkgreen!10!pagebackground,inner sep=0pt,minimum size=1.3mm},
	xib/.style={thin,solid,circle,fill=symbols!10!pagebackground,draw=symbols,inner sep=0pt,minimum size=1.6mm},
	xib1/.style={thin,solid,regular polygon,regular polygon sides=5,fill=darkgreen!10!pagebackground,draw=darkgreen,inner sep=0pt,minimum size=1.7mm},
	xib2/.style={thin,solid,star, star points=6,fill=red!10!pagebackground,draw=red,inner sep=0pt,minimum size=1.8mm},
	xibn/.style={thin,solid,circle,fill=darkgreen!10!pagebackground,draw=darkgreen,inner sep=0pt,minimum size=1.7mm},
	noise/.style={snake=zigzag,segment amplitude=1pt,segment length=3pt, draw=symbols},
	bnoise/.style={snake=zigzag,segment amplitude=1pt,segment length=3pt, draw=darkgreen},
	bbnoise/.style={snake=zigzag,segment amplitude=1pt,segment length=3pt, draw=red},
	kernels0/.style={thick,draw=symbols,densely dotted,segment length= 35pt},
	kernels01/.style={thick,draw=darkgreen,densely dotted,segment length= 35pt},
	kernels1/.style={thick,draw=symbols,segment length= 35pt},
	kernels2/.style={double,draw=symbols,segment length= 35pt},
	bkernels1/.style={thick,draw=darkgreen,segment length= 35pt},
	bkernels2/.style={double,draw=darkgreen,segment length= 35pt},
	bbkernels2/.style={double,draw=red,segment length= 35pt},
	not/.style={very thin,solid,circle,draw=black,fill=black,inner sep=0pt,minimum size=0.5mm},
	not0/.style={very thin,solid,circle,draw=symbols,fill=symbols,inner sep=0pt,minimum size=0.6mm},
	not1/.style={thin,solid,regular polygon,regular polygon sides=5,draw=darkgreen,fill=darkgreen,inner sep=0pt,minimum size=0.7mm},
	not2/.style={thin,solid,star, star points=6,draw=red,fill=red,inner sep=0pt,minimum size=0.7mm},
}

\makeatletter
\def\DeclareSymbol#1#2#3{%
	\expandafter\gdef\csname MH@symb@#1\endcsname{\tikzsetnextfilename{symbol#1}%
		\tikz[baseline=#2,scale=0.15,draw=symbols,line join=round]{#3}}%
	\expandafter\gdef\csname MH@symb@#1s\endcsname{\scalebox{0.75}{\tikzsetnextfilename{symbol#1}%
			\tikz[baseline=#2,scale=0.15,draw=symbols,line join=round]{#3}}}%
	\expandafter\gdef\csname MH@symb@#1ss\endcsname{\scalebox{0.65}{\tikzsetnextfilename{symbol#1}%
			\tikz[baseline=#2,scale=0.15,draw=symbols,line join=round]{#3}}}%
}
\def\<#1>{\ifthenelse{\boolean{mmode}}{\mathchoice{\csname MH@symb@#1\endcsname}{\csname MH@symb@#1\endcsname}{\csname MH@symb@#1s\endcsname}{\csname MH@symb@#1ss\endcsname}}{\csname MH@symb@#1\endcsname}}
\makeatother

\DeclareSymbol{node}{-3}{\node[dot] {};}
\DeclareSymbol{Jl}{1}{\draw[noise] (0,0) node[not] {} -- (0,2) node[not0] {};} 
\DeclareSymbol{Jl1}{1}{\draw[bnoise] (0,0) node[not] {} -- (0,2) node[not1] {};} 
\DeclareSymbol{Jl2}{1}{\draw[bbnoise] (0,0) node[not] {} -- (0,2) node[not2] {};} 
\DeclareSymbol{Jt0}{1}{\draw[kernels0] (0,0) node[not] {} -- (0,2) node[not] {};} 
\DeclareSymbol{Jt1}{1}{\draw[kernels01] (0,0) node[not] {} -- (0,2) node[not] {};} 
\DeclareSymbol{J1t0}{1}{\draw[kernels1] (0,0) node[not] {} -- (0,2) node[not] {};} 
\DeclareSymbol{J1t1}{1}{\draw[bkernels1] (0,0) node[not] {} -- (0,2) node[not] {};} 
\DeclareSymbol{J2t1}{1}{\draw[bkernels2] (0,0) node[not] {} -- (0,2) node[not] {};} 
\DeclareSymbol{J2t2}{1}{\draw[bbkernels2] (0,0) node[not] {} -- (0,2) node[not] {};} 
\DeclareSymbol{t1t1}{1}{\draw[bkernels1] (-1,2) node[xi1] {} -- (0,0) node[not] {}  -- (1,2) node[xi1] {};
} 

\DeclareSymbol{Xi}{-2.5}{\node[xib] {};
} 
\DeclareSymbol{Xi+1}{-2.5}{\node[xib] at (0, 0) {}; \node(1) at (1, 0) {\tiny 1};
} 
\DeclareSymbol{Ch}{1}{\draw[kernels1] (-1,2) node[xi] {} -- (0,0) node[not] {}  -- (1,2) node[xi] {};
} 
\DeclareSymbol{HXi}{1}{\draw[kernels0] (-1,2) node[xi] {} -- (0,0) node[not] {}; 
	\draw[noise] (0,0) node[not] {} -- (1,2) node[not0] {};
}
\DeclareSymbol{H1Xi}{1}{\draw[kernels01] (-1,2) node[xi1] {} -- (0,0) node[not] {}; 
	\draw[noise] (0,0) node[not] {} -- (1,2) node[not0] {};
}
\DeclareSymbol{Xi1}{-2.5}{\node[xib1] {};
}
\DeclareSymbol{t2t1}{1}{\draw[bkernels2] (-1,2) node[xi1] {} -- (0,0) node[not] {}; 
	\draw[bkernels1] (0,0) node[not] {} -- (1,2) node[xi1] {};
}
\DeclareSymbol{t1t2t1}{1}{\draw[bkernels1] (-1.2,1.7) node[xi1] {} -- (0,0) node[not] {}  -- (1.2,1.7) node[xi1] {}; \draw[bkernels2] (0,0) node[not] {} -- (0,2.4) node[xi1] {};
}
\DeclareSymbol{J[HXi]Xi}{1.5}{\draw[kernels0] (-2,2.5) node[xi] {} -- (-1,1) node[not] {}  -- (0,0) node[not] {}; \draw[noise] (-1,1) node[not] {} -- (0,2.5) node[not0] {}; \draw[noise] (0,0) node[not] {} -- (1,1.5) node[not0] {};
}
\DeclareSymbol{J[Ch]H}{1.5}{\draw[kernels1] (-2,2.5) node[xi] {} -- (-1,1) node[not] {}  -- (0,0) node[not] {}; \draw[kernels1] (-1,1) node[not] {} -- (0,2.5) node[xi] {}; \draw[kernels1] (0,0) node[not] {} -- (1,1.5) node[xi] {};
}
\DeclareSymbol{J[HXi]H}{1.5}{\draw[kernels0] (-2,2.5) node[xi] {} -- (-1,1) node[not] {}; \draw[noise] (-1,1) node[not] {} -- (0,2.5) node[not0] {}; \draw[kernels1] (-1,1) node[not] {}  -- (0,0) node[not] {} -- (1,1.5) node[xi] {};
}
\DeclareSymbol{HXiH}{1}{\draw[kernels0] (-1.2,1.7) node[xi] {} -- (0,0) node[not] {}  -- (1.2,1.7) node[xi] {}; \draw[noise] (0,0) node[not] {} -- (0,2.4) node[not0] {};
}
\DeclareSymbol{bt1Xi1H}{1}{\draw[bkernels1] (-1.2,1.7) node[xi1] {} -- (0,0) node[not] {}; \draw[kernels0] (0,0) node[not] {}  -- (1.2,1.7) node[xi] {}; \draw[bnoise] (0,0) node[not] {} -- (0,2.4) node[not1] {};
}
\DeclareSymbol{J[Ch]Xi}{1.5}{\draw[kernels1] (-2,2.5) node[xi] {} -- (-1,1) node[not] {}  -- (0,2.5) node[xi] {}; \draw[kernels0] (-1,1) node[not] {} -- (0,0) node[not] {}; \draw[noise] (0,0) node[not] {} -- (1,1.5) node[not0] {};
}
\DeclareSymbol{HXi1}{1}{\draw[kernels0] (-1,2) node[xi] {} -- (0,0) node[not] {}; 
	\draw[bnoise] (0,0) node[not] {} -- (1,2) node[not1] {};
}
\DeclareSymbol{t2J1[H']}{1}{\draw[bkernels2] (-1,1.5) node[xi1] {} -- (0,0) node[not] {}; 
	\draw[bkernels1] (0,0) node[not] {} -- (1,1) node[not] {};
	\draw[kernels1] (1,1) node[not] {} -- (1,2.5) node[xi] {};
}
\DeclareSymbol{t1J2[H']}{1}{\draw[bkernels1] (-1,1.5) node[xi1] {} -- (0,0) node[not] {}; 
	\draw[bkernels2] (0,0) node[not] {} -- (1,1) node[not] {};
	\draw[kernels1] (1,1) node[not] {} -- (1,2.5) node[xi] {};
}
\DeclareSymbol{Ch1}{1}{\draw[kernels1] (-1,2) node[xi] {} -- (0,0) node[not] {}; 
	\draw[bkernels1] (0,0) node[not] {} -- (1,2) node[xi1] {};
}
\DeclareSymbol{Ch1a}{1}{\draw[kernels1] (-1,2) node[xi] {} -- (0,0) node[not] {}; 
	\draw[kernels1] (0,0) node[not] {} -- (1,2) node[xi1] {};
}
\DeclareSymbol{H'J1[t2]}{1}{\draw[kernels1] (-1,1.5) node[xi] {} -- (0,0) node[not] {}; 
	\draw[kernels1] (0,0) node[not] {} -- (1,1) node[not] {};
	\draw[bkernels2] (1,1) node[not] {} -- (1,2.5) node[xi1] {};
}
\DeclareSymbol{4Xi1}{1}{\draw[bkernels1] (-1.5,1.5) node[xi1] {} -- (0,0) node[not] {}  -- (-0.6,2) node[xi1] {}; \draw[bkernels1] (0.6,2) node[xi1] {} -- (0,0) node[not] {};
	\draw[bkernels2] (0,0) node[not] {} -- (1.5,1.5) node[xi1] {};
} 
\DeclareSymbol{J[t1t1]Xi}{1}{\draw[bkernels1] (-2,2.5) node[xi1] {} -- (-1,1) node[not] {}  -- (0,2.5) node[xi1] {}; \draw[kernels0] (0,0) node[not] {} -- (-1,1) node[not] {}; \draw[noise] (0,0) node[not] {} -- (1,1.5) node[not0] {};
}
\DeclareSymbol{t1J2[H']t1}{1}{\draw[bkernels1] (-1.2,1.5) node[xi1] {} -- (0,0) node[not] {}  -- (1.2,1.5) node[xi1] {}; \draw[bkernels2] (0,0) node[not] {} -- (0,1.4) node[not] {}; \draw[kernels1] (0,1.4) node[not] {} -- (0,2.8) node[xi] {};
}
\DeclareSymbol{J1[HXi1]t2}{1}{\draw[kernels0] (-2,2.5) node[xi] {} -- (-1,1) node[not] {}; \draw[bnoise] (-1,1) node[not] {}  -- (0,2.5) node[not1] {}; \draw[bkernels1] (0,0) node[not] {} -- (-1,1) node[not] {}; \draw[bkernels2] (0,0) node[not] {} -- (1,1.5) node[xi1] {};
}
\DeclareSymbol{t1t2J1[H']}{1}{\draw[bkernels1] (-1.2,1.3) node[xi1] {} -- (0,0) node[not] {}  -- (1.2,1.3) node[not] {}; \draw[bkernels2] (0,0) node[not] {} -- (0,1.6) node[xi1] {}; \draw[kernels1] (1.2,1.3) node[not] {} -- (1.2,2.8) node[xi] {};
}
\DeclareSymbol{Xi2}{-2.5}{\node[xib2] {};} 
\DeclareSymbol{J[t1t2]Xi}{-1}{\draw[bkernels1] (-2,2.5) node[xi1] {} -- (-1,1) node[not] {}; \draw[bkernels2] (-1,1) node[not] {} -- (0,2.5) node[xi1] {}; \draw[kernels0] (0,0) node[not] {} -- (-1,1) node[not] {}; \draw[noise] (0,0) node[not] {} -- (1,1.5) node[not0] {};
}
\DeclareSymbol{J1[Ch1]t2}{-1}{\draw[bkernels1] (-2,2.5) node[xi1] {} -- (-1,1) node[not] {}; \draw[kernels1] (-1,1) node[not] {} -- (0,2.5) node[xi] {}; \draw[bkernels1] (0,0) node[not] {} -- (-1,1) node[not] {}; \draw[bkernels2] (0,0) node[not] {} -- (1,1.5) node[xi1] {};
}
\DeclareSymbol{t1H't1}{-1}{\draw[bkernels1] (-1.2,1.7) node[xi1] {} -- (0,0) node[not] {}  -- (1.2,1.7) node[xi1] {}; \draw[kernels1] (0,0) node[not] {} -- (0,2.4) node[xi] {};
}
\DeclareSymbol{J2[Ch1]t1}{-1}{\draw[kernels1] (-2,2.5) node[xi] {} -- (-1,1) node[not] {}; \draw[bkernels1] (-1,1) node[not] {} -- (0,2.5) node[xi1] {}; \draw[bkernels2] (0,0) node[not] {} -- (-1,1) node[not] {}; \draw[bkernels1] (0,0) node[not] {} -- (1,1.5) node[xi1] {};
}
\DeclareSymbol{J1[t1t2]H'}{-1}{\draw[bkernels2] (-2,2.5) node[xi1] {} -- (-1,1) node[not] {}; \draw[bkernels1] (-1,1) node[not] {} -- (0,2.5) node[xi1] {}; \draw[kernels1] (0,0) node[not] {} -- (-1,1) node[not] {}; \draw[kernels1] (0,0) node[not] {} -- (1,1.5) node[xi] {};
}
\DeclareSymbol{J1[2bt1Xi1]H'}{-1}{\draw[bkernels1] (-2,2.5) node[xi1] {} -- (-1,1) node[xi1] {}; \draw[bkernels1] (-1,1) node[xi1] {} -- (0,2.5) node[xi1] {}; \draw[kernels1] (0,0) node[not] {} -- (-1,1) node[xi1] {}; \draw[kernels1] (0,0) node[not] {} -- (1,1.5) node[xi] {};
}
\DeclareSymbol{J2[HXi1]t1}{-1}{\draw[kernels0] (-2,2.5) node[xi] {} -- (-1,1) node[not] {}; \draw[bnoise] (-1,1) node[not] {} -- (0,2.5) node[not1] {}; \draw[bkernels2] (0,0) node[not] {} -- (-1,1) node[not] {}; \draw[bkernels1] (0,0) node[not] {} -- (1,1.5) node[xi1] {};
}
\DeclareSymbol{5Xi1}{-1}{\draw[bkernels1] (-1.5,0.9) node[xi1] {} -- (0,0) node[not] {}  -- (-0.8,1.7) node[xi1] {}; \draw[bkernels1] (0.8,1.7) node[xi1] {} -- (0,0) node[not] {};
	\draw[bkernels1] (0,0) node[not] {} -- (1.5,0.9) node[xi1] {};
	\draw[bkernels2] (0,0) node[not] {} -- (0,2.5) node[xi1] {};
} 
\DeclareSymbol{6Xi1}{1.5}{\draw[bkernels1] (-1.7,1) node[xi1] {} -- (0,0) node[not] {}  -- (1.7,1) node[xi1] {}; \draw[bkernels1] (-1.1,1.6) node[xi1] {} -- (0,0) node[not] {}  -- (1.1,1.6) node[xi1] {};
	\draw[bkernels1] (0,0) node[not] {} -- (-0.5,2.2) node[xi1] {};
	\draw[bkernels2] (0,0) node[not] {} -- (0.5,2.2) node[xi1] {};
} 
\DeclareSymbol{LHXiH'Xi}{-1}{\draw[kernels0] (-2,2.5) node[xi] {} -- (-1,1) node[not] {}; \draw[kernels1] (-1,1) node[not] {} -- (0,0) node[not] {} -- (1,1.5) node[xi] {}; \draw[noise] (-1,1) node[not] {} -- (0,2.5) node[not0] {}; \draw[kernels0] (0,0) node[not] {} -- (1,-1) node[not] {};
	\draw[noise] (1,-1) node[not] {} -- (2,0.5) node[not0] {};
}
\DeclareSymbol{t2[Xi2]t1}{-1}{\draw[bkernels2] (-1,2) node[xi2] {} -- (0,0) node[not] {}; 
	\draw[bkernels1] (0,0) node[not] {} -- (1,2) node[xi1] {};
}
\DeclareSymbol{H'J1[H']}{-1}{\draw[kernels1] (-1,1.5) node[xi] {} -- (0,0) node[not] {}; 
	\draw[bkernels1] (0,0) node[not] {} -- (1,1) node[not] {};
	\draw[kernels1] (1,1) node[not] {} -- (1,2.5) node[xi] {};
}
\DeclareSymbol{LH'H'H'Xi}{-1}{\draw[kernels1] (-2,2.5) node[xi] {} -- (-1,1) node[not] {}  -- (0,2.5) node[xi] {}; \draw[kernels1] (-1,1) node[not] {} -- (0,0) node[not] {}  -- (1,1.5) node[xi] {}; \draw[kernels0] (0,0) node[not] {} -- (1,-1) node[not] {}; \draw[noise] (1,-1) node[not] {} -- (2,0.5) node[not0] {};
}
\DeclareSymbol{J[HXi]XiH}{1}{\draw[kernels0] (-1.3,1.1) node[not] {} -- (0,0) node[not] {}  -- (1.2,1.4) node[xi] {}; \draw[noise] (0,0) node[not] {} -- (0,2) node[not0] {};\draw[noise] (-1.3,1.1) node[not] {} -- (-0.6,2.6) node[not0] {}; \draw[kernels0] (-1.3,1.1) node[not] {} -- (-2,2.6) node[xi] {};
}
\DeclareSymbol{3HXi}{1}{\draw[kernels0] (-1.5,1.5) node[xi] {} -- (0,0) node[not] {}  -- (-0.6,2) node[xi] {}; \draw[kernels0] (0,0) node[not] {} -- (0.6,2) node[xi] {};
	\draw[noise] (1.5,1.5) node[not0] {} -- (0,0) node[not] {};
} 
\DeclareSymbol{2bt1Xi1H}{1}{\draw[bkernels1] (-1.5,1.5) node[xi1] {} -- (0,0) node[not] {}  -- (-0.6,2) node[xi1] {}; \draw[bnoise] (0,0) node[not] {} -- (0.6,2) node[not1] {};
	\draw[kernels0] (1.5,1.5) node[xi] {} -- (0,0) node[not] {};
} 
\DeclareSymbol{bt2t1}{1}{\draw[bbkernels2] (-1,2) node[xi2] {} -- (0,0) node[not] {}; 
	\draw[bkernels1] (0,0) node[not] {} -- (1,2) node[xi1] {};
}
\DeclareSymbol{XHXi}{-1}{\draw[kernels0] (-1,2) node[xi] {} -- (0,0) node[not] {}; \draw[noise] (0,0) node[not] {}  -- (1,2) node[not0] {}; \node(1) at (1, 0) {\tiny 1};
}
\DeclareSymbol{XHXia}{-1}{\draw[kernels0] (-1,2) node[xi] {} -- (0,0) node[not] {}; \draw[noise] (0,0) node[not] {}  -- (1,2) node[not0] {}; \node(1) at (1.8, 2) {\tiny 1};
}
\DeclareSymbol{t2t1[Xi2]}{-1}{\draw[bkernels2] (-1,2) node[xi1] {} -- (0,0) node[not] {}; 
	\draw[bkernels1] (0,0) node[not] {} -- (1,2) node[xi2] {};
}
\DeclareSymbol{LKPZ}{-1}{\draw[kernels1] (-2,2.5) node[xi] {} -- (-1,1) node[not] {}  -- (0,0) node[not] {} -- (1,-1) node[not] {}; \draw[kernels1] (-1,1) node[not] {} -- (0,2.5) node[xi] {}; \draw[kernels1] (0,0) node[not] {} -- (1,1.5) node[xi] {};
	\draw[kernels1] (1,-1) node[not] {} -- (2,0.5) node[xi] {};
}
\DeclareSymbol{LHXiXiXi}{-1}{\draw[kernels0] (-2,2.5) node[xi] {} -- (-1,1) node[not] {}  -- (0,0) node[not] {} -- (1,-1) node[not] {}; \draw[noise] (-1,1) node[not] {} -- (0,2.5) node[not0] {}; \draw[noise] (0,0) node[not] {} -- (1,1.5) node[not0] {};
	\draw[noise] (1,-1) node[not] {} -- (2,0.5) node[not0] {};
}
\DeclareSymbol{H'J1[H']+}{1.5}{\draw[kernels1] (-1,1.5) node[xi] {} -- (0,0) node[not] {}; 
	\draw[kernels1] (0,0) node[not] {} -- (1,1) node[not] {};
	\draw[kernels1] (1,1) node[not] {} -- (1,2.5) node[xi] {};
}
\DeclareSymbol{H'J1[Xi1bt1]}{1.5}{\draw[kernels1] (-1,1.5) node[xi] {} -- (0,0) node[not] {}; 
	\draw[kernels1] (0,0) node[not] {} -- (1,1) node[not] {};
	\draw[bkernels1] (1,1) node[xi1] {} -- (1,2.5) node[xi1] {};
}
\DeclareSymbol{t2J1[bt2]}{-1}{\draw[bkernels2] (-1,1.5) node[xi1] {} -- (0,0) node[not] {}; 
	\draw[bkernels1] (0,0) node[not] {} -- (1,1) node[not] {};
	\draw[bbkernels2] (1,1) node[not] {} -- (1,2.5) node[xi2] {};
}
\DeclareSymbol{Ch2}{1}{\draw[kernels1] (0,0) node[not] {} -- (-1,1) node[not] {};
	\draw[kernels1] (0,0) node[not] {} -- (1,1) node[not] {};
	\draw[kernels1] (-1,1) node[not] {} -- (-1.5,2.5) node[xi] {};
	\draw[kernels1] (-1,1) node[not] {}  -- (-0.5,2.5) node[xi] {};
	\draw[kernels1] (1,1) node[not] {} -- (0.5,2.5) node[xi] {};
	\draw[kernels1] (1,1) node[not] {}  -- (1.5,2.5) node[xi] {}; 
}
\DeclareSymbol{Ch+1}{1}{\draw[kernels1] (-1,2) node[xi] {} -- (0,0) node[not] {}  -- (1,2) node[xi] {};
	\node(1) at (1, 0) {\tiny 1};
} 
\DeclareSymbol{Ch+1a}{1}{\draw[kernels1] (-1,2) node[xi] {} -- (0,0) node[not] {}  -- (1,2) node[xi] {};
	\node(1) at (1.8, 2) {\tiny 1};
} 
\DeclareSymbol{bt1Xi1bt1}{1}{\draw[bkernels1] (-1,2) node[xi1] {} -- (0,0) node[xi1] {}  -- (1,2) node[xi1] {};
} 
\DeclareSymbol{J1[t1t2t1]H'}{1}{\draw[bkernels1] (-2.1,2.3) node[xi1] {} -- (-1,1) node[not] {}; \draw[bkernels1] (-1,1) node[not] {} -- (0.1,2.3) node[xi1] {}; \draw[bkernels2] (-1,2.8) node[xi1] {} -- (-1,1) node[not] {};
	\draw[kernels1] (0,0) node[not] {} -- (-1,1) node[not] {}; \draw[kernels1] (0,0) node[not] {} -- (1.2,1.3) node[xi] {};
}
\DeclareSymbol{t2t1t1J1[H']}{-1}{\draw[bkernels1] (-0.6,2) node[xi1] {} -- (0,0) node[not] {}  -- (0.6,2) node[xi1] {}; \draw[bkernels2] (-1.5,1.5) node[xi1] {} -- (0,0) node[not] {};
	\draw[bkernels1] (0,0) node[not] {} -- (1.5,1) node[not] {};
	\draw[kernels1] (1.5,1) node[not] {} -- (1.5,2.6) node[xi] {}; 
}
\DeclareSymbol{LH'H'XiXi}{-1}{\draw[kernels1] (-2,2.5) node[xi] {} -- (-1,1) node[not] {}  -- (0,2.5) node[xi] {}; \draw[kernels0] (-1,1) node[not] {} -- (0,0) node[not] {}; \draw[kernels0] (0,0) node[not] {} -- (1,-1) node[not] {}; \draw[noise] (0,0) node[not] {} -- (1,1.5) node[not0] {};
	\draw[noise] (1,-1) node[not] {} -- (2,0.5) node[not0] {};
}
\DeclareSymbol{J[Ch]XiH}{-1}{\draw[kernels0] (-1.3,1.1) node[not] {} -- (0,0) node[not] {}  -- (1.2,1.4) node[xi] {}; \draw[noise] (0,0) node[not] {} -- (0,2) node[not0] {};\draw[kernels1] (-0.6,2.6) node[xi] {} -- (-1.3,1.1) node[not] {} -- (-2,2.6) node[xi] {};
}
\DeclareSymbol{ChHXi}{-1}{\draw[kernels1] (0,0) node[not] {} -- (-1,1) node[not] {};
	\draw[kernels1] (0,0) node[not] {} -- (1,1) node[not] {};
	\draw[kernels1] (-1,1) node[not] {} -- (-1.5,2.5) node[xi] {};
	\draw[kernels1] (-1,1) node[not] {}  -- (-0.5,2.5) node[xi] {};
	\draw[kernels0] (1,1) node[not] {} -- (0.5,2.5) node[xi] {};
	\draw[noise] (1,1) node[not] {}  -- (1.5,2.5) node[not0] {}; 
}
\DeclareSymbol{t2[H']t1[H']}{-1}{\draw[kernels1] (-1,1) node[not] {} -- (-1,2.5) node[xi] {};
	\draw[bkernels1] (-1,1) node[not] {} -- (0,0) node[not] {}; 
	\draw[bkernels2] (0,0) node[not] {} -- (1,1) node[not] {};
	\draw[kernels1] (1,1) node[not] {} -- (1,2.5) node[xi] {};
}
\DeclareSymbol{J[HHXi]Xi}{-1}{\draw[kernels0] (-2.1,2.3) node[xi] {} -- (-1,1) node[not] {}; \draw[kernels0] (-1,1) node[not] {} -- (0.1,2.3) node[xi] {}; \draw[noise] (-1,2.8) node[not0] {} -- (-1,1) node[not] {};
	\draw[kernels0] (0,0) node[not] {} -- (-1,1) node[not] {}; \draw[noise] (0,0) node[not] {} -- (1.2,1.3) node[not0] {};
}
\DeclareSymbol{J1[H't1]t2t1}{1}{\draw[bkernels1] (-1.3,1.1) node[not] {} -- (0,0) node[not] {}  -- (1.2,1.4) node[xi1] {}; \draw[bkernels2] (0,0) node[not] {} -- (0,1.8) node[xi1] {};\draw[bkernels1] (-1.3,1.1) node[not] {} -- (-0.7,2.6) node[xi1] {}; \draw[kernels1] (-1.3,1.1) node[not] {} -- (-1.9,2.6) node[xi] {};
}
\DeclareSymbol{XiJ[H']+}{1.5}{\draw[noise] (-1,1.5) node[not0] {} -- (0,0) node[not] {}; 
	\draw[kernels0] (0,0) node[not] {} -- (1,1) node[not] {};
	\draw[kernels1] (1,1) node[not] {} -- (1,2.5) node[xi] {};
}
\DeclareSymbol{H't1^3}{-1}{\draw[bkernels1] (-1.5,1.5) node[xi1] {} -- (0,0) node[not] {}  -- (-0.6,2) node[xi1] {}; \draw[bkernels1] (0.6,2) node[xi1] {} -- (0,0) node[not] {};
	\draw[kernels1] (0,0) node[not] {} -- (1.5,1.5) node[xi] {};
} 
\DeclareSymbol{J1[HXi]t2t1}{-1}{\draw[bkernels1] (-1.3,1.1) node[not] {} -- (0,0) node[not] {}  -- (1.2,1.4) node[xi1] {}; \draw[bkernels2] (0,0) node[not] {} -- (0,1.8) node[xi1] {};\draw[noise] (-1.3,1.1) node[not] {} -- (-0.7,2.6) node[not0] {}; \draw[kernels0] (-1.3,1.1) node[not] {} -- (-1.9,2.6) node[xi] {};
}
\DeclareSymbol{J[t2[H']]Xi}{1}{\draw[noise] (-1,1.1) node[not0] {} -- (0,0) node[not] {}; 
	\draw[kernels0] (0,0) node[not] {} -- (0.8,0.8) node[not] {};
	\draw[bkernels2] (0.8,0.8) node[not] {} -- (0.8,1.8) node[not] {};
	\draw[kernels1] (0.8,1.8) node[not] {} -- (0,2.8) node[xi] {};
}
\DeclareSymbol{LHXiXiH'}{-1}{\draw[kernels0] (-2,2.5) node[xi] {} -- (-1,1) node[not] {}  -- (0,0) node[not] {}; \draw[kernels1] (0,0) node[not] {} -- (1,-1) node[not] {} -- (2,0.5) node[xi] {}; \draw[noise] (-1,1) node[not] {} -- (0,2.5) node[not0] {}; \draw[noise] (0,0) node[not] {} -- (1,1.5) node[not0] {};
}
\DeclareSymbol{J1[HXi]J1[HXi]}{-1}{\draw[kernels1] (0,0) node[not] {} -- (-1,1) node[not] {};
	\draw[kernels1] (0,0) node[not] {} -- (1,1) node[not] {};
	\draw[kernels0] (-1,1) node[not] {} -- (-1.5,2.5) node[xi] {};
	\draw[noise] (-1,1) node[not] {}  -- (-0.5,2.5) node[not0] {};
	\draw[kernels0] (1,1) node[not] {} -- (0.5,2.5) node[xi] {};
	\draw[noise] (1,1) node[not] {}  -- (1.5,2.5) node[not0] {}; 
}
\DeclareSymbol{J1[t1t2t1]Xi}{-1}{\draw[bkernels1] (-2.1,2.3) node[xi1] {} -- (-1,1) node[not] {}; \draw[bkernels1] (-1,1) node[not] {} -- (0.1,2.3) node[xi1] {}; \draw[bkernels2] (-1,2.8) node[xi1] {} -- (-1,1) node[not] {};
	\draw[kernels1] (0,0) node[not] {} -- (-1,1) node[not] {}; \draw[noise] (0,0) node[not] {} -- (1.2,1.3) node[not0] {};
}
\DeclareSymbol{LH'H'XiH'}{-1}{\draw[kernels1] (-2,2.5) node[xi] {} -- (-1,1) node[not] {}  -- (0,2.5) node[xi] {}; \draw[kernels0] (-1,1) node[not] {} -- (0,0) node[not] {}; \draw[kernels1] (0,0) node[not] {} -- (1,-1) node[not] {}; \draw[noise] (0,0) node[not] {} -- (1,1.5) node[not0] {};
	\draw[kernels1] (1,-1) node[not] {} -- (2,0.5) node[xi] {};
}
\DeclareSymbol{t2[H't1]t1t1}{-1}{\draw[bkernels1] (0,1.8) node[xi1] {} -- (0,0) node[not] {}  -- (1.2,1.4) node[xi1] {}; \draw[bkernels2] (0,0) node[not] {} -- (-1.3,1.1) node[not] {};\draw[bkernels1] (-1.3,1.1) node[not] {} -- (-0.7,2.6) node[xi1] {}; \draw[kernels1] (-1.3,1.1) node[not] {} -- (-1.9,2.6) node[xi] {};
}
\DeclareSymbol{J1[t2[H']]H'}{1}{\draw[kernels1] (-0.8,0.8) node[xi] {} -- (0,0) node[not] {}; 
	\draw[kernels1] (0,0) node[not] {} -- (0.8,0.8) node[not] {};
	\draw[bkernels2] (0.8,0.8) node[not] {} -- (0.8,1.8) node[not] {};
	\draw[kernels1] (0.8,1.8) node[not] {} -- (0,2.8) node[xi] {};
}
\DeclareSymbol{t2[HXi1]t1t1}{-1}{\draw[bkernels1] (0,1.8) node[xi1] {} -- (0,0) node[not] {}  -- (1.2,1.4) node[xi1] {}; \draw[bkernels2] (0,0) node[not] {} -- (-1.3,1.1) node[not] {};\draw[bnoise] (-1.3,1.1) node[not] {} -- (-0.7,2.6) node[not1] {}; \draw[kernels0] (-1.3,1.1) node[not] {} -- (-1.9,2.6) node[xi] {};
}
\DeclareSymbol{J1[HXiH]H'}{-1}{\draw[kernels0] (-2.1,2.3) node[xi] {} -- (-1,1) node[not] {}; \draw[kernels0] (-1,1) node[not] {} -- (0.1,2.3) node[xi] {}; \draw[noise] (-1,2.8) node[not0] {} -- (-1,1) node[not] {};
	\draw[kernels1] (0,0) node[not] {} -- (-1,1) node[not] {}; \draw[kernels1] (0,0) node[not] {} -- (1.2,1.3) node[xi] {};
}
\DeclareSymbol{LHXiH'H'}{-1}{\draw[kernels0] (-2,2.5) node[xi] {} -- (-1,1) node[not] {}; \draw[kernels1] (-1,1) node[not] {}  -- (0,0) node[not] {}; \draw[kernels1] (0,0) node[not] {} -- (1,-1) node[not] {} -- (2,0.5) node[xi] {}; \draw[noise] (-1,1) node[not] {} -- (0,2.5) node[not0] {}; \draw[kernels1] (0,0) node[not] {} -- (1,1.5) node[xi] {};
}
\DeclareSymbol{t1Xi2}{1}{\draw[bkernels1] (0,2) node[xi1] {} -- (0,0) node[xi2] {};
}
\DeclareSymbol{t1Xi1}{1}{\draw[bkernels1] (0,2) node[xi1] {} -- (0,0) node[xi1] {};
}
\DeclareSymbol{t1t2[bt2]}{-1}{\draw[bkernels1] (-1,1.5) node[xi1] {} -- (0,0) node[not] {}; 
	\draw[bkernels2] (0,0) node[not] {} -- (1,1) node[not] {};
	\draw[bbkernels2] (1,1) node[not] {} -- (1,2.5) node[xi2] {};
}
\DeclareSymbol{t1t1t1J2[H']}{-1}{\draw[bkernels1] (-0.6,2) node[xi1] {} -- (0,0) node[not] {}  -- (0.6,2) node[xi1] {}; \draw[bkernels1] (-1.5,1.5) node[xi1] {} -- (0,0) node[not] {};
	\draw[bkernels2] (0,0) node[not] {} -- (1.5,1) node[not] {};
	\draw[kernels1] (1.5,1) node[not] {} -- (1.5,2.6) node[xi] {};
}
\DeclareSymbol{XiX}{1}{\draw[noise] (0,0) node[not] {} -- (0,2) node[not0] {}; \node(1) at (0.8, 2) {\tiny 1};
}
\DeclareSymbol{bXiX}{1}{\draw[bnoise] (0,0) node[not] {} -- (0,2) node[not1] {}; \node(1) at (0.8, 2) {\tiny 1};
}
\DeclareSymbol{bbXiX}{1}{\draw[bbnoise] (0,0) node[not] {} -- (0,2) node[not2] {}; \node(1) at (0.8, 2) {\tiny 1};
}
\DeclareSymbol{H'J1[XiX]}{1.5}{\draw[kernels1] (-1,1.5) node[xi] {} -- (0,0) node[not] {} -- (1,1) node[not] {};
	\draw[noise] (1,1) node[not] {} -- (1,2.5) node[not0] {}; \node(1) at (1.7, 2.5) {\tiny 1};
}
\DeclareSymbol{XiJ[XiX]}{1}{\draw[noise] (-1,1.5) node[not0] {} -- (0,0) node[not] {}; 
	\draw[kernels0] (0,0) node[not] {} -- (1,1) node[not] {};
	\draw[noise] (1,1) node[not] {} -- (1,2.5) node[not0] {}; \node(1) at (1.7, 2.5) {\tiny 1};
}
\DeclareSymbol{XiJ[Xi1bt1]}{1}{\draw[noise] (-1,1.5) node[not0] {} -- (0,0) node[not] {}; 
	\draw[kernels0] (0,0) node[not] {} -- (1,1) node[not] {};
	\draw[bkernels1] (1,1) node[xi1] {} -- (1,2.5) node[xi1] {};
}

\let\f\frac
\def\downop{\mathop{\downarrow}}

\def\dash{\leavevmode\unskip\kern0.18em--\penalty\exhyphenpenalty\kern0.18em}
\def\slash{\leavevmode\unskip\kern0.15em/\penalty\exhyphenpenalty\kern0.15em}

\begin{document}

\date{\today}
\title{Directed mean curvature flow in noisy environment}
\author{A.~Gerasimovi\v cs$^1$, M.~Hairer$^2$ and K.~Matetski$^3$}
\institute{University of Bath, \email{ag2616@bath.ac.uk} \and EPFL and Imperial College London, \email{martin.hairer@epfl.ch} \and Michigan State University, \email{matetski@msu.edu}}
\date{\today}
\titleindent=0.65cm

\maketitle

\begin{abstract}
We consider the directed mean curvature flow on the plane in a weak  Gaussian random 
environment. We prove that, when started from a sufficiently flat initial condition, a 
rescaled and 
recentred solution converges to the Cole--Hopf solution of the KPZ equation. This result follows from the
analysis of a more general system of nonlinear SPDEs driven by inhomogeneous noises, using the 
theory of regularity structures. However, due to inhomogeneity of the noise, the ``black box'' 
result developed in the series of works~\cite{Regularity, BHZ, HC, BCCH} cannot be applied directly and 
requires significant extension to infinite-dimensional regularity structures. 

Analysis of this general system of SPDEs gives two more interesting results. First, we prove that the 
solution of the quenched KPZ equation with a very strong force also 
converges to the Cole--Hopf solution 
of the KPZ equation. Second, we show that a properly rescaled and renormalised quenched 
Edwards--Wilkinson model in any dimension converges to the stochastic heat equation. \\[.4em]
\end{abstract}

\setcounter{tocdepth}{1}
\tableofcontents

\section{Introduction}
\label{sec:intro}

Consider the time evolution of a curve $(t,x) \mapsto U_t(x) \in \R^2$ with instantaneous normal 
velocity at any $u \in \Image U_t$ given by $v(u) = \kappa(u) + 1 + \sqrt{\eps}\eta(u)$, where 
$\kappa(u)$ is the curvature of $U_t$ at $u$ and $u\mapsto\eta(u)$ is a stationary
Gaussian random field with finite dependence length. The parameter $\eps > 0$ is considered to be 
small and will be taken tending to zero. If the initial state $U_0$ is the graph of some function $u_0$, i.e.\ 
$U_0(x) = (x,u_0(x))$, then one can expect that for sufficiently regular $\eta$ the curve $U_t$ remains 
the graph of some function $u_t$, at least for a short time. In this case the curvature at the point 
$(x,u_t(x))$ is given by~\cite[p.~14]{Ecker}
\begin{equ}
\kappa (x,u_t(x)) = \d_x \biggl( \frac{\d_x u_t(x)}{\sqrt{1 + (\d_x u_t(x))^2}} \biggr) = \frac{\d^2_x u_t(x)}{(1 + (\d_x u_t(x))^2)^{3/2}}\;.
\end{equ}
A calculation analogous to that given for example in~\cite[Eq.~2.17]{Ecker} then yields for $u_t$ the 
random PDE
\begin{equ}
\d_t u_t(x) = v (x,u_t(x)) \sqrt{1 + (\d_x u_t(x))^2}\;,
\end{equ}
which can be written explicitly as
\begin{equ}[eq:main_before_scaling]
\d_t u = {\d_x^2 u \over {1+(\d_x u)^2}}  + \sqrt{1+(\d_x u)^2}\bigl(1 + \sqrt\eps\, \eta^u\bigr)\;,
\end{equ}
where $\eta^u(t,x)$ is now a $u$-dependent driving noise given by $\eta^u(t,x) \eqdef \eta(x,u(t,x))$. The 
time interval over which $U_t$ remains a graph then agrees with the existence time of 
\eqref{eq:main_before_scaling}. In order to avoid problems with solving~\eqref{eq:main_before_scaling} 
on the full space we will assume that both the initial state $u(0,x)$, and the random field $\eta$ are 
periodic in space with period $\eps^{-1}$. To indicate this $\eps$-dependence, we will write $u_\eps^0$ 
for the initial state and $\eta_\eps$ for the random field. 

The noisy environment $\eta_\eps$ is taken to be a mollified space-time white noise. More precisely, let 
$\rho$ be a compactly supported smooth function on $\R^2$ that integrates to $1$, and let 
$\zeta_\eps$ be a space-time white noise on $\R \times (\R / \eps^{-1} \Z)$, which we extend periodically 
to $\R^2$. Then the noisy environment $\eta_\eps$ is given by space-time convolution $\eta_\eps(x, y) = (\rho * \zeta_\eps)(y, x)$. In particular, $\eta_\eps$ is a stationary, centred Gaussian random field.

If one formally takes $\eps \to 0$ then $u$ converges to the deterministic directed mean curvature flow. We aim 
to prove that under a suitable rescaling and in the adequate moving frame, we observe convergence to 
the Cole--Hopf solution of the KPZ equation~\cite{BertiniGiacomin} on the unit circle $\T \eqdef \R / \Z$
\begin{equ}[eq:KPZ]
\d_t h = \d_x^2 h + \lambda\, (\d_x h)^2+ \sigma \xi\;, \qquad h(0, \bigcdot) = h^0(\bigcdot)\;,
\end{equ}
for a space-time white noise $\xi$ on $\R \times \T$ and some $\lambda,\sigma \in \R$. We shall see in Section~\ref{sec:abstract_system} that due to renormalisation one actually obtains a 
convergence to the KPZ equation in a moving frame. We define the speed of the moving frame to be
\begin{equ}[eq:speed]
	\nu = 2 \lambda \sigma^2 \int \int\int \partial_t\rho(z_1)\rho(z_2) (x_1 - x_2 - x_3) G(z_3) \partial^2_x G(z_1-z_2 - z_3)dz_1 dz_2 dz_3\;,
\end{equ}
where $G(t,x) = \1_{t>0} e^{-x^2 / (4 t)} / \sqrt{4\pi t}$\label{label:heat-kernel} is the heat kernel on $\R$, $x_i$ is the spatial component of $z_i \in \R^2$, and $\lambda$ is as in~\eqref{eq:KPZ}. In particular, if $\rho$ is an even function in the spatial variable, 
then $\nu = 0$. The main result of this paper is the following theorem, which is a particular case of a 
more general result provided in Theorem~\ref{thm:main}.

\begin{theorem}\label{thm:main_intro}
Let the initial state $u^{0}_\eps : \eps^{-1} \T \to \R$ of~\eqref{eq:main_before_scaling} be a $\CC^{7/3 + \kappa}$ function for some $\kappa > 0$, such that $\eps^{\frac{2k}{3}} \|u^0_\eps(\eps^{-1} \bigcdot)\|_{\CC^{(1+2k) / 3 + \kappa}}$ is bounded 
uniformly in $\eps \in (0,1]$ for $k \in \{0,1,2,3\}$. Moreover, assume that the functions $x \mapsto u^0_\eps(\eps^{-1} x)$ 
converge as $\eps \to 0$ to a function $h^0 \in \CC^{\beta}(\T)$ for some $\beta > \frac{1}{3}$. Let 
$\nu$ be as in~\eqref{eq:speed} with $\lambda = \frac 12, \sigma = 1$. Then there is a choice of divergent 
constants $C_\eps \sim \eps^{-1}$ such that, as $\eps \to 0$, the functions 
\begin{equ}
	u \bigl(\eps^{-2} t , \eps^{-1}(x + \nu t)\bigr) - \eps^{-2} t - C_\eps t
\end{equ}
converge in probability in the locally uniform topology on $\R_+\times \T$ to the Cole--Hopf solution of 
the KPZ equation~\eqref{eq:KPZ} with $\lambda = \frac{1}{2}$, $\sigma = 1$ and initial condition $h^0$.
\end{theorem}

The reason why even mollifiers do not require a moving frame ($\nu=0$) is a 
consequence of the $x \mapsto -x$ symmetry of the equation. 

We could in principle also consider non-Gaussian random fields $\eta_\eps$ driving
\eqref{eq:main_before_scaling} which is not expect to cause any major difficulties.
It would however make it more tedious to formulate assumptions guaranteeing that 
the renormalisation constants are bounded by the ``correct'' powers of $\eps$.
Another natural generalisation would be to allow for Neumann boundary conditions instead of 
periodic, but this would lead to substantial additional difficulties due to the presence
of boundary renormalisation \cite{MateBoundary,IvanHao}.

The high regularity $\CC^{7/3 + \kappa}$ of the initial conditions as well as the uniform bound on $\eps^{\frac{2k}{3}} 
\|u^0_\eps(\eps^{-1} \bigcdot)\|_{\CC^{(1+2k) / 3 + \kappa}}$ are dictated by our method of 
proof (see~\eqref{eq:main2} and Lemma~\ref{prop:system_soln}). Nonetheless, it is not at all 
unrealistic since a smooth enough mollification of a $\frac{1}{3}$-H\"older continuous function does 
satisfy the required uniform in $\eps$ bound. At this stage it is unclear though whether the 
 exponent $\f73$ could be improved substantially.

To derive an SPDE for the above rescaling of $u$, we set
\begin{equ}[eq:rescaled_solution]
	h_\eps(t,x) \eqdef u \bigl(\eps^{-2} t , \eps^{-1}x\bigr) - \eps^{-2} t - C_\eps t
\end{equ}
as well as
\begin{equs}[eq:noise]
	\xi_\eps(t,x) \eqdef \eps^{-\frac{3}{2}} \eta_\eps \bigl(\eps^{-1} x, \eps^{-2} t\bigr)\;,
\end{equs}
which is now $1$-periodic in the spatial variable. Then $h_\eps$ solves
\begin{equs}[eq:rescaledSPDE]
\d_t h_\eps = {\d_x^2 h_\eps \over {1+ \eps^2 (\d_x h_\eps)^2}} + \eps^{-2}\bigl(\sqrt{1+ \eps^2 (\d_x h_\eps)^2}-1\bigr) - C_\eps \\
+ \sqrt{1+ \eps^2 (\d_x h_\eps)^2}\xi^{h_\eps}_\eps\,,
\end{equs}
with the initial state $h_\eps(0, x) = u^0_\eps(\eps^{-1}x)$, where we set the inhomogeneous noise
\begin{equ}[eq:rescaled_noise]
\xi^{h_\eps}_\eps(t,x) \eqdef \xi_\eps \bigl(t + \eps^2 C_\eps t + \eps^2 h_{\eps}(t,x), x\bigr)\;.
\end{equ}
We refrain from incorporating the additional drift $\nu t$ into the definition of the
rescaling~\eqref{eq:rescaled_solution} at this point 
because it will just add an additional term $\nu\partial_x h$ into equation~\eqref{eq:rescaledSPDE}. We will 
simply show the equivalent statement that $h_\eps$ converges to $h$ such that $(t,x) 
\mapsto h(t, x-\nu t)$ solves the KPZ equation. We refer the reader to~\cite{CLTforKPZ} 
where the authors 
keep such a translation inside the rescaled equation. This leads to minor technical 
complications with the choice of renormalisation which our setting allows to bypass. 

Equation~\eqref{eq:rescaledSPDE} is a special case of the class of equations
\begin{equ}[eq:main]
\d_t h_\eps = \d_x^2 h_\eps + F_1(\eps\d_x h_\eps) \d_x^2 h_\eps + F_2(\eps \d_x h_\eps ) (\d_x h_\eps)^2 - C_\eps +   F_3(\eps\d_x h_\eps)\xi^{h_\eps}_\eps,
\end{equ}
where the functions $F_i$ have the following properties.

\begin{assumption}\label{assum:functions}
	The functions $F_i : \R \to \R$ are of class $\CC^7$, with derivatives growing subexponentially at infinity, and such that
	\begin{equ}[eq:Functions]
	F_1(0) = F_1'(0) = 0\;, \qquad F_2'(0) = 0\;, \qquad F_3'(0) = F_3'''(0) = 0\;.
\end{equ}
	Moreover, $F''_1, F'''_1 \in L^\infty(\R)$.
\end{assumption}

\noindent Indeed, if we make the particular choice of the functions
\begin{equ}[eq:functions_main]
F_1(x) = - \frac{x^2}{1 + x^2}\;, \quad F_2(x) = \frac{1}{x^2} \left( \sqrt{1 + x^2} - 1\right)\;, \quad F_3(x) = \sqrt{1+x^2}\;,
\end{equ}
we obtain~\eqref{eq:rescaledSPDE}. Here, we extend the function $F_2$ continuously to $x = 0$.

On the other hand, in the case $F_1 \equiv 0$, $F_2 \equiv \lambda$ and $F_3 \equiv \sigma$, 
equation~\eqref{eq:main} is
\begin{equ}[eq:qKPZ]
\d_t h_\eps = \d_x^2 h_\eps + \lambda (\d_x h_\eps)^2 + \sigma \xi^{h_\eps}_\eps - C_\eps\;, \qquad h_\eps(0, \bigcdot) = h^0_\eps(\bigcdot)\;.
\end{equ}
This equation is obtained by rescaling \eqref{eq:rescaled_solution} from the \emph{quenched KPZ (qKPZ) equation} in the intermediate disorder regime
\begin{equ}
\d_t u = \d_x^2 u + \lambda (\d_x u)^2  + 1 + \sigma \sqrt\eps\, \eta^u\;,
\end{equ}
in which the driving noise is as in \eqref{eq:main_before_scaling}. Together with Theorem~\ref{thm:main_intro} we prove that for a suitable choice of $C_\eps$ the qKPZ equation converges to the standard KPZ equation. Moreover, one expects \cite{Takeuchi} that under the 1:2:3 scaling the solution of the qKPZ equation converges to the KPZ fixed point constructed in \cite{fixedpt} (see also \cite{sarkar2020fixedpoint, virag2020fixedpoint} for the recent proof of convergence of the KPZ equation to the KPZ fixed point). A proof of this much harder conjecture is out of scope of the presented result.

\begin{theorem}\label{thm:qKPZ_intro}
Let the initial states $h_\eps^0 : \T \to \R$ of~\eqref{eq:qKPZ} converge as $\eps \to 0$ in $\CC^{\beta}$ 
to a function $h^0$ for some $\beta > \frac 13$.\footnote{The reason why one can take a less regular 
initial condition than in Theorem~\ref{thm:main_intro} is explained on page~\pageref{cor:qKPZ}.} Then there is a choice of divergent constants $C_\eps \sim \eps^{-1}$ such 
that, if $h_\eps$ solves the qKPZ equation~\eqref{eq:qKPZ} with initial condition $h_\eps^0$, then as $\eps 
\to 0$ the functions $\tilde h_\eps (t,x) \eqdef h_\eps (t,x+ \nu t)$ converge in probability in the locally 
uniform topology on $\R_+\times \T$ to the Cole--Hopf solution of the KPZ equation~\eqref{eq:KPZ} 
with the initial state $h^0$, where $\nu$ is as in~\eqref{eq:speed}.
\end{theorem}

One can see that equation~\eqref{eq:main} is not locally subcritical in the sense of~\cite{Regularity}, and 
the theory of regularity structures 
cannot be applied directly to the equation written in this form. More precisely, the space where the 
limiting driving noise $\xi = \lim_{\eps 
\to 0} \xi_\eps$ lives is $\CC^\alpha_\s$ for any $\alpha < -\frac{3}{2}$ (see Section~\ref{ss:notation} for the definition of the spaces). Then 
the Schauder estimate implies that the best we can get for $h_\eps$ is to converge in $\CC^{\alpha + 2}_\s$. If one Taylor expands $F_2$, then 
one can observe a $(\partial_x h_\eps)^4$ term appearing on the right-hand side of the equation, which 
a priory diverges in $\CC^\alpha_\s$ 
even after renormalisation.\footnote{Formal power counting suggests that $(\partial_x h_\eps)^4$ 
converges in $\CC^{4\alpha + 4}_\s$, since 
$\partial_x h_\eps$ converges in $\CC^{\alpha + 1}_\s$. Note that $4\alpha + 4 < \alpha$ for $\alpha < - \frac{3}{2}$.} The way around this 
problem is to observe that all such divergent terms are actually always multiplied by $\eps^\beta$ with a 
high enough power $\beta > 0$. More precisely,~\eqref{eq:Functions} yields the Taylor's expansion
\begin{equ}
	F_2(\eps\d_x h_\eps)\, (\d_x h_\eps)^2 = \sum_{n \geq 0} \frac{F_2^{(n)}(0)}{n!} (\eps\d_x h_\eps)^{n} (\d_x h_\eps)^2\;,
\end{equ}
and we can use the fact that $(\d_x h_\eps)^{n+2}$ is always multiplied by $\eps^n$. Formally, assuming 
that we can define the product $(\d_x 
h_\eps)^{n+2}$ uniformly in $\eps > 0$, it is expected to converge in $\CC^{(n+2)(\alpha + 1)}_\s$. If we 
now consider multiplication by $\eps^n$ as an 
``improvement'' of regularity by $n$, we expect convergence of $\eps^n(\d_x h_\eps)^{n+2}$ in $\CC^{(n+2)\alpha + 2n + 2}_\s$, which is 
a more regular space than $\CC^\alpha_\s$ as soon as $\alpha > -\frac{7}{4}$. One can make this argument 
rigorous, similarly to~\cite{KPZJeremy}, 
where in the framework of regularity structures multiplication by $\eps^n$ was 
implemented by an ``integration map'' $\CE^n$ at the level of 
the regularity structure. However, this approach does not quite fall into the scope 
of~\cite{HC,Rhys}, which prevents us from using the general framework of the BPHZ renormalisation.

Instead, we use the trick of rewriting~\eqref{eq:main} as a system of four equations. By considering an 
apparently more complicated 
problem, we make it locally subcritical and hence amenable to solving by using the framework of 
regularity structures without having to introduce the operator $\CE$ of~\cite{KPZJeremy}. 
More precisely, for integers $i \geq 0$ we define the functions $h_{i, \eps}$, the 
random noises $\xi_{i, \eps}$ and the renormalisation constants $C_{i, \eps}$ by
\begin{equ}[eq:new_functions]
	h_{i, \eps} \eqdef \eps^{\frac{2 i}{3}} h_\eps\;, \qquad \xi_{i, \eps} \eqdef \eps^{\frac{2 i}{3}} \xi_\eps\;, \qquad C_{i, \eps} \eqdef \eps^{\frac{2 i}{3}} C_\eps\;,
\end{equ}
so that $h_{0, \eps} = h_{\eps}$, $\xi_{0, \eps} = \xi_{\eps}$ and $C_{0, \eps} = C_{\eps}$. Moreover, we define the functions 
\begin{equs}[eq:F_eps]
F_{1,\eps}(u) &= \eps^{-\frac{2}{3}} F_{1} \bigl(\eps^{\frac{1}{3}} u\bigr)\;,\qquad F_{2,\eps}(u) = \eps^{- \frac{2}{3}} \big( F_{2}\bigl(\eps^{\frac{1}{3}} u\bigr) - F_{2}(0)\big) u\;,\quad \\
&F_{3,\eps}(u) = \eps^{-\frac{4}{3}} \Big(F_{3}\bigl(\eps^{\frac{1}{3}} u\bigr) - F_{3}(0) - {1\over 2}F_{3}''(0) \eps^{\frac{2}{3}} u^2\Big)\;.
\end{equs}
Then from~\eqref{eq:main}, for the constants $\lambda = F_2(0)$, $\sigma = F_3(0)$, $\sigma_1 = \frac{1}{2}F_3''(0)$ and for $i = 0, 1, 2$, we obtain the system of equations
\begin{equs}[eq:main2]
		\d_t h_{i, \eps} &= \d_x^2 h_{i, \eps} + F_{1,\eps}( \d_x h_{1, \eps})\, \d_x^2 h_{i+1, \eps} + \lambda\, \d_x h_{\lceil i/2\rceil, \eps} \, \d_x h_{\lfloor i/2\rfloor, \eps}\\
		 &\qquad + F_{2,\eps}(\d_x h_{1, \eps})\, \d_x h_{i, \eps}  - C_{i, \eps} + \sigma \xi^{h_{0, \eps}}_{i, \eps} \\
		 &\qquad + \sigma_1 (\d_x h_{1, \eps})^2\, \xi^{h_{0, \eps}}_{i+1, \eps} + F_{3,\eps}(\d_x h_{1, \eps})\, \xi^{h_{0, \eps}}_{i+2, \eps}\;, \\
		\d_t h_{3, \eps} &= \d_x^2 h_{3, \eps} + \eps^{\frac{2}{3}} F_{1,\eps}(\d_x h_{1, \eps})\, \d_x^2 h_{3, \eps} \\
		 &\qquad + F_{2}\bigl( \eps^{\frac{1}{3}} \d_x h_{1, \eps}\bigr) \, \d_x h_{1, \eps} \, \d_x h_{2, \eps} - C_{3, \eps} +  F_{3}\bigl(\eps^{\frac{1}{3}}\d_x h_{1, \eps}\bigr) \, \xi^{h_{0, \eps}}_{3, \eps}\;,
\end{equs}
with initial conditions $h^0_{i, \eps} = \eps^{\frac{2i}{3}} h^0_{\eps}$. The 
inhomogeneous noises are defined by analogy with~\eqref{eq:rescaled_noise} as
\begin{equ}[eq:non-hom_noises]
\xi^{h_{0, \eps}}_{i, \eps}(t,x) \eqdef \xi_{i, \eps} \bigl(t + \eps^2 C_\eps t + \eps^2 h_{0, \eps}(t,x), x\bigr)\;, \qquad i = 0, \ldots, 3\;.
\end{equ}
Now, the regularity (uniform in $\eps$) of the noise $\xi_{i, \eps}$ in~\eqref{eq:new_functions} is $\CC^{\alpha + 2i/ 3}_\s$, for any $\alpha < -\frac{3}{2}$ so that, by Schauder estimates, one expects the
functions $h_{i, \eps}$ appearing in~\eqref{eq:new_functions} to be in $\CC^{\alpha + 2 + 2 i / 3}_\s$. 

Given the just described expected regularities, it follows that $\d_x h_{1, \eps}$, $\d_x h_{2, \eps}$, $\d_x^2 h_{3, \eps}$ and $\xi^{h_{0, \eps}}_{3, \eps}$ are Hölder continuous functions (again, uniformly in $\eps$) and that all the products appearing in the equation for $h_{3, \eps}$ are classically defined. Hence, the equation for $h_{3, \eps}$ is a classically well-posed quasi-linear SPDE (uniformly in $\eps \le 1$) \cite[Ch.~V]{Ladyzenskaja}, while we use regularity structures to solve the other three equations. One can check that the system of the first three equations in~\eqref{eq:main2} is locally subcritical in the sense of~\cite[Ass.~8.3]{Regularity}. This is implied by the fact that the coefficients multiplying the noises in \eqref{eq:main2} are of positive regularity (uniformly in $\eps$), while the regularities of all other terms appearing on the right-hand side of the $i$-th equation in \eqref{eq:main2} are strictly larger than the regularity of the noise $\xi^{h_{0, \eps}}_{i, \eps}$. In this power-counting argument, one considers that the regularity of a product of terms of respective regularities $\gamma_1$ and $\gamma_2$ is given by $\gamma_1 \wedge \gamma_2 \wedge (\gamma_1+\gamma_2)$, even 
when $\gamma_1 + \gamma_2 \le 0$ so that the product isn't classically well-defined. We refer to \cite[Def.~5.14]{BHZ} for a more general definition of local subcriticality. 

\begin{remark}
There are of course many different ways of turning~\eqref{eq:main} into a system of 
subcritical equations, and~\eqref{eq:main2} is just one of them. We do however believe that 
this is one of the most convenient choices (systems of only $3$ equations do not seem 
to work) and produces in the end close to a minimum possible number of trees of negative degrees 
(see Appendix~\ref{app:constants}).
\end{remark}

\begin{remark}
At a formal level, one would expect solutions to \eqref{eq:main} to converge to solutions to
the KPZ equation
\begin{equ}[eq:KPZ]
\d_t h = \d_x^2 h + F_2(0) (\d_x h)^2 +   F_3(0)\xi\;.
\end{equ}
Note however that while one might expect that $\xi^{h_{0, \eps}}_{0, \eps}(t,x) \approx \xi_{0, 
\eps}(t,x)$, the pointwise difference between these two terms is actually
of order $\eps^{-3/2}$, i.e.\ it is just as large as $\xi_{0, \eps}$ itself! This is because one has the 
identity~\eqref{eq:non-hom_noises} and, if all goes well, $h_{0, \eps}$ is expected to be of order one in 
the limit, while $\eta_\eps$ has a correlation length
of order $1$. Our main result shows that these diverging error terms 
do actually average out to zero on the scales under consideration. 
\end{remark}

\subsection{A quenched Edwards--Wilkinson model}
\label{sec:QuenchedEW}

Let us look at a much simpler model, the \emph{quenched Edwards--Wilkinson (qEW) model}:
\begin{equ}[eq:QuenchedEW]
\d_t u (t,x) = \Delta u(t,x) + 1 + \eps^\alpha \eta_\eps(u(t,x),x)\;, \qquad u (0,\bigcdot) = u^0_\eps (\bigcdot)\;,
\end{equ}
for $x \in \R^d$ and $d \geq 1$ and some $\alpha > 0$. Similarly to before, the driving noise $\eta_\eps$ and the initial state 
$u^0$ are $\eps^{-1}$-periodic in every spatial direction, and $\eta_\eps$ is a mollified (on scales of order $1$) spatially 
periodic space-time white noise. If the coefficient in front of $\Delta$ and the constant term is not $1$, 
it can always be made such by rescaling time and $u$, and by changing the value of~$\eps$.

For $\beta \in \R$, we rescale space by $\eps^{-1}$ and the solution by $\eps^{-\beta}$, 
namely we set 
\begin{equ}
	h_\eps(t,x) = \eps^{-\beta} u(\eps^{-2} t, \eps^{-1} x) - \eps^{-\beta-2}t- C_\eps t\;.
\end{equ}
Furthermore, we define the rescaled noise $\xi_\eps(t,x) = \eps^{-\frac{d+2}{2}}\eta_\eps 
\bigl(\eps^{-2} t, \eps^{-1} x\bigr)$. One then has
\begin{equ}[eq:QuenchedEW_rescaled]
\d_t h_\eps = \Delta h_\eps - C_\eps + \eps^{\alpha - \beta + \frac{d-2}{2}} \xi_\eps^{h_\eps}\;,
\end{equ}
with the initial state $h_\eps(0,x) = \eps^{-\beta} u_\eps^0(\eps^{-1} x)$, and where the driving noise is 
\begin{equ}[eq:noise_qEW]
\xi_\eps^{h_\eps}(t,x) \eqdef \xi_\eps \bigl(t + \eps^{2 + \beta} C_\eps t + \eps^{2 + \beta} h_\eps(t,x), x\bigr)\;.
\end{equ}
The rescaled noise $\xi_\eps$ weakly converges as $\eps \to 0$ to a space-time white noise, which is 
$1$-periodic in the spatial variable. In order to see non-trivial fluctuations, we therefore choose $\beta$ 
in such a way that
\begin{equ}[eq:alpha_and_beta]
\alpha - \beta + \frac{d-2}{2} = 0 \qquad\Rightarrow\qquad \beta = \alpha + \frac{d-2}{2}\;.
\end{equ}
Note that as soon as $d \ge 2$, one has $\beta > 0$ and if $d = 1$ then $\beta \geq 0$ for $\alpha \geq \frac 12$. Furthermore, the typical size of $h_\eps$ is expected to be of order $\eps^{(2 - d) /2}$ for $d \ge 3$ and of order $1$ for $d \leq 2$. This can be explained as follows: let us remove the shift of the 
time variable in the noise in~\eqref{eq:QuenchedEW_rescaled} by $h_\eps$. Then in the case $d = 1$, 
$h_\eps$ behaves as a function (rather than a distribution) in the limit $\eps \to 0$. On the other hand, 
in the case $d \geq 2$ the rescaled noise has fluctuations of order $\eps^{-(d+2) /2}$ and its 
space-time convolution with the heat kernel increases the power by $2$. Then $\eps^{\beta} 
h_\eps$ is expected to be of order $\eps^{\alpha}$ for $d \geq 2$ and of order $\eps^{\alpha - 1/2}$ for $d = 1$. This justifies (at least at a 
formal level) the expectation
that in $d\ge 2$ and for $\eps$ small, $h_\eps$ is close to $\tilde h_\eps$, where the latter solves 
\begin{equ}
\d_t \tilde h_\eps = \Delta \tilde h_\eps +  \xi_\eps\;,
\end{equ}
whatever the value of $\alpha > 0$ is. In dimension $1$, one would expect the same statement to hold for $\alpha > \frac{1}{2}$. In the following theorem, which is proved in Section~\ref{sec:qEW}, we show that this heuristics is correct. 

\begin{theorem}\label{thm:qEW}
Let $\alpha$ and $\delta$ be such that $\alpha \geq \frac 12$ and $\delta > 0$ in the case $d = 1$, and $\alpha > 0$ and $\delta + \beta + 2 > 0$ in the case $d \geq 2$. Let the initial state $u^{0}_\eps : (\eps^{-1} \T)^d \to \R$ of~\eqref{eq:QuenchedEW} be such that for $\beta$ as in~\eqref{eq:alpha_and_beta} the 
functions $\eps^{-\beta} u^0_\eps(\eps^{-1} \bigcdot)$ converge as $\eps \to 0$ in $\CC^{\delta}(\T^d)$ 
to $h^0 \in \CC^{\delta}(\T^d)$ and if $d \geq 2$ then $\|u^0_\eps(\eps^{-1} 
\bigcdot)\|_{\CC^{\delta+\beta}}$ are also uniformly bounded in $\eps$. Then there is a choice of $C_\eps$  such that for any $\nu < \frac{2 - d}{2}$ the solutions $h_\eps$ of~\eqref{eq:QuenchedEW_rescaled} converge in probability in the topology of $\CC_\s^{\nu \wedge 0}$ to the solution of the stochastic heat equation
 \begin{equ}[eq:heat_qEW]
  	\d_t h = \Delta h +  \xi\;, \qquad h(0, \bigcdot) = h^0(\bigcdot)\;,
 \end{equ}
driven by a space-time white noise $\xi$ on $\R \times \T^d$. Finally, the above renormalisation constant satisfies the following: $C_\eps  \sim \eps^{-1}$ if $d = 1$, $C_\eps \sim \eps^{\alpha - \frac{d-2}{2}}$ if $d \geq 2$ and in particular $C_\eps = 0$ if $\alpha > \frac{d-2}{2}$.
\end{theorem}

The proof of this theorem is different for $d = 1$ and $d \geq 2$, because in the latter case the solution of \eqref{eq:heat_qEW} is a distribution and the framework which we develop in the following sections cannot be applied directly (it is important for this framework that we perturb the noise \eqref{eq:noise_qEW} by a function $h_\eps$, which is also a function in the limit $\eps \to 0$). To resolve this problem we use a trick as in the system \eqref{eq:main2}: we define $h_{1, \eps} \eqdef \eps^{\beta} h_\eps$, $\xi_{1, \eps} \eqdef \eps^{\beta} \xi_\eps$ and $C_{1, \eps} \eqdef \eps^{\beta} C_\eps$, and label these objects without the multiplier $\eps^\beta$ with the subscript $0$. Then we obtain the system of equations 
\begin{equ}[eq:QuenchedEW_system]
\d_t h_{0, \eps} = \Delta h_{0, \eps} - C_{0, \eps} + \xi_{0, \eps}^{h_{1, \eps}}\;, \qquad \d_t h_{1, \eps} = \Delta h_{1, \eps} - C_{1, \eps} + \xi_{1, \eps}^{h_{1, \eps}}\;,
\end{equ}
where the perturbed noises are defined as
\begin{equ}
\xi_{i, \eps}^{h_{1, \eps}}(t,x) = \xi_{i, \eps} \bigl(t + \eps^{2 + \beta} C_\eps t + \eps^{2} h_{1, \eps}(t,x), x\bigr)\;.
\end{equ}
Because the solution of~\eqref{eq:heat_qEW} has regularity $\frac{2-d}{2}-$, the function $h_{1, \eps}$ is expected to be a function of regularity $\beta + \frac{2-d}{2} - = \alpha -$, which allows to apply our framework to the system \eqref{eq:QuenchedEW_system}.

\subsection{A general system of SPDEs}

For locally subcritical SPDEs with homogeneous noises renormalisation has been derived in~\cite{BCCH}. 
This result relates the renormalisation group constructed in~\cite{BHZ} with the renormalised canonical 
lifts of driving noises obtained in~\cite{HC}, and gives a general existence of solution and stability 
theorem for a wide class of nonlinear SPDEs. Unfortunately, these results cannot be applied directly in 
our case even after rewriting it in the form \eqref{eq:main2}, because of the inhomogenous noises~\eqref{eq:rescaled_noise}. For this we will have to 
adapt the results from~\cite{BCCH} and for the first part of the paper will consider a general system of 
SPDEs, driven by different inhomogeneous noises. More precisely, for integers $m, n \geq 
1$ and for functions $u_{j} : \R_+ \times \T^{d} \to \R$, with $d \geq 1$ and $j=1,\ldots,m$, 
we consider 
a collection of inhomogeneous noises $\xi^{u,c} = (\xi^{u,c}_i)_{i = 1}^n$, defined by
\begin{equ}[eq:non-hom_noise_intro]
\xi^{u,c}_i(z) \eqdef \xi_{i} \biggl(z + cz + \sum_{j = 1}^{m}\sum_{k \in \N^{d+1}} a^k_{i j} D^k u_{j}(z)\biggr)\;,
\end{equ}
where $a^k_{i j} \in \R^{d+1}$ is a fixed 
collection of vectors such that $a^k_{i j} \neq 0$ only for finitely many $k \in \N^{d+1}$, $c \in L(\R^{d+1}, \R^{d+1})$ is a constant matrix which is needed in order to represent a translation by $\eps^2C_\eps$ in 
\eqref{eq:rescaled_noise}, and $D^k$ is a mixed space-time derivative. In other words, we shift the 
space-time variable $z$ in the noises by a linear combination of the functions $D^k u_j(z)$ and $cz$. 
Then we consider a system of SPDEs
\begin{equ}[eq:system_intro]
\d_t u_i = \SL_i u_i + F_{i} \bigl(\{\d^p_x u_{j} : p \in \N^d\}, \{D^q \xi^{u,c}_{j} : q \in \N^{d+1}\}\bigr)\;, \quad u_i(0, \bigcdot) = u_i^0(\bigcdot)\;,
\end{equ}
where $\SL_i$ is an elliptic differential operator, $\d^p_x$ is a mixed spatial derivative, and where 
$(F_i)_{i=1}^m$ is a collection of local nonlinearities, which are affine in the noises $\xi^{u,c}$ and 
possibly a finite number of their derivatives. Moreover, the nonlinearities depend only on finitely many 
elements $\d^p_x u_{j}$ and are smooth with respect to them. We assume that the system of equations 
\eqref{eq:system_intro} is locally subcritical, which we define rigorously in Section~\ref{sec:reg_str}. 
One can readily see that~\eqref{eq:main2} and~\eqref{eq:QuenchedEW_rescaled} are special cases 
of the general system~\eqref{eq:system_intro}.

The smooth noises $\xi_{j, \eps}$ we are interested in depend on a parameter $\eps > 0$ and have limits 
$\xi_j$ as $\eps \to 0$ in respective spaces of distributions. A typical example of such noises can be 
$\xi_{j, \eps} = \sigma_{j, \eps} (\xi_j * \rho_\eps)$, for some $\eps$-dependent constants $\sigma_{j, 
\eps}$, for space-time white noises $\xi_j$, which are not necessarily independent,  and for a smooth 
mollifier $\rho_\eps$, converging to the Dirac delta as $\eps \to 0$. Then the classical solutions $u_{i, 
\eps}$ of~\eqref{eq:system_intro}, driven by the noises $\xi_{j, \eps}$, do not typically converge to a 
non-trivial limit as $\eps \to 0$. This is due to the fact that the nonlinearities $F_{i}$ are not 
well-defined in the limit, since the driving noises have low regularities. In general, one would like to 
perform renormalisation of equations~\eqref{eq:system_intro} which allows to obtain a non-trivial limit. 
More precisely, for every $i = 1, \ldots, m$ one would like to find a ``natural'' modification $F_{i, \eps}$ of 
$F_{i}$ such that the classical solutions of the renormalised system 
\begin{equ}[eq:system_intro_renorm]
\d_t \hat u_{i, \eps} = \SL_i \hat u_{i, \eps} + F_{i, \eps} \bigl(\{\d^p_x \hat u_{j, \eps} : p \in \N^d\}, \{D^q \xi_j^{\hat u, c} : q \in \N^{d+1}\}\bigr)\;,
\end{equ}
converge as $\eps \to 0$ to a non-trivial collection of functions $(u_i)_{i=1}^m$. If these limits do not 
depend on a particular choice of approximation of the noises $\xi_j$ (although the functions $F_{i, 
\eps}$ might depend on such choice), then they are considered to be solutions of the 
system~\eqref{eq:system_intro}. However, construction of such renormalised equations can be very 
non-trivial, especially when the number of equations is large and the nonlinearities are complex. 

\subsection{Counterterms for the inhomogeneous noises}
\label{sec:counterterms}

In order to deal with the inhomogeneous noises in~\eqref{eq:system_intro}, we are going to consider 
regularity structures such 
that each instance of the noise $\Xi_i$ in a tree $\tau$ has an infinite-dimensional space $\CB$ 
``attached'' 
to it (see Section~\ref{sec:first} for the definition of the trees in the regularity structure). More 
precisely, instead of one instance of the noise $\Xi_i$ we consider a couple $\mu \otimes \Xi_{i}$ for an 
element $\mu \in \CB$. This approach is similar to the quasilinear setting of~\cite{FuGu,HendrikFelix,MH17}, where the authors 
attach infinite-dimensional 
spaces to each instance of the kernel. (In the case of \cite{FuGu,HendrikFelix} this is 
equivalently formulated as a parametrised family of ``models'' or paracontrolled
distributions.)
As in those papers, we take $\CB$ to be the dual of a weighted version of $\CC^{k_\star}$ 
for some sufficiently large value of $k_\star > 0$ which will be determined later,\footnote{Throughout 
the article, we have several global constants which need to be specified. Similarly to $k_\star$, we 
decorate them with ``$\star$'' to distinguish them from other constants.} and we set up the model such that 
for 
$\mu \in \CB$
\begin{equ}[eq:model_intro]
	\big(\Pi^\eps_{z} (\mu \otimes \Xi_{i}) \big)(\bar z) = \int_\R \xi_{i,\eps} \bigl(\bar z + c_\eps \bar z + \eps^2u\bigr) \mu(du)\;,
\end{equ}
where we use the shorthand notation $\bar z + \eps^2u = \bar z + (\eps^2u, 0)$ and $c_\eps = \text{diag}(\eps^2C_\eps, 0)$. A precise definition of models is provided in Section~\ref{sec:models}. 
Note that the regularity (in $\bar z$) of these distributions is the same for every $\mu$ and
is uniform in $\eps$. This is
thanks to the fact that we have $\eps^2$ multiplying $u$, so that potential derivatives hitting 
$\xi_{i,\eps}$ are precisely compensated by the powers of $\eps$ that they generate.

We need to take $k_\star$ big enough, so that $\CB$ contains derivatives of Dirac delta functions of 
high enough order. Delta functions play a special role in our regularity structure because~\eqref{eq:model_intro} implies
\begin{equ}
	\big(\Pi^\eps_{z} (\delta_{h_\eps(t,x)} \otimes \Xi_{i,\eps}) \big)(t,x) = \xi_{i,\eps} \bigl(t+\eps^2C_\eps t + \eps^2 h_\eps(t,x), x\bigr)\;,
\end{equ}
where $\delta_{h}$ denotes the Dirac delta function centred at a point $h$.

With this notation, given a function-like sector $V$, one can define
an operator $\hat\Xi_i \colon \CD^\gamma(V) \to 
\CD^{\gamma-\alpha_i}$ (where $-\alpha_i < 0$ is the degree of $\Xi_i$) by
\begin{equ}[eq:Xi_hat]
	\hat\Xi_i(H)(z) = \sum_{n \ge 0}{1\over n!} \bigl(\delta^{(n)}_{h(z)}\otimes \Xi_i\bigr) \bigl(\tilde h(z)\bigr)^n\;,
\end{equ}
for modelled distribution $H \in \CD^\gamma(V)$ of the form $H(z) = h(z)\one + 
\tilde h(z)$. Here $\delta_u^{(n)}$ denotes the $n^{\text{th}}$ derivative of the Dirac distribution located at $u$. The 
reason to include higher order terms ($n > 0$) in~\eqref{eq:Xi_hat} is analogous to the higher order 
terms in the definition of the composition of modelled distributions with smooth functions (see 
\cite[Sec.~4.2]{Regularity}). We will see that this guarantees boundedness of the 
$\hat \Xi_i$ operators between the spaces of modelled distributions just mentioned. 

In practice, we need to be more careful with what we mean by attaching vector spaces to noises. We are 
going to use the notion of vector-valued regularity structures developed in~\cite{CCHS}. For this, we 
shall actually view noises as edges $\J_{\fl_i}$ rather than leaves $\Xi_i$ in the trees $\tau$. Moreover,  
we shall see that it is more convenient to view the multiplication in~\eqref{eq:Xi_hat} as 
$\J_{\fl_i}[\delta^{(n)}_{h(z)} \otimes \bigl(\tilde h(z)\bigr)^n]$, which should be understood as attaching 
$\delta^{(n)}_{h(z)}$ to the edge $\J_{\fl_i}$ and ``drawing'' $\bigl(\tilde h(z)\bigr)^n$ above it. This is 
convenient for several reasons. First, this allows to distinguish multiplication with $\xi_{i,\eps}$ that 
comes from a multiplicative noise in~\eqref{eq:main2} from the ``multiplication'' that comes from the 
shift by $h_\eps$. Second, if we observe that $n$ edges leaving the edge $\J_{\fl_i}$, it means that 
$\delta^{(k)}_{h}$ for some $k\geq n$ should be attached to that edge.\footnote{If there is no polynomial on top of $\J_{\fl_i}$ then precisely $\delta^{(n)}_{h}$ will be attached to that edge. For the effect of polynomials on the derivatives of $\delta_{h}$ see~\eqref{eq:d_i} and~\eqref{eq:LeibnizFhat}.} This allows us to perform part of the algebra on the 
trees $\tau$ with no vector spaces attached. 

\subsection{Outline of the paper}

In Section~\ref{sec:algebra} we recall the main algebraic definitions from the theory of the regularity 
structures. We use the framework of~\cite{CCHS} to construct vector-valued regularity structures that 
reflect the assignment of an infinite-dimensional spaces on noises from~\eqref{eq:model_intro}. In fact, we construct
two regularity structures $\CT_{\CD}$ and $\CT_\CB$ as well as an evaluation map $\Ev$ between them. 
The regularity structure $\CT_\CD$ is going to be used for an algebraic formulation of 
the renormalisation of general equations like~\eqref{eq:system_intro} and has the advantage 
that its homogeneous subspaces are all finite-dimensional.
The structure $\CT_\CB$ is genuinely infinite-dimensional and will be used to formulate 
the analytic solution theory. Section~\ref{sec:coherence} is an adaptation of results from~\cite[Sec. 3]{BCCH} to 
the inhomogeneous noise setting. We define an operator $\Upsilon$ that is going to produce 
counterterms in the renormalised equation~\eqref{eq:system_intro_renorm} as well as discuss the 
notion of coherence which is an algebraic equivalent of a notion of a solution. Moreover, we define an 
action of the renormalisation group on the non-liearities. In Section~\ref{sec:models} we 
introduce models on the space $\CT_\CB$ and prove the continuity of the map $\hat\Xi$ 
from~\eqref{eq:Xi_hat}. In Section~\ref{sec:renorm_SPDE} we apply the renormalisation procedure to 
SPDEs, i.e.\ we show how an abstract equation on the space of modelled distributions is reconstructed 
for a renormalised model. In Section~\ref{sec:application} we apply all this general machinery to the 
system~\eqref{eq:main2} and prove Theorems~\ref{thm:main_intro} and \ref{thm:qKPZ_intro}. In particular, we show convergences of 
the noises $\xi_{i,\eps}$ from~\eqref{eq:new_functions} and construction of the BPHZ model on 
$\CT_\CB$. We prove Theorem~\ref{thm:qEW} in Section~\ref{sec:qEW}. Appendix~\ref{app:constants} contains computations of renormalisation constants for 
equation~\eqref{eq:main2}. Appendix~\ref{app:notation} contains notations of frequently used symbols.

\subsection{Notation and definitions}
\label{ss:notation}

We use the notations $\N = \{0,1,\dots\}$, $\R_+ = [0, \infty)$, and $\T^d = \R^d/\Z^d$, 
which we simply denote by $\T$ in the case $d = 1$. 

We always work on the time-space domain $\R^{d+1}$, equipped with the parabolic scaling $\s = (\s_0,\s_1, 
\ldots, \s_d) \in \R^{d+1}_+$, where $\s_0=2$ is the scaling of the time variable and $\s_i=1$ are the scalings of the spatial variables for $i = 1,\dots,d$. 
We write $|\s| = \sum_{i=0}^d \s_i$. Although we aim to work with the equations~\eqref{eq:main_before_scaling} 
with $d=1$, we develop a solution theory for a more general system of SPDEs~\eqref{eq:system_intro} 
with arbitrary spatial dimension $d \geq 1$. For $z = (t,x_1, \ldots, x_d) \in \R^{d+1}$ we define the norm 
$\|z\|_\s \eqdef \max \{\sqrt{|t|}, |x_1|, \ldots, |x_d|\}$, and for a multi-index $k = (k_0, k_1, \ldots, k_d) \in 
\N^{d+1}$ we set $|k|_\s \eqdef \sum_{i=0}^d \s_i k_i$. For vectors $v \in \R^m$ we often use the 
$\ell^1$ norm $|v| = \sum_{i = 1}^m |v_i|$, and the norm $|v|_\infty = \max_{1 \leq i \leq m} |v_i|$. For a 
sufficiently many times differentiable function $f : \R^{d+1} \to \R^{d+1}$ and for $k = (k_0, k_1, \ldots, 
k_d) \in \N^{d+1}$, we let $D^k f$ be the mixed derivative of $f$, where we differentiate $k_0$ times in 
the time variable and $k_i$ times in the spatial variable $x_i$, for $1 \leq i \leq d$. 

For topological vector spaces $V,W$ we write $L(V,W)$ for the space of continuous linear maps $V 
\mapsto W$ and set $L(V) \eqdef L(V,V)$. 

We use the standard notation $\CC^\alpha(\R^d)$ for the H\"{o}lder space
 and the Besov spaces of distributions (when $\alpha < 0$), defined in 
\cite[Def.~3.7]{Regularity}, whose (semi-)norms we denote by $\| \bigcdot \|_{\CC^\alpha}$. In the case 
$\alpha \in \N$, the space $\CC^\alpha(\R^d)$ is the space of all $\alpha$ times continuously 
differentiable functions. We note that in contrast to~\cite[Def.~3.7]{Regularity} we do not define the 
semi-norms locally; this is because all the distributions we consider are spatially periodic and are 
defined on a finite time interval. For $\alpha \geq 0$, the space $\CC^\alpha_{\loc}(\R^d)$ refers to the 
space of locally $\alpha$-H\"{o}lder continuous functions. The Besov space of distributions on 
$\R^{d+1}$ with the parabolic scaling $\s$ we denote by $\CC_\s^\alpha(\R^{d+1})$, with the respective 
semi-norm $\| \bigcdot \|_{\CC^\alpha_\s}$. In the case $\alpha > 0$ is non-integer, we denote by 
$\CC_\s^\alpha(\R^{d+1})$ the H\"{o}lder space with respect to the metric $\|\bigcdot\|_\s$. If we want to 
specify a domain of functions/distributions, we write $\CC^\alpha(\fK)$ and $\CC^\alpha_\s(\fK)$, for 
respective domains $\fK$. Whenever the domain is clear from the context, we prefer to write 
$\CC^\alpha$ and $\CC^\alpha_\s$.

For $r \in \N$, sometimes we use the set $\SB^r_\s$ containing all smooth test functions $\phi : \R^{d+1} \to \R$ supported on a unit ball with $\|\phi\|_{\CC^r_\s} \leq 1$. The respective set of functions on $\R$ we denote by $\SB^r$.

To measure regularities of distributions, we use rescaled test functions. More precisely, for a function 
$\phi : \R^d \to \R$, a point $x \in \R^d$ and a scaling parameter $\lambda \in (0,1]$, we define its 
rescaled and recentred version $\phi^\lambda_{x}(y) \eqdef \lambda^{-d} \phi 
\bigl(\lambda^{-1}(y-x)\bigr)$; and in the case of a function $\phi : \R^{d+1} \to \R$ on the space-time 
domain we define 
\begin{equ}[eq:rescaling]
	\phi^\lambda_{(t,x)}(s,y) \eqdef \lambda^{-|\s|} \phi \bigl(\lambda^{-2}(s-t), \lambda^{-1}(y-x)\bigr)\;,
\end{equ}
for $x,y \in \R^d$ and $t, s \in \R$, where we use the parabolic scaling $\s$. We will use both 
$\scal{\zeta,\phi}$ and $\zeta(\phi)$ notations for a distribution $\zeta$ applied to a test function 
$\phi$. As a rule of thumb $\scal{\zeta,\phi}$ is mostly used for spaces $\CB$ from 
Section~\ref{sec:spaces} and $\zeta(\phi)$ is mostly used for models (see Section~\ref{sec:models}).

To make our exposition lighter, we prefer to use ``$\lesssim$'' for a bound ``$\leq$'' with a constant 
multiplier, independent of the relevant quantities, which will always be clear from the context.

\subsection*{Acknowledgments}

The authors would like to thank Ajay Chandra for numerous useful discussions of the renormalisation 
and coherence in Section~\ref{sec:coherence}, as well as Francesco Pedullà for pointing
out a mistake in an earlier version of Lemma~\ref{lem:counterterms}. AG gratefully acknowledges the financial support by the Leverhulme Trust through Hendrik Weber’s Philip Leverhulme Prize. MH was supported by the
Royal Society through a research professorship. KM was partially supported by NSF grant DMS-1953859.

\section{Algebraic framework}
\label{sec:algebra}

In this section, we define a regularity structure that allows to formulate the system of equations 
\eqref{eq:system_intro} in a suitable space of modelled distributions. However, in contrast to the 
original definition in~\cite{Regularity}, each element in our regularity structure will take values in an 
infinite-dimensional Banach space (a similar idea was used in~\cite{MH17} for solving quasilinear singular 
SPDEs). This is due to our approach described in Section~\ref{sec:counterterms}, where each instance 
of the driving noise comes together with a suitably localized derivative of the Dirac delta function.

\subsection{Rules, Trees and Subcriticality}\label{sec:reg_str}
\label{sec:rules}
	
Our aim is to introduce a rich enough algebraic framework which is going to be used to describe a 
generalized Taylor expansion of the solution to the system~\eqref{eq:system_intro}. As in~\cite{BHZ, BCCH} 
let $\fL_{-}$ denote a finite index set of noise labels and $\fL_+$ denote a finite index set of kernel labels. Recall that 
$\fL_+$ can also be associated with the components of the given system of 
equations~\eqref{eq:system_intro} since each component 
corresponds to precisely one differential operator and thus to precisely one integration kernel. We 
denote $\fL \eqdef \fL_-\sqcup\fL_+$ and assign degrees $\deg : \fL \to \R$ to its elements, such that the 
elements of $\fL_-$ have strictly negative degrees and those in $\fL_+$ have strictly positive degrees. Each 
element $\fl \in \fL_-$ represents a driving noise of regularity $\deg \fl$, and each element $\ft \in \fL_+$ 
represents a linear differential operator whose inverse improves regularity by $\deg 
\ft$.\label{lab:index_set}

We define $\Eps \eqdef \fL\times \N^{d+1}$ and $\CO \eqdef \fL_{+}\times \N^{d+1}\subset 
\Eps$\label{p:defCO}, where $d \geq 1$ is the spatial dimension of the equation. For $(\ft, p) \in \Eps$ we write 
$\deg(\ft,p) \eqdef \deg \ft-|p|_\s$, and one should think of $\CO$ as an indexing set of all the 
solutions and their derivatives of our system of SPDEs.\label{lab:CO} On  the other hand, $\Eps$ 
indexes all the components of the system and their derivatives, which includes both the 
solutions and the noises.

As we describe below, we work with decorated trees, such that each edge has a type from $\Eps$. In 
addition to the edge types, we define a set of node types $\Nodes \eqdef \hat{\mcbP} (\Eps)$,  the set of all multisets\footnote{More precisely, $\hat{\mcbP}(A) \eqdef 
\bigsqcup_{n \geq 0} [A]^n$, where $[A]^n$ equals $A^n$ quotiented by permutation of entries.} 
with elements from $\Eps$. The order describing which type of node can follow which type of edge is 
determined by a \emph{rule} $\Rule: \fL \to \mcbP(\Nodes) \backslash \{\emptyset\}$, where $\mcbP(A)$ 
denotes the power set of $A$. We extend $\Rule$ to $\Eps$ by postulating $\Rule(\ft, p) = \Rule(\ft)$.  For any $A \subset \Eps$, we  always identify $\N^{A}$ with $\hat\mcbP(A)$.
 We use angled brackets to build multisets, so for example $[a,a,b]$ is the multiset containing $a$ twice and $b$ once.

Given a map $\reg : \fL \to \R$ we also extend it to $\Eps$ by $\reg(\ft,p) = \reg(\ft) - |p|_\s$
and we define a partition $\Eps = \Eps_- \sqcup \Eps_+$ by  
setting $\Eps_+ \eqdef \{o \in \Eps\,:\,\reg o \geq 0\}$. Heuristically, $\Eps_-$ indexes the 
components of the system which are space-time distributions of negative regularities and $\Eps_+$ 
indexes the function like components of positive regularities. We shall see later that a map $\reg $ that 
we are going to use will always satisfy $\reg(\fL_-) \subset (-\infty,0)$ thus implying that $\fL_-\times \N^{d+1} \subset \Eps_-$. We also have $\CO \not\subset \Eps_+$ since sufficiently high derivatives of 
solutions are distributions. For $\fl \in \fL_-$, define 
\begin{equ}[eq:Eps_ell]
	\Eps_+(\fl) \eqdef  \{o \in \Eps_+ \,:\, [o] \in \Rule(\fl)\} \subset \Eps_+\;,
\end{equ}
to be a set representing the components of the equation (and their derivatives) that are present in the 
inhomogenous noise $\xi^{u}_\fl$ (i.e.\ the set $\Eps_+(\fl)$ represents coefficients $a^k_{ij}$ 
in~\eqref{eq:non-hom_noise_intro} that are non zero).\label{lab:Eell} We postulate also 
$\Eps_+\big((\fl,p)\big) \eqdef \Eps_+(\fl)$ and it will be convenient to set $\Eps_+(\fl) = \emptyset$
for $\fl \in \fL_+$. Note that thanks to the property $\reg(\ft,p) = \reg(\ft) - |p|_\s$, the
sets $\Eps_+$ and $\Eps_+(\fl)$ are finite. Both $\Eps_+$ and $\Eps_+(\fl)$ depend on a 
choice of $\reg$ but this choice will be fixed for each rule.

We now state the assumptions on the rule that we are going to fix throughout the whole article. 

\begin{assumption}\label{ass:rule} 
We assume that the rule $\Rule : \fL \to \mcbP(\Nodes) \backslash \{\emptyset\}$ has the following properties:
	\begin{enumerate}[label=(R\arabic*)]
		\item\label{R1} $\Rule$ is \emph{normal}: for every $\ft\in \fL$, if $M \in \Rule(\ft)$ then $N \in \Rule(\ft)$ for every $N \subset M$;
		\item\label{R2} $\Rule$ is \emph{subcritical} with respect to the scaling $\s$: there exists a map $\reg : \fL \to \R$ such that for every $\ft \in \fL$ one has
		\begin{equ}
			\reg(\ft) < \deg \ft + \inf_{N \in \Rule(\ft)} \sum_{(\fl, k) \in N} \bigl(\reg(\fl) - |k|_\s\bigr)\;;
		\end{equ}
		\item\label{R3} $\Rule$ is \emph{complete} according to~\cite[Def.~5.20]{BHZ};
		\item\label{R5} With $\Eps_+$ defined using $\reg$ as in~\ref{R2}, one has $\Rule(\fl) = \N^{\Eps_+(\fl)}$ for every $\fl \in \fL_-$.
	\end{enumerate}
\end{assumption}

More discussion of properties \ref{R1}--\ref{R3} can be found in~\cite[Sec.~5]{BHZ}.\label{lab:Rule}  In 
contrast to~\cite{BCCH} (but as in \cite{BHZ}), we represent noises by edges and \ref{R5}
clearly implies that these edges need not be terminal. The latter is a reflection of the presence
of inhomogeneous noise and allows us to distinguish between the function-like elements appearing 
because of the multiplicativity of the noise and those appearing because of its inhomogeneity.
The reason why we allow $\Rule(\fl)$ to be all of $\N^{\Eps_+(\fl)}$ is to guarantee that  formulae like~\eqref{eq:Xi_hat} are meaningful.

Let $\fT$\label{lab:trees} be the set of rooted decorated combinatorial trees $\tau = (T, \ff, \fm)$, where $T$ is a tree with root $\rho$, nodes $N_T$, and edges $E_T$, and $\tau$ admits 
edge decorations $\ff : E_T \to \Eps$ and node decorations $\fm : N_T \to \N^{d+1}$. For $e \in E_T$ we 
write $\ff(e) = (\ft(e), p(e))$, where $\ft(e) \in \fL$ and $p(e) \in \N^{d+1}$ and refer to $\ft(e)$ as the 
type of $e$. To simplify notations,  we also write $\tau = T^\fm_\ff$ and we set
\begin{equ}[eq:degree-not-extended]
\deg \tau = \sum_{v \in N_T} | \fm(v) |_\s + \sum_{e \in E_T} \deg \ff(e)\;.
\end{equ}
Given $T^\fm_\ff \in \fT$, we equip its edges with the standard orientation pointing away 
from the root. For 
each node $v \in N_T$ we write $v_{\uparrow} \subset E_T$ for the set of edges leaving $v$, which is 
empty if $v$ is a leaf. If $v$ is not the root, then $v_{\downarrow}  \in E_T$ is the unique edge coming 
into $v$. We associate to each node $v \in N_T$ a node type $\Nodes[v] \in \Nodes$ by $\Nodes[v] \eqdef \bigl[\ff(e) : e \in v_{\uparrow}\bigr]$.

For a rule $\Rule$ satisfying Assumption~\ref{ass:rule}, we denote by $\fT(\Rule) \subset \fT$\label{lab:trees_rule} the set of all trees $T \in \fT$ \emph{conforming} to $\Rule$ in the sense of~\cite[Def. 5.8]{BHZ}, i.e.\ for the root $\rho$ of $T$ there is $\ft \in \fL$ such that $\rho \in \Rule(\ft)$, and for every other node $v \in N_T \setminus \{\rho\}$ one has $\Nodes[v] \in \Rule(v_{\downarrow})$.

We say that the tree $\tau = T^\fm_\ff$ is \emph{planted} if $\Nodes[\rho]$ consists of a single 
edge type $\ff(e) \in \CO$ and we say that $\tau$ is \emph{unplanted} otherwise.\footnote{Note that trees that are planted in the usual sense, but with a trunk of a type in $\fL_-$ are not considered planted in our sense. The reason for this will be 
apparent in Section~\ref{sec:admissible}.} We define the set
\begin{equ}[eq:unplanted]
	\fT_-(\Rule) \eqdef \{\tau \in \fT(\Rule): \deg \tau < 0,\,\text{$\tau$ is unplanted}\}\;.
\end{equ}

We are extensively going to use the following alternative notation for
$\tau = T^\fm_\ff \in \fT(\Rule)$. Let $m = \fm(\rho) \in \N^{d+1}$
and let $\tau_i$ for $i \le N$ be the subtrees rooted atop each of the $N$
edges (of respective types $o_i \in \Eps$) incident to $\rho$.
In this case, we write 
\begin{equ}[eq:recursive_tree]
	\tau = \X^m \prod_{i=1}^{N} \J_{o_i}[\tau_i]\;.
\end{equ}
Sometimes we also write
\begin{equ}[eq:recursive_tree_new]
	\tau = \X^m \prod_{i=1}^{n} \J_{o_i}[\tau_i]^{\beta_i}\;, 
\end{equ}
with the implicit convention that then $(o_i, \tau_i)\neq (o_j, \tau_j)$
whenever $i\neq j$.
Given $\tau$ written as~\eqref{eq:recursive_tree_new} we define its symmetry factor by
\begin{equ}[eq:symmetry_factor]
	S(\tau) = m! \Big(\prod_{i= 1}^{n} S(\tau_i)^{\beta_i} \beta_i! \Big)\;,
\end{equ}
with the usual convention that empty products evaluate to $1$.

\subsubsection{The rule for the system~\eqref{eq:main2}}
\label{sec:first}

In the case of the system of SPDEs~\eqref{eq:main2}, we have $d = 1$ and $\s = (2,1)$. The three 
driving noises: $\xi_{k, \eps}$, $k = 0, 1, 2$, we label respectively by $\fL_- = \{\fl_k : k = 0, 1, 2\}$, and 
the three equations for $h_{k, \eps}$, $k = 0, 1, 2$, we label by $\fL_+ = \{\ft_k : k = 0, 1, 2\}$. We do 
not consider a label corresponding to the equation for $h_{3, \eps}$, because as we stated 
below~\eqref{eq:main2}, this equation can be solved classically and does not require using regularity 
structures. We set the degrees of the noises to be 
\begin{equ}[eq:deg_main]
	\deg \fl_k = -\frac{3}{2} + \frac{2 k}{3} - \kappa_\star\;, \qquad k = 0, 1, 2\;,
\end{equ}
for a fixed constant $0 < \kappa_\star < \frac{1}{42}$ (the upper bound by $\frac{1}{42}$ is explained at 
the beginning of Section~\ref{sec:smooth_models} and is also used in the proof of 
Lemma~\ref{prop:system_soln} to guarantee that $\partial_x h_{1, \eps}$ is a function, rather than a distribution). The 
value of $\deg \fl_0$ corresponds to the regularity of a space-time white noise in one dimension, i.e.\ 
$\xi_{0, \eps} \in \CC_\s^{-3/2 - \kappa_\star}$ for all $\eps > 0$. The values of the other $\deg \fl_k$ 
reflect the increase of regularity, which we obtain after multiplying $\xi_\eps$ by a positive power of 
$\eps$ in~\eqref{eq:new_functions}. Since all equations in~\eqref{eq:main2} have the linear parts $\d_t - 
\d_x^2$, we assign the degrees $\deg \ft_k = 2$, for $k = 0,1,2$, which is the gain of regularity provided 
by the heat kernel. In other words, we expect that $h_{k, \eps} \in \CC_\s^{\deg \fl_k + 2}$. 

All $h_{k, \eps}$ have positive regularities, but their derivatives can be distributions. We 
define the map $\reg$ as $\reg(\fl_k) = \deg \fl_k - \kappa$ and $\reg(\ft_k) = \reg(\fl_k) + 2$, 
for sufficiently small $\kappa > 0$, which yields
\begin{equ}
	\Eps_+ = \{o \in \Eps\,:\, \reg o > 0\}
	 = \{\ft_0, \ft_1, \ft_2, \d\ft_1,\d\ft_2 \}\;,
\end{equ}
where, given $o = (\ft,p) \in \Eps$, we set $\d o = (\ft, p+(0,1))$. 

Regarding the rule $\Rule$ suitable to describe the system~\eqref{eq:main2}, its 
right-hand sides suggest that $\Rule$ should be taken to 
be the smallest complete rule such that
\begin{equs}
\Rule(\fl_k) &\supset \{\ft_{0}^n \,:\, n \geq 0\}\;,\\
	\Rule(\ft_k) &\supset \{\fl_k, \fl_{k+1} \d \ft_1^2,  \fl_{k+2} \d \ft_1^n, 
	\d\ft_{\lceil k / 2 \rceil}\d\ft_{\lfloor k / 2 \rfloor}, \d\ft_{k} \d \ft_1^n , \d^2\ft_{k+1} \d \ft_1^n\,:\,
	n \ge 0\}\;,
\end{equs}
where we use the abuse of notation that $\fl_3 = \fl_4 = []$ and the product denotes the 
concatenation of multisets. We refer to Appendix~\ref{app:constants} for examples of trees generated by this rule.

It is straightforward to see that $\Rule$ is indeed normal~\ref{R1}. From the 
definition of $\Rule(\fl_k)$ the rule $\Rule$ clearly satisfies~\ref{R5}. It is a 
straightforward computation to see that the rule $\Rule$ is \emph{subcritical} with respect to the 
scaling $\s$~\ref{R2} for the choice of $\reg$ as above. The completeness of 
$\Rule$~\ref{R3} is guaranteed by~\cite[Prop.~5.21]{BHZ}. Moreover, the 
completed rule still satisfies~\ref{R5}, because the renormalisation won't affect 
trees above the noises, since only edges from $\Eps_+$ could leave the noise edges to begin with.

As described in Section~\ref{sec:rules}, the rule $\Rule$ generates the set $\fT(\Rule)$ of labelled trees 
conforming to $\Rule$. A sufficiently large subset of $\fT(\Rule)$ forms a basis in the structure space 
of a respective regularity structure, which one can use to solve the system~\eqref{eq:main2}. To bound 
a model for this regularity structure, one typically needs to bound it on the unplanted trees of negative 
degrees $\fT_-(\Rule)$, and a bound on the other trees follows automatically 
(see~\cite[Thm.~10.7]{Regularity}). Using a computer program, 
we found that the number of trees in $\fT_-(\Rule)$ is $72$ when $\kappa_\star$ in~\eqref{eq:deg_main} is sufficiently small. Of course, analysing each tree in $\fT_-(\Rule)$ by hand, as it was done for example for the KPZ equation in~\cite{FrizHairer}, would be tedious and one requires a more automated approach.

\subsubsection{The rule for the KPZ equation}

We can define the rule $\Rule^\KPZ$ for the KPZ equation~\eqref{eq:KPZ}. For this, we use the settings 
and notation of Section~\ref{sec:first}. More precisely, we have $d = 1$ and $\s = (2,1)$. We label the 
driving noise by $\fL^\KPZ_- = \{\fl_0\}$ and the equation by $\fL^\KPZ_+ = \{\ft_0\}$, so that $\fL^\KPZ = 
\fL^\KPZ_- \sqcup \fL^\KPZ_+$; and assign the degrees $\deg$ and regularity $\reg$ on $\ft_0$ and $\fl_0$ exactly the same as in Section~\ref{sec:first}. Then we have $\Eps^\KPZ = \Eps^\KPZ_- \sqcup 
\Eps^\KPZ_+$, where $\Eps^\KPZ_+ = \{\ft_0\}$. The rule in this case is defined as 
\begin{equ}
	\Rule^{\KPZ}(\fl_0) = \{[]\}\qquad\text{and}\qquad \Rule^{\KPZ}(\ft_0) = \{[], \fl_0, \partial\fl_{0}, \partial\fl_{0}^2 \}\;.
\end{equ}
These definitions are related to those from Section~\ref{sec:first} as $\fL^\KPZ \subset \fL$, $\Eps^\KPZ 
\subset \Eps$ and $\Rule^{\KPZ}(\ft_0) \subset \Rule(\ft_0)$. This implies that the regularity structure 
generated by $\Rule^{\KPZ}$ is a sector of the regularity structure generated by $\Rule$ 
(see~\cite[Def.~2.5]{Regularity} for the definition of a sector). 

\subsubsection{The rule for the qEW model~\eqref{eq:QuenchedEW_rescaled}}
\label{sec:qEW_rule}

In the case of the equation~\eqref{eq:QuenchedEW_rescaled} we have $d \geq 1$ and a parabolic 
scaling $\s = (2, 1, \ldots, 1)$. As we explain after Theorem~\ref{thm:qEW}, the cases $d = 1$ and $d \geq 2$ should be treated differently.

In the case $d = 1$ the noise is labelled by $\fL^{\qEW}_- = \{\fl_0\}$ and the equation is 
labelled by $\fL^{\qEW}_+ = \{\ft_0\}$. Mappings $\reg$ and $\deg$ are the same as in $\Rule^\KPZ$. Then $\Eps^\qEW = \Eps^\qEW_- \sqcup \Eps^\qEW_+$, where again $\Eps^\qEW_+ =  \{\ft_0\}$. Finally, the rule is given by 
\begin{equ}
	\Rule^{\qEW}(\fl_0) = \{\ft_0^n : n \geq 0\}\;, \qquad \Rule^{\qEW}(\ft_0) = \{[], \fl_0 \}\;.
\end{equ}

In the case $d \geq 2$ we consider the system \eqref{eq:QuenchedEW_system} with noises
 labelled by $\fL^{\qEW}_- = \{\fl_0, \fl_1\}$ and components  
labelled by $\fL^{\qEW}_+ = \{\ft_0, \ft_1\}$. Then $\fL^\qEW = \fL^\qEW_- \sqcup 
\fL^\qEW_+$ with $\deg \fl_0 = -\frac{d+2}{2} - \kappa_\star$, $\deg \fl_1 = \deg \fl_0 + \beta$, where the constant $\beta$ is fixed in~\eqref{eq:alpha_and_beta}, and $\deg \ft_0 = \deg \ft_1 = 2$. We postulate, $\reg(\fl_k) = \deg (\fl_k) -\kappa$ and $\reg (\ft_k) = \reg(\fl_k) + 2$. With $\alpha > 0$ and $\beta$ defined in~\eqref{eq:alpha_and_beta} it is guaranteed that $\Eps^\qEW_+ = \{\ft_1\}$. Finally, the rule is given by $\Rule^{\qEW}(\ft_0) = \{[], \fl_0 \}$, $\Rule^{\qEW}(\ft_1) = \{[], \fl_1 \}$ and
\begin{equ}
	\Rule^{\qEW}(\fl_0) = \Rule^{\qEW}(\fl_1) = \{\ft_0^n : n \geq 0\}\;.
\end{equ}

\subsection{Nonlinearities}
\label{sec:nonlinearities}

We introduce a family of commuting indeterminates $\fX = (\fX_o)_{o \in \Eps}$ and denote by 
$\SP$\label{lab:SP} the real algebra of smooth functions on $\R^\Eps$, i.e.\ functions of $\fX$. We also define $\SQ_+ \eqdef 
\SP^{\fL_+}$, which corresponds to the $\fL_+$ components of the system of SPDEs under 
consideration.\footnote{The set $\SQ_+$ is denoted by $\mathring{\SQ}$ in~\cite{BCCH}} For $F \in \SP$ 
denote $\Eps(F)$ for the minimal subset of $\Eps$ such that $F$ depends only on component $\fX_o$ with 
$o \in \Eps(F)$. Depending on the context, the indeterminates $\fX_o$ will serve as placeholders for either 
an abstract expansion for the component of the solution\slash noise indexed by $o$, or for a reconstruction of 
that expression.

We define two families of differential operators on $\SP$: $\{D_o\}_{o \in \Eps}$ and $\{\d_i\}_{i=0}^d$. For 
$o \in \Eps$ the operator $D_o : \SP \to \SP$\label{lab:Do} simply denotes differentiation with respect to $\fX_o$. Moreover, for every $0 \leq i \leq d$ and $(\ft,p) \in \Eps$ we set $\d_i 
\fX_{(\ft,p)} = \fX_{(\ft,p+e_i)}$, where $e_i$ is the $i^{\text{th}}$ element of the canonical basis of 
$\R^{d+1}$, and extend it to $\d_i : \SP \to \SP$ by imposing the chain rule
\begin{equ}[eq:chain_rule]
	\d_i F \eqdef \sum_{o \in \Eps(F)} \d_i \fX_o\, D_oF\;.
\end{equ}
We also use the shorthand notation $\d_i^k$ for $k \in \N$ consecutive applications of $\d_i$. For a 
multi-index $k \in \N^{d+1}$ we use the standard notation $\d^k \eqdef \prod_{i = 0}^{d} \d^{k_i}_i$ for 
the mixed derivatives which are well-defined since the $\d_i$ all commute. 
Moreover, for a multi-index $\alpha \in \N^\Eps$ we use the shorthands
\begin{equ}
	\fX^\alpha \eqdef \prod_{o \in \alpha} \fX_o\qquad\text{and}\qquad D^\alpha \eqdef \prod_{o \in \alpha} D_o\;.
\end{equ}

\begin{remark}
	Note that in~\cite{BCCH} instead of $\Eps$ only the set $\CO$ was considered for labelling the 
	indeterminates $\fX$. There are several reasons to consider not only the labels that correspond to 
	the components of the equations (and their derivatives), but also the ``noise labels'' $\fL_{-}\times \N^{d+1}$. First, we are going to use the algebraic formalism of~\cite{BHZ}, where both noises and 
	kernels correspond to labels on the edges of the trees of $\fT$. Thus having a language that does 
	not distinguish at the algebraic level $\fL_{+}\times \N^{d+1}$ from $\fL_{-}\times \N^{d+1}$ will 
	be useful. Second, and most important is that we want to consider trees where one can ``draw'' 
	edges above the noises. This is needed in order to reflect the fact that the 
	noises in~\eqref{eq:system_intro} are inhomogeneous.
\end{remark}

Given a rule $\Rule$ satisfying Assumption~\ref{ass:rule}, we would like to define those functions which 
conform in some sense to this rule.

\begin{definition}\label{def:conform}
	We say that $F \in \SQ_+$ \emph{conforms} to a rule $\Rule$ satisfying Assumption~\ref{ass:rule} if 
	for each $\ft \in \fL_+$ and each $\alpha \in \mcbP(\Nodes) \backslash \{\emptyset\}$ such that 
	$\alpha \notin \Rule(\ft)$ we have $D^\alpha F_\ft = 0$. We define $\SQ(\Rule)\subset \SQ_+$ to be 
	the set of all functions $F$ conforming to $\Rule$.\label{lab:QRule}
\end{definition}

\begin{example}\label{ex:Examplenonlin}
	To show how~\eqref{eq:main2} fits into the above formalism, we use the 
	objects defined in Section~\ref{sec:first}. We set $\fX_{(\ft_i,p)} = \d^p h_i$ and $\fX_{(\fl_i,0)} = \xi_i$, and we use the shorthand notation $\fX_{i,m} = \fX_{(\ft_i,(0,m))}$ for integer $m \geq 0$. 
	Then the right-hand sides of the first three equations in~\eqref{eq:main2} can be written as
	\begin{equs}[eq:abstract_functions]
		F_{\!\ft_i} &= F_{1,\eps}(\fX_{1, 1}) \fX_{i+1, 2} + \lambda\, \fX_{\lceil i/2\rceil, 1} \fX_{\lfloor i/2\rfloor, 1} + F_{2,\eps}(\fX_{1, 1}) \fX_{i, 1} \\
		 &\hspace{2cm}+ \sigma \fX_{(\fl_i, 0)} + \sigma_1 (\fX_{1, 1})^2 \fX_{(\fl_{i+1}, 0)} + F_{3,\eps}(\fX_{i+1, 1}) \fX_{(\fl_{i+2}, 0)}\;,
	\end{equs}
for $i = 0,1,2$, where we do not consider the renormalisation constants $C_{i, \eps}$. It is clear that we have $F_{\!\ft_i} \in \SP$. We do not define a function $F_{\!\ft_3}$ because the corresponding equation will be solved classically. 
\end{example}

\begin{example}
Using notation from Example~\ref{ex:Examplenonlin}, the right-hand side of the KPZ 
equation~\eqref{eq:KPZ} can be written as $F_{\!\ft_0} = \lambda\, (\fX_{0, 1})^2 + \sigma \fX_{(\fl_0, 0)}$.
\end{example}

\begin{example}\label{ex:qEWnonlin}
The right-hand side of~\eqref{eq:QuenchedEW_rescaled} is given simply by $F_{\!\ft_0} = \fX_{(\fl_0, 0)}$.
\end{example}

One can readily check that the nonlinearities described in the preceding examples conform to the rules described in the respective Sections~\ref{sec:first}--\ref{sec:qEW_rule}.

\subsection{Inhomogeneous noises and Green's functions}
\label{sec:general_system}

Using the nonlinearities introduced in the preceding section, we would like to introduce a general 
system of SPDEs, driven by different inhomogeneous noises, which we are going to solve using the 
framework of regularity structures. 

For this, let $\fl \in \fL_{-}$ label smooth noises $\xi_\fl$ and let $\ft \in \fL_+$ label equations for smooth functions $u_{\ft} : \R^{d+1} \to \R$ with $d \geq 1$. For any $z$ in the domain, we write $\bu(z) \eqdef (\d^p u_{\ft}(z) : (\ft, p) \in \CO) \in \R^{\CO}$, where $\d^p$ is a mixed space-time derivative. 
Furthermore,  recalling the definition of the set~\eqref{eq:Eps_ell}, for each $\fl  \in \fL_{-}$ we consider fixed $\ba_{\fl} \eqdef (a_{\fl, o})_{o \in \Eps_+(\fl)} \in \mathcal{L}(\R^{\Eps_+(\fl)}, \R^{d+1})$ and for each $\bu \in \R^{\Eps_+}$ we write
\begin{equ}[eq:a-product]
	\ba_{\fl} \cdot \bu \eqdef \sum_{o \in \Eps_+(\fl)} a_{\fl, o} u_{o}\;. 
\end{equ}
Using this notation, we define for every $c \in L(\R^{d+1})$ the tuple of inhomogeneous noises 
\begin{equ}
\bxi^{\bu,c}(z) \eqdef \bigl((\d^p \xi_{\fl})^{\bu, c}(z) : (\fl, p) \in \Eps \setminus \CO\bigr) \in \R^{\Eps \setminus \CO}\;,
\end{equ}
where each individual noise is defined as
\begin{equ}[eq:non-hom_noise]
(\d^p \xi_{\fl})^{\bu,c}(z) \eqdef (\d^p \xi_{\fl}) \bigl(z + cz + \ba_{\fl} \cdot \bu(z) \bigr)\;.
\end{equ}
We also prefer to write for brevity $\xi^{\bu,c}_{(\fl, p)}(z) \eqdef (\d^p \xi_{\fl})^{\bu,c}(z)$. Once again 
we point out that translation by $c$ is needed to represent a translation by $\eps^2C_\eps$ 
in~\eqref{eq:rescaled_noise}.

Then, for a tuple $F = (F_\ft)_{\ft \in \fL_+}$ conforming to the rule $\Rule$, we consider the system of equations labelled by $\ft \in \fL_+$, 
\begin{equ}[eq:system]
\d_t u_\ft = \SL_\ft u_\ft + F_{\ft} (\bu, \bxi^{\bu,c})\;, \qquad u_\ft(0, \bigcdot) = u_\ft^0(\bigcdot)\;,
\end{equ}
where $\SL_\ft$ is an elliptic differential operator which we specify below, and where we write $(\bu, \bxi^{\bu,c})(z) \subset \R^{\CE}$ for the tuple with the elements
	\begin{equs}\label{eq:uxi}
		(\bu,\bxi^{\bu,c})_o(z) = 
		\begin{cases}
			\d^p u_\ft(z)\,, &\text{for}~o = (\ft,p)\in \CO\;,\\
			(\d^p\xi_\fl)^{\bu,c}(z)\,, &\text{for}~o = (\fl, p) \in \Eps \setminus \CO\;.
		\end{cases}
	\end{equs}
We also use the shorthand notation $F_{\ft} (\bu, \bxi^{\bu,c})(z) = F_{\ft} (z, \bu(z), \bxi^{\bu,c}(z))$. 

Note that~\eqref{eq:uxi} guarantees that both $\xi_\fl$ and $\d^p\xi_\fl$ are translated by 
the same $\ba_\fl\cdot \bu(z)$. If one wishes $\d^p\xi_\fl$ to be translated by some $\ba_{(\fl,p)}\cdot \bu(z)$ with $\ba_{(\fl,p)}\neq\ba_\fl$, then one should view $\d^p\xi_\fl$ as a separate noise $\xi_{\bar\fl} \eqdef \d^p\xi_\fl$ and enlarge $\bar\fL = \fL\cup \{\bar\fl\}$, with $\ba_{\bar\fl} = \ba_{(\fl,p)}$. We would 
like to point out that this might produce renormalisation functions instead of renormalisation constants 
since then $\xi_{\bar\fl}$ is not independent of $\xi_{\fl}$ (see Proposition~\ref{prop:stationarity}).

The differential operator $\SL_\ft$ is assumed to involve only spatial partial derivatives 
$\{\partial_i\}_{i=1}^d$ and is such that the Green's function of $\partial_t-\SL_\ft$ 
satisfies~\cite[Assum.~2.8]{BCCH}, so that it is a regularising kernel of order $\deg \ft$ in the 
sense of~\cite[Assum.~5.1]{Regularity}. 

\begin{example}
The typical example is $\SL_\ft = Q(\nabla_x) - 1$ for a homogeneous polynomial $Q$ of even degree 
$2q$. Then the scaling $\s$ should be taken $\s = (2q, 1, \ldots, 1) \in \R^{d+1}$ and $\deg \ft = 2q$. In 
particular the heat operator with unit mass falls into this framework: $\SL_\ft = \Delta - 1$ making $q=1$ 
and thus $\deg \ft = 2$. This is the operator that we are going to use for both our 
system~\eqref{eq:main2} ($d = 1$) and the quenched Edwards--Wilkinson~\eqref{eq:QuenchedEW}. 
Note that strictly speaking in both of these equations $\SL_\ft = \Delta$ but we can easily add a linear 
term to both sides, which affects neither well-posedness nor renormalisation.
\end{example}

Local subcriticality of the system~\eqref{eq:system} guarantees existence of a regularity structure, built 
in~\cite{BHZ}, and reformulation of the equations~\eqref{eq:system} (with homogeneous noises) in 
terms of modelled distributions on this regularity structure. In order to accommodate the 
inhomogeneous noises however, we will need to consider an associated ``vector-valued'' regularity structure,
the construction of which is recalled in Section~\ref{sec:VectorRS}. Our aim is then to perform renormalisation of the general 
system~\eqref{eq:system} and then apply it to the stochastic mean curvature flow~\eqref{eq:main2} 
and the qEW model~\eqref{eq:QuenchedEW_rescaled}. We will prove existence of a local solution only 
for these particular equations.

\subsection{Vector-valued regularity structures}
\label{sec:VectorRS}

As mentioned in Section~\ref{sec:counterterms}, we would like our regularity structures to include
distributions, ``attached'' to edges of noise types, which will then be mapped to inhomogeneous noises by our model. 
Such a construction does not fall into the 
standard construction of regularity structures given in~\cite{BHZ}, but the formalism of~\cite[Sec.~5]{CCHS} 
(which was already used in \cite{MH17} in a somewhat informal way) is designed precisely to allow this. 

Assume that we are given a \textit{space assignment}, namely a
collection of vector spaces $V = (V_\ft)_{\ft \in \fL}$. (In our cases we will take $V_\ft = \R$ for 
$\ft \in \fL_+$.) Then, given any tree 
$\tau = T^\fm_\ff \in \fT$, it determines a ``symmetric set'' $\scal{\tau}$ 
(with elements given by edges of $\tau$) as well as
a vector space $\scal{\tau}_V = \BF_V \big(\scal{\tau}\big)$, where $\BF_V \colon \TStruc \to \Vec$ 
is the functor described in~\cite[Sec.~5.2]{CCHS}.
In a slightly informal way, $\scal{\tau}_V$ is nothing but the subspace
of $\bigotimes_{e \in E_T} V_{\ft(e)}$ respecting the natural symmetries of the tree $\tau$.

\begin{example}
	For $\tau = \J_{(\ft,p)}[\one]^2\J_{(\bar\ft,\bar p)}[\one]$ with some $(\ft,p), (\bar\ft,\bar p) \in \Eps$ such that $(\ft,p) \neq (\bar\ft,\bar p)$, $\scal{\tau}_V$ is 
	canonically isomorphic to $(V_\ft \otimes_s V_\ft) \otimes V_{\bar \ft}$, where $\otimes_s$ denotes the symmetric 
	tensor product.
\end{example}

Given a rule $\Rule$ satisfying Assumption~\ref{ass:rule} and a space assignment $V$, we then define the space\label{lab:T_V}\footnote{One 
may have expected to have a direct sum in~\eqref{eq:CT_V} instead of a Cartesian product. 
The definition of $\bUpsilon$ 
in~\eqref{eq:bUpsilon} will however be more natural if we allow for infinite sums. 
Since all the  analysis can be carried out in spaces where we restrict ourselves to 
finitely many trees, the difference is mainly cosmetic.}
\begin{equ}[eq:CT_V]
	\CT_V \eqdef \BF_V \Bigl(\prod_{\tau \in \fT(\Rule)} \scal{\tau}\Bigr)	\simeq \prod_{\tau \in \fT(\Rule)} \CT_V[\tau]\;,\qquad \CT_V[\tau] \simeq \scal{\tau}_V\;,
\end{equ}
which we call a $V$-\emph{valued regularity structure}. 
It was shown in \cite{CCHS} how to build a regularity structure in $\TStruc$, which 
$\BF_V$ then allows to transport onto a ``usual'' regularity structure with underlying space
$\CT_V$.
In particular, the machinery developed in~\cite{Regularity, HC, BHZ, BCCH} immediately applies also to 
vector-valued regularity structures.

\begin{remark}\label{rem:quotient}
In our case we have one slight deviation from the construction given in \cite{BHZ,CCHS}, namely the fact that we allow for edges 
of noise type that are not ``terminal'' in our trees. This does not really matter since the construction there does not in principle
need to distinguish between the two types of edges, but it will be convenient for our purpose to guarantee that
$\Deltap_V$ commutes with abstract integration operators corresponding to noises. We implement this by simply quotienting
$\CT_V^+$ by the ideal (which is easily seen to be a Hopf ideal) 
generated by all elements of the type $\CJ_o(\tau)$ with $o \not\in \CO$ and $\tau \in \CT_V$.
From now on, we call this space $\CT_V^+$ again.
\end{remark}

For a vector space assignment $V = (V_\ft)_{\ft \in \fL}$ we are going to use the operators $\tilde\CA^-_V$, $\tilde\CA^+_V$, $\Deltam_V$ and $\Deltap_V$, which are respectively the negative and positive twisted antipodes and coproducts on $\CT_V$ (see \cite{BHZ} for the definitions). Whenever the space $V$ is clear from context, we prefer to omit the subscripts of these operators.

Recall that, for $o = (\ft,p) \in \Eps$, $\J_{o}$ is a map from 
$V_{\ft} \otimes \CT_V$ to $\CT_V$ (which in particular maps
$V_\ft \otimes \CT_V[\tau]$ to $\CT_V[\J_{o}(\tau)]$). Then, given a tree $\tau$ written in the 
form~\eqref{eq:recursive_tree}, we can express a general element of $\CT_V[\tau]$ recursively as
a linear combination of elements of the type 
\begin{equ}[eq:recursive_tree_V]
	\sigma = \X^m \prod_{i=1}^{n} \J_{o_i}[v_i\otimes \sigma_i]\;,
\end{equ}
for some $\sigma_i \in \scal{\tau_i}_V$, $o_i = (\ft_i,p_i) \in \Eps$ and $v_i \in V_{\ft_i}$.

As usual, we write 
$\CT_{V, \gamma} \subset \CT_V$ for the subspace spanned by the $\CT_V[\tau]$ 
with $\deg \tau = \gamma$, and $\CQ_{\gamma}$\label{lab:projectionQ} for the projection onto $\CT_{V,\gamma}$. 
We similarly define $\CQ_{\le \gamma}$, etc.
We will usually denote elements of $\fT(\Rule)$ by $\tau$ and elements of $\CT_V$ by $\sigma$ to 
distinguish between the combinatorial trees and those with vectors assigned to their edges.
Given $\tau\in\fT(\Rule)$, we write $\Proj_\tau: \CT_V \to \scal{\tau}_V$ for the natural projection. 
Similarly, for any subset $Q \subset \fT(\Rule)$ we write $\Proj_Q : \CT_V \to \prod_{\tau \in Q} \scal{\tau}_V$.\label{lab:Proj_Q}

Finally, we define the last type of projection that is going to be used. For a tree $\tau = T^\fm_\ff \in \fT$ we define the \emph{truncation 
parameter} $L(\tau) \eqdef |E^+_\tau| + \sum_{v \in N_T} |\fm(v)|_\s$, where $E^+_\tau \eqdef \{ e \in E_T : \ff(e) \in \CO\}$ is a set of edges of $T$ which do not correspond to noises. For $L \in \N \cup \{\infty\}$ we 
then write $\bp_{\leq L} = \Proj_{\{\tau\in \fT\,:\, L(\tau)\le L\}}$. Moreover, we define a space $\CW_{V, \leq L} = \prod_{\tau \in \fT(\Rule) : L(\tau) \leq L} \scal{\tau}_V$.\footnote{The reason 
for introducing $\bp_{\leq L}$ is that it interacts nicely with negative renormalisation, which does not 
commute with $\CQ_{\leq \gamma}$, see~\cite[Sec. 3.6]{BCCH} and Section~\ref{sec:renorm_SPDE} for more details.}\label{lab:projectionL}

Given another space assignment $(W_\ft)_{\ft \in \fL}$ and a family of linear maps $E_\ft : V_\ft \to W_\ft$, we naturally obtain a linear map $E :\CT_V \to \CT_W$ as 
in~\cite[Rem.~5.19]{CCHS}, which can also be seen in the following recursive way: given $\sigma \in \scal{\tau}_V$ of the form~\eqref{eq:recursive_tree_V}, we obtain an element $E[\sigma] \in \scal{\tau}_W$ by setting
\begin{equ}[eq:recursive_map]
	E[\sigma] = \X^m \prod_{i=1}^{n} \J_{o_i} \bigl[E_{\ft_i}v_i \otimes E[\sigma_i]\bigr]\;,
\end{equ}
and then extending this linearly to all of $\CT_V$. It was shown in \cite{CCHS} that 
this extension is a natural transformation,
which implies in particular that these maps commute with all the operations $\Delta$, $\Deltap$, $\Deltam$, $\tilde\CA^+$ and $\tilde\CA^-$ (see \cite[Rem.~5.19]{CCHS}). 

\subsection{Spaces of abstract derivatives and distributions}
\label{sec:spaces}

In this section we define precisely which vector spaces we use for the regularity structures defined in 
Section~\ref{sec:VectorRS}. In order to solve~\eqref{eq:system}, we would like to assign to each 
noise edge of a tree $\tau \in \fT(\Rule)$ a Dirac delta function $\delta_u$ or its derivative. This means 
that the spaces $V_\ft$ with $\ft \in \fL_-$ should be large enough to contain delta functions and their 
derivatives of high enough 
orders. While this is not a problem from the analytic point of view, it creates technical problems at the
algebraic level since some of the arguments of \cite{CCHS} rely on the spaces $V_\fl$ being finite-dimensional. 

We circumvent this by introducing auxiliary finite-dimensional vector spaces, elements of which should be thought of as constant-coefficient differential operators that will later be applied to $\delta_u$. 
We will then use natural transformations of the form described in \eqref{eq:recursive_map} to switch between
these points of view, see also
Section~\ref{sec:evaluation}. The finite-dimensional auxiliary vector spaces will be used in 
Section~\ref{sec:coherence} to describe the renormalisation of nonlinearities.

Let us start by introducing a space to which delta functions belong. For a fixed $k_\star \geq 2$, 
which is going to be determined later, we define the set $\SB^{k_\star+1}$ as those functions 
$\phi \in \CC_0^{k_\star + 1}(\R)$ that are supported on the ball $\{u \in \R\,:\, |u| \leq 1\}$ and such that 
$\|\phi\|_{\CC^{k_\star + 1}} \leq 1$. Then we define the space $\CB$ as those tempered
distributions $\zeta \in \SS'(\R)$ belonging to the dual of $\CC^{k_\star + 1}_0(\R)$ and for which
\begin{equ}[eq:norm_weighted]
\| \zeta \|_{\CB} \eqdef \sup_{u \in \R} \sup_{\lambda \in (0,1]} \sup_{\phi \in \SB^{k_\star+1}} \w(u) \lambda^{k_\star} |\scal{\zeta,\phi^\lambda_u}| < \infty\;.
\end{equ}
Here, $\w(u) \eqdef 1 + |u|^{2+\eta_\star}$ is a weight function, and where we use a recentred and 
rescaled function $\phi^\lambda_u(v) \eqdef \lambda^{- 1} \phi(\lambda^{-1}(v - u))$.
We leave the value of $\eta_\star$ free for the moment. It will be fixed to a sufficiently large
value later on.

Note that $u \mapsto \delta^{(k)}_u$ is a Lipschitz continuous function $\R \to \CB$ for all $k \leq k_\star - 2$ 
where $\delta^{(k)}_u$ is the $k^{\text{th}}$ distributional derivative of the Dirac delta function at $u$.

\begin{remark}\label{rem:weights}
The role of the weight $\w$ is twofold. First, it avoids concentration of mass at infinity since 
we have for example $ \| \delta_u \|_{\CB} \simeq |u|^{2+\eta_\star}$ as $|u| \to \infty$. Second, the function $1/\w(u)$ is 
integrable on $\R$, which allows to pair any $\zeta \in \CB$ with functions $\psi \in 
\CC^{k_\star + 1}(\R)$ such that $\lim_{u \to \pm \infty} |D^m \psi(u)| / |u|^{\delta} = 0$, for some $\delta \in (0,1)$ and all $m \le k_\star+1$ (See~\cite[Eq.~3.15]{Regularity} for the proof of a very similar statement.)
\end{remark}

Write now $\CB^{\otimes \ell}$ for the $\ell$-fold projective tensor product of 
$\CB$ and let $\CB_\star^{(\ell)}$ be the completion of $\CC_0^\infty(\R^\ell)$ under the 
norm dual to that of $\CB^{\otimes \ell}$. (We follow the usual convention identifying 
the tensor product of $\ell$ distributions on $\R$ with a distribution on $\R^\ell$.) The 
following simple result is then useful.
\begin{lemma}\label{lem:usefulB}
Given any $\alpha > 0$, we have the bound
\begin{equ}[e:boundBstar]
\|f\|_{\CB_\star^{(\ell)}}
\lesssim \sum_{\bu \in \Z^\ell} \Bigl( \prod_{j \le \ell} \f1{1+|\bu_j|^{2+\eta_\star}}\Bigr) \sup_{|k| \le \ell k_\star} \Big(|\d^k f(\bu)| + \sup_{|\bv-\bu| \le 1\atop |\bar\bv-\bu| \le 1} {|\d^{k} f(\bv)- \d^{k} f(\bar\bv)| \over|\bv-\bar\bv|^\alpha}\Bigr)\;.
\end{equ}
Here the variables $\bv$ and $\bar\bv$ take values in $\R^\ell$.
\end{lemma}
\begin{proof}
Choose $f$ such that the right-hand side of \eqref{e:boundBstar} equals $1$,
fix a wavelet basis $\psi_\bu^{i,n}$, $\phi_\bu$ on $\R^\ell$ of regularity $k_\star + 1$ (with
basis elements scaled in such a way that their $L^1$ norm \dash not the $L^2$ norm \dash is of order $1$)
and write
\begin{equ}
f =  \sum_{n \ge 0} 2^{-n\ell} \sum_{\bu \in 2^{-n}\Z^\ell}\sum_{i\in I} f_\bu^{i,n} \psi_\bu^{i,n} + \sum_{\bu \in \Z^\ell} f_\bu \phi_\bu\;,
\end{equ}
where $I$ is some fixed finite set (see for example \cite[Sec.~3.1]{Regularity}). The normalisation chosen
here is such that $f_\bu^{i,n} = \scal{f,\psi_\bu^{i,n}}$ and the sum over $\bu$ is a Riemann sum approximation
of an integral.
It also follows from standard estimates (see for example \cite{MR1228209}) that
\begin{equ}
|f_\bu^{i,n}| \lesssim 2^{-\ell k_\star n} \sup_{|k| \le \ell k_\star} \Big(|\d^k f(\bu)| + \sup_{|\bv-\bu| \le C2^{-n} \atop |\bar\bv-\bu| \le C2^{-n}} {|\d^{k} f(\bv)- \d^{k} f(\bar\bv)| \over|\bv-\bar\bv|^\alpha}\Bigr)\;.
\end{equ}
In particular, the sum of $2^{-n\ell}|f_{\bar \bu}^{i,n}|$ over all $2^{n\ell}$ dyadic cubes $\bar \bu$ at scale $2^n$ that are
closest to some fixed $\bu \in \Z^d$ is bounded by the same expression, but with the condition $|\bv-\bu| \le C2^{-n}$
replaced by $|\bv-\bu| \le 1$ and similarly for $\bar\bv$.

It follows at once that we have the bound
\begin{equ}
2^{-n\ell}\sum_{\bu \in 2^{-n}\Z^\ell}  \Bigl( \prod_{j \le \ell} \f1{1+|\bu_j|^{2+\eta_\star}}\Bigr) |f_\bu^{i,n}| \lesssim 2^{-(\ell k_\star+\alpha) n} \;,
\end{equ}
and similarly for $f_\bu$. On the other hand, it follows immediately
from the scaling properties of the wavelet basis and the definition of $\CB$ that 
\begin{equ}
|(\zeta_1\otimes \cdots \otimes \zeta_\ell)(\psi_\bu^{i,n})| \lesssim 2^{k_\star n \ell} \prod_{j \leq \ell} \f{\|\zeta_j\|_\CB}{1+|\bu_j|^{2+\eta_\star}}\;.
\end{equ}
Summing over $\bu$ and $n$, 
the required bound $|\zeta(f)|\lesssim 1$ then follows when $\zeta$ is a pure tensor.
The general case follows from the definition of the projective tensor product.
\end{proof}

Given a normed space $E$, we also write $\CB_\star^{(\ell)}(E)$ for the space of bounded linear maps
$\CB^{(\ell)} \to E$ endowed with its operator norm. Note that we have a canonical embedding of
$\CC_0^\infty(\R^\ell,E)$ into $\CB_\star^{(\ell)}(E)$. We then have the following result:

\begin{corollary}\label{cor:useful}
There exists $p > 0$ such that, for every $f \in \CC_0^\infty(\R^\ell,E)$, one has
\begin{equ}
\E \|f\|_{\CB_\star^{(\ell)}(E)}^p \lesssim \sup_{\bu \in \R^\ell} \sup_{|\alpha| \le \ell k_\star+1} {\E |\d^\alpha f(\bu)|_E^p \over 1+|\bu|^{\eta_\star p}}\;.
\end{equ}
\end{corollary}

\begin{proof}
This is a straightforward consequence of the lemma, combined with Kolmogorov's continuity test,
the fact that the function $\bu \mapsto (1+|\bu|)/\prod_{j \le \ell}(1+|\bu_j|)$ is bounded, and the fact that the weight $\prod_{j \leq \ell} 1/(1+|\bu_j|^2)$ is summable.
\end{proof}

We now introduce auxiliary \textit{finite-dimensional} vector spaces, whose elements are interpreted as constant-coefficient
differential operators. For 
each $\fl \in \fL$,  we write $\CD_\fl$ for the free unital algebra generated by $\Eps_+(\fl)$, with generators written as $\{\fd_o\}_{o \in \Eps_+(\fl)}$ \label{lab:fd}
and unit denoted by $\one_\fl$ (recall that $\Eps_+(\fl)$ is defined in~\eqref{eq:Eps_ell}), quotiented by the ideal
generated by monomials of degree strictly greater than $k_\star$.\footnote{The reason for not considering abstract derivatives $\fd$ with $\deg \fd > k_\star$ comes from the fact that $\sigma \in \scal{\tau}_\CD$ which contains such $\fd$ cannot 
appear in the expansion of the solution to our equations. This is because~\eqref{eq:Xi_hat_new} implies 
that in this case $\deg \tau$ would itself be too big.}
Here, the degree of a monomial $\fd^\alpha$ with $\alpha \in \N^{\Eps_+(\fl)}$ is
given by $\deg(\fd^\alpha) = |\alpha|_\infty \eqdef \max\{\alpha_o\,:\, o \in \Eps_+(\fl)\}$.
 We identify $\CD_\fl$ with the vector space generated by
\begin{equ}[eq:D-ell-k]
\fD^\fl_{k_\star} \eqdef \{\fd^\alpha\,:\, \deg(\fd^\alpha) \le k_\star\}\;.
\end{equ}
The reason for writing $o \mapsto \fd_o$ for the canonical embedding $\Eps_+(\fl) \hookrightarrow \CD_\fl$
is that one should think of $\fd_o$ as representing differentiation in the $o$-th direction on
$\R^{\Eps_+(\fl)}$. In fact, we will 
make use in Section~\ref{sec:evaluation} of an ``evaluation map'' which maps $\one_\fl$ 
to a Dirac delta function located at some point of $\R^{\Eps_+(\fl)}$ and $\fd^\alpha$ to 
its corresponding derivative. Note that since
we set $\Eps_+(\fl) = \emptyset$  whenever $\fl \in \fL_+$ by convention, we 
automatically have $\CD_\fl \simeq \R$ in that case.
 
With these definitions at hand, we define two regularity structures $\CT_\CD$ and $\CT_\CB$\label{lab:TD_TB} according to~\eqref{eq:CT_V}, with $\CD_\ft$ as just described and \label{lab:Bell}
$\CB_\ft  \eqdef  \CB^{\otimes \Eps_+(\ft)}$
for the space of distributions $\CB$ defined above~\eqref{eq:norm_weighted} and where $\CB^{\otimes A}$ denotes the $A$-fold tensor product of $\CB$ with itself, completed with respect to the projective cross norm. We also postulate 
that $\CD_o = \CD_\ft$ and $\CB_o = \CB_\ft$ for $o = (\ft,p) \in \Eps$.

We need to introduce the regularity structure $\CT_{\CB}$ to be able to work with the object \eqref{eq:Xi_hat}. The latter is an abstract description of the Taylor expansion of the driving noise, where each abstract instance of the noise $\Xi_i$ should be equipped with a suitable Dirac delta function which tells at which point the noise is evaluated. The regularity 
structure $\CT_\CB$ will be used later in the analytic part and will allow us to set up a Picard iteration on the space of modelled distributions in $\CT_\CB$ (see Sections~\ref{sec:application}). On the other hand, $\CT_\CD$ is the regularity structure on which most of the algebra and renormalisation is going to take 
place (see Sections~\ref{sec:coherence} and~\ref{sec:renorm_SPDE}). It is advantageous to work with $\CT_\CD$, because the space $\CD$ is finite dimensional (in contrast to $\CB$) and the previously developed theory can be used for $\CT_\CD$. Later we will define an evaluation map, which gives a correspondence between elements of the two regularity structures. 

We extend now the differentiation operations defined on $\SP$ in Section~\ref{sec:nonlinearities} to 
$\CD_\fl$ as follows. For $o \in \Eps$ and $\fd \in \fD^\fl$ we set
\begin{equs}[eq:D_o]
	D_o \fd \eqdef \begin{cases}
		\fd_o \cdot \fd & \text{if $o \in \Eps_+(\fl)$}\,,\\
		0 & \text{otherwise}\,,
	\end{cases} \qquad 
\end{equs}
and we extend them to $\CD_\fl \otimes \SP$ by the Leibniz rule 
\minilab{e:defDer}
\begin{equ}[e:extendDo]
D_o (\fd \otimes F) \eqdef \fd \otimes D_o F + D_o\fd \otimes  F \;.
\end{equ}
Definition \eqref{eq:D_o} should be understood in light of the ``evaluation map'' 
mentioned just after \eqref{eq:D-ell-k}. If we take $\fd^\alpha \in \fD^\fl$ to represent $\delta^{(\alpha)}$,
then $\fd_o \cdot \fd^\alpha$ does indeed represent $D_o \delta^{(\alpha)}$. The rule that $D_o \fd = 0$
for  $o \not\in \Eps_+(\fl)$ can be interpreted as stating that distributions on $\R^{\Eps_+(\fl)}$ are extended to
distributions on $\R^{\Eps}$ by tensorising with the constant function $1$, so that their derivatives in directions other than those in $\Eps_+(\fl)$ do vanish.
Definition~\eqref{e:extendDo} corresponds of course simply to the standard Leibniz rule.

Moreover, for each $i \in \{0,\dots, d\}$ we define $\d_i \colon \CD_\fl \otimes \SP \to \CD_\fl \otimes \SP$ by
\minilab{e:defDer}
\begin{equ}[eq:d_i]
\d_i (\fd \otimes F) \eqdef \fd \otimes \d_i F + \sum_{o \in \Eps} D_o\fd \otimes  F \d_i\fX_{o}\;.
\end{equ}
Having in mind the definition of $\d_i\fX_{o}$, provided in Section~\ref{sec:nonlinearities}, this identity
 is again just the Leibniz rule for the differential operator $\d_i$. The above sum is finite by~\eqref{eq:D_o} and the fact that $\Eps_+(\fl)$ is finite. We note that definition~\eqref{eq:d_i} gives a natural 
restriction of $\d_i$ to $\CD_\fl$ as $\d_i \fd = \d_i (\fd \otimes 1)$. 

\begin{example}\label{ex:derivatives}
Using the definitions from Section~\ref{sec:first} and Example~\ref{ex:Examplenonlin}, since $\Eps \cap \Rule(\fl_k) = \{(\ft_0, 0)\}$, for $i = 0, 1$ we have
\begin{equ}
\d_i \one_{\fl_k} = D_{(\ft_0, 0)} \one_{\fl_k} \otimes \fX_{(\ft_0, e_i)} =\fd_{(\ft_0, 0)} \otimes \fX_{(\ft_0, e_i)}\;,
\end{equ}
where $e_i$, $i = 0, 1$, are the elements of the canonical basis of $\R^2$. Moreover, \eqref{eq:d_i} yields 
\begin{equ}
\d_i^2 \one_{\fl_k} = \d_i \bigl(\fd_{(\ft_0, 0)} \otimes \fX_{(\ft_0, e_i)}\bigr) = \fd_{(\ft_0, 0)} \otimes \fX_{(\ft_0, 2e_i)} + \fd^2_{(\ft_0, 0)} \otimes \fX^2_{(\ft_0, e_i)}\;.
\end{equ}
\end{example}

\subsection{The evaluation map}
\label{sec:evaluation}

As already mentioned in the preceding section, the regularity structure $\CT_\CB$ will be used for the 
analytic part (solving PDEs) and the regularity structure $\CT_\CD$ will be used in the algebraic part 
(renormalisation). Of course, we would like to somehow connect these in order to use both parts for 
the system of equations~\eqref{eq:system_intro}. This connection is given by evaluation maps $\Ev_\bu$ 
mapping a derivative $\fd^\alpha$ to $\d^\alpha \delta_\bu$.

Given $\bu \in \R^{\Eps_+}$ and $\ft \in \fL_-$, the \emph{evaluation map} $\Ev_\bu \colon \CD_\ft \to \CB_\ft$ is defined by\label{lab:Eval}
\begin{equ}[eq:Eval]
	\Ev_\bu[\, \fd^\alpha] = \delta^{(\alpha)}_{\bu} \eqdef \d^\alpha \delta_{\bu} = \bigotimes_{o \in  \Eps_+(\ft)} \d^{\alpha_o}\delta_{\bu_o}\;,
\end{equ}
for every $\fd \in \CD_\ft$ and $\alpha \in \N^{\Eps_+(\ft)}$. For $\ft \in \fL_+$, we set $\Ev_\bu = \id \colon \CD_\ft \to \CB_\ft$ (both spaces are equal to $\R$ in this case). Note that if $\Eps_+(\ft) = \emptyset$, 
then $\CB_\ft \simeq \R$ and the empty tensor product above is just $1$.
By~\eqref{eq:recursive_map}, this determines linear maps $\Ev_\bu\colon \CT_\CD \to \CT_\CB$.

Whenever it is more convenient, we sometimes write $\Ev(\bu,\bigcdot)$ instead of $\Ev_\bu$. Also, with a little 
abuse of notation, we will write $\Ev_\bu$ as a shorthand for $\Ev_\bu \otimes \id \colon \CD_\ft  \otimes \SP \to \CB_\ft  \otimes \SP$.

\begin{example}
Continuing Example~\ref{ex:derivatives}, for $\bh \in \R^{\Eps_+}$ we denote $h = h_{(\ft_0, 0)}$. Then we have
\begin{equ}
	\Ev_\bh \bigl [\d_i\one_{\fl_k}\bigr] = \delta^{(1)}_{h} \otimes \fX_{(\ft_0, e_i)}\;, \qquad \Ev_\bh \bigl [\d_i^2\one_{\fl_k}\bigr] = \delta^{(1)}_{h} \otimes \fX_{(\ft_0, 2 e_i)} + \delta^{(2)}_{h} \otimes \fX^2_{(\ft_0, e_i)}\;.
\end{equ}
\end{example}

Let $u \in \R$ we define a translation operator $T_u : \CB \to \CB$ by setting $\scal{T_u\mu, \varphi} \eqdef \scal{\mu, \varphi(\cdot + u)}$. For $\bu \in \R^{\Eps_+}$ and $\ft \in \fL_-$ the translation operator $T_\bu :\CB_\ft \to \CB_\ft$ is simply a tensor product $T_\bu = \bigotimes_{o \in \Eps_+(\ft)} T_{u_o}$. By postulating that $T_\bu$ is simply an identity on $\CB_\ft$ for $\ft \in \fL_+$ we can further lift $T_\bu$ to a linear operator $\CT_\CB \to \CT_\CB$ similarly as for the evaluation map. Translation operator acts naturally on the evaluation map: for $\bv,\bu \in \R^{\Eps_+}$
\begin{equ}[eq:translation]
	T_\bv \circ \Ev_\bu = \Ev_{\bu+\bv}\,,
\end{equ}
which is an easy consequence of the fact that $T_v \delta^{(n)}_u = \delta^{(n)}_{u+v}$.

\section{Renormalisation of nonlinearities}
\label{sec:coherence}

We now describe the action of the renormalisation group $\fR$ constructed in~\cite{BHZ}, on 
equations of the type~\eqref{eq:main2}. Our methodology is similar to the construction given 
in~\cite[Sec.~6]{CCHS}, but the results of that article do not quite apply to our situation 
due to the $h$-dependence of the noise.
As in \cite{CCHS}, we will give a recursive construction of the map 
$\fT(\Rule) \ni \tau \mapsto \Upsilon[\tau] \in \scal{\tau}_\CB$ that simultaneously describes 
the counterterm generated by the renormalisation constant associated to the tree $\tau$
and the corresponding term 
appearing in the generalised Taylor expansion of the solution. In our case, $\Upsilon$ must of course 
take into account 
the nature of the operators $\hat \Xi$ described in~\eqref{eq:Xi_hat}. Once again, we 
will use the spaces of abstract derivatives $\CD^\fl_{k_\star}$ to implement this.

Starting from this section we adopt the following abuse of notation. Assume that we are given a collection of 
vector spaces $(B_i)_{i=1}^n,$ such 
that $B_m \neq B_\ell$ for $m\neq \ell$, and a linear map $T : B_j \to B_j$ for some $j \in \{1,\dots, n\}$. Then 
we extend $T$ to a map on 
$\bigotimes_{i = 1}^n B_i$ by postulating that
\begin{equ}[eq:tensor_notation]
	T \bigl(\otimes_{i=1}^n b_i \bigr) = \bigl(\otimes_{i=1}^{j-1} b_i \bigr) \otimes T(b_j) \otimes \bigl(\otimes_{i=j+1}^{n} b_i \bigr)\;,
\end{equ} 
for all $b_i \in B_i$, i.e.\ $T$ acts as the identity on factors other than the $j^{\text{th}}$. 

\subsection{Definition of $\Upsilon$}

Given $F \in \SQ_+$ (recall that $\SQ_+$ was defined in Section~\ref{sec:nonlinearities}), where $F_\ft$ 
is interpreted as describing the right-hand side of the $\ft^{\text{th}}$ component of a system of SPDEs, we define 
a tuple $(\hat F_\ft)_{\ft\in \fL}$ with $\hat F_\ft \in \CD_\ft \otimes \SP$ by setting
\begin{equs}\label{eq:hatF}
	\hat F_\ft \eqdef 
		\one_{\ft} \otimes F_\ft\;,
\end{equs}
with the convention that $F_\ft = 1$ for $\ft \in \fL_-$.
Recall that~\eqref{e:defDer} gives a natural action of the ``partial derivatives'' $\d_i$  
and the ``functional derivatives'' $D_o$ on the spaces $\CD_\ft \otimes \SP$. Note that for $\hat F \in \CD_\ft \otimes \SP$, \eqref{eq:chain_rule} and~\eqref{eq:d_i} can be rewritten as
\begin{equ}[eq:LeibnizFhat]
	\d_i \hat F_\ft = \sum_{o \in \Eps} \bigl(\id \otimes \d_i\fX_{o}\bigr) D_{o}\hat F_\ft\;,
\end{equ}
where $\d_i\fX_{o}$ is identified with the corresponding multiplication operator on $\SP$. 

For $o = (\ft,p) \in \Eps$ and $\tau \in \fT(\Rule)$ of the form~\eqref{eq:recursive_tree}, we define inductively elements $\bar \Upsilon^F_o[\tau] \in \CD_\ft \otimes \scal{\tau}_\CD \otimes \SP$ and $\Upsilon^F_o[\tau] \in \scal{\J_o[\tau]}_\CD \otimes \SP$ by
\begin{equs}[eq:upsilon]
	\bar \Upsilon^F_o[\tau] &= \X^m \Bigl(\prod_{i=1}^{N} \Upsilon^F_{o_i}[\tau_i] \Bigr) \d^m \bigl(D_{o_1}\cdots D_{o_N}\bigr)\hat F_\ft\;,\\
	\Upsilon^F_o[\tau] &= \J_{o}\big[\bar \Upsilon^F_o[\tau] \big]\;.
\end{equs}
In the case $N = 0$, i.e.\ when $\tau = \X^m$, these definitions give $\bar \Upsilon^F_o[\X^m] = \X^m \d^m \hat F_\ft$, thus providing a starting 
point for the induction. However, this definition requires 
some explanation, because it might be difficult to keep track of the spaces. The Leibniz 
rule and the definition of $\hat F_\ft$ yields
\begin{equ}
\d^m \bigl(D_{o_1}\dots D_{o_N}\bigr)\hat F_\ft = 
\begin{cases}
		1 \otimes \d^m \bigl(D_{o_1}\cdots D_{o_N}\bigr) F_\ft  &\text{for}~\ft \in \fL_{+}\;,\\
		\d^m \bigl(\fd_{o_1}\cdots \fd_{o_N}  \otimes 1\bigr) &\text{for}~\ft \in \fL_{-}\;,
	\end{cases}
\end{equ}
and this expression belongs to $\CD_\ft \otimes \SP$. Furthermore, since $\Upsilon^F_{o_i}[\tau_i] \in \scal{\J_{o_i}[\tau_i]}_\CD \otimes \SP$, the product $\prod_{i=1}^{N} \Upsilon^F_{o_i}[\tau_i]$ belongs to 
$\scal{\prod_{i=1}^{N} \J_{o_i}[\tau_i]}_\CD \otimes \SP$. As described 
in~\eqref{eq:tensor_notation}, the multiplication by $\X^m$ is extended to $\scal{\prod_{i=1}^{N} \J_{o_i}[\tau_i]}_\CD \otimes \SP$, 
so that the product $\X^m \prod_{i=1}^{N} \Upsilon^F_{o_i}[\tau_i]$ yields an element of $\scal{\tau}_\CD \otimes \SP$. 
Finally, the product between an element 
$\sigma \otimes f_1 \in \scal{\tau}_\CD \otimes \SP$ and an element $\fd \otimes f_2 \in \CD_\ft \otimes \SP$ is given by
\begin{equ}
	(\sigma \otimes f_1)\; (\fd \otimes f_2) = (\fd \otimes \sigma) \otimes f_1f_2 \;\in\; \CD_\ft \otimes \scal{\tau}_\CD \otimes \SP\;.
\end{equ}
Since $\CD_\ft \simeq \R$  for $\ft \in \fL_+$, one then has
$\CD_\ft \otimes \scal{\tau}_\CD \otimes \SP \simeq \scal{\tau}_\CD \otimes \SP$, so that
in this case we will usually interpret $\bar\Upsilon^F_\ft[\tau]$ as an element 
of $\scal{\tau}_\CD \otimes \SP$.
Regarding the second definition in~\eqref{eq:upsilon}, recall from 
Section~\ref{sec:VectorRS} that $\J_o \colon \CD_\ft \otimes \scal{\tau}_\CD \to \scal{\J_o[\tau]}_\CD$, which we again extend by tensorising with
the identity on the factor $\SP$.

Using the two maps from~\eqref{eq:upsilon}, we set  
\begin{equ}[eq:maps_bold_Upsilon]
\bUpsilon^F_o[\tau] \eqdef \Upsilon^F_o[\tau]/S(\tau)\;, \qquad \bbUpsilon^F_o[\tau] \eqdef \bar\Upsilon^F_o[\tau]/S(\tau)\;, 
\end{equ}
where $S$ is the symmetry factor introduced in~\eqref{eq:symmetry_factor}. In the same way as for $\J_\ft$, 
we set $\Upsilon_\ft  \eqdef  \Upsilon_{(\ft,0)}$, and analogously for $\bar\Upsilon_\ft,\bUpsilon_\ft$ and 
$\bbUpsilon_\ft$.

\subsection{Coherence}
\label{subsec:coherence}

Let $\Poly \eqdef \{\X^k: k \in \N^{d+1}\}$\label{lab:Poly} be the set of all abstract Taylor monomials and let 
$\TT_{\Poly} \eqdef \text{Span}(\Poly)$ be the corresponding polynomial sector inside $\CT_{D}$. 
For every $o \in \Eps$ we define a set of trees 
$\fT_o(\Rule) \subset \fT (\Rule) $ and two subspaces $\TT_o \subset \CT_{D}$ and $\bar\TT_o \subset \CD_o \otimes \CT_{D}$ by
\begin{equs}[eq:jets]
	\fT_o(\Rule)   \eqdef  \{\tau \in \fT(\Rule):\; \tau &= \J_{o}[\btau]\; \text{for some}\; \btau \in \fT(\Rule)\}\,,  \\[0.4em]
	\bar\fT_o(\Rule)   \eqdef  \{\tau \in \fT(\Rule)&:\; \J_{o}[\tau] \in \fT(\Rule)\}\,,  \\[0.4em]
	\TT_o  \eqdef  \TT_{\Poly} \oplus \Im \J_{o}\,,\qquad &\bar\TT_o  \eqdef  \{\sigma \in \CT_{D}\,: \J_{o}[\sigma] \in \TT_o\}\,.
\end{equs}
As before, we use the shorthand notation $\fT_\ft(\Rule) = \fT_{(\ft, 0)}(\Rule)$ and similarly for other sets.
We also set $\SH_\CD  \eqdef  \bigoplus_{\ft \in \fL_{+}} \TT_\ft$ and $\bar\SH_\CD = \bigoplus_{\ft \in \fL_{+}} \bar\TT_\ft$.\label{lab:H-spaces} The space $\TT_\ft$ contains the ``jets'' used to describe the left-hand side 
of the $\ft$-component of~\eqref{eq:system}, while $\bar\TT_\ft$ contains those used to describe 
the right-hand side. We also use the notations $\TT_{\ft, \leq \gamma} \eqdef \CQ_{\leq \gamma} \TT_\ft$ and 
$\bar\TT_{\ft, \leq \gamma} \eqdef \CQ_{\leq \gamma} \bar\TT_\ft$.\label{lab:projected-T}

\begin{remark}\label{rem:sector}
	The result~\cite[Eq.~5.11]{BHZ} shows that $\TT_o$ and $\bar \TT_o$ are sectors \cite[Def.~2.5]{Regularity} of $\CT_D$ of 
	respective regularities $\reg(o) \wedge 0$ and $(\reg(o) - \deg o) \wedge 0$, where the map $\reg$ is 
	from Assumption~\ref{ass:rule}~\ref{R2}. In particular $\TT_o$ is function-like for $o \in \Eps_+$ 
	since $\reg o > 0$.
\end{remark}

Given $F \in \SQ(\Rule)$ we define for every $o \in \Eps$ two elements $\bUpsilon^F_o \in \TT_o \otimes \SP$ 
and $\bar \bUpsilon^F_o \in \CD_o \otimes \bar\TT_o \otimes \SP$ by
\begin{equ}[eq:bUpsilon]
	\bUpsilon^F_o \eqdef \sum_{\tau \in \fT} \bUpsilon^F_o[\tau] \qquad\text{and} \qquad \bbUpsilon^F_o \eqdef \sum_{\tau \in \fT} \bbUpsilon^F_o[\tau]\;,
\end{equ}
where we use the two maps~\eqref{eq:maps_bold_Upsilon}. The reason for using Cartesian 
products in the definition of $\CT_V$~\eqref{eq:CT_V} is precisely because we want to allow such infinite 
sums.

\begin{remark}\label{rem:Upsilon}
	In fact, the summands  
	in~\eqref{eq:bUpsilon} vanish unless $\tau \in \fT$ is such that $\J_o[\tau] \in \fT(\Rule)$. Indeed, if 
	$\J_{o}[\tau] \notin \fT(\Rule)$ then there exist an edge $e$ within $\J_{o}[\tau]$ with the type $\ff(e) = (\bar\ft,\bar p) \in \Eps$ and a set of edge types $\alpha \in \N^{\Eps}$ leaving the edge $e$, such that 
	$\alpha \notin \Rule(\bar\ft)$. This implies that
	\begin{equs}
		D^\alpha \hat F_{\bar\ft} = \begin{cases}
			D^\alpha (1 \otimes F_{\bar\ft}) = 1 \otimes D^\alpha F_{\bar\ft} = 0 & \text{ if } \bar\ft \in \fL_{+}\;,\\
			D^\alpha (\one_{\bar\ft} \otimes 1) = D^\alpha \one_{\bar\ft} \otimes 1 = 0 & \text{ if } \bar\ft \in \fL_{-}\;,
		\end{cases}
	\end{equs}
	where the first line follows from the fact that $F$ conforms to the rule $\Rule$ and the second line follows 
	from~\eqref{eq:D_o}. This, together with the recursive definition of the $\Upsilon$ map~\eqref{eq:upsilon}, 
	implies that $\Upsilon_o^F[\tau] = 0$ and $\bar\Upsilon_o^F[\tau] = 0$. 
\end{remark}

For $i \in \{0,\ldots,d\}$ we define abstract differentiation operators $\SD_i$ on $\CT_{\CD}$ by setting
\begin{equ}
\SD_i \X_j = \delta_{ij}\one\;,\qquad \SD_i\J_{(\ft, p)}[\fd \otimes \sigma] = \J_{(\ft, p+e_i)}[\fd \otimes \sigma]\;,
\end{equ}
and extending it to all of $\CT_{\CD}$ by the Leibniz rule. 
Given $p \in \N^{d+1}$, we also write $\SD^p = \prod_i \SD_i^{p_i}$, noting that these
operators commute, so this is well-defined. Given $U = (U_\ft)_{\ft \in \fL_+} \in \SH_\CD$, we then set 
\begin{equ}[eq:U_coefficients]
U_{(\ft,p)} = \SD^p U_\ft \qquad \text{ and } \qquad  u^U_{(\ft,p)} = \Proj_{\1} U_{(\ft,p)}\;.
\end{equ}

We now define rigorously the operators $\hat \Xi$ that were motivated in~\eqref{eq:Xi_hat}. For $U \in \SH_\CD$ and $o \in \Eps_+$, we can write $U_{o} = u^U_{o} \one + \tilde u_{o}$. Note that by 
Remark~\ref{rem:sector} $\tilde u_{o}$ has components only of strictly positive degrees unless $\tilde u_o = 0$, since $o \in \Eps_+$.
For $o = (\fl,p) \in \Eps\setminus \CO$, we define the operators $\hXi_{\CD,o}$ on $\SH_\CD$
\begin{equ}[eq:Xi_hat_new]
\hXi_{\CD,o}(U)  \eqdef  \sum_{\alpha \in \N^{\Eps_+(o)}}{1\over \alpha!} \J_{o} \bigl[\fd^{\alpha} \otimes \tilde u^{\alpha}\bigr]\;.
\end{equ} 
The operator $\hXi_{\CD,o}$ is the ``abstract'' counterpart of the inhomogeneous 
noise~\eqref{eq:non-hom_noise}, which is really nothing but \eqref{eq:Xi_hat}, but interpreted in $\CT_\CD$. 
In view of \eqref{eq:non-hom_noise}, it may be surprising 
that the constants $a_{\fl, o}$ do not appear explicitly in this 
expression. They will appear later in the choice of a suitable model, see~\eqref{eq:canonical}.
If $\Rule(o) = \{()\}$, we have $\Eps_+(o) = \emptyset$, and $\hXi_{\CD,o}(U) = \J_{o}[\one_\fl \otimes \one]$, which corresponds to a ``classical'' homogeneous noise. If on the other 
hand there is no $\ft \in \fL_{+}$ such that $o \notin \Rule(\ft)$ we simply set $\hat\Xi_{\CD,o} = 0$.

Given $U \in \SH_\CD$, we write
\begin{equ}
	\textbf{U} = (U_{o})_{o \in \CO}\,, \quad \bu^U = (u^U_o)_{o \in \CO}\,, \quad \bXi_\CD^U = \big( \hXi_{\CD,o}(U)\big)_{o\in \Eps\setminus \CO}\,.
\end{equ}
Finally, let us denote $\SW_\CD \eqdef \CT_\CD^{\fL_{+}}$ and for $F \in \SQ(\Rule)$ let us define $\BF_{\CD} : \SH_\CD \to \SW_\CD$ by\footnote{Formally $\BF_\CD(U)$ can be viewed as a Taylor expansion of $F(\textbf{U},\bXi^U_\CD)$ around the point $(\bu^U,0)$.}
	 \begin{equ}[eq:FofU]
	 	\BF_{\CD, \ft}(U)  \eqdef  \sum_{\alpha \in \N^\Eps} \frac{D^\alpha F_{\ft}\big(\bu^U,0\big)}{\alpha!} \bigl(\textbf{U}-\bu^U\one, \bXi_\CD^U\bigr)^\alpha\;.
	 \end{equ}

\begin{lemma}\label{lem:BF}
	 Let the rule $\Rule$ satisfy Assumption~\ref{ass:rule} and $F \in \SQ(\Rule)$. Then $\BF_\CD$ maps $\SH_\CD$ to $\bar{\SH_\CD}$.
\end{lemma}
\begin{proof}
	Let $U\in \SH_\CD$ and $\ft \in \fL_+$. Let $\alpha \in \N^\Eps$ be such that $D^\alpha F_\ft \neq 0$. 
	Then, since $F$ conforms to the rule $\Rule$, Definition~\ref{def:conform} yields $\alpha \in \Rule(\ft)$. We 
	can split $\alpha$ as $\alpha = \alpha_+ \sqcup \alpha_-$ for $\alpha_+ \in \N^{\CO}$ and $\alpha_- \in \N^{\Eps\setminus\CO}$, which allows us to write
	\begin{equs}[eq:U^alpha]
		\bigl(\textbf{U}-\bu^U\one, \bXi_\CD^U\bigr)^\alpha = \bigl(\textbf{U} - \bu^U \1 \bigr)^{\alpha_+}\bXi_{\CD}^{\alpha_-}(U)\;,
	\end{equs}
	where $\bXi_{\CD}^{\alpha_-}(U) \eqdef \prod_{o \in \alpha_-} \hXi_{\CD,o}(U)$. Note that the definition of 
	$\Eps_+(o)$ in~\eqref{eq:Eps_ell}, the fact that $U \in \SH_\CD$ and Assumption~\ref{ass:rule}\ref{R5} 
	imply $\hXi_{\CD,o}(U) \in \CT_\CD$ for every $o \in \Eps\setminus \CO$. Therefore, 
	$\bXi_{\CD}^{\alpha_-}(U) \in \CT_\CD$ and $\bigl(\textbf{U} - \bu^U \1 \bigr)^{\alpha_+}\in\CT_\CD$, and 
	together with the fact that $\alpha \in \Rule(\ft)$ we see that $\bigl(\textbf{U}-\bu^U\one, \bXi_\CD^U\bigr)^\alpha \in \bar\TT_\ft$. This implies that $\BF_\CD$ takes values in $\bar{\SH_\CD}$, as required.
\end{proof}

For $U \in \SH_\CD$ and $\ft \in \fL_+$ we define $U^R$ as the unique element of $\bar\SH_\CD$ such that
\begin{equ}[eq:UR]
	U_\ft = \sum_{p \in \N^{d+1}} \frac{1}{p!} u^U_{(\ft,p)} \X^p + \J_{\ft}[U^R_\ft]\;,
\end{equ}
where we use the coefficients $u^U_{(\ft,p)}$, defined in~\eqref{eq:U_coefficients}. With this at hand we are 
ready to present a definition of coherence. 

\begin{definition}\label{def:coherence}
	We say that $U \in \SH_\CD$ is \emph{coherent} to order $L \in \N \cup \{\infty\}$ with $F \in \SQ(\Rule)$, if for all $\ft \in \fL_+$
	\begin{equ}[eq:coherence]
		\bp_{\leq L} U^R_\ft = \bp_{\leq L} \bbUpsilon^F_\ft\big(\bu^U\big)\;,
	\end{equ}
	where $U^R$ is defined in~\eqref{eq:UR}, the map $\bbUpsilon^F_\ft$ is defined in~\eqref{eq:maps_bold_Upsilon}, and the projection $\bp_{\leq L}$ is defined on p.\pageref{lab:projectionL}.
\end{definition}

Regarding the interpretation of the right-hand side of~\eqref{eq:coherence}, one has
$\bbUpsilon^F_\ft \in \CD_\ft \otimes \bar\TT_\ft \otimes \SP$. For every $\bv \in \R^{\Eps}$, we have a 
natural evaluation map $\iota_\bv\colon \SP \to \R$ such that $\iota_\bv(F) = F(\bv)$. For
$\bu \in \R^{\Eps_+}$ we simply extend it with zeros to an element of $\R^{\Eps}$, and then 
interpret $\bbUpsilon^F_\ft\big(\bu\big)$ as 
\begin{equ}
	\bigl(\id_{\CD} \otimes \id_{\TT} \otimes \iota_{(\bu, 0)}\bigr) \bbUpsilon^F_\ft \in \CD_\ft \otimes \bar\TT_\ft \otimes \R\;.
\end{equ}
Since $\CD_\ft \simeq \R$ for $\ft \in \fL_+$, we have $\CD_\ft \otimes \bar\TT_\ft \otimes \R \simeq \bar\TT_\ft$, which allows to consider $\bbUpsilon^F_\ft\big(\bu^U\big)$ as an element of $\bar\TT_\ft$.

Note that such an extension by zeros of $\bu \in \R^{\Eps_+} \hookleftarrow \R^{\Eps}$ will only be used in $\bbUpsilon$ and not in explicit functions $F \in \SQ(\Rule)$ like in~\eqref{eq:FofU}. This is because noise components of $\bbUpsilon$ most of the times be evaluated at $0$ which might not be the case for some specific functions $F \in \SQ(\Rule)$ (see Lemma~\ref{lem:RF}). The only exception is when Theorem~\ref{thm:renormalised_PDE} where $\bbUpsilon[\tau]$ is evaluated at $(\bu, \xi^{\bu, c})$ for $\tau \in \fT_-(\Rule)$. 

Before stating a generalisation of the Fa\`{a} di Bruno's formula, we make some definitions. We use the 
lexicographic order ``$<$'' on $\N^{d+1}$, so that $0$ is the smallest element. For every $r \in \N$ and $k \in \N^{d+1}$ we define the set 
\begin{equ}
	I(r,k)  \eqdef  \Bigl\{ (\vec q, \vec m) \in \bigl(\N^{d+1}\bigr)^r \times \bigl(\N^{\Eps} \setminus \{0\}\bigr)^{r} : 0 < q_1 < \cdots < q_r,\; \sum_{i = 1}^r |m_i| \cdot q_i = k \Bigr\},
\end{equ}
and for $(\vec q, \vec m) \in I(r,k)$ we use the shorthands $r_{(\vec q, \vec m)} = r$ and $|\vec m| = \sum_{i = 1}^r m_i$. Furthermore, we define $I(k)  \eqdef  \bigsqcup_{r = 0}^\infty I(r,k)$. Then the following is a  version of the Fa\`{a} di Bruno's formula.

\begin{lemma}\label{lem:FaaDiBruno}
For any $k \in \N^{d+1}$, $\fl \in \fL$ and any $\hat F \in \CD_\fl \otimes \SP$ one has
\begin{equ}[eq:FaaDiBruno]
	\d^k \hat F = k! \sum_{(\vec q, \vec m) \in I(k)} \Biggl[ \prod_{1 \leq i \leq r_{(\vec q, \vec m)}} \prod_{(\ft, p) \in \Eps} \frac{1}{m_i[(\ft, p)]!} \Bigl( \frac{1}{q_i!} \fX_{(\ft, p + q_i)}\Bigr)^{m_i[(\ft, p)]} \Biggr] D^m \hat F\;,
\end{equ}
where $m_i[(\ft, p)]$ is the $(\ft, p)$-component of $m_i \in \N^{\Eps}$, and where multiplication of the 
monomials $\fX_{o}$ and the function $\hat F$ is on the level of $\SP$-components.
\end{lemma}

\begin{proof}
For $\fl \in \fL_+$, the definition~\eqref{eq:hatF} implies that the function $\hat F$ does not contain abstract derivatives (i.e.\ $\CD_\fl = \R$) and formula~\eqref{eq:FaaDiBruno} was proved in~\cite[Lem.~A.1]{BCCH}. 

For $\fl \in \fL_- $, formula~\eqref{eq:FaaDiBruno} can be proved by induction over $k \in \N^{d+1}$. More 
precisely, for $k = e_i$ (the latter is an element of the canonical basis of $\R^{d+1}$) the right-hand side 
of~\eqref{eq:FaaDiBruno} is exactly~\eqref{eq:LeibnizFhat}. Furthermore, every $k \in \N^{d+1}$ such that 
$|k| \geq 2$ can be written as $k = \bar k + e_i$, for $\bar k < k$ and some $0 \leq i \leq d$. Then writing 
$\d^k \hat F = \d_i (\d^{\bar k} \hat F)$, formula~\eqref{eq:FaaDiBruno} follows from~\eqref{eq:LeibnizFhat} 
and the induction hypothesis for $\bar k$.
\end{proof}

The following lemma is an analogue of~\cite[Lem.~4.6]{BCCH}. 

\begin{lemma}\label{lem:CoherenceEquivalence}
Let $U \in \SH_\CD$ and $F \in \SQ(\Rule)$ for a rule $\Rule$ satisfying Assumption~\ref{ass:rule}. Then $U$ is coherent to order $L$ with $F$ if and only if, for every $\ft \in \fL_+$,
\begin{equ}
\bp_{\leq L} U^R_\ft = \bp_{\leq L} \BF_{\CD, \ft}(U)\;.
\end{equ}
\end{lemma}

Before proving Lemma~\ref{lem:CoherenceEquivalence}, we obtain a couple of auxiliary results. 
To state them, we consider the set $\CA = \Eps \times \bar \CA$ with $\bar \CA  \eqdef  (\N^{d+1}\setminus\{0\}) \sqcup \fT$ and for $\nu \in \N^{\CA}$ we define $\bar \nu \in \N^{\Eps}$ as
\begin{equ}[eq:nu_bar]
\bar \nu[o]  \eqdef  \sum_{q \in \bar \CA} \nu[(o, q)]\;,
\end{equ}
where $o \in \Eps$. Furthermore, for a collection $\bu^U = (u^U_o)_{o \in \CO}$ and for $\nu \in \N^{\CA}$, 
such that $|\nu| < \infty$, we set 
\begin{equ}
(u^U)^\nu  \eqdef  \prod_{(\ft, p) \in \Eps} \prod_{q \in \N^{d+1}\setminus\{0\}} \Bigl( \frac{u^U_{(\ft, p + q)}}{q!}\Bigr)^{\nu[((\ft, p), q)]}, \quad \bigl(\Upsilon^F\bigr)^\nu  \eqdef  \prod_{o \in \Eps} \prod_{\tau \in \fT} \bigl( \Upsilon^F_{o}[\tau]\bigr)^{\nu[(o, \tau)]}\;,
\end{equ}
as well as $\sigma(\nu)  \eqdef  \sum_{o \in \Eps} \sum_{q \in \N^{d+1}\setminus\{0\}} q \cdot \nu[(o, q)] \in \N^{d+1}$. Note 
that $(u^U)^\nu \in \R$ and $\bigl(\Upsilon^F\bigr)^\nu \in \CT_\CD \otimes \SP$. The following lemma then provides 
another expression for the map $\bar \bUpsilon^F$ defined in~\eqref{eq:bUpsilon}.

\begin{lemma}\label{lem:Upsilon_auxiliary}
For any $\ft \in \fL$ one has\footnote{Here for $(\fd\otimes f) \in \CD\otimes \SP$ we view $(\fd\otimes f)(\bu^U,0) = (\fd\otimes f(\bu^U,0)) \in \CD\otimes \R \simeq \CD$. In particular $(\fd\otimes 1)(\bu^U,0) = (\fd\otimes 1)$.}
\begin{equ}[eq:Upsilon_auxiliary]
\bar \bUpsilon^F_{\ft} (\bu^U,0) = \sum_{\nu \in \N^{\CA}} \frac{D^{\bar \nu} \hat F_{\ft}(\bu^U,0)}{\nu!} (u^U)^\nu \X^{\sigma(\nu)} \cdot \bigl(\Upsilon^F\bigr)^\nu (\bu^U)\;.
\end{equ}
\end{lemma}

\begin{proof}
From the definitions \eqref{eq:bUpsilon} and \eqref{eq:maps_bold_Upsilon} we have
\begin{equ}[eq:Upsilon-bar-expansion]
\bar \bUpsilon^F_{\ft} = \sum_{\tau \in \fT} \bar\Upsilon^F_\ft[\tau]/S(\tau)\;,
\end{equ}
where the symmetry factor $S(\tau)$ is introduced in~\eqref{eq:symmetry_factor}. Let us take $\tau$ of the form \eqref{eq:recursive_tree_new}. Then from the recursive definition~\eqref{eq:upsilon} and the generalised Fa\`{a} di Bruno's formula~\eqref{eq:FaaDiBruno} we get
\begin{equs}
\bar\Upsilon^F_\ft[\tau] 
& = \X^m \Bigl(\prod_{i=1}^{n} \bigl(\Upsilon^F_{o_i}[\tau_i] \bigr)^{\beta_i}\Bigr) \\
&\quad \times m! \sum_{(\vec q, \vec m') \in I(m)} \Biggl[ \prod_{1 \leq i \leq r_{(\vec q, \vec m')}} \prod_{(\ft', p) \in \Eps} \frac{1}{m'_i[(\ft', p)]!} \Bigl( \frac{1}{q_i!} \fX_{(\ft', p + q_i)}\Bigr)^{m'_i[(\ft', p)]} \Biggr] \\
&\hspace{5cm} \times D^{m'} \bigl(D^{\beta_1}_{o_1}\cdots D^{\beta_n}_{o_n}\bigr)\hat F_\ft\;.
\end{equs}
Evaluating these expressions on $(\bu^U,0)$ means replacing $\fX_{(\ft', p + q_i)}$ by $u^U_{(\ft', p + q_i)}$. 

Let us now define $\nu \in \N^\CA$ as follows: $\nu[(o, q)] = m'_i[o]$ if $q = q_i$ for some $1 \leq i \leq r_{(\vec q, \vec m')}$, and $\nu[(o, q)] = 0$ otherwise; $\nu[(o, \tau)] = \beta_i$ if $(o, \tau) = (o_i, \tau_i)$ for some $1 \leq i \leq n$, and $\nu[(o, \tau)] = 0$ otherwise. In particular, $(\vec q, \vec m') \in I(m)$ implies $\sigma(\nu) = m$. Then using the definitions provided above this lemma allow to write the preceding expression as 
\begin{equ}
\bar\Upsilon^F_\ft[\tau](\bu^U,0) = S(\tau) \sum_{\nu \in \N^\CA: \sigma(\nu) = m} \frac{D^{\bar \nu} \hat F_{\ft}}{\nu!} (u^U)^\nu \X^{\sigma(\nu)} \cdot \bigl(\Upsilon^F\bigr)^\nu (\bu^U)\;.
\end{equ}
Plugging this identity into \eqref{eq:Upsilon-bar-expansion} we get the required identity \eqref{eq:Upsilon_auxiliary}.
\end{proof}

The following lemma shows a relation between the maps~\eqref{eq:Xi_hat_new} and~\eqref{eq:bUpsilon}.

\begin{lemma}\label{lem:Xi_Upsilon}
If $U$ is coherent of all orders with $F$, then for every $o \in \Eps \setminus \CO$ one has
\begin{equ}[eq:Xi_Upsilon_coherence]
\hat \Xi_{\CD, o}(U) = \bUpsilon^F_o (\bu^U)\;.
\end{equ}
\end{lemma}

\begin{proof}
Let us assume that $U$ is coherent of all orders with $F$. Then for $\ft \in \fL_+$ and $\J_{\ft}[\tau] \in \fT(\Rule)$, Definition~\ref{def:coherence} and~\eqref{eq:upsilon} yield
\begin{equ}[eq:Xi_Upsilon_coherence_some_identity]
\Proj_{\J_{\ft}[\tau]}(U_\ft) = \Proj_{\J_{\ft}[\tau]} \bigl(\J_{\ft}[U^R_\ft] \bigr) = \Upsilon^F_{\ft}[\tau] (\bu^U)\;,
\end{equ}
where the remainder $U^R_\ft$ is defined in~\eqref{eq:UR}.
This implies that for each $(\ft, p) \in \CO$ we can write
\begin{equ}[eq:U_expansion]
 U_{(\ft, p)} = u_{(\ft, p)} \1 + \sum_{q \in \N^{d+1}\setminus\{0\}} \frac{u_{(\ft, p + q)}}{q!} \X^q + \sum_{\tau \in \fT \,:\, \J_\ft[\tau] \in \fT(\Rule)} \Upsilon^F_{\ft}[\tau] (\bu^U)\;,
\end{equ}
and we denote the last two sums by $\mathring{U}_{(\ft, p)}$ and $\hat{U}_{(\ft, p)}$ respectively. Let us also 
we define $\tilde{U}_{(\ft, p)}  \eqdef  \mathring{U}_{(\ft, p)} + \hat{U}_{(\ft, p)}$. Then for $o \in \Eps \setminus \CO$ definition~\eqref{eq:Xi_hat_new} yields 
\begin{equ}[eq:Xi_Upsilon]
\hat \Xi_{\CD, o}(U)  \eqdef  \sum_{\alpha \in \N^{\Eps_+(o)}} {1\over \alpha!} \J_{o} \Bigl[\fd^{\alpha} \otimes \prod_{(\fb, p) \in \Eps_+(o)} \bigl(\tilde{U}_{(\fb, p)}\bigr)^{\alpha[(\fb, p)]} \Bigr]\;.
\end{equ}
Using the definitions from the beginning of this section, for any $\alpha \in \N^{\Eps_+}$ we can write 
\begin{equs}
\prod_{(\fb, p) \in \Eps_+} \bigl(\tilde{U}_{(\fb, p)}\bigr)^{\alpha[(\fb, p)]} = \sum_{\nu \in \N^{\CA} : \bar \nu = \alpha} \frac{\alpha!}{\nu!} (u^U)^\nu \X^{\sigma(\nu)} \cdot (\Upsilon^F)^\nu (\bu^U)\;.
\end{equs}
Plugging this expression into~\eqref{eq:Xi_Upsilon} and using Lemma~\ref{lem:Upsilon_auxiliary}, we get
\begin{equ}
\hat \Xi_{\CD, o}(U) = \sum_{\nu \in \N^{{\CA}}} \frac{1}{\nu !} \J_{o} \Bigl[ \fd^{\bar \nu} \otimes u^\nu \X^{\sigma(\nu)} \cdot (\Upsilon^F)^\nu (\bu^U)\Bigr] = \J_{o} \Bigl[\bar \bUpsilon^F_o (\bu^U)\Bigr]\;.
\end{equ}
According to~\eqref{eq:upsilon}, the preceding expression equals $\bUpsilon^F_o (\bu^U)$, which is our claim.
\end{proof}

\begin{proof}[of Lemma~\ref{lem:CoherenceEquivalence}]
We will prove a more general result, namely that $U$ is coherent of all orders with $F$ if and only if for every $\ft \in \fL_+$ one has
\begin{equ}[eq:coherence_more_general]
U^R_\ft = \BF_{\CD, \ft}(U)\;.
\end{equ}
Then Lemma~\ref{lem:CoherenceEquivalence} follows from the fact that the truncation parameter 
$L(\tau)$, defined on p.~\pageref{lab:projectionL} satisfies $L(\tau_1\tau_2) = L(\tau_1)+L(\tau_2)$.

Lemma~\ref{lem:Xi_Upsilon}, coherence and~\eqref{eq:U^alpha} imply that for all $\alpha \in \Eps$
\begin{equ}
	\prod_{(\fb, p) \in \Eps} \bigl(\textbf{U} - \bu^U\1, \bXi^U_D \bigr)^{\alpha[(\fb,p)]} = \sum_{\substack{\nu \in \N^{\CA} \\ \bar \nu = \alpha}} \frac{\alpha!}{\nu!} (u^U)^\nu \X^{\sigma(\nu)} \cdot \bigl(\Upsilon^F\bigr)^\nu (\bu^U)\;.
\end{equ}
Then Lemma~\ref{lem:BF} and~\eqref{eq:Upsilon_auxiliary} allow to write
\begin{equ}
\BF_{\CD, \ft}(U) = \sum_{\nu \in \N^{{\CA}}} \frac{D^{\bar \nu} F_{\ft}(\bu^U,0)}{\nu!} (u^U)^\nu \X^{\sigma(\nu)} \cdot (\Upsilon^F)^\nu (\bu^U) = \bar \bUpsilon^F_{\ft} (\bu^U)\;.
\end{equ}
This equals $U^R_\ft$ as we proved in~\eqref{eq:Xi_Upsilon_coherence_some_identity}.

Now, we will prove that~\eqref{eq:coherence_more_general} implies coherence of all orders. For this, let us 
take $V \in \SH_\CD$, which is coherent of all orders with $F$ and such that $\bu^V = \bu^U$. Then, in order 
to prove coherence, it is enough to show that $\Proj_{\J_{\ft}[\tau]}(V_\ft) = \Proj_{\J_{\ft}[\tau]}(U_\ft)$ for all 
$\ft \in \fL_+$ and all $\J_{\ft}[\tau] \in \fT(\Rule)$. We can prove it by induction over the number of edges in 
$\tau$. The case when $\tau$ does not have edges is trivial. Let $\tau$ have $k \geq 1$ edges and let us 
assume that the claim is proved for all trees with at most $k-1$ edges. Then we have 
\begin{equ}
	\Proj_{\J_{\ft}[\tau]}(U_\ft) = \Proj_{\tau} \bigl(\BF_{\CD, \ft}(U)\bigr) = \Proj_{\tau} \bigl(\BF_{\CD, \ft}(V)\bigr) = \Proj_{\J_{\ft}[\tau]}(V_\ft)
\end{equ}
as required, where the first identity follows from~\eqref{eq:coherence_more_general}, the second identity 
follows from the induction hypothesis, and the last identity follows from coherence.
\end{proof}

Similarly to the above definitions, we define the spaces $\SW_\CB, \SH_\CB$ and $\bar\SH_\CB$ by 
changing $\CD$ to $\CB$ 
in their respective definitions. For $U \in \SH_\CB$ and $o \in \CO$, we also define in a similar manner $U_o$, 
$u^U_o$ and write $U_o = u^U_o 
\one + \tilde u_o$ in case $o\in \Eps_+$. We can also define the respective maps on the regularity structure $\CT_{\CB}$. For $o \in \Eps\setminus \CO$ we set
	\begin{equ}[eq:Xi_hat_B]
		\hXi_{\CB, o}(U) \eqdef \sum_{\alpha \in \N^{\Eps_+(o)} : |\alpha|_\infty \leq k_\star} {1\over \alpha!} \J_{o} \Bigl[\; \delta^{(\alpha)}_{\bu^U\restr_{\Eps_+(o)}} \otimes \tilde u^{\alpha}\Bigr]\,.
	\end{equ}
Note the similarity of this definition with the usual way of lifting a nonlinear function
$F$ as in~\cite[Eq.~4.11]{Regularity}, with the role of the derivatives of the function being played by $\J_{(\fl,p)} \bigl[\,\delta^{(\alpha)}_{\bu^U\restr_{\Eps_+(o)}} \otimes \bigcdot\, \bigr]$. We introduce the notation $\textbf{U}$, $\bXi_\CB$ and $\bu^U$ similarly to above and define the function $\BF_{\CB}$ by
\begin{equ}[eq:FofU_B]
	\BF_{\CB, \ft}(U)  \eqdef  \sum_{\alpha \in \N^\Eps} \frac{D^\alpha F_{\ft}\big(\bu^U,0\big)}{\alpha!} \bigl(\textbf{U}-\bu^U \1, \bXi_\CB^U\bigr)^\alpha.
\end{equ}
One can carry out the proof of Lemma~\ref{lem:BF} \emph{mutatis mutandis} to show that $\BF_{\CB}$ 
maps $\SH_\CB$ to $\bar\SH_\CB$.

The next result shows a connection between $\BF_{\CB}$ and $\BF_{\CD}$ through the evaluation map.

\begin{lemma}
For $U \in \SH_{\CD} $ and for each $o \in \Eps\setminus\CO$ and $\ft \in \fL_+$ one has
\minilab{eqs:eval}
\begin{equs}
\Ev_{\bu^U}\hXi_{\CD, o}(U) &=  \hXi_{\CB, o}\bigl(\Ev_{\bu^U}U\bigr)\;, \label{eq:Xi_eval}\\
\Ev_{\bu^U}\BF_{\CD, \ft}(U ) &= \BF_{\CB, \ft}\bigl(\Ev_{\bu^U}U \bigr)\;,\label{eq:F_eval}
\end{equs}
where we use the evaluation maps defined in Section~\ref{sec:evaluation}.
\end{lemma}

\begin{proof}
 Both identities in~\eqref{eqs:eval} follow immediately from the fact that the evaluation map $\Ev$ is a natural transformation from $\CT_\CD$ to $\CT_\CB$. More precisely, using the definition~\eqref{eq:Xi_hat_new}, we obtain
\begin{equs}
\Ev_{\bu^U} \hXi_{\CD, o}(U) &= \sum_{\alpha \in \N^{\Eps_+(o)}}{1\over \alpha!}  \J_{o} \big[\Ev_{\bu^U} [\fd^{\alpha}] \otimes \Ev_{\bu^U}[\tilde u^{\alpha}]\big] \\
& = \sum_{\alpha \in \N^{\Eps_+(o)}: |\alpha|_\infty \leq k_\star}{1\over \alpha!} \J_{o} \Bigl[ \delta^{(\alpha)}_{\bu^U\restr_{\Eps_+(o)}} \otimes \Ev_{\bu^U}[\tilde u]^{\alpha}\Bigr]\,,\label{eq:Xi_for_B_proof}
\end{equs}
where we applied the identity~\eqref{eq:recursive_map} for the evaluation map, as well as used its 
definition~\eqref{eq:Eval} and the fact that $\fd^\alpha = 0$ if $|\alpha|_\infty > k_\star$. Note that the 
identity $\Ev_{\bu^U}\big[U_o\big] = u_o \one + \Ev_{\bu^U}[\tilde u_o]$ for $o \in \CO$, implies 
that~\eqref{eq:Xi_for_B_proof} equals~\eqref{eq:Xi_hat_B}, applied to the element $\Ev(\bu^U, U)$. This gives 
the identity~\eqref{eq:Xi_eval}.

Identity~\eqref{eq:F_eval} follows from~\eqref{eq:Xi_eval} and the above mentioned properties of the evaluation map.
\end{proof}

\begin{lemma}\label{lem:twoFs}
Let $U \in \SH_\CD$ be coherent to order $L$ with $F$. Then one has
\begin{equ}[eq:twoFs]
\bp_{\leq L} \BF_{\CB, \ft} \bigl( \Ev_{\bu^U}U\bigr) = \bp_{\leq L} \Ev_{\bu^U}\big[\bar \bUpsilon^{F}_\ft\big(\bu^U\big)\big]\;,
\end{equ}
where the map $\bar \bUpsilon^{F}_\ft$ is defined in~\eqref{eq:maps_bold_Upsilon} and the projection $\bp_{\leq L}$ is defined on page~\pageref{lab:projectionL}.
\end{lemma}

\begin{proof}
Combining the definition of coherence~\eqref{eq:coherence} and Lemma~\ref{lem:CoherenceEquivalence}, we 
obtain the identity $\bp_{\leq L} \BF_{\CD, \ft}(U) = \bp_{\leq L} \bar \bUpsilon^{F}_\ft\big(\bu^U\big)$, which then implies that $\Ev_{\bu^U} \bigl[\bp_{\leq L} \BF_{\CD, \ft}(U)\bigr] = \Ev_{\bu^U}\big[\bp_{\leq L} \bar \bUpsilon^{F}_\ft\big(\bu^U\big)\big]$. Since $\bp_{\leq L}$ 
and $\Ev_{\bu^U}$ commute,~\eqref{eq:twoFs} follows from~\eqref{eq:F_eval}.
\end{proof}

We now address the choice of $k_\star$ in the definition of $\CB_\fl$ and $\CD_\fl$ from Section~\ref{sec:spaces}. This requires to look at truncated versions of the operators $\hat \Xi_\CB$. 

\begin{lemma}\label{lem:k_star}
	Let $\Rule$ satisfy Assumption~\ref{ass:rule} and let $\gamma > 0$. There exists $k_\star \in \N$ such that for every $o \in \Eps\setminus\CO$ and every $U \in \SH_{\CB}$ the following relation is true
	\begin{equ}[e:idenProjXi]
		\CQ_{\leq \gamma + \deg o} \hat\Xi_{\CB,o}(U) = \sum_{\alpha \in \N^{\Eps_+(o)} } {1\over \alpha!} \J_{o} \Bigl[\; \delta^{(\alpha)}_{\bu^U\restr_{\Eps_+(o)}} \otimes \CQ_{\leq \gamma} \tilde u^{\alpha}\Bigr]\;,
	\end{equ}
	 where the space $\CB$ is defined in Section~\ref{sec:spaces} with this choice of $k_\star$ and where 
	 we interpret both sides to be equal to $0$ if there is no $\ft \in \fL_+$ such that $o \in \Rule(\ft)$.
\end{lemma}

\begin{proof}
The only obstruction to \eqref{e:idenProjXi} is the fact that the sum in 
\eqref{eq:Xi_hat_B} is restricted to $|\alpha|_\infty \leq k_\star$.
It then suffices to take $k_\star$ large enough so that $\CQ_{\leq \gamma} \tilde u^{\alpha} = 0$
as soon as $|\alpha|_\infty > k_\star$, which can be done since
there exists a smallest strictly positive degree $\alpha_0$ in our regularity structure: simply
take $k_\star$ so that $\alpha_0 k_\star > \gamma$.
\end{proof}

\subsection{Renormalisation of the nonlinearities}
\label{sec:renormalisation_of_nonlinearities}

Using the results from the previous section, we can establish a relation between coherence and 
renormalisation. Let us first recall some notions from the renormalisation in regularity structures. Recall that 
$\fT_-(\Rule)$ from~\eqref{eq:unplanted} is the set of all unplanted trees conforming to the rule which have 
negative degrees. We then define $\fF_-$ 
to be the free unital monoid generated by $\fT_-(\Rule)$. Elements of $\fF_-$ are interpreted as linear 
combinations of ``forests'',
i.e.\ disjoint unions of trees, with the product given by the ``forest product'', namely 
for $\tau_i = (T_i, \ff_i, \fm_i) \in \fT$, $i = 1,2$ we set
 $\tau_1 \cdot \tau_2  = (T_1\sqcup T_2, \ff_1\sqcup \ff_2, \fm_1\sqcup\fm_2)$.  The unit is 
 naturally represented by the empty forest.
Given a vector space assignment $V$, we define $\CT^-_V = \bigoplus_{\tau \in \fF_-} 
\scal{\tau}_V$ which is a commutative unital algebra under the forest product (which 
extends naturally to elements of $\scal{\tau_i}_V$, $\tau_i \in \fF_-$).

A coproduct $\Delta^-_V : \CT_V \to \CT^-_V \otimes \CT_V$ is defined in~\cite[Sec.~5-6]{BHZ} and is adapted 
to the setting of vector-valued regularity structures in~\cite[Sec.~5.5]{CCHS}. As usual, we write $\CG^-_V$ 
for the group of characters on $\CT^-_V$, called the renormalisation group (see~\cite[Def.~5.13]{BHZ}). We 
say that $M : \CT_V \to \CT_V$ is a renormalisation map if $M = (g \otimes \id )\Delta^-_V$ for some  
$g \in \CG^-_V$. We denote by $\fR(\CT_V)$ the space of all renormalisation maps on $\CT_V$.

\begin{definition}
	For $M \in \fR(\CT_\CD)$, we define the action $\hat M$ on $\SQ_+$ by setting for every $F \in \SQ_+$ and 
	$\ft \in \fL_+$
	\begin{equ}[eq:MF]
		(\hat M F)_{\ft} \eqdef \Proj_\one M \bbUpsilon^F_\ft = \Proj_\one \sum_{\tau \in \fT} M \bbUpsilon^F_\ft[\tau] \;.
	\end{equ}
(Recall \eqref{eq:bUpsilon} for the second identity.) Here, we use the identification 
	$\scal{\one}_\CD \otimes \SP \cong \SP$, since $\scal{\one}_\CD \cong \R$.
\end{definition}

We now state two auxiliary results, Lemmas~\ref{lem:MFinQ} and~\ref{lem:MF}, which justify our notations.

\begin{lemma}\label{lem:MFinQ}
	Let the rule $\Rule$ satisfy Assumption~\ref{ass:rule}, let $F \in \SQ(\Rule)$ and $M\in \fR(\CT_\CD)$. Then 
	$\hat MF \in \SQ(\Rule)$.
\end{lemma}

\begin{proof}
Let $g \in \CG^-(\CT_\CD)$ be such that $M = (g \otimes \id) \Delta^-_\CD$. In order to show that 
	$\hat M F$ conforms to the rule it is sufficient to prove that $g\bar\Upsilon^F[\tau]$ conforms to the 
	rule for every $\tau \in \fT(\Rule)$. For this fix $\tau \in \fT_-(\Rule)$, let $\ft \in \fL_+$ and assume that 
	$\alpha \notin \Rule(\ft)$. We want to show that $D^\alpha g \bar\Upsilon_\ft^F[\tau] = 0$. Because 
	of the normality of the rule $\Rule$ (assumption~\ref{R1}) we can assume without loss of generality that 
	$\alpha \in \N^\Eps$ contains no repeated elements. Following the same ideas as in~\cite[A.1]{BCCH} 
	there exist a finite set of trees $\tau_i \in \fT$ such that 
	\begin{equ}[eq:upsilon_conforms]
		D^\alpha g \bar\Upsilon_\ft^F[\tau] = \sum_i g \bar\Upsilon_\ft^F[\tau_i]\;,
	\end{equ}
	and such that the set of all edges types of $\tau_i$ contains $\alpha$. If $\tau_i \notin \fT_-$ then 
	$g \bar\Upsilon^F[\tau_i] = 0$ since the character $g$ only acts on forest of trees from $\fT_-(\Rule)$. Now assume that 
	$\tau_i \in \fT_-$. By completeness of the rule~\ref{R3} we must have $\J_\ft[\tau_i] \notin \fT(\Rule)$. 
	Indeed, since $\tau_i$ contains within itself all the edges from $\alpha$ assuming $\J_\ft[\tau_i] \in \fT(\Rule)$ together with completeness would imply that $\alpha \in \Rule(\ft)$ which is a contradiction. 
	The fact that $\J_\ft[\tau_i] \notin \fT(\Rule)$ implies that $\bar\Upsilon^F_\ft[\tau_i] = 0$ by 
	Remark~\ref{rem:Upsilon}. Therefore, for all $\tau_i$ that are present on the right-hand side 
	of~\eqref{eq:upsilon_conforms} we have $g\bar\Upsilon_\ft^F[\tau_i] = 0$ thus finishing the proof. 
\end{proof}

\begin{lemma}\label{lem:MF}
For $M \in \fR(\CT_\CD)$, let $\hat M$ be defined in~\eqref{eq:MF}. Then for every $F \in \SQ(\Rule)$ and $\ft \in \fL_{+}$ we have $M \bar \bUpsilon^F_\ft = \bar \bUpsilon^{\hat M F}_\ft$.
\end{lemma}

The proof relies strongly on the grafting operators introduced in~\cite[Sec 4.3]{BCCH}. In order to use the 
properties of such grafting operators we introduce a notion of the inner product on $\CT_\CD$. First, set 
$\fD^{\fl}$ and $\fD^{\fl}_{k_\star}$ be respectively canonical bases of $\CD^{\fl}$ and $\CD^{\fl}_{k_\star}$ (i.e.\ 
vectors of the form $\fd^\alpha$ for $\alpha \in \N^{\Eps_+(\fl)} $). We define an inner product $\scal{\bigcdot\,,\bigcdot}$ on 
$\CD = \sqcup_{\ft \in \fL} \CD_\ft$ in the following way: for $\fl_i \in \fL_-$ and $\fd_i\in \fD^{\fl_i}$, $i\in \{1,2\}$ we write:
\begin{equ}
	\scal{\fd_1,\fd_2} = \delta_{\fl_1,\fl_2}\delta_{\fd_1,\fd_2}\;,
\end{equ}
and extend this inner product on elements of $\CD_{\fl_i}$ by linearity. Now for general $\ft_i \in \fL_+$  and 
$a,b \in \CD_{\ft_i} \equiv \R$ we set $\scal{a,b} = \delta_{\ft_1,\ft_2} a.b$ and the scalar product between 
elements of $\CD_{\fl}$ and $\CD_{\ft}$ for $\fl \in \fL_{-}$ and $\ft \in \fL_{+}$ are set to be zero. It will be 
beneficial for us to also identify $\fD^\ft_k = \{1\}$ for $\ft \in \fL_{+}$.

Consider $\tau \in \fT(\Rule)$ of the form~\eqref{eq:recursive_tree} and an element $\sigma \in \scal{\tau}_\CD$ of the form
\begin{equ}[eq:general_sigma]
	\sigma = \X^m \prod_{i \in J} \J_{o_i}[\fd_i \otimes \sigma_i ]\;,
\end{equ}
for some set of indices $J$, $m \in \N^{d+1}$, $o_i = (\ft_i,p_i) \in \Eps$, $\fd_i \in \CD^{\ft_i}_{k_\star}$, 
$\sigma_i \in \scal{\tau_i}_\CD$ and $\tau_i \in \fT(\Rule)$, and where we again view $\J_{o_i}$ as a map 
$\CD_{\ft_i} \otimes \scal{\tau_i}_\CD \to \scal{\J_{o_i}\tau_i}_\CD$ as in~\eqref{eq:upsilon}. For $\btau \in \fT(\Rule)$ and $\bar\sigma \in \scal{\btau}_\CD$ that has a form as in~\eqref{eq:general_sigma} but where all 
$m,J,\fd_i,\ft_i,p_i,\sigma_i,\tau_i$ are replaced by $\bar m, \bar J,\bar \fd_i,\bo_i, \bar \sigma_i,\btau_i$ we 
inductively define
\begin{equ}[eq:scalar_prod]
	\scal{\sigma,\bar\sigma} =  \delta_{m,\bar m}\, m!\; \sum_{s \in S(J,\bar J)} \prod_{j \in J} \delta_{o_j,\bo_{s(j)}} \scal{\fd_j, \bar \fd_{s(j)}} \scal{\sigma_j, \bar \sigma_{s(j)}}\;,
\end{equ}
where $S(J,\bar J)$ is the set of bijections on from $J$ to $\bar J$ which is trivially is equal to an empty set 
if $|J|\ne |\bar J|$. We extend this inner product to the whole $\CT_\CD$ linearly.
Note that a similar inner product can be defined on the elements of $\fT(\Rule)$ by simply 
dropping $\fd_i$ and replacing $\sigma_i$ by $\tau_i$. This actually gives $\scal{\tau,\btau} = \delta_{\tau,\btau} S(\tau)$. 

We will denote the basis elements of $\CT_\CD$ by $B_\CD = \sqcup_{\ell \in \N} B^{\ell}_\CD$ where the 
sets on the right-hand side are defined inductively in the following way:
$B^{0}_\CD = \Poly$, and every element $\sigma \in B^{\ell}_\CD$ has the form~\eqref{eq:general_sigma} where 
$\sigma_i \in B^{\ell-1}_\CD$ and $\fd_i \in \fD^{\ft_i}_{k_\star}$. Elements of $B_\CD$ form an orthogonal (not 
orthonormal) basis of $\CT_\CD$ with respect to the inner product~\eqref{eq:scalar_prod}. We also set $B_\CD[\tau]  \eqdef  B_\CD \cap \scal{\tau}_\CD$. 

For $\tau \in \fT$ and $\sigma \in B_\CD[\tau]$ we postulate $S(\sigma) = S(\tau)$. One can 
easily see using the definition of the symmetry factor~\eqref{eq:symmetry_factor} that for $\tau$ of the 
form~\eqref{eq:recursive_tree_new} and $o = (\ft,p)$
\begin{equ}[eq:symmetry_ups]
	\bbUpsilon^F_o[\tau] = S(\tau)^{-1} \X^m \Big(\prod_{i=1}^{n} \bUpsilon^F_{o_i}[\tau_i]^{\beta_i}\Big) \d^m \prod_{i = 1}^n D_{o_i}^{\beta_i} \hat F_\ft\;.
\end{equ}
For simplicity, we also denote by $D^J = \prod_{j \in J} D_{o_j}$. With this at hand.

\begin{lemma}
	Let $\sigma \in B_\CD$ be of the form~\eqref{eq:general_sigma} then for any $o = (\ft,p) \in \Eps$ and $\fd \in \fD^\ft_{k_\star}$
\begin{equ}[eq:sigma_ups]
	\scal{\fd \otimes \sigma, \bbUpsilon^F_o} = \scal{\fd, \d^m D^J \hat F_\ft} \prod_{j \in J} \scal{\fd_j \otimes \sigma_j, \bbUpsilon^F_{o_j}}\;.
\end{equ}
\end{lemma}

\begin{proof}
	First, there exists $\tau \in \fT(\Rule)$ such that $\sigma \in B_\CD[\tau]$. Then, since for every $\bar \tau$, 
	$\bar \bUpsilon^F_\ft[\bar \tau] \in \fD^\ft_k \otimes \scal{\btau}_\CD \otimes \SP$ we have writing $\tau = \X^m \prod_{j\in J}\J_{o_j}[\tau_j]$ and using~\eqref{eq:symmetry_ups}
	\begin{equs}[eq:sigma_ups_computation]
		\scal{\fd \otimes \sigma, \bbUpsilon^F_o} &= \scal{\fd \otimes \sigma, \bbUpsilon^F_o[\tau]}\\
		&= \scal{\fd, \d^m D^J \hat F_\ft} m! S(\sigma)^{-1} \scal{\prod_{i \in J} \J_{o_i}[\fd_i \otimes \sigma_i ], \prod_{j \in J} \bUpsilon^F_{o_j}[\tau_j]}\\
		&= \scal{\fd, \d^m D^J \hat F_\ft} m! S(\sigma)^{-1} \sum_{s \in S(J,J)} \prod_{j\in J} \delta_{o_j, o_{s(j)}} \scal{\fd_j \otimes \sigma_j , \bbUpsilon^F_{o_{s(j)}}[\tau_{s(j)}]}\;.
	\end{equs}
	Now assume that exist $i,j \in J$, $i\neq j$ such that $o_i = o_j$ and $\sigma_i,\sigma_j \in B_\CD[\tau_i]$ 
	but $\fd_i \otimes \sigma_i \neq \fd_j \otimes \sigma_j$. Then we must have 
	\begin{equ}
		\scal{\fd_i \otimes \sigma_i, \bbUpsilon^F_{o_i}} \scal{\fd_j \otimes \sigma_j, \bbUpsilon^F_{o_j}} = \scal{\fd_i \otimes \sigma_i, \bbUpsilon^F_{o_i}[\tau_i]} \scal{\fd_j \otimes \sigma_j, \bbUpsilon^F_{o_i}[\tau_i]} = 0\;,
	\end{equ}
	because $\fd_i \otimes \sigma_i$ and $\fd_j \otimes \sigma_j$ are orthogonal. Therefore, if such 
	assumption is satisfied then by~\eqref{eq:sigma_ups_computation} we see that 
	equality~\eqref{eq:sigma_ups} is true since both sides are equal to zero. If the aforementioned assumption 
	is false then in fact one has (after re-indexing) $\sigma = \X^m \prod_{i = 1}^n \J_{o_i}[\fd_i \otimes \sigma_i]^{\beta_i}$ such that either $o_i \neq o_j$ or $\sigma_i \in B_\CD[\tau_i], \sigma_i \in B_\CD[\tau_j]$ 
	for some $\tau_i \neq \tau_j$. For such $\sigma$ we see using $\scal{\fd_j \otimes \sigma_j , \bbUpsilon^F_{o_{s(j)}}[\tau_{s(j)}]} = 0$ if $\sigma_j \notin B_\CD[\tau_{s(j)}]$ that
	\begin{equs}
		\sum_{s \in S(J,J)} \prod_{j\in J} \delta_{o_j, o_{s(j)}} \scal{\fd_j \otimes \sigma_j , \bbUpsilon^F_{o_{s(j)}}[\tau_{s(j)}]} &= \prod_{j=1}^n \beta_j! \scal{\fd_j \otimes \sigma_j , \bbUpsilon^F_{o_{j}}[\tau_j]}^{\beta_j}\\
		&= \Bigl(\prod_{j=1}^n \beta_j!  \Bigr) \prod_{i\in J} \scal{\fd_i \otimes \sigma_i , \bbUpsilon^F_{o_{i}}}\;,
	\end{equs}
	where in the second equality we have used $\scal{\fd_i \otimes \sigma_i , \bbUpsilon^F_{o_{i}}[\tau_i]} = \scal{\fd_i \otimes \sigma_i , \bbUpsilon^F_{o_{i}}}$ since $\sigma\in B_\CD[i]$ as well as used the 
	re-indexing from $j \in \{1,\dots,n\}$ to $i\in J$. Plugging in the above 
	into~\eqref{eq:sigma_ups_computation} and recalling the definition of the symmetry factor \eqref{eq:symmetry_factor} gives the desired result. 
\end{proof}

Using the above notion of inner product we can write every element of $\sigma \in \CT_\CD$ in the dual basis 
as $\sigma = \sum_{\bar\sigma \in B_\CD} \bar\sigma \scal{\sigma,\bar\sigma} \scal{\bar\sigma,\bar\sigma}^{-1}$. For example for $M \in \fR(\CT_\CD)$ we have 
\begin{equ}[eq:mupsilon]
	M \bbUpsilon^F_\ft = \sum_{\sigma \in B_\CD} \frac{\scal{\sigma,M\bbUpsilon^F_\ft}}{\scal{\sigma,\sigma}}\sigma = \sum_{\sigma \in B_\CD} \frac{\scal{M^*\sigma,\bbUpsilon^F_\ft}}{\scal{\sigma,\sigma}}\sigma\;,
\end{equ}
where $M^*$ is the adjoint of $M$ with respect to $\scal{\bigcdot,\bigcdot}$.

The next result is about an existence of the grafting operators in our context.

\begin{proposition}\label{prop:graft}
	For every $(\ft,p) \in \Eps$, $\fd \in \fD^{\ft}_{k_\star}$ and $i \in \{0,\dots,d\}$ there exist grafting operators $\graft_{(\ft,p)}^\fd : \CT_\CD \otimes \CT_\CD \to \CT_\CD$ and $\uparrow_i : \CT_\CD \to \CT_\CD$ with the following properties:
	\begin{enumerate}
		\item  $B_\CD$ is generated by the operations $\{\graft_{(\ft,p)}^\fd: (\ft,p) \in \Eps,\fd \in \fD^\ft_{k}\}$ 
		and generators in $\Poly$ in the sense that every $\sigma \in B_\CD$ can be written as a linear 
		combination of terms of the form 
		\begin{equ}
			\sigma_0 \graft_{o_1}^{\fd_1}\, \sigma_1 \graft_{o_2}^{\fd_2} \dots \graft_{o_n}^{\fd_n}\, \sigma_n\;,
		\end{equ}
		for some $(\sigma_i)_{i = 0}^n \subset \Poly$, $(o_i)_{i = 1}^n \subset \Eps$, $(\fd_i)_{i = 1}^n$ such that 
		$o_i = (\ft_i,p_i)$ and $\fd_i \in \fD^{\ft_i}_{k_\star}$. (We interpret the grafting operators as being 
		right-associative.)
		
		\item Let $o, \bo \in \Eps$ with $o = (\ft,p), \bo = (\fl,q)$ and let $\fd \in \fD^\ft_{k_\star}, \bfd \in \fD^{\fl}_{k_\star}$. Then for all $\sigma,\bar\sigma \in B_\CD$ and $F \in \SQ(\Rule)$ we have
		\begin{equs}
			\scal{\fd \otimes (\bar\sigma \graft_{\bo}^{\bfd} \sigma), \bbUpsilon^F_{o}} &= \scal{\bfd \otimes \bar\sigma, \bbUpsilon^F_{\bo}} \scal{\fd \otimes \sigma, D_{\bo}  \bbUpsilon^F_{o}}\;,\label{eq:graft_ups}
			\\
			\scal{\fd\otimes \up_i\sigma, \bbUpsilon^F_{o}} &=  \scal{\fd \otimes \sigma, \d_i \bbUpsilon^F_{o}}\;,\label{eq:up_ups}
		\end{equs}
		where we use the fact that $\scal{\fd\otimes \sigma, \bbUpsilon^F_{\ft}} \in \SP$.
		\item For all $o = (\ft,p) \in \Eps$, $\fd \in \fD^\ft_{k_\star}$ and $\sigma,\bar\sigma \in B_\CD$ we have 
		for every $M\in \fR(\CT_\CD)$, that
		\begin{equ}[eq:mgraft]
			M^* (\bar\sigma \graft_o^\fd \sigma) = (M^*\bar\sigma) \graft_o^\fd (M^*\sigma) \qquad\text{and}\qquad M^* \up_i \sigma = \up_i M^* \sigma
			\;.
		\end{equ}
	\end{enumerate}
\end{proposition}
\begin{proof}
	First, we define a grafting operator on the space $\tilde\CT_\CD = \prod_{\tau \in \fT} \scal{\tau}_\CD$. For $\sigma \in \tilde\CT_\CD$ given by~\eqref{eq:general_sigma} and $o = (\ft,p) \in \Eps, \fd \in \fD^{\ft}_{k_\star}$ we define $\bar\curvearrowright^{\fd}_{(\ft,p)} : \CT_\CD \otimes \CT_\CD \to \tilde\CT_\CD$ recursively as 
	\begin{equs}[eq:graft_D]
		\bar\sigma \bar\curvearrowright^{\fd}_{(\ft,p)} \sigma &= \sum_{\ell} {m \choose \ell} \X^{m-\ell} \J_{(\ft,p-\ell)}[\fd \otimes \bar\sigma] \prod_{i \in J} \J_{o_i}[\fd_i \otimes \sigma_i] \\
		&\qquad + \sum_{j \in J} \X^m \J_{o_j}[\fd_j \otimes (\bar\sigma \bar\curvearrowright^{\fd}_{(\ft,p)} \sigma_j)] \prod_{i \in J \setminus \{j\}} \J_{o_i}[\fd_i \otimes \sigma_i]\;,
	\end{equs} 
	where for $p \ngeqslant \ell$ we understand the operator $\J_{(\ft,p-\ell)}$ as an operator that maps 
	everything to $0$. Moreover, we set 
	\begin{equ}[eq:up_i]
		\up_i \sigma = \X^{m+i} \prod_{j \in J} \J_{o_j}[\fd_j \otimes \sigma_j] + \sum_{\jmath \in J} \X^m \J_{o_\jmath}[\fd_\jmath \otimes \up_i \sigma_\jmath] \prod_{j \in J \setminus \{\jmath\}} \J_{o_j}[\fd_j \otimes \sigma_j]\;,
	\end{equ}
	where $m+i \eqdef m+e_i$.
	Let $o, \bo \in \Eps$ with $o = (\ft,p), \bo = (\bar\ft,\bar p)$ and let $\fd \in \fD^\ft_{k_\star}, \bfd \in \fD^{\bar\ft}_{k_\star}$. Assume that $\sigma \in B_\CD[\tau]$ is of the form~\eqref{eq:general_sigma} for 
	some $\tau \in \fT(\Rule)$. We will first show by induction on the number of edges of $\tau$ that 
	$\bar\curvearrowright_{\bo}^{\bfd}$ satisfies~\eqref{eq:graft_ups}. We use bilinearity of the scalar product, 
	linearity of operators $\J_{o_j}$ as well as formula~\eqref{eq:sigma_ups} to deduce
\begin{equs}
		\scal{\fd &\otimes (\bar\sigma \bar\curvearrowright_{\bo}^{\bfd} \sigma), \bbUpsilon^F_{o}} = \sum_\ell {m \choose \ell} \scal{\fd \otimes \X^{m-\ell} \J_{(\bar\ft,\bar p-\ell)}[\bar \fd \otimes \bar\sigma] \prod_{i \in J} \J_{o_i}[\fd_i \otimes \sigma_i], \bbUpsilon^F_{o}}\\
		&\quad + \sum_{j \in J} \scal{\fd \otimes \X^m \J_{o_j}[\fd_j \otimes (\bar\sigma \bar\curvearrowright^{\bar\fd}_{(\bar\ft,\bar p)} \sigma_j)] \prod_{i \in J \setminus \{j\}} \J_{o_i}[\fd_i \otimes \sigma_i], \bbUpsilon^F_{o}}\\
		&= \sum_{\ell} {m \choose \ell} \scal{\fd,\d^{m-\ell} D_{(\bar\ft,\bar p-\ell)} D^{J} \hat F_\ft} \scal{\bfd \otimes \bar\sigma, \bbUpsilon^F_{(\bar\ft,\bar p-\ell)}} \prod_{i \in J} \scal{\fd_i \otimes \sigma_i, \bbUpsilon^F_{o_i}}\\
		&\quad+ \sum_{j \in J} \scal{\fd,\d^{m} D^{J} \hat F_\ft} \scal{\fd_j \otimes (\bar\sigma \bar\curvearrowright_{\bo}^{\bfd} \sigma_j), \bbUpsilon^F_{o_j}} \prod_{i \in J\setminus\{j\}} \scal{\fd_i \otimes \sigma_i, \bbUpsilon^F_{o_i}}\\
		&= \sum_{\ell} {m \choose \ell} \scal{\fd,\d^{m-\ell} D_{(\bar\ft,\bar p-\ell)} D^{J} \hat F_\ft} \scal{\bfd \otimes \bar\sigma, \bbUpsilon^F_{(\bar\ft,\bar p-\ell)}} \prod_{i \in J} \scal{\fd_i \otimes \sigma_i, \bbUpsilon^F_{o_i}}\\
		&\quad+ \sum_{j \in J} \scal{\fd,\d^{m} D^{J} \hat F_\ft} \scal{\bfd \otimes \bar\sigma, \bbUpsilon^F_{\bo}} \scal{\fd_j \otimes \sigma_j, D_{\bo}  \bbUpsilon^F_{o_j}} \prod_{i \in J\setminus\{j\}} \scal{\fd_i \otimes \sigma_i, \bbUpsilon^F_{o_i}}\;,
\end{equs}
where we used induction hypothesis in the last equality. Note that 	$\bbUpsilon^F_{\bo} = \bbUpsilon^F_{\fl}= \bbUpsilon^F_{(\bar\ft,\bar p-\ell)}$, which leads to 
\begin{equs}
	\scal{\fd \otimes &(\bar\sigma \bar\curvearrowright_{\bo}^{\bfd} \sigma), \bbUpsilon^F_{o}}\\
	 &= \scal{\fd,\sum_{\ell} {m \choose \ell} \d^{m-\ell} D_{(\bar\ft,\bar p-\ell)} D^{J} \hat F_\ft} m! \scal{\bfd \otimes \bar\sigma, \bbUpsilon^F_{\bo}} \prod_{i \in J} \scal{\fd_i \otimes \sigma_i, \bbUpsilon^F_{o_i} }\\
	 &\quad+ \scal{\fd,\d^{m} D^{J} \hat F_\ft} \scal{\bfd \otimes \bar\sigma, \bbUpsilon^F_{\bo}} \sum_{j \in J} \scal{\fd_j \otimes \sigma_j, D_{\bo} \bbUpsilon^F_{o_j}} \prod_{i \in J\setminus\{j\}} \scal{\fd_i \otimes \sigma_i, \bbUpsilon^F_{o_i} }\;.
\end{equs}
On the other hand using the Leibniz rule for $D_o \bbUpsilon^F_o$ it is easy to see that
\begin{equs}
	\scal{\bfd \otimes &\bar\sigma, \bbUpsilon^F_{\bo}} \scal{\fd \otimes \sigma, D_{\bo}  \bbUpsilon^F_{o}} = \scal{\bfd \otimes \bar\sigma, \bbUpsilon^F_{\bo}} \scal{\fd, D_{\bo} \d^m D^J \hat F_\ft} \prod_{i \in J} \scal{\fd_i \otimes \sigma_i, \bbUpsilon^F_{o_i}}\\ 
	&+ \scal{\bfd \otimes \bar\sigma, \bbUpsilon^F_{\bo}} \scal{\fd, \d^m D^J \hat F_\ft} \sum_{j \in J } \scal{\fd_j \otimes \sigma_j, D_{\bo} \bbUpsilon^F_{o_j}} \prod_{i \in J\setminus\{j\}} \scal{\fd_i \otimes \sigma_i, \bbUpsilon^F_{o_i} }\;.
\end{equs}
It remains therefore to show that 
\begin{equ}[eq:twisted_Leibniz]
	\sum_{\ell} {m \choose \ell} \d^{m-\ell} D_{(\bar\ft,\bar p-\ell)} D^{J} \hat F_\ft = D_{\bo}\, \d^{m} D^{J} \hat F_\ft\;,
\end{equ}
which follows from Lemma~\ref{lem:derivatives_swap} below, applied to $D^{J} \hat F_\ft$. We leave it to the reader 
to show the base case for induction when $\sigma = \X^m$. 

The true grafting operator $\graft_{o}^{\fd}$ is then given by $\graft_{o}^{\fd}  \eqdef  \Proj_{\fT(\Rule)} \circ \bar\curvearrowright_{o}^{\fd}$. Note that since $\bUpsilon^F_o[\tau] = 0$ for $\J_o\tau \notin \fT(\Rule)$ and 
since~\eqref{eq:graft_ups} holds true for grafting $\bar\curvearrowright_{o}^{\fd}$ then the same is true for 
grafting $\graft_{o}^{\fd}$.

Now, we will show that our grafting operator~\eqref{eq:up_i} satisfies~\eqref{eq:up_ups}. Let us take $\sigma \in B_\CD[\tau]$ of the form~\eqref{eq:general_sigma} for some $\tau \in \fT(\Rule)$, and we will show the 
claim by induction over the number of edges in $\tau$. More precisely, we assume that~\eqref{eq:up_ups} 
holds for all such $\bar \sigma \in B_\CD[\bar \tau]$ with the number of edges strictly smaller than in $\tau$. 
Then we prove~\eqref{eq:up_ups} for $\sigma$ (the induction base can be proved in exactly the same way, 
that's why we do not provide it separately). Definition~\eqref{eq:up_i} and identity~\eqref{eq:sigma_ups} yield
\begin{equs}
\scal{\fd\otimes \up_i\sigma, \bbUpsilon^F_{o}} &= \scal{\fd \otimes \X^{m+i} \prod_{j \in J} \J_{o_j}[\fd_j \otimes \sigma_j], \bbUpsilon^F_{o}} \\
&\qquad + \sum_{\jmath \in J} \scal{\fd \otimes \X^m \J_{o_\jmath}[\fd_\jmath \otimes \up_i \sigma_\jmath] \prod_{j \in J \setminus \{\jmath\}} \J_{o_j}[\fd_j \otimes \sigma_j], \bbUpsilon^F_{o}}\\
&= \scal{\fd, \d^{m+i} D^J \hat F_\ft} \prod_{j \in J} \scal{\fd_j \otimes \sigma_j, \bbUpsilon^F_{o_j}} \label{eq:up_ups_interm}\\
&\qquad + \sum_{\jmath \in J} \scal{\fd, \d^{m} D^J \hat F_\ft} \scal{\fd_\jmath \otimes \up_i \sigma_\jmath, \bbUpsilon^F_{o_\jmath}} \prod_{j \in J \setminus \{\jmath\}} \scal{\fd_j \otimes \sigma_j, \bbUpsilon^F_{o_j}}\;.
\end{equs}
Then, since the tree of $\up_i \sigma_\jmath$ has strictly smaller number of edges than in $\tau$, we can use 
the induction hypothesis to write $\scal{\fd_\jmath \otimes \up_i \sigma_\jmath, \bbUpsilon^F_{o_\jmath}} = \scal{\fd_\jmath \otimes \sigma_\jmath, \d_{i} \bbUpsilon^F_{o_\jmath}}$. Next, we will show 
that~\eqref{eq:up_ups_interm} is equal to the right-hand side of~\eqref{eq:up_ups}. 
Identity~\eqref{eq:symmetry_ups} and the Leibniz rule yield 
\begin{equs}
\d_i &\bbUpsilon^F_{o}[\tau] = S(\tau)^{-1} \X^m \Big(\prod_{j=1}^{n} \bUpsilon^F_{o_j}[\tau_j]^{\beta_j}\Big)\, \Bigl(\d^{m+i} \prod_{j = 1}^n D_{o_j}^{\beta_j} \hat F_\ft\Bigr)\\
& + \sum_{\jmath = 1}^n \beta_{\jmath} S(\tau)^{-1} \X^m \d_i \bUpsilon^F_{o_\jmath}[\tau_\jmath] \Big(\bUpsilon^F_{o_\jmath}[\tau_\jmath]^{\beta_\jmath - 1} \prod_{j \neq \jmath} \bUpsilon^F_{o_j}[\tau_j]^{\beta_j}\Big)\, \Bigl(\d^{m} \prod_{j = 1}^n D_{o_j}^{\beta_j} \hat F_\ft\Bigr)\;.
\end{equs}
Applying~\eqref{eq:sigma_ups} this identity turns into~\eqref{eq:up_ups} as required.

First part can be proven similarly to the~\cite[Cor 4.23]{BCCH} by showing that grafting 
$\bar\curvearrowright_{o}^{\fd}$ satisfies the pre-Lie type identity: 
\begin{equs}
	( &\sigma_1 \bar\curvearrowright_{o_1}^{\fd_1} \sigma_2 ) \bar\curvearrowright_{o_2}^{\fd_2} \sigma_3 - \sigma_1 \bar\curvearrowright_{o_1}^{\fd_1} (\sigma_2  \bar\curvearrowright_{o_2}^{\fd_2} \sigma_3) \\
	= ( &\sigma_2 \bar\curvearrowright_{o_2}^{\fd_2} \sigma_1 ) \bar\curvearrowright_{o_1}^{\fd_1} \sigma_3 - \sigma_2 \bar\curvearrowright_{o_2}^{\fd_2} (\sigma_1  \bar\curvearrowright_{o_1}^{\fd_1} \sigma_3)\;.
\end{equs}

Third part of Proposition~\ref{prop:graft} can be shown \textit{mutatis mutandis} as the proof of~\cite[Prop 4.18]{BCCH} using the fact that $\scal{\tau}_\CD$ is finite-dimensional for each $\tau \in \fT(\Rule)$. This is 
because $\bar\curvearrowright$ and $\up_i$ have almost the same form as their analogues 
in~\cite[Sec.~4]{BCCH}. The only difference is the attachment of the corresponding abstract derivatives 
$\fd$ from $\CD$ (see~\cite[Rem.~4.12]{BCCH}) which interact with the negative coproduct in a ``trivial'' way.
\end{proof}

The following lemma shows how one can swap derivatives $\d^k$ and $D_o$. This result is known and can be found in \cite{Wilcox}, as we explain in the proof. In the framework of regularity structures it was previously used in \cite{2101.11949, MR4537770}.

\begin{lemma}\label{lem:derivatives_swap}
For $k \in \N^{d+1}$, $(\ft, p) \in \Eps$, $\fl \in \fL$ and any $\hat F \in \CD_\fl \otimes \SP$ one has
\begin{equ}[ex:derivatives_swap]
	\sum_{\ell} \binom{k}{\ell}  \d^{k - \ell} D_{(\ft, p-\ell)} \hat F = D_{(\ft, p)}\, \d^{k} \hat F\;,
\end{equ}
with the convention that $\d^q = D_{(\ft, q)} = 0$ unless $q_i \ge 0$ for all $i$.
\end{lemma}

\begin{proof}
For $i \in \{0,\ldots,d\}$ it follows from the definition \eqref{eq:chain_rule} that
$D_{(\ft, p)} \d_i = \d_i D_{(\ft, p)} + D_{(\ft, p-e_i)}$, which is the same commutation 
relation satisfied by $\d_i$ and multiplication by $(-x)^p/p!$. The claim is therefore a particular instance of
a well-known result in elementary quantum mechanics, see for example \cite[Eq.~10.7]{Wilcox}.
\end{proof}

Before we turn to the proof of the Lemma~\ref{lem:MF} we would like to make sense of the renormalisation of $\hat F_\ft$ where $\hat F$ is given by~\eqref{eq:hatF}. We postulate that for every $\ft \in \fL$
\begin{equ}
	(\hat M \hat F)_\ft = \Proj_\1 M \bbUpsilon^F_\ft \in \CD_\ft \otimes \SP\;.
\end{equ}
For $\ft \in \fL_+$, this definition yields $(\hat M \hat F)_\ft = (\hat M F)_\ft$, where the latter 
was defined in \eqref{eq:MF}.

\begin{proof}[of Lemma~\ref{lem:MF}]
	We are going to prove inductively that for all $\ft \in \fL$, $\fd \in \CD^\ft_{k_\star}$ and $\sigma \in B_\CD$
	\begin{equ}[eq:scalar_ups]
		\scal{\fd \otimes \sigma, M\bbUpsilon^F_\ft} = \scal{\fd \otimes \sigma, \bbUpsilon^{\hat MF}_\ft}\;.
	\end{equ}
	This together with~\eqref{eq:mupsilon} and the fact that for $\ft \in \fL_+$ we have $\CD^\ft_{k_\star} \equiv \R$ will give Lemma~\ref{lem:MF}. Note that in~\eqref{eq:scalar_ups} we view $M$ as an operator on 
	$\CD_{\ft} \otimes \CT_{\CD} \otimes \SP$ which acts as identity on the first and third vector space.
	We first check~\eqref{eq:scalar_ups} for $\sigma =  \1$ where $\fl \in \fL_{-}$ and $\bfd \in \CD^\fl_{k_\star}$. Note that $\Proj_\1 \sigma = \scal{\1, \sigma}$ for every $\sigma \in \CT_\CD$ and therefore
	\begin{equ}
		\scal{\fd \otimes \1, \bbUpsilon^{\hat MF}_\ft} = \scal{\fd \otimes \1, \bbUpsilon^{\hat MF}_\ft[\1]} = \scal{\fd , (\hat M \hat{F})_\ft } = \scal{\fd , \Proj_\1 M \bbUpsilon^F_\ft} = \scal{\fd \otimes \1, M \bbUpsilon^F_\ft}\,.
	\end{equ}
	Now for $m \in \N^{d+1}$ denote by $\up^m = \prod_{i = 0}^{d} \up_i^{m[i]}$ and note that $\up^m\one = \X^m$ by definition~\eqref{eq:up_i}. Therefore, using~\eqref{eq:up_ups} and~\eqref{eq:mgraft}
	\begin{equs}
		\scal{\fd \otimes \X^m , \bbUpsilon^{\hat MF}_\ft} &= \scal{\fd \otimes \up^m\1, \bbUpsilon^{\hat MF}_\ft} 
			= \scal{\fd \otimes\1, \d^m \bbUpsilon^{\hat MF}_\ft}\\ 
		&= \scal{\fd \otimes \1, \d^m (M \bUpsilon^F)_\ft} = \scal{\fd \otimes M^*\1 , \d^m  \bbUpsilon^{F}_\ft}\\
		&= \scal{\fd \otimes \up^mM^*\1, \bbUpsilon^F_\ft} = \scal{\fd \otimes M^*\up^m \1 , \bbUpsilon^{F}_\ft}\\
		& = \scal{\fd \otimes X^m , M\bbUpsilon^{F}_\ft}\;,
	\end{equs}
	where in the fourth equality we use the fact that $\d^m M = M \d^m$ since here $M$ acts $\CT_\CD$ and not on $\SQ$. 
	For inductive hypothesis assume that~\eqref{eq:scalar_ups} is true for some $\sigma,\bar\sigma \in B_\CD$. Let $o = (\fl,p) \in \Eps$, and $\bfd \in \fD^{\fl}_{k_\star}$ then 
	\begin{equs}
		\scal{\fd \otimes (\bar\sigma \graft_{o}^{\bfd} \sigma), \bbUpsilon^{\hat MF}_{\ft}} &= \scal{\bfd \otimes \bar\sigma, \bbUpsilon^{\hat MF}_o} \scal{\fd \otimes \sigma, D_o  \bbUpsilon^{\hat MF}_{\ft}}\\
		&= \scal{\bfd \otimes \bar\sigma, M\bbUpsilon^{F}_o}  \scal{\fd \otimes \sigma, D_o M\bbUpsilon^{F}_{\ft}}\\
		&= \scal{\bfd \otimes M^*\bar\sigma, \bbUpsilon^{F}_o}  \scal{\fd \otimes M^*\sigma, D_o\bbUpsilon^{F}_{\ft}}\\
		&= \scal{\fd \otimes \big((M^*\bar\sigma) \graft_{o}^{\bfd} (M^*\sigma)\big), \bbUpsilon^{F}_{\ft}}\\
		&= \scal{\fd \otimes M^*(\bar\sigma \graft_{o}^{\bfd} \sigma), \bbUpsilon^{F}_{\ft}}\\
		&= \scal{\fd \otimes (\bar\sigma \graft_{o}^{\bfd} \sigma), M\bbUpsilon^{F}_{\ft}}\;,
	\end{equs}
	where we use~\eqref{eq:graft_ups} in the first and fourth equality, $M D_o = D_o M$ in the third 
	equality,~\eqref{eq:mgraft} in the fifth and induction hypothesis in the second. We conclude the proof, 
	since the set $\Poly$ together with the family $\{\graft_{(\ft,p)}^\fd: (\ft,p) \in \Eps,\fd \in \fD^\ft_{k_\star}\}$ 
	generates $B_\CD$ by Proposition~\ref{prop:graft}.
\end{proof}

With these results at hand we are ready to prove one of the most important properties of the renormalised functions.

\begin{proposition}\label{prop:MU_coherent}
Let  $\Rule$ be a rule satisfying Assumption~\ref{ass:rule}, let $F \in \SQ(\Rule)$, let $M \in \fR(\CT_\CD)$,
and let $\hat M$ be given by~\eqref{eq:MF}. Then for every $L \in \N \cup \{\infty\}$ there exists $\bar L \in \N \cup \{\infty\}$, which is finite if $L$ is finite, such that if $U$ is coherent with $F$ to order $\bar L$, then 
$MU$ is coherent with $\hat MF$ to order $L$.
\end{proposition}
\begin{proof}
Applying $M$ to the expansion~\eqref{eq:UR} componentwise, we get
\begin{equs}[eq:MU1]
	(MU)_\ft \eqdef MU_\ft &= \sum_{p \in \N^{d+1}} \frac{1}{p!} u^U_{(\ft,p)} M \X^p + M\J_\ft\big[U^R_\ft\big] \\
	&= \sum_{p \in \N^{d+1}} \frac{1}{p!} u^U_{(\ft,p)} \X^p + \J_\ft\big[MU^R_\ft\big]\;,
\end{equs}
where in the last equality we used $M \X^p = \X^p$ and the fact that $\J_\ft$ commutes with $M$ for each 
$\ft \in \fL_{+}$ (see~\cite[Sec. 6.4.3]{BHZ}). On the other hand, applying~\eqref{eq:UR} directly for $MU$ yields
\begin{equ}[eq:MU2]
(M U)_\ft = \sum_{p \in \N^{d+1}} \frac{1}{p!} u^{M U}_{(\ft,p)}\X^p + \J_\ft\big[(M U)^R_\ft\big]\;.
\end{equ}
Applying $\Proj_{\X^p}$ to~\eqref{eq:MU1} and~\eqref{eq:MU2} we get $\bu^{MU} = \bu^{U}$ which therefore 
implies $(MU)^R_\ft = MU^R_\ft$. Moreover,~\cite[Def~3.16 \& Eq.~3.6]{BCCH} implies that for every $L \in \N \cup \{\infty\}$ there exists $\bar L \in \N \cup \{\infty\}$ (which is finite if $L$ is finite and which satisfies 
$\bar L \geq L$), such that $M \sigma \in \CW_{\CD, \leq \bar L}$ for every $\sigma \in \CT_\CD$ 
satisfying $L(\sigma) \leq L$. Hence, coherence~\eqref{eq:coherence} to order $\bar L$ yields
\begin{equs}
\bp_{\leq L} M U^R_\ft &= \bp_{\leq L} \bigl(M \bp_{\leq \bar L} U^R_\ft\bigr) = \bp_{\leq L} \bigl(M \bp_{\leq \bar L} \bar \bUpsilon^F_\ft(\bu^U)\bigr) \\
&= \bp_{\leq L} \bar \bUpsilon^{\hat M F}_\ft(\bu^U) = \bp_{\leq L} \bar \bUpsilon^{\hat MF}_\ft(\bu^{MU})\;,
\end{equs}
where in the second identity we simply omitted $\bp_{\leq \bar L}$, in the third equality we used 
Lemma~\ref{lem:MF} and in the last equality we used $(\bu^{U})= (\bu^{MU})$. This is precisely the 
coherence of $MU$ with $\hat M F$ to order $L$, as defined in~\eqref{eq:coherence}.
\end{proof}

We now connect renormalisation on $\CT_\CB$ with renormalisation on $\CT_\CD$ via the evaluation map $\Ev$. Given a character $g \in \CG^-_\CB$ and denoting $M = (g \otimes \id) \Delta^-_\CB$ we define for every $\bu \in \R^{\Eps_+}$
\begin{equ}[eq:Mu]
	g^\bu = g\circ \Ev_\bu\qquad\text{and}\qquad M_\bu = (g^\bu\otimes \id) \Delta^-_\CD\;,
\end{equ}
where we extend $\Ev_\bu$ to be multiplicative on a forest product. The fact that $\Ev_\bu$ is a natural transformation $\BF_\CD \to \BF_\CB$ implies by~\cite[Rem.~5.8]{CCHS} that $g^\bu \in \CG^-_\CD$ and $M_\bu \in \fR(\CT_\CD)$ for every $\bu \in \R^{\Eps_+}$.

\begin{lemma}\label{lem:EvM}
	For every $\bu \in \R^{\Eps_+}$ and $M \in \fR(\CT_\CB)$ one has $M \Ev_\bu = \Ev_\bu M_\bu$.
\end{lemma}

\begin{proof}
	This follows easily from the definitions~\eqref{eq:Mu} and the identities $\tilde{\CA}^-_\CB \Ev_\bu = \Ev_\bu \tilde{\CA}^-_\CD$ and $(\Ev_\bu\otimes \Ev_\bu) \Delta^-_\CD = \Delta^-_\CB \Ev_\bu$, which are 
	direct consequences of the fact that $\Ev_\bu$ induces a natural transformation $\BF_\CD \to \BF_\CB$ 
	by~\cite[Rem.~5.8]{CCHS}.
\end{proof}

\begin{lemma}\label{lem:MF_u}
	Let $\bu \in \R^{\Eps_+}$ and $M = (g \otimes \id) \Delta^-_\CB \in \fR(\CT_\CB)$. Let $\hat M_\bu$ denote the action of $M_\bu$ onto the space of nonlinearites given by~\eqref{eq:MF}. Then for every $F \in \SQ_+$ and $\ft \in \fL_+$ 
	we have
	\begin{equ}[eq:MF_u]
		\bigl(\hat M_\bu F\bigr)_\ft = F_\ft + \sum_{\tau \in \fT_-(\Rule)} g^\bu \bbUpsilon^{F}_\ft[\tau]\;,
	\end{equ}
	where $g^\bu$ acts on the factor in $\scal{\tau}_\CD$ of $\bbUpsilon^{F}_\ft[\tau]$ as in~\eqref{eq:tensor_notation}.
\end{lemma}

\begin{proof}
	By definition $(\hat M_\bu F)_\ft = \Proj_\one \sum_{\tau \in \fT(\Rule)} (g\circ \Ev_\bu \otimes \id) \Delta^-_\CD \bar{\bUpsilon}^F_\ft[\tau]$. Note that for $\tau \in \fT(\Rule)$ and $\sigma \in \scal{\tau}_\CD$ one has 
	$\Delta^-_\CD \sigma = \sum_i \sigma^i_1\otimes \sigma^i_2$ and we have $\sigma^i_2 = \one$ only when 
	$\sigma^i_1=\sigma$. Moreover, this can only happen either if $\sigma = \one$ or if $\tau \in \fT_-(\Rule)$. 
	Therefore, since $\bar{\bUpsilon}^F_\ft[\tau] \in \scal{\tau}_\CD \otimes \SP$, from~\eqref{eq:MF} we get
	\begin{equs}
		\bigl(\hat M_\bu F\bigr)_\ft &= \bigl(g^\bu \otimes \id\bigr) \bar{\bUpsilon}^F_\ft[\one] + \sum_{\tau \in \fT_-(\Rule)} \bigl(g^{\,\bu} \otimes \id\bigr) \bar{\bUpsilon}^F_\ft[\tau]\\
		&= \bigl(g^\bu \otimes \id\bigr) (\one \otimes F_\ft) + \sum_{\tau \in \fT_-(\Rule)} g^{\,\bu}\bar{\bUpsilon}^F_\ft[\tau]\;,
	\end{equs}
	which is exactly~\eqref{eq:MF_u}.
\end{proof}

We say that a character $g \in \CG^-_\CB$ is translation invariant if $g \circ T_\bu = g$ for all $\bu \in \R^{\Eps_+}$. Here $T_\bu$ is the translation map defined in Section~\ref{sec:evaluation} extended multiplicatively on the forest product of $\CT^-_\CB$. 

\begin{proposition}\label{prop:translation}
	Let $g \in \CG^-_\CB$ be translation invariant then $g^\bu$ is independent of $\bu$.
\end{proposition}
\begin{proof}
	As a consequence of~\eqref{eq:translation}, we have
	\begin{equ}
		g^\bu = g\circ \Ev_\bu = g \circ T_{-\bu} \circ \Ev_\bu = g \circ \Ev_{\bu-\bu} = g \circ \Ev_\0 = g^\0\;,
	\end{equ}
	as claimed.
\end{proof}

\begin{remark}\label{rem:F_z}
	The content of this section can be applied to functions $F \in \CC(\R^{d+1}, \SQ(\Rule))$ of the form
	\begin{equ}
		F(z,\fX) = F^1(\fX) + F^2(z)\;,
	\end{equ}
	by simply postulating that $\Upsilon^F = \Upsilon^{F^1}$ and $\hat M F =\hat M F^1 + F^2$ 
	i.e.\ by treating $F^2(z)$ as a constant. The coherence in Lemma~\ref{lem:CoherenceEquivalence} is then simply changed to $\bp_{\leq L} U^R = \bp_{\leq L} \BF^1_{\CD}$. This will be useful in Section~\ref{sec:fixed_point} where the nonlinearity will be of the above form and $F^2$ will depend regularly enough on the reconstruction of the solution.
\end{remark}

\section{Models and modelled distributions on $\CT_{\CB}$}
\label{sec:models}

Recall the definition of a model from~\cite[Sec.~2.3]{Regularity}.

\begin{definition}
A model $Z = (\Pi, \Gamma)$ for the regularity structure $\CT_{\CB}$ 
consists of two collections of continuous linear maps $\Pi_{z} : \CT_{\CB} \to \CS'(\R^{d+1})$ and elements $\Gamma_{z \bar z} \in \CG$, parametrised by $z, \bar z \in \R^{d+1}$, such that the following three algebraic identities hold: 
$\Pi_{\bar z} = \Pi_z \Gamma_{z \bar z}$, $\Gamma_{z \bar{\bar z}} = \Gamma_{z \bar z} \Gamma_{\bar z \bar{\bar z}}$ and $\Gamma_{z z} = 1$, for any points $z, \bar z, \bar{\bar z} \in \R^{d+1}$. Moreover, for 
every $\gamma \in \R$ and every compact set $\fK \subset \R^{d+1}$ there is a constant $C_{\gamma, \fK} \geq 0$ such that the following analytical bounds hold:
\begin{equs}[eq:model]
\sup_{\sigma, \lambda, \phi} \frac{\lambda^{-\deg\sigma}}{\|\sigma\|}\bigl|(\Pi_{z} \sigma)( \phi_{z}^\lambda)\bigr| \leq C_{\gamma, \fK}\;, \qquad \sup_{\sigma, \zeta} \frac{\| \CQ_\zeta (\Gamma_{z \bar z} \sigma)\|}{\|\sigma\| \|z - \bar z\|^{\deg\sigma - \zeta}_\s} \leq C_{\gamma, \fK}\;,
\end{equs}
uniformly over points $z \neq \bar z \in \fK$, where the supremum is taken over all elements $\sigma \in \CT_{\CB}$ such that $\deg\sigma \leq \gamma$, all degrees $\zeta < \deg\sigma$, all functions $\phi \in \SB^r_\s$ and all $\lambda \in (0,1]$. 
\end{definition}

Here $r$ is the 
smallest integer such that $\deg \tau > - r$ for every $\tau \in \fT(\Rule)$ which is guaranteed to be finite 
by~\cite[Prop. 5.15]{BHZ} and the subcriticality of the rule~\ref{R2}. In practice, we will only 
consider 
models defined on $\CT_{\CB,\gamma}$ for some $\gamma > 0$ rather than on the whole $\CT_\CB$. 

\begin{definition}
We denote by $\$ Z \$_{\fK}$ the smallest constant $C_{\gamma, \fK}$, for which the 
bounds \eqref{eq:model} hold. Respectively, we define a ``distance'' $\$ Z; \bar Z \$_{\fK}$ 
between two models $Z = (\Pi, \Gamma)$ and $\bar Z = (\bar \Pi, \bar \Gamma)$ to be the smallest constant $C_{\gamma; \fK} \geq 0$ such that the bounds~\eqref{eq:model} hold for the pair of maps $\Pi - \bar \Pi$ 
and $\Gamma - \bar \Gamma$.
\end{definition}

Let us denote the time-space domain $\Lambda \eqdef  \R_+ \times \T^{d}$,\label{lab:Lambda} where $\T$ is 
the circle as in~\eqref{eq:system_intro}. We define the set $P = \{(t,x) \in \Lambda : t = 0\}$ at which a 
modelled distribution can have a singularity. Let $\gamma, \eta \in \R$ and let $Z$ be a model on the 
regularity structure $\CT_{\CB,\gamma}$. 

\begin{definition}
For a sector $V$ of $\CT_{\CB,\gamma}$ we define as 
in~\cite[Def.~6.2]{Regularity} the set $\CD^{\gamma, \eta}(V)$ of modelled distributions $U :\Lambda \setminus P \to V$ with singularities at $P$. Whenever the sector $V$ is clear from the context, we will simply 
write $\CD^{\gamma,\eta}$. 
\end{definition}

When solving a fixed point problem in a space of modelled distributions, it is important to restrict the domain 
to a fixed time interval $[0, T]$. In this case, we write $\CD^{\gamma, \eta}_T$ for the respective space, and 
we denote by $\$\bigcdot\$_{\gamma, \eta; T}$ the respective norm.

Recalling the definitions of the spaces $\SH_\CB$ and $\bar \SH_\CB$ in Section~\ref{subsec:coherence}, we 
denote by $\SU$ the set of all maps $U : \Lambda \setminus P \to \SH_\CB$ and by $\bar \SU$ the set of all 
maps $\bar U : \Lambda \setminus P \to \bar \SH_\CB$. Then for these maps we write $U = (U_\ft)_{\ft \in \fL_+}$ and $\bar U = (\bar U_\ft)_{\ft \in \fL_+}$. For $U \in \SU$ and a model $(\Pi, \Gamma)$ we shall 
understand $\Pi_z U$ and $\Gamma_{z \bar z} U$ component-wise, i.e.\ $\Pi_z U = \big(\Pi_z U_\ft\big)_{\ft \in \fL_{+}}$ and $\Gamma_{z \bar z} U = \big(\Gamma_{z \bar z} U_\ft\big)_{\ft \in \fL_{+}}$.

\begin{definition}
For a choice of constants $\gamma_\ft, \eta_\ft \in \R$ with $\ft \in \fL_{+}$ and a model $Z$ on $\CT_{\CB, \bar\gamma}$ for $\bar\gamma = \max_{\ft \in \fL_{+}} \gamma_\ft$ we define the space
\begin{equ}[eq:U-space]
	\SU^{\gamma,\eta} \eqdef  \bigoplus_{\ft \in \fL_{+}} \CD^{\gamma_\ft, \eta_\ft}(\TT_{\ft, \leq \gamma_\ft})\,,
\end{equ}
which satisfies $\SU^{\gamma,\eta} \subset \SU$. In the case when the time variable is restricted to the 
interval $[0, T]$, we write $\SU^{\gamma,\eta}_T$. The corresponding norm on $\SU^{\gamma,\eta}$ 
is denoted by $\$ \bigcdot \$_{\gamma,\eta}$.
\end{definition}

In the following result we prove that the map~\eqref{eq:Xi_hat_B} acts on the appropriate spaces of modelled 
distributions. 

\begin{lemma}\label{lem:Xi}
Let $\gamma = (\gamma_\ft)_{\ft \in \fL_{+}} \in \R_+^{\fL_+}$, and let $\eta = (\eta_\ft)_{\ft \in \fL_{+}} \in \R_+^{\fL_+}$ be such that $\eta_\ft \in [0, \gamma_\ft]$ for $\ft \in \fL_{+}$. For $o \in \Eps\setminus \CO$ define $\hat \gamma_o = \min_{(\ft, p) \in \Eps_+(o)} (\gamma_\ft - |p|_\s)$ and $\hat \eta_o = \min_{(\ft, p) \in \Eps_+(o)} (\eta_\ft - |p|_\s)$ and let $k_\star$ be given in Lemma~\ref{lem:k_star} for this value $\hat \gamma_o$. Then the map $\hXi_{\CB, o}$, defined in~\eqref{eq:Xi_hat_B}, is locally Lipschitz continuous 
from $\SU^{\gamma,\eta}$ to $\CD^{\hat \gamma_o + \deg o, \hat \eta_o + \deg o}$, with range in a sector of 
regularity $\deg o$.
\end{lemma}

\begin{proof}
For $U \in \SU^{\gamma,\eta}$, let us denote $V(z) = \CQ_{\leq \hat \gamma_o + \deg o} \hXi_{\CB, o}(U)(z)$. 
Then it obviously takes values in a sector of regularity $\deg o$, and we need to prove the bounds 
\minilab{eqs:Xi-bounds}
\begin{equs}
	 \| V(z) \|_{\zeta} &\lesssim \$U\$_{\gamma,\eta} |z|_0^{(\hat \eta_o + \deg o - \zeta) \wedge 0}\;, \label{eq:Xi-bound1} \\
	  \|\Gamma_{z \bar z} V(\bar z) - V(z)\|_\zeta &\lesssim \$U\$_{\gamma,\eta} \|z - \bar z\|_\s^{\hat \gamma_o + \deg o - \zeta} | z, \bar z |_0^{\hat \eta_o - \hat \gamma_o}\;, \label{eq:Xi-bound2}
\end{equs}
for any $\zeta \in \CA \eqdef \{\deg \tau : \tau \in \fT(\Rule)\}$ satisfying $\zeta < \hat \gamma_o + \deg o$. 
Here, we use the norm $\| \bigcdot \|_{\zeta} \eqdef \| \CQ_\zeta (\bigcdot) \|$ on the regularity structure $\CT_\CB$, and the quantities $|t,x|_0 \eqdef \sqrt{|t|} \wedge 1$ and $|z, \bar z|_0 \eqdef |z|_0 \wedge |\bar z|_0$. Let us 
furthermore denote $\bar V(z) = \CQ_{\leq \hat \gamma_o + \deg o} \hXi_{\CB, o}(\bar U)(z)$ for $\bar U \in \SU^{\gamma,\eta}$. Then the 
Lipschitz continuity will follow if we prove the bounds 
\begin{equs}
	 \| V(z) - \bar V(z) \|_{\zeta} &\lesssim \$U - \bar U\$_{\gamma,\eta} |z|_0^{(\hat \eta_o + \deg o - \zeta) \wedge 0}\;, \label{eq:Xi-bound-difference} \\
	  \|\Gamma_{z \bar z} (V - \bar V)(\bar z) - (V - \bar V)(z)\|_\zeta &\lesssim \$U - \bar U\$_{\gamma,\eta} \|z - \bar z\|_\s^{\hat \gamma_o + \deg o - \zeta} | z, \bar z |_0^{\hat \eta_o - \hat \gamma_o}\;,
\end{equs}
where the proportionality constants are linear in $\$U\$_{\gamma,\eta} + \$\bar U\$_{\gamma,\eta}$.

We will prove only~\eqref{eqs:Xi-bounds} and the bounds~\eqref{eq:Xi-bound-difference} can be proved by 
analogy. From Lemma~\ref{lem:k_star} we have 
\begin{equ}[eq:Xi_0]
	V(z) = \sum_{\alpha \in \N^{\Eps_+(o)} } {1\over \alpha!} \J_{o} \Bigl[\; \delta^{(\alpha)}_{\bu^U(z)\restr_{\Eps_+(o)}} \otimes \CQ_{\leq \hat \gamma_o} \tilde u(z)^{\alpha}\Bigr]\;.
\end{equ}
Only the modelled distributions $U_{\bar o}$ for $\bar o \in \Eps_+(o)$ contribute to the sum 
in~\eqref{eq:Xi_0}, which are all function-like. Then as in the proof of Lemma~\ref{lem:k_star} that writing $U_{\bar o}(z) = u_{\bar o}(z)\one + \tilde u_{\bar o}(z)$, one has $\tilde u_{\bar o}(z) = \sum_{\sigma : \reg\bar o \leq \deg \sigma < \gamma_{\bar o}} \tilde u_{\bar o, \sigma}(z) \sigma$, for $\tilde u_{\bar o, \sigma}(z) \in \R$. For $\zeta \in \CA$, let $\hat \zeta = \zeta - \deg o$. Then we have for $\hat\zeta < \hat\gamma_o$
\begin{equ}
	 \| V(z) \|_{\zeta} \lesssim \sum_{\alpha \in \N^{\Eps_+(o)}} \Bigl\| \delta^{(\alpha)}_{\bu^U(z)\restr_{\Eps_+(o)}}\Bigr\|_{\CB^{\otimes \Eps_+(o)}}\, \|\tilde u(z)^{\alpha}\|_{\hat \zeta}\;.
\end{equ}
Since all values $\eta_\ft$ are positive, the functions $\bu^U(z)\restr_{\Eps_+(o)}$ are uniformly 
bounded, and from the definition~\eqref{eq:norm_weighted} we conclude $\bigl\| \delta^{(\alpha)}_{\bu^U(z)|_{\Eps_+(o)}}\bigr\|_{\CB^{\otimes \Eps_+(o)}} \lesssim 1$. Furthermore, 
\begin{equ}
	\|\tilde u^{\alpha}(z)\|_{\hat \zeta} \lesssim \sum_{A_\alpha, \zeta_{\bar o, i}} \prod_{\bar o \in A_\alpha} \prod_{i = 1}^{\alpha_{\bar o}} \|\tilde u_{\bar o}(z)\|_{\zeta_{\bar o, i}}\;,
\end{equ}
where the summation is over all subsets $A_\alpha \subset \Eps_+(o)$ containing $\bar o$ for which 
$\alpha_{\bar o} \geq 1$, and all strictly positive values $\zeta_{\bar o, i} \in \CA$ parametrised by $\bar o \in A_\alpha$ and $1 \leq i \leq \alpha_{\bar o}$, and satisfying $\sum_{\bar o, i} \zeta_{\bar o, i} = \hat \zeta$. 
From the definition of modelled distributions, we get 
\begin{equ}[eq:Xi_1]
	\|\tilde u^{\alpha}(z)\|_{\hat \zeta} \lesssim \sum_{A_\alpha, \zeta_{\bar o, i}} \prod_{\bar o \in A_\alpha} \prod_{i = 1}^{\alpha_{\bar o}} |z|_0^{(\eta_{\bar o} - \zeta_{\bar o, i}) \wedge 0}\;.
\end{equ}
One can readily see that if $|z|_0 \leq 1$, then the latter is bounded by $|z|_0^{(\hat \eta_o - \hat \zeta) \wedge 0}$, where $\hat \eta_o$ is as in the statement of this lemma. Indeed, if we have $\eta_{\bar o} - \zeta_{\bar o, i} \geq 0$ for all $\bar o$ and $i$, then the bound is trivial. In the case $\eta_{\tilde o} - \zeta_{\tilde o, \tilde i} < 0$ for some $\tilde o$ and $\tilde i$, then we simply bound $(\eta_{\bar o} - \zeta_{\bar o, i}) \wedge 0 \geq - \zeta_{\bar o, i}$ for all $(\bar o, i) \neq (\tilde o, \tilde i)$, which yields the 
required bound. Combining everything together, we have proved~\eqref{eq:Xi-bound1}.

Now we turn to the proof of~\eqref{eq:Xi-bound2}. Recall that $\Eps_+(o) = \min_{\bar o \in \Eps_+(o)} \reg(o)$. Let $L = \lfloor\hat \gamma_o / \reg \Eps_+(o) \rfloor$ be the largest value $|\alpha| = \sum_{o \in \Eps_+(o)} \alpha_o$ of the multiindex $\alpha$ 
which can contribute to the sum in~\eqref{eq:Xi_0}. Then as in~\cite[Eq.~4.13]{Regularity} we can use 
only the algebraic properties of the model and modelled distributions to write
\begin{equ}[eq:Xi_2]
\Gamma_{z \bar z} V(\bar z) = \sum_{\substack{\alpha \in \N^{\Eps_+(o)} \\ |\alpha| \leq L}}{1\over \alpha!} \J_{o} \Bigl[\; \delta^{(\alpha)}_{\bu^U(\bar z)\restr_{\Eps_+(o)}} \otimes \bigl(\tilde u(z) + (u(z) - u(\bar z))\one\bigr)^{\alpha}\Bigr] + R(z, \bar z)\,,
\end{equ}
where the remainder $R(z,\bar z)$ satisfies $\| R_1(z,\bar z)\|_\zeta \lesssim \|z - \bar z\|_\s^{\hat \gamma_o + \deg o - \zeta} | z, \bar z |_0^{\hat \eta_o - \hat \gamma_o}$, for any $\zeta \in \CA$ satisfying $\zeta < \hat \gamma_o + \deg o$ (see~\cite[Eqs.~6.15, 6.17]{Regularity} for a detailed explanation of this bound).  Here, we used $|u_{\bar o}(\bar z) - u_{\bar o}(z)| \lesssim \|z - \bar z\|_\s^{\reg \Eps_+(o)} \|z, \bar z\|_0^{(\eta_{\bar o} - \reg \Eps_+(o)) \wedge 0}$, which can be proved as in~\cite[Eq.~6.18]{Regularity}. Furthermore, we have the distributional identity 
\begin{equ}[eq:Xi_3]
\delta^{(\alpha)}_{\bu^U(\bar z)\restr_{\Eps_+(o)}} = \sum_{\substack{\hat \alpha \in \N^{\Eps_+(o)} \\|\alpha + \hat \alpha| \leq L}} {1 \over \hat \alpha!} \delta^{(\alpha + \hat \alpha)}_{\bu^U( z)\restr_{\Eps_+(o)}} (u(\bar z) - u(z))^{\bar \alpha} + \nu^\alpha_{z, \bar z}\;,
\end{equ}
for a distribution $\nu^\alpha_{z, \bar z} \in \CB^{\otimes \Eps_+(o)}$ is such that 
\begin{equ}
	\| \nu^\alpha_{z, \bar z} \|_{\CB^{\otimes \Eps_+(o)}} \lesssim \|z - \bar z\|_\s^{\hat\gamma_o - |\alpha| \reg \Eps_+(o)} \|z, \bar z\|_0^{\chi_\alpha}\,,
\end{equ}
following from our assumption on $U_{\bar o}$, where $\chi_\alpha = (|\alpha| \reg \Eps_+(o) - \hat \gamma_o - |\alpha| \hat \eta_o + \hat \gamma_o \hat \eta_o / \reg \Eps_+(o)) \wedge 0$ 
(see also similar computations in~\cite[Eq.~6.19]{Regularity} for regular functions). Hence, using the 
binomial identity 
\begin{equ}
\sum_{\alpha + \hat \alpha = \tilde \alpha} {1\over \alpha! \hat \alpha!} (u(\bar z) - u(z))^{\hat \alpha} \bigl(\tilde u(z) + (u(z) - u(\bar z))\one\bigr)^{\alpha} = {1\over \tilde \alpha!} \tilde u(z)^{\tilde \alpha}\;,
\end{equ}
where $\tilde \alpha \in \Eps_+(o)$ and the summation is over $\alpha, \hat \alpha \in \Eps_+(o)$, we obtain
\begin{equ}
\Gamma_{z \bar z} V(\bar z) = V(\bar z) + \sum_{\alpha} {1\over \alpha!} \J_{o} \Bigl[ \nu^\alpha_{z, \bar z} \otimes \bigl(\tilde u(z) + (u(z) - u(\bar z))\one\bigr)^{\alpha}\Bigr] + R(z, \bar z)\;,
\end{equ}
where $\alpha$ is as in~\eqref{eq:Xi_2}. For $\zeta \in \CA$, let us denote as above $\hat \zeta = \zeta - \deg o$. Then
\begin{equ}[eq:Xi_4]
\bigl\|\Gamma_{z \bar z} V(\bar z) - V(\bar z)\bigr\|_\zeta \lesssim \sum_{\alpha} \bigl\| \nu^\alpha_{z, \bar z}\bigr\|_{\CB} \|\bigl(\tilde u(z) + (u(z) - u(\bar z))\one\bigr)^{\alpha}\|_{\hat \zeta} + \bigl\|R(z, \bar z)\bigr\|_\zeta\;.
\end{equ}
The required bound on the norm of $R(z, \bar z)$ is provided above, and we need to bound the sum over 
$\alpha$. As in the proof of~\cite[Prop.~6.13]{Regularity} we get
\begin{equ}
\|\bigl(\tilde u(z) + (u(z) - u(\bar z))\one\bigr)^{\alpha}\|_{\hat \zeta} \lesssim \|z - \bar z\|_\s^{\reg \Eps_+(o) |\alpha| - \hat \zeta} \|z, \bar z\|_0^{|\alpha|(\hat \eta_o - \reg \Eps_+(o)) \wedge 0}\;.
\end{equ}
Then using the bound on the norm of $\nu^\alpha_{z, \bar z}$, the sum over $\alpha$ in~\eqref{eq:Xi_4} is 
bounded by $\|z - \bar z\|_\s^{\hat \gamma_o - \hat \zeta} | z, \bar z |_0^{\hat \eta_o - \hat \gamma_o}$, where 
we have used $|\alpha|(\hat \eta_o - \reg \Eps_+(o)) \wedge 0 + \chi_\alpha \geq \hat \eta_o - \hat \gamma_o$.
\end{proof}

\subsection{Admissible models}
\label{sec:admissible}

Assume that for every $\ft \in \fL_{+}$ we are given a Green's function $G_\ft$ 
satisfying~\cite[Assum.~2.8]{BCCH}. Such Green's function can be written as $G_\ft = K_\ft + R_\ft$, such 
that the $K_\ft$ and $R_\ft$ satisfy assumptions at the beginning of~\cite[Sec. 5.1]{BCCH}. In particular 
$K_\ft(z)$ is supported on a unit ball $\|z\|_\s \leq 1$ and agrees with $G_\ft$ on the ball $\|z\|_\s \leq \frac{1}{2}$ as 
well as $R_\ft$ is a smooth function. Construction of such decomposition of $G_\ft$ can be found 
in~\cite[Lem.~7.7]{Regularity}. Given a collection of such Green's functions $G = (G_\ft)_{\ft \in \fL_+}$ we 
define an admissible map on the vector-valued regularity structures $\CT_V$ with $V_\ft = \R$ for $\ft \in \fL_+$.

\begin{definition}\label{def:admissible}
 We say that a linear map $\PPi : \CT_V \to \SC^\infty$ is \emph{admissible} if it satisfies 
 \minilab{eqs:admissible}
 \begin{equ}[eq:admissible1]
	\PPi \one = 1\;,\qquad 
	\PPi \X^m \sigma = z^m \PPi \sigma\;,\qquad
	\PPi \J_{(\ft, p)}[\sigma] = D^p K_\ft * \PPi \sigma\;,
\end{equ}
 \minilab{eqs:admissible}
as well as 
\begin{equs}[eq:admissible2]
	\PPi \J_{(\fl,p)}[\mu  \otimes \sigma] &= \PPi \bigl( \sigma \J_{(\fl,p)}[\mu \otimes \one] \bigr)\;, \\
	\PPi \J_{(\fl,p)}[\mu  \otimes \X^m] &=  z^m \d^p \big(\PPi \J_{\fl}[\mu \otimes \one]\big)\;,
\end{equs}
for all $\sigma \in \CT_V$, $p,m \in \N^{d+1}$, $\ft \in \fL_{+}$, $\fl \in \fL_{-}$ and $\mu \in V_\fl$, where $*$ is distributional convolution, $\J_{(\ft, p)}$ is viewed as a map on $\CT_V$ since $V_\ft = \R$ for $\ft \in \fL_+$ and where we use the notation
\begin{equ}
	\d^p \eqdef \prod_{i = 0}^{d} \frac{\d^{p_i}}{\d y^{p_i}_i}\,,\qquad z^m : \R^{d+1} \to \R\,,\quad z^m(y) = \prod_{i = 0}^{d} y^{m_i}_i\;.
\end{equ}
\end{definition}
In practice, we will only care about admissible maps on $\CT_{V,\gamma}$ for some $\gamma>0$ rather than 
the whole $\CT_V$. Another natural notion, which will however be broken by the 
renormalisation procedure, is the following.

\begin{definition}\label{def:multiplicative}	
We say that an admissible map $\PPi : \CT_V \to \SC^\infty$ is \textit{multiplicative}, if
\begin{equ}
	\PPi \sigma\bar\sigma = (\PPi\sigma)\, (\PPi\bar\sigma)\;, \label{eq:multiplicative1}
\end{equ}
for all $\sigma,\bar\sigma \in \CT_V$. 
\end{definition}

\begin{definition}\label{def:canonical}	
Let smooth functions $\xi = (\xi_\fl)_{\fl \in \fL_-}$ and all their derivatives grow sublinearly at infinity. Let $c \in L(\R^{d+1})$ and $\ba = (\ba_\fl)_{\fl\in\fL_-}$ with $\ba_\fl \in \CL(\R^{\Eps_+(\fl)}, \R^{d+1})$ for every $\fl\in\fL_-$. In the case when $V = \CB$ we say that a multiplicative admissible map $\PPi : \CT_\CB\to \SC^\infty$ is the 
\emph{canonical lift} of $\xi$ translated by $c$ with inhomogeneous shift $\ba$, if
for every $\fl \in \fL_-$, for every distribution $\mu \in \CB_\fl$ and for $z \in \R^{d+1}$
\begin{equ}[eq:canonical]
		\PPi \J_{\fl}[\mu\otimes \one](z) = \int_{\R^{\Eps_+(\fl)}}\xi_\fl(z + cz+ \ba_\fl \cdot \bu) \mu\big(d\bu\big)\;,
\end{equ}
where we use the shorthand notation~\eqref{eq:a-product} with some fixed constants $a_{\fl, o}$. As we 
explain in Remark~\ref{rem:weights}, our definition of the weighted spaces~\eqref{eq:norm_weighted} allows 
testing distributions with smooth functions; hence,~\eqref{eq:canonical} is well-defined.
\end{definition}

The additional translation by $c \in L(\R^{d+1})$ is needed in order to later allow the 
noise $\xi$ to be shifted by a 
renormalisation $\eps^2 C_\eps $ like in~\eqref{eq:rescaled_noise}. When the precise value of the translation $c$ or a shift $\ba$ is irrelevant we simply refer to $\PPi$ as a canonical lift of $\xi$.

\begin{definition}
	Given an admissible map $\PPi$ on $\CT_V$ and a character $g \in \CG^-_V$ we write $\PPi^g$ for 
	\begin{equ}[eq:RenormalisedPi]
		\PPi^g  \eqdef  \PPi \circ M\;,
	\end{equ}
	where $M = (g \otimes \id) \Deltam_V$.
\end{definition}

Let the map $\ZZ : \PPi \mapsto (\Pi,\Gamma)$ be defined as in~\cite[Def.~6.8]{BHZ}. Then the results 
of~\cite[Sec.~6]{BHZ} imply that for $\gamma>0$ if $\PPi : \CT_{V,\gamma} \to \CC^\infty$ is multiplicative, 
then both $(\Pi, \Gamma) = \ZZ(\PPi)$ and $(\hat\Pi, \hat\Gamma) = \ZZ(\PPi \circ M)$ are in fact models on 
$\CT_{V,\gamma}$ for every renormalisation map $M \in \fR(\CT_{V,\gamma})$. We shall say that a model 
$\ZZ(\PPi)$ is admissible if $\PPi$ is, and we denote $\$ \PPi \$_\fK \eqdef \$ \ZZ(\PPi) \$_\fK$.

\begin{remark}\label{rem:models}
	Strictly speaking, \cite{BHZ} does not deal with trees of the form  $\J_{(\fl, p)}[\mu\otimes \sigma]$ with $\fl \in \fL_-$ and $p \in \N^{d+1}$, but the results still apply since for admissible models we have \eqref{eq:admissible2} which is consistent with the fact that Remark~\ref{rem:quotient} guarantees that the
	action of the structure group commutes with that of $\J_{(\fl,p)}[\mu \otimes \bigcdot]$. 
	Furthermore, Lemma~\ref{lem:down} guarantees
	that models satisfying this property are preserved under the action of the renormalisation group.
\end{remark}

We denote by $\M_\infty(\CT_{\CB,\gamma})$ the space of all smooth models on $\CT_\CB$ of the form 
$\ZZ(\PPi)$, for some admissible map $\PPi$. We note that for every $\gamma > 0$ the ``distances'' between 
two models $\$\bigcdot\,; \bigcdot\$_{\fK}$ for all compact $\fK \subset \R^{d+1}$, introduced in the 
beginning of Section~\ref{sec:models}, define a metric $d_\gamma$ on $\M_\infty(\CT_{\CB,\gamma})$. We 
then define $\M_0(\CT_{\CB,\gamma})$ to be the completion of $\M_\infty(\CT_{\CB,\gamma})$ with respect to 
this metric $d_\gamma$.\label{lab:models}

\begin{remark}
	Given a multiplicative admissible map $\PPi$ and $\bu \in \R^{\Eps_+}$, we can define 
	$\PPi^\bu = \PPi \circ \Ev_\bu$ which gives a multiplicative admissible map on $\CT_\CD$ as well a model 
	$(\Pi^\bu, \Gamma^\bu) = \ZZ(\PPi^\bu)$. Unfortunately this does not allow us to later formulate 
	our PDE in the 
	space of modelled distributions on $\CT_\CD$. This is because ultimately $\bu$ is going to represent the 
	reconstruction of the solution itself and therefore is not a fixed number, but rather a function $\bu : \R^{d+1} \to \R^{\Eps_+}$. It is also not true in general that $\PPi_z^{\bu(z)}$ is a model. This is another 
	motivation why we want to separate the regularity structure $\CT_\CB$, where we perform the analysis, from 
	the regularity structure $\CT_\CD$, where we perform the algebraic renormalisation.
\end{remark}

\subsection{Properties of the reconstruction map}

We now take a look at the reconstruction of the abstract function $\BF_{\CB,\ft}$. For this we first 
examine the structure of the functions that conform to our rule. Note that Assumption~\ref{ass:rule}\ref{R2} 
on the rule $\Rule$ guarantees that for $F \in \SQ(\Rule)$ and $\ft \in \fL_+$, the function $F_\ft$ must be 
polynomial in $(\fX_{o})_{o \in \Eps_-}$ so that
\begin{equ}[eq:Fconforms]
	F_\ft(\fX) = \sum_{\alpha \in \N^{\Eps_-}} F^\alpha_\ft(\fX)\fX^\alpha\,,
\end{equ}
and $F^\alpha_\ft$ is non-zero only for finitely many $\alpha$. Furthermore, $\Eps(F^\alpha_\ft) \subset \mathcal \Eps_+$ for each $\alpha\in \Eps_-$ and $\ft \in \fL_+$.

We recall from~\cite[Rem. 3.15]{Regularity} that for a smooth model $(\Pi, \Gamma)$, the reconstruction map $\CR$ is given by
\begin{equ}[eq:reconstruction]
	(\CR U)(z) = (\Pi_z U(z) \bigr)(z)\;,
\end{equ}
where $U$ is a modelled distribution. The following result shows how the reconstruction map acts on 
nonlinearities that conform to the rule.

\begin{lemma}\label{lem:RF}
	Let the noise $\xi$ satisfy assumptions of the Definition~\ref{def:canonical}. Let $\PPi: \CT_\CB \to \SC^\infty$ be the canonical lift of $\xi$ translated by $c \in L(\R^{d+1})$, and 
	$(\Pi,\Gamma) = \ZZ(\PPi)$ be the canonical model. Let $U \in \SU$ and let $u_\ft(z) = (\CR U_\ft)(z)$ for 
	$\ft \in \fL_{+}$. Then for all $z \in \Lambda \setminus P$ and $\ft \in \fL_{+}$
	\begin{equ}[eq:RF]
		\CR \Big(\CQ_{\leq 0} \BF_{\CB,\ft}(U)\Big)(z) = F_\ft\big(\bu, \xi^{\bu,c}\big)(z)\;,
	\end{equ}
	where $(\bu, \xi^{\bu,c})$ is defined in~\eqref{eq:uxi}.
\end{lemma}
\begin{proof}
	First, note that for $o = (\ft, p) \in \CO$
	\begin{equ}
		\d^p u_\ft(z) = \d^p \Pi_z(U(z))(z) = \Pi_z(\SD^p U_\ft(z))(z) = \Pi_z(U_o(z))(z)\,.
	\end{equ}
	Using the fact that $\Pi_z(\sigma)(z) = 0$ for all $\sigma \in \scal{\tau}_\CB$ with $\deg(\tau) > 0$ as 
	well as~\eqref{eq:admissible2}, the above implies in particular that $\d^p u_\ft(z) = u^U_o(z)$ for $o \in \Eps_+$ by Remark~\ref{rem:sector}. Therefore, for $(\fl,p) \in \fL_- \times \N^{d+1}$ the definitions \eqref{eq:Xi_hat_B} and \eqref{eq:admissible2} yield
	\begin{equs}
		\Pi_z\bigl(\hXi_{\CB, (\fl,p)}(U(z))\bigr)(z) &= \Pi_z \bigl( \J_{(\fl,p)} \bigl[ \delta_{\bu^U|_{\Eps_+(\fl)}(z) } \otimes \1\bigr]\bigr)(z) \\
		&= \PPi \bigl( \J_{(\fl,p)} \bigl[ \delta_{\bu^U|_{\Eps_+(\fl)}(z) } \otimes \1\bigr]\bigr)(z) \\
		&= \d^p \Big(\PPi \bigl( \J_{\fl} \bigl[ \delta_{\bu^U|_{\Eps_+(\fl)}(z) } \otimes \1\bigr]\bigr)(x) \Big)\Big|_{x=z}\\
		&= \int_{\R^{\Eps_+(\fl)}} (\d^p\xi_\fl)(z + cz+ \ba_\fl \cdot \bv) \delta_{\bu^U|_{\Eps_+(\fl)}(z) }(\bv) d\bv\\ 
		&= (\d^p\xi_\fl)\big(z + cz + \ba_\fl \cdot \bu^U(z)\big) = \xi_{(\fl,p)}^{\bu,c}(z), \label{eq:RF_proof1}
	\end{equs}
	where in the second equality we have used $\deg(\fl,p) < 0$.
	
	Second, let $\beta \in \N^{\Eps_-}$, we can view $\fX^\beta$ as a function $\R^{\Eps} \to \R$ by setting 
	$\fX^\beta(\bv)= \bv^\beta$. Since polynomials equal their Taylor expansion, 
	we get for $\bv, \tilde{\bv} \in \R^{\Eps}$
	\begin{equ}
		\fX^\beta(\bv) = \sum_{\alpha \in \N^{\Eps_-}} \frac{D^\alpha \fX^\beta (\tilde\bv)}{\alpha!} \fX^\alpha(\bv-\tilde\bv)\;.\label{eq:RF_proof2}
	\end{equ}
	Note also that for $\alpha \in \N^{\Eps}$ such that there is $o \in \Eps_+$ with $\alpha(o) > 0$ one has
\begin{equ}
	\Pi_z\bigl(\textbf{U}(z)-\bu^U(z) \1, \bXi_\CB^U(z)\bigr)^\alpha(z) = 0\;,
\end{equ}
	 by multiplicativity of $\Pi$. Finally, for $F \in \SQ(\Rule)$ we use~\eqref{eq:Fconforms} to obtain (dropping $z$ for an easier presentation)
	\begin{equs}
		\Pi \Big(\CQ_{\leq 0} \BF_\CB\big(U\big)\Big) &= \sum_{\alpha \in \N^{\Eps_-}} \frac{D^\alpha F (\bu^U,0)}{\alpha!} \Pi \bigl(\textbf{U}-\bu^U \1, \bXi_\CB^U\bigr)^\alpha \\
		&= \sum_{\alpha \in \N^{\Eps_-}}\sum_{\beta \in \N^{\Eps_-}} \frac{ D^\alpha  F^\beta(\bu^U,0) \fX^\beta (\bu^U,0)}{\alpha!} \Pi \bigl(\textbf{U}-\bu^U \1, \bXi_\CB^U\bigr)^\alpha\\
		&= \sum_{\beta \in \N^{\Eps_-}}F^\beta(\bu^U,0)  \sum_{\alpha \in \N^{\Eps_-}} \frac{ D^\alpha  \fX^\beta (\bu^U,0)}{\alpha!} \Pi \bigl(\textbf{U}-\bu^U \1, \bXi_\CB^U\bigr)^\alpha\\
		&= \sum_{\beta \in \N^{\Eps_-}}F^\beta(\bu^U,0)\; \Pi \bigl(\textbf{U}-\bu^U \1, \bXi_\CB^U\bigr)^\beta\;,\label{eq:RF_proof3}
	\end{equs}
	where in the third equality we use the fact that $\Eps(F^\beta) \subset \Eps_+$ and that $F^\beta \neq 0$ 
	only for finitely many $\beta$, and in the fourth equality we use~\eqref{eq:RF_proof2} together with the 
	multiplicativity of $\Pi$ and $\Pi \one = 1$. Now since $\Eps(F^\beta) \subset \Eps_+$ we have 
	$F^\beta(\bu^U,0) = F^\beta(\bu,\xi^{\bu,c})$. Using the fact that $\d^p u_\ft = \Pi(U_{(\ft,p)})$ 
	and~\eqref{eq:RF_proof1} we also get for $\beta \in \N^{\Eps_-}$ that $\Pi \bigl(\textbf{U}-\bu^U \1, \bXi_\CB^U\bigr)^\beta = \fX^\beta(\bu,\xi^{\bu,c})$. Plugging this into~\eqref{eq:RF_proof3} and using 
	again~\eqref{eq:Fconforms}, we obtain
	\begin{equ}
		\Pi_z \Big(\CQ_{\leq 0} \BF_\CB\big(U(z)\big)\Big) (z) = \sum_{\beta \in \N^{\Eps_-}} F^\beta(\bu,\xi^{\bu,c})(z) \fX^\beta(\bu,\xi^{\bu,c})(z) = F_\ft\big(\bu, \xi^{\bu,c}\big)(z)\,,
	\end{equ}
	thus finishing the proof.
\end{proof}

\begin{remark}\label{rem:reconstruction}
	The result of Lemma~\ref{lem:RF} still holds if one replaces $\CQ_{\leq 0}$ in~\eqref{eq:RF} by 
	$\CQ_{\leq a}$ for any $a > 0$. This is due to the fact that analytical properties of 
	models~\eqref{eq:model} imply that $\Pi_z(\sigma)(z) = 0$ for all $\sigma \in \scal{\tau}_\CB$ with 
	$\deg(\tau) > 0$.
\end{remark}

\subsection{BPHZ character}
\label{sec:renorm-models}

\begin{definition}\label{def:BPHZ}
	Assume that we are given a collection of random stationary smooth functions $\xi = (\xi_\fl)_{\fl \in \fL_-}$ 
	 satisfying the assumption of Definition~\ref{def:canonical} almost surely. Let $\PPi$ be the canonical lift of $\xi$ to $\CT_\CB$. We define characters $g^-(\PPi)$, $\ren \in \CG^-_\CB$ on $\CT^-_\CB$ by
	\begin{equ}[eq:character]
		g^-(\PPi)(\sigma) \eqdef \E(\PPi \sigma) (0)\;,\qquad\text{and}\qquad \ren(\sigma) = g^-(\PPi)\tilde\CA^-_\CB \sigma\,,
	\end{equ}
	where $\tilde\CA^-_\CB$ is the negative twisted antipode defined for the vector space assignment $\CB$ (see Section~\ref{sec:VectorRS}).
	The character $\ren$ is called the BPHZ character for $\PPi$. We will call $M^\BPHZ = (\ren \otimes \id)\Delta^-_\CB$ the BPHZ renormalisation on $\CT_\CB$ and $\ren^{\,\bu} = \ren \circ \Ev_\bu$ and $M^\BPHZ_\bu$ respectively are called BPHZ character and a BPHZ renormalisation on $\CT_\CD$. Finally, $\PPi^\BPHZ$ is called the BPHZ lift of $\xi$.\label{lab:BPHZ}
\end{definition}

If we are given a random collection of functions $\xi^\eps$ satisfying the above assumptions for all $\eps > 0$ we shall write $\ren^\eps$ and $\ren^{\,\bu,\eps}$ the corresponding BPHZ characters which are coming from the canonical lift $\PPi^\eps$ of $\xi^\eps$. Note that as in the proof of Lemma~\ref{lem:EvM} we have $\tilde\CA^-_\CB \Ev_\bu = \Ev_\bu \tilde\CA^-_\CD$ thus implying $\ren^{\,\bu} = g^-(\PPi) \Ev_\bu \tilde\CA^-_\CD$.

The next proposition presents a sufficient condition Minor comment (9) on the noises guaranteeing that the BPHZ character $\ren$ 
is translation invariant.

For this, assume that we are given a partition of the set of noise labels $\fL_- = \sqcup_{i \in I} \fL^i_-$ such that for every $i\in I$ there exists a finite subset $\Eps_+^i \subset \Eps_+$ and $\ba^i \in \mathcal{L}(\R^{\Eps_+^i},\R^{d+1})$ with $\ba^i = \ba_\fl$ and $\Eps_+^i = \Eps_+(\fl)$, for all $\fl \in \fL^i_-$.

\begin{proposition}\label{prop:stationarity}
	Let $\xi = (\xi_\fl)_{\fl \in \fL_{-}}$ be a collection of smooth random functions as in Definition~\ref{def:canonical}, such that for all 
	$i,j \in I$ with $i\neq j$ and all $\fl \in \fL_{-}^i$, $\bar\fl \in \fL_{-}^j$ the function $\xi_\fl$ is independent of 
	$\xi_{\bar \fl}$. Then the character $\ren$ is translation invariant.
\end{proposition}

\begin{proof} 
	By linearity of $\tilde\CA^-_\CD$ it is sufficient 
	to show that $g^-(\PPi) T_\bu \sigma = g^-(\PPi) \sigma$ for every $\sigma \in \scal{\tau}_\CB$,$\tau \in \fT_-(\Rule)$ and every $\bu \in \R^{\Eps_+}$. 
	Using explicit formulas for the BPHZ character~\cite[Eq.~3.4 \& Lem.~4.14]{HC} it is easy to see that 
	for each $\bu \in \R^{\Eps_+}$
	\begin{equ}
		g^-(\PPi)T_\bu \sigma = C \big((\ba^{i}\cdot \bu)_{i \in I}\big)\;,
	\end{equ} 
	for some function $C : (\R^{d+1})^{I} \to \R$. Using the assumption on independence of the noises from different sets $\fL^i_-$ we can further obtain
	\begin{equ}
		C \big((\ba^{i}\cdot \bu)_{i \in I}\big) = \int_{(\R^{d+1})^{I}} \hat{K} \big((z_i)_{i\in I}\big) \prod_{i \in I} C_i(\ba^{i}\cdot \bu - z_i)\, dz_i\;,
	\end{equ}
	for some kernel $\hat K : (\R^{d+1})^{I} \to \R$ and functions $C_i : \R^{d+1} \to \R$. Stationarity of the noises imply that $C_i(\ba^{i}\cdot \bu - z_i) = C_i(0 - z_i)$ and therefore
	\begin{equ}
		g^-(\PPi) T_\bu \sigma = C \big((\ba^{i}\cdot \bu)_{i \in I}\big) = C(0) = g^-(\PPi) T_\0 \sigma = g^-(\PPi)\sigma\;,
	\end{equ}
	thus finishing the proof.
\end{proof}

\begin{remark}\label{rem:bphz_abuse}
	For the noise as in the above proposition and the BPHZ character $\ren$ (or $\ren^\eps$) on $\CT_\CB$ we will use an abuse of notation and call $\ren$ (or $\ren^\eps$) to be also the BPHZ character $\CT_\CD$ by writing $\ren = \ren^\0$ (or $\ren^\eps= \ren^{\0,\eps}$) thanks to the Proposition~\ref{prop:translation}.
\end{remark}

\section{Renormalised equations}
\label{sec:renorm_SPDE}

In this section we apply the results from previous sections in order to derive a renormalised equation of a 
certain class of PDEs and show the analogue of~\cite[Thm.~5.7]{BCCH} for an inhomogeneous noise.

In order to write a renormalised equation, we cannot in general use the regularity structure $\CT_\CB$ introduced in Section~\ref{sec:spaces}. The reason is that the positive and negative coproducts $\Delta^+_\CB$ and $\Delta^-_\CB$, which define respectively the positive and negative renormalisations, do not interact correctly on $\CT_\CB$. In particular, the natural identity $\Pi_z^g = \Pi_z M$, for the models in \eqref{eq:RenormalisedPi}, does not always hold. To resolve this problem, we need to use an \textit{extended} regularity structure $\CT^\ex_\CB$ constructed in \cite{BHZ}.

We denote by $\fT^\ex$ be the set of rooted decorated combinatorial trees $\tau = (T, \ff, \fm, \fo)$, where the triple $(T, \ff, \fm)$ is as in Section~\ref{sec:rules} and $\fo : N_T \to \Z^{d+1} \oplus \Z(\fL)$ is an extra node decoration. Here, $\Z(\fL)$ is the free abelian group generated by $\fL$. 
We define the degree assignment, which takes into account the extended decorations $\fo$ (cf. \cite[Def.~5.3]{BHZ}), 
\begin{equ}[eq:degree-extended]
{\deg}^\ex \tau \eqdef \sum_{v \in N_T} (| \fm(v) |_\s + | \fo(v) |_\s) + \sum_{e \in E_T} \deg \ff(e)\;,
\end{equ}
where $| \fo(v) |_\s \eqdef |k|_\s + \sum_{\ft \in \fL} a_\ft \deg \ft$ for $\fo(v) = \bigl(k, \sum_{\ft \in \fL} a_\ft \ft\bigr) \in \Z^{d+1} \oplus \Z(\fL)$ with $a_\ft \in \Z$. In the notation of \cite[Def.~5.3]{BHZ}, the degree \eqref{eq:degree-extended} corresponds to $|\bigcdot|_+$, while \eqref{eq:degree-not-extended} corresponds to $| \bigcdot |_-$. Then we denote by $\CT^\ex_\CB$ and $\CT^\ex_\CD$ the extended regularity structures, defined in \cite{BHZ} and in Section~\ref{sec:VectorRS} for the respective vector assignments $\CB_\ft$ and $\CD_\ft$. Respectively, we can define the sets as in the beginning of Section~\ref{subsec:coherence} for the the extended regularity structure. Since we defined $\deg^\ex$ only on trees, we will keep using $\deg$ on $\fL$.

We note that many definitions from the previous sections can be readily formulated for the regularity structure $\CT^\ex_\CB$. Namely, we can define the projection $\CQ^\ex_{\le \gamma}$ using the degree ${\deg}^\ex$. Then the definitions of models and modelled distributions from Section~\ref{sec:models} hold also for the regularity structure $\CT^\ex_\CB$, where we use $\deg^\ex\tau$ and $\CQ^\ex_{\le \gamma}$ in place of $\deg \tau$ and $\CQ_{\le \gamma}$. Respectively, the rest of Section~\ref{sec:models} can be reformulated on $\CT^\ex_\CB$, since all the differences then lie in the definition of coproducts and twisted antipodes on extended and reduced regularity structures. We will denote an elements of renormalisation group by $M^\ex$ emphasizing that the negative coproduct that is used in $M^\ex$ is an extended negative coproduct.

\begin{remark}
	The main difference between the extended and reduced regularity structures is that the extended negative coproduct $\Delta^{-}_{V,\ex}$ on $\CT^\ex_V$ not only extracts negative trees but also assigns an extended decoration $\fo$ at the node to which the negative tree was contracted. This is done in such a way so that if $\Delta^{-}_{V,\ex} \sigma = \sum_i \sigma^1_i \otimes \sigma^2_i$ then $\deg^\ex \sigma = \deg^\ex \sigma^2_i$ even though $\deg \sigma \leq \deg \sigma^2_i$. Such an assignment of an extended decoration prevents an unnecessary positive renormalisation of the trees $\sigma^2_i$ which leads to a relation $\Pi^g_z = \Pi_z M^\ex$.
\end{remark}

Let $\PPi : \CT^\ex_\CB \to \SC^\infty$ be an admissible multiplicative map, and let $Z = \ZZ(\PPi)$ be the 
corresponding model on $\CT^\ex_\CB$. Let us furthermore write the Green's functions as $G_\ft = K_\ft + R_\ft$, 
where the kernels have the properties listed in Section~\ref{sec:admissible}. Given $\gamma = (\gamma_\ft)_{\ft \in \fL_+} \in \R_+^{\fL_+}$, we denote $\partial \gamma_\ft \eqdef \gamma_\ft - \deg \ft$ and $\partial \gamma \eqdef (\partial \gamma_\ft)_{\ft \in \fL_{+}} \in \R^{\fL_{+}}$.
Then for every $\ft \in \fL_+$ we introduce the integration operator $\CP_\ft$, which acts on modelled 
distributions on $\CT^\ex_\CB$ as
\begin{equ}[eq:P_t]
\CP_\ft  \eqdef  \CK^{\ft}_{\partial \gamma_\ft} + R^{\ft}_{\partial \gamma_\ft} \CR\;,
\end{equ}
where the abstract integral operator $\CK^{\ft}_{\partial\gamma_\ft}$ is defined in~\cite[Eq.~5.15]{Regularity} 
for the value $\partial\gamma_{\ft}$ and for the kernel $K_\ft$, $R^{\ft}_{\partial\gamma_\ft}$ is defined 
in~\cite[Eq.~7.7]{Regularity} for the kernel $R_\ft$, and $\CR$ is the reconstruction map associated to the 
model $Z$ (see~\cite[Sec.~6.1]{Regularity}). The above definition implies that $\CP_\ft$ depends on the 
choice of the model $Z$. Since the underlying model will be always specified, we prefer not to stress this 
dependence in the notation. Let us denote by $\one_+ : \R^{d+1} \to \{0,1\}$ the indicator function of the set $\{(t,x) \in 
\R^{d+1} : t > 0 \}$, which can be canonically identified with an element of $\CD_0^{\infty, \infty}$ (we recall 
that the subscript $0$ means the lowest degree of modelled distributions).

For a suitable choice of values $\gamma_\ft > 0$, we can formulate the abstract version of the 
system~\eqref{eq:system}:
\begin{equ}[eq:AbstractProblemB]
U_{\ft} = \CP_{\ft} \Bigl[\one_+ \CQ^\ex_{\leq \partial\gamma_{\ft}} \BF_{\CB, \ft}(U) \Bigr] + G_\ft u_\ft^0\;,
\end{equ}
where the nonlinearity $\BF_{\CB, \ft}$ is defined in~\eqref{eq:FofU_B}, and where $G_\ft u_\ft^0$ is the 
harmonic extension, given by $G_\ft u_\ft^0(t, x)  \eqdef  \int_{\T^d} G_\ft(t, x-y) u_\ft^0(y)\, dy$ and lifted to 
an abstract polynomial via its Taylor expansion as in~\cite[Eq.~2.6]{Regularity}.

We are not going to formulate a fixed point problem on the regularity structure $\CT^\ex_\CD$. However, given a 
solution $U$ of~\eqref{eq:AbstractProblemB} on $\CT^\ex_\CB$, we can construct a modelled distribution 
$U_\CD$ on $\CT^\ex_\CD$, which is coherent with the nonlinearity $\BF_{\CD, \ft}$ defined in~\eqref{eq:FofU}. 
More precisely, let us denote $\bu^U = (u^{U}_{o} : o \in \CO) \in \R^{\CO}$ with $u^{U}_{(\ft,p)} = \Proj_\one U_{(\ft,p)} = p!\, \Proj_{\X^p}U_\ft$. Then we define 
\begin{equ}[eq:U_on_D]
U_{\CD, \ft} = \sum_{p \in \N^{d+1}} \frac{1}{p!} u^{U}_{(\ft,p)} \X^p + \bUpsilon^{F}_\ft(\bu^U)\;,
\end{equ}
where we have used the map~\eqref{eq:bUpsilon}.

\begin{lemma}\label{lem:U=U_D}
Let $Z \in \M_0$ be an admissible model on $\CT^\ex_\CB$. Let $\gamma,\eta \in \R^{\fL_+}$ such that $\eta_\ft,\partial\eta_\ft > - \s_0$ and 
$\partial\gamma_\ft > 0$ for every $\ft \in \fL_{+}$. Assume that $U \in \SU^{\gamma,\eta}$ and $\one_+ \BF_{\CB, \ft}(U) \in \SU^{\partial\gamma,\partial\eta}$ with respect to $Z$. Then if $U$ is a solution 
of~\eqref{eq:AbstractProblemB} and $U_{\CD}$ is defined by~\eqref{eq:U_on_D}, then for every $\ft \in \fL_+$
\begin{equ}[eq:U=U_D]
	 U_\ft = \CQ^\ex_{\leq \gamma_\ft} \Ev_{\bu^U}\big[U_{\CD,\ft}\big]\,.
\end{equ}
\end{lemma}

\begin{proof}
First, note that since $\Proj_{\X^p} U_{\CD, \ft} = \frac{1}{p!} u^U_{(\ft,p)}$ for $\ft \in \fL_+$ and $p \in \N^{d+1}$, we trivially have $\bu^U = \bu^{U_\CD}$ . This implies that, $U_\CD$ is coherent with $F$ 
to all orders and therefore by Lemma~\ref{lem:twoFs} 
\begin{equ}
	\BF_{\CB, \ft} \bigl( \Ev_{\bu^U}[U_\CD]\bigr) = \Ev_{\bu^U}\big[\bar \bUpsilon^{F}_\ft(\bu^{U_\CD})\big] = \Ev_{\bu^U}\big[\bar \bUpsilon^{F}_\ft(\bu^U)\big]\;.
\end{equ}
Using the above equality and applying evaluation map $\Ev_{\bu^U}$ to~\eqref{eq:U_on_D} yields
\begin{equ}[eq:U_Dcoherent]
\CQ^\ex_{\leq \gamma_\ft}\Ev_{\bu^U}\big[U_{\CD, \ft}\big] = \sum_{|p|_\s \leq \gamma_\ft} \frac{1}{p!} u^{U}_{(\ft,p)} \X^p + \CQ^\ex_{\leq \gamma_\ft}\J_{\ft} \bigl[\BF_{\CB, \ft} \bigl( \Ev_{\bu^U}[U_\CD]\bigr)\bigr]\,.
\end{equ}
Note that by the definition~\eqref{eq:P_t} and~\cite[Eq.~5.15]{Regularity}, 
\begin{equ}[eq:Ucoherent]
	U_\ft = \sum_{|p|_\s \leq \gamma_\ft} \frac{1}{p!} u^{U}_{(\ft,p)} \X^p + \J_{\ft} \bigl[\CQ^\ex_{\leq \partial\gamma_\ft}\BF_{\CB, \ft}( U)\bigr]\;.
\end{equ} 
Now, in order to prove~\eqref{eq:U=U_D} it suffices to show that for all $\ft \in \fL_+$, $\tau \in \fT^\ex$ with 
$\J_\ft[\tau] \in \fT^\ex(\Rule)$ and $\CQ^\ex_{\leq \gamma_\ft} \J_\ft[\tau] \neq 0$ one has $\Proj_{\J_\ft[\tau]} U_{\ft} = \Proj_{\J_\ft[\tau]} \CQ^\ex_{\leq \gamma_\ft} \Ev_{\bu^U}[U_{\CD, \ft}]$. We show this by induction over the 
number of edges in the tree $\tau$. 
It is therefore easy to see that the base case of induction (when $\tau$ has no edges) follows from the fact 
that the right-hand sides of~\eqref{eq:U_Dcoherent} and~\eqref{eq:Ucoherent} have the same polynomial 
parts. Assume that $\tau$ has $k\geq 1$ edges and that the statement holds true for all $\bar\ft \in \fL_{+}$ 
and trees with number of edges less or equal to $k-1$. Then using~\eqref{eq:Ucoherent} we get
\begin{equs}
	\Proj_{\J_\ft[\tau]} U_{\ft} = \Proj_\tau \BF_{\CB, \ft}( U) = \Proj_\tau \BF_{\CB, \ft}\big( \Ev_{\bu^U}[U_\CD]\big)
	=  \Proj_{\J_\ft[\tau]} \Ev_{\bu^U}\big[U_{\CD, \ft}\big]\,,
\end{equs}
where we used the induction hypothesis to justify the equality in the middle and~\eqref{eq:U_Dcoherent} for the last equality. 
\end{proof}

We can now exploit this self-consistent identity to show that if the underlying
model is obtained from a canonical lift by the action of a $\bu$-independent element 
of the renormalisation group, then the solution obeys a renormalised PDE.

\begin{theorem}\label{thm:renormalised_PDE}
Let $\xi = (\xi_\fl)_{\fl \in \fL_{-}}$ be a family of smooth functions satisfying the assumptions of 
Definition~\ref{def:canonical} and let 
$\gamma,\eta$ be as in Lemma~\ref{lem:U=U_D}. Let $\PPi$ be a canonical lift of $\xi$ on translated by 
$c \in L(\R^{d+1})$ and shifted by $\ba = (\ba_\fl)_{\fl \in \fL_-}$ (see Definition~\ref{def:canonical}). Let $g \in \CG^-_{\CB,\ex}$ be a translation invariant character and $(\Pi^g, \Gamma^g) = \ZZ(\PPi^g)$ be a renormalised model on $\CT^\ex_{\CB, \bar \gamma}$ where $\bar\gamma = \max_{\ft \in \fL_+} \gamma_\ft$, and let $\CR^g$ be the respective reconstruction map. Assume that $U \in \SU^{\gamma,\eta}_T$ and $\one_+ \BF_{\CB, \ft}(U) \in \SU^{\partial\gamma,\partial\eta}_T$ are such that $U$ is a local solution to the abstract 
problem~\eqref{eq:AbstractProblemB} on an interval $[0, T]$. For $(\ft,p) \in \CO$ set $u_{\ft} = \CR ^gU_\ft$, $u_{(\ft,p)}  \eqdef  \d^p u_\ft$ and let $(\bu,\xi^{\bu,c})$ be defined by~\eqref{eq:uxi}. Then the 
functions $u_\ft$ give a solution on $[0, T]$ of the following renormalised version of the system of 
PDEs~\eqref{eq:system}:
\begin{equ}[eq:reconstructed_PDE]
\d_t u_{\ft} = \SL_\ft u_{\ft} + F_{\ft} (\bu, \xi^{\bu,c})\; + \sum_{\tau \in \fT^\ex_-(\Rule)} g^\0 \bbUpsilon^{F}_\ft[\tau](\bu,\xi^{\bu,c})\;,
\end{equ}
with the initial condition $u_{\ft}(0, \bigcdot) = u^0_{\ft}(\bigcdot)$.
\end{theorem}

\begin{proof}
Using the properties of the integration maps~\cite[Thm.~5.12, Eq.~7.7]{Regularity}, 
Remark~\ref{rem:reconstruction} and~\eqref{eq:reconstruction},  we obtain
from~\eqref{eq:AbstractProblemB} the identity
\begin{equ}
u_{\ft} = G_\ft * \CR^g  \bigl(\one_+ \CQ^\ex_{\leq \partial\gamma_{\ft}} \BF_{\CB, \ft}(U) \bigr) + G_\ft u_\ft^0\;.
\end{equ}
Note that model $(\Pi^g, \Gamma^g)$ is smooth, therefore we can apply~\cite[Rem.~3.15]{Regularity} and obtain
\begin{equs}[eq:expression]
\CR^g  \bigl(\one_+ \CQ^\ex_{\leq \partial\gamma_{\ft}} \BF_{\CB, \ft}(U) \bigr)(z) =  \one_+(z)\Pi^g_z \bigl(\one_+(z) \CQ^\ex_{\leq \partial\gamma_{\ft}} \BF_{\CB, \ft}(U)(z)\bigr)(z)\;.
\end{equs}
Denote $\frac{1}{p!}u^U_{(\ft,p)} = \Proj_{\X^p} U_\ft$ where $u^U_{(\ft,p)} = 0$ if $|p|_\s > \gamma_\ft$. As in 
the proof of Lemma~\ref{lem:RF} we see that $u^U_{(\ft,p)} = u_{(\ft,p)}$ for $(\ft,p) \in \Eps_+$. The 
evaluation map only depends on the components in $\R^{\Eps_+}$, thus $\Ev_{\bu^U(z)} = \Ev_{\bu(z)}$. Let 
$L \in \N$ be maximal such that for $\gamma_L$ defined in~\cite[Def. 3.15]{BCCH} we have 
$\partial\gamma_\ft \geq \gamma_L$ for all $\ft \in \fL_+$. In particular for such $L$ we have
\begin{equs}[eq:PLQ]
	\Pi_z^g (\CQ^\ex_{\leq \partial\gamma_{\ft}}  \sigma) (z) &= \Pi_z^g (\bp_{\leq L}  \sigma) (z)\,,\\
	\Pi_z^g (\bp_{\leq L}  \sigma) (z) &= \Pi_z (\bp_{\leq L} M^\ex \sigma) (z) 
\end{equs}
for all $\ft \in \fL_+$ and $\sigma \in \CT^\ex_{\CB, \bar \gamma}$, where $M^\ex = (g \otimes \id) \Delta^-_{\CB,\ex}$. First equality in follows from properties of $\gamma_L$ and second follows from the fact that $\Pi^g = \Pi M^\ex$ and the fact that $M^\ex$ preserves the $\deg^\ex$- degree (see proof of~\cite[Thm.~5.7]{BCCH} for more details). This together with Lemmas~\ref{lem:CoherenceEquivalence} and~\ref{lem:U=U_D} yields 
$\bp_{\leq L} \BF_{\CB, \ft}(U) = \bp_{\leq L} \Ev_{\bu} \bbUpsilon^F_\ft (\bu^U)$. Note that this together with Lemma~\ref{lem:EvM} 
suggests that the expression~\eqref{eq:expression}, without the multiplier $\one_+(z)$, is equal to
\begin{equs}[eq:R(MF)]
\Pi^g_z \bigl( \bp_{\leq L} \BF_{\CB, \ft}(U)(z)\bigr)(z) &= \Pi_z \bigl(\bp_{\leq L} M^\ex \Ev_{\bu(z)} \bbUpsilon^F_\ft (\bu^U)\bigr)(z)\\
 &= \Pi_z \bigl(\Ev_{\bu(z)}
  \bp_{\leq L} M^\ex_{\bu(z)} \bbUpsilon^F_\ft (\bu^U)\bigr)(z)\,.
\end{equs}
Let $U_\CD$ be given by~\eqref{eq:U_on_D}, which is coherent with $F$ to any order by definition. By 
Proposition~\ref{prop:MU_coherent}, we have that $M^\ex_{\bu(z)}U_\CD$ is coherent with $\hat M^\ex_{\bu(z)}F$ to 
all orders, which implies that $M^\ex_{\bu(z)} \bbUpsilon^F_\ft (\bu^U)(z) = (\hat M^\ex_{\bu(z)}F)_{\CD,\ft} (M^\ex_{\bu(z)}U_D(z))$, where $(\bigcdot)_{\CD,\ft}$ stands for lifting function in $\SQ(\Rule)$ to a function $\SH^\ex_\CD \to \bar\SH^\ex_\CD$ as defined in~\eqref{eq:FofU}. Putting this into~\eqref{eq:R(MF)} gives 
\begin{equs}
	\Pi^g_z \bigl( \bp_{\leq L} \BF_{\CB, \ft}(U)(z)\bigr)(z) &= \Pi_z \bigl(\Ev_{\bu(z)}
	\bp_{\leq L} (\hat M^\ex_{\bu(z)}F)_{\CD,\ft} \big(M^\ex_{\bu(z)}U_D(z))\big) \\
	&= \Pi_z \bigl(\bp_{\leq L} (\hat M^\ex_{\bu(z)}F)_{\CB,\ft} \big(\Ev_{\bu(z)} M^\ex_{\bu(z)}U_D(z))\big) \\
	&= \Pi_z \bigl(\bp_{\leq L} (\hat M^\ex_{\bu(z)}F)_{\CB,\ft} \big(M^\ex U(z))\big) \\
 	&= (\hat M^\ex_{\bu(z)} F)_\ft\big((\bu,\xi^{\bu,c})\big)(z)\\
 	&= F_{\ft} (\bu, \xi^{\bu,c})(z)\; + \sum_{\tau \in \fT^\ex_-(\Rule)} g^\0\bbUpsilon^{F}_\ft[\tau](\bu,\xi^{\bu,c})(z)\;,
\end{equs}
where in the second equality we have used~\eqref{eq:F_eval}, in the third equality we have used 
Lemma~\ref{lem:U=U_D} and~\eqref{eq:PLQ} to deduce $\bp_{\leq L}\Ev_{\bu(z)} M^\ex_{\bu(z)}U_\CD(z) = M^\ex \CQ^\ex_{\leq \gamma_\ft} \Ev_{\bu(z)} U_\CD(z) = \bp_{\leq L} M^\ex U(z)$, in 
the fourth equality we used Lemma~\ref{lem:RF} together with~\eqref{eq:PLQ} and the definition of 
$(\bu,\xi^{\bu,c})$, and finally in the last equality Lemma~\ref{lem:MF_u} was used together with translation invariance of $g$ and Proposition~\ref{prop:translation}.
\end{proof}

\begin{remark}\label{rem:reduced_reg_str}
	Given a canonical lift on a reduced regularity structure $\PPi \colon \CT_\CB \to \SC^\infty$ one can lift this map on an extended regularity structure simply by setting $\PPi^\ex = \PPi \circ Q$ where operator $Q$ simply ignores the extended decoration. This together with results of~\cite[Sec.~6.4]{BHZ} imply that the above theorem works on the reconstruction of the equation~\eqref{eq:AbstractProblemB} on the reduced regularity structure with $\fT^\ex_-(\Rule)$ changed to $\fT_-(\Rule)$ in~\eqref{eq:reconstructed_PDE}.
\end{remark}

\section{Fixed point problem}\label{sec:fixed_point}

We now turn our attention to solving the system of abstract equations with non-homogeneous noises. 
The abstract version of equation~\eqref{eq:main2} contains extra difficulties like (seemingly) 
quasilinear terms and the presence of nonlinearities that depend on the reconstruction. For this 
reason we will not consider the most general case of equation~\eqref{eq:AbstractProblemB}, but 
rather a slight generalisation of the abstract analogue of~\eqref{eq:main2}. Nevertheless, such 
a system will contain all the necessary ingredients to treat equations with non-homogeneous noises. 

We let $d = 1$, take the scaling $\s = (2,1)$, and let the rule $\Rule$ be as defined in Section~\ref{sec:first}. For some $\gamma_i,\eta_i$ with $\gamma_i > 0$ for $i = 0, \ldots, 3$, let us define $\TT_{\ft_{i}, \leq \gamma_i} \eqdef \CT_{\CB,\leq \gamma_i} \cap \TT_{\ft_{i}} = \CQ_{\leq \gamma_i} \TT_{\ft_{i}}$, where the spaces $\TT_{\ft_{i}}$ are given in~\eqref{eq:jets}. Note that since there is no symbol $\ft_3$ present in the rule then we simply have $\TT_{\ft_{3}, \leq \gamma_3} = \{\X^k\,\colon\, |k|_\s \leq \gamma_3\}$. Let $Z \in \SM_0(\CT_{\CB, \max\{\gamma_i\}})$ be a (possibly non smooth) admissible model and set $\CD^{\gamma_i, \eta_i} = \CD^{\gamma_i, \eta_i}(\TT_{\ft_{i}, \leq \gamma_i}, Z)$, for $i = 0,\ldots,3$, to be the spaces of modelled distributions on $\TT_{\ft_{i}, \leq \gamma_i}$ with respect to $Z$. We 
let $k_\star$ be as given in Lemma~\ref{lem:k_star}, which is used in the definition of the spaces $\CB_\fl$ in Section~\ref{sec:spaces}. Assume that we are given an element
\begin{equ}
	H = (H_0,H_1,H_2, H_3) \in \SU^{\gamma, \eta}\;,
\end{equ}
where $\SU^{\gamma, \eta}$ is defined in \eqref{eq:U-space}.

We would like to define the functions on the right-hand sides in \eqref{eq:main2} at 
the level of modelled distributions. For this, we look at the slightly more general 
class of equations given in \eqref{eq:abstract-functions} below. 
The reason is that since we will use the renormalised BPHZ model, we need some freedom to adjust
the equation in such a way that even with the renormalised model it reduces to \eqref{eq:main2} 
after reconstruction (see Proposition~\ref{prop:abstract_system}). More precisely, the functions $\BG$ and $\bar \BG$ in \eqref{eq:main2} are needed to cancel the terms which appear due to 
renormalisation. We gather in the following assumption all the ingredients which we need for 
our solution theory. We make it a standing assumption for this section. Let $m > 0$ and $p \in \N$ we write $\CC^\func_{m,p}$ for the space of functions $F\colon \R \to \R$ such that
\begin{equ}
\|F\|_{\CC^\func_{m,p}} \eqdef \sup_{\ell \le p} \sup_{x \in \R} e^{-m|x|} |F^{(\ell)}(x)| < \infty\;.
\end{equ}
Here, the exponential weight is somewhat arbitrary, any other weight growing faster 
than a sufficiently high power of $x$ would also do.

\begin{assumption}\label{assum:all-functions}
\begin{enumerate}
\item The functions $F_i$, for $i = 1, 2, 3$, satisfy Assumption~\ref{assum:functions}.
\item One has $\tilde{F}_1 \in \CC^\func_{m, 7}$ with $\| \tilde F_1''\|_{L^\infty(\R)}  < \infty$ and $\| \tilde F_1'''\|_{L^\infty(\R)} < \infty$.
\item One has $\tilde{F}_i \in \CC^\func_{m, 7}$ for $i = 2, 3$ and $G, \bar G \in \CC^\func_{\bar m, 5}$ for the some value $m, \bar m > 0$.\footnote{The values of the exponential weights for the functions $G$ and $\bar G$ are dictated by their definitions in the proof of Proposition~\ref{prop:abstract_system} for the mean curvature flow. For this section precise values of $m$ and $\bar m$ are irrelevant.}
\item\label{it:all-functions-xi} The function $(t,x,h) \mapsto \xi^{h}_{3}(t,x)$ satisfies $\sup_{h \in \R}(1+|h|)^{-1}\|\d^\ell_h \xi^{h}_{3}\|_{\CC^{1/2 - \kappa_\star}_{\s, 1}} < \infty$
 for $\ell = 0, 1, 2$ and for $\kappa_\star$ as in \eqref{eq:deg_main}.
\end{enumerate}
\end{assumption}

\begin{remark}
	Strictly speaking we need $\tilde F_{1} \in \CC^\func_{m, 7}$, $\tilde F_{2} \in \CC^\func_{m, 5}$, $\tilde F_{3} \in \CC^\func_{m, 3}$, $\bar G \in \CC^\func_{\bar m, 5}$, $G \in \CC^\func_{\bar m, 1}$ in order to make the corresponding terms of the equation of positive degree after Taylor expansion as in~\eqref{eq:abstract_functions}. We use a suboptimal assumption on derivatives just for less heavy notation and because these will be satisfied for the functions from Section~\ref{sec:intro}. Same applies to $F_i$.
\end{remark}

We lift the functions $\tilde{F}_i, \bar G$ to the respective functions $\tilde{\BF}_i, \bar \BG \colon \SH_\CB^4 \to \bar \SH_\CB$ via~\eqref{eq:FofU_B}.\footnote{In~\eqref{eq:FofU_B} functions are assumed to be smooth but this is not an issue since higher order derivatives will be ignored because of the truncation $\CQ_{\le \partial\gamma_\ft}$ in the abstract equation~\eqref{eq:AbstractProblemB}.}
We also denote by $\SD = \SD_1$ the abstract ``spatial'' derivative defined in Section~\ref{subsec:coherence}.
 More precisely, for any constants $\bar C_i$, $i = 1,2,3$, for $\eps > 0$ and for $H = (H_0,H_1,H_2, H_3) \in \SH_\CB^4$ we define functions $\BF_{\ft_i}\colon \SH_\CB^4 \times \R_+ \times \T \to \bar \SH_\CB$ by
\begin{equs}[eq:abstract-functions]
	\BF_{\ft_0} (H,t,x) &= \tilde \BF_{1}( \SD H_{1})\, \SD^2 H_{1} + \lambda\, (\SD H_0)^2 + \SD H_{0} \tilde \BF_{2}(\SD H_{1}) \\
	&\qquad + \SD H_{0} \bar \BG(\SD H_{1}) + G (\d_x h_1) + \sigma \hat\Xi_{\CB, \fl_0}(H) \\
	&\qquad + \sigma_1 (\SD H_1)^2 \hat\Xi_{\CB, \fl_1}(H) + \tilde \BF_{3}(\SD H_{1}) \hat\Xi_{\CB, \fl_2}(H)\;, \\[1.0em]
	\BF_{\ft_1} (H,t,x) &= \tilde \BF_{1}( \SD H_{1})\, \SD^2 H_{2} + \lambda\, \SD H_0 \SD H_1 + \d_x h_{1} \tilde F_{2}(\d_x h_{1}) - \bar C_{1}\\ 
	&\qquad+ \sigma \hat\Xi_{\CB, \fl_1}(H) + \sigma_1 (\SD H_1)^2 \hat\Xi_{\CB, \fl_2}(H) + \tilde F_{3}(\d_x h_1) \xi^{h_{0}}_{3}\;, \\[1em]
	\BF_{\ft_2} (H,t,x) &= \tilde F_{1}(\partial_x h_1)\, \partial^2_x h_3 + \lambda\, (\d_x h_1)^2 + \d_x h_{2} \tilde F_{2}(\d_x h_{1}) - \bar C_{2} \\ 
		&\qquad+ \sigma \hat\Xi_{\CB, \fl_2}(H) + \bigl( \sigma_1 (\d_x h_1)^2 + \eps^{\frac 23} \tilde F_{3}(\partial_x h_1)\bigr)\xi^{h_{0}}_{3}\;, \\[1.0em]
	\BF_{\ft_3}(H,t,x) &= \eps^{\frac{2}{3}} \tilde F_{1}(\partial_x h_1)\, \d_x^2 h_{3} + F_{2}\bigl(\eps^{\frac{1}{3}}\partial_x h_1 \bigr) \partial_x h_1 \partial_x h_2 - \bar C_{3} \\
	&\qquad + F_{3}\bigl(\eps^{\frac{1}{3}} \partial_x h_1\bigr)\xi^{h_{0}}_{3}\;,
\end{equs}
where the constants $\lambda$, $\sigma$ and $\sigma_1$ are as in \eqref{eq:main2}
and where $h_i$ denotes the  $\one$-component of $H_i$, 
$\d_x h_i$ its $\X_1$-component, etc. The functions $h_i$, $\d_x h_i$ and $\xi^{h_{0}}_{3}$ on the right-hand side are evaluated at $(t,x)$ which we don't write to simplify the notation.
The terms  in these 
functions only involving real numbers should be interpreted as being proportional to $\1$, which we 
prefer to leave implicit in the notation. (In particular, $\BF_{\ft_3}$ only takes values 
proportional to $\1$.) Note that the first three components of this equation take values in some spaces of modelled distribution, while the last component is a real-valued function. 

\begin{remark}\label{rem:local_functions}
	The functions $\BF_{\ft_i}$ are of the form
	\begin{equ}
		\BF_{\ft_i}(H, t, x) =: \BF^1_{\ft_i}(H) + F^2_{\ft_i} \big(\partial_x h_1(t,x), \partial_x h_2(t,x), \partial^2_x h_3(t,x), \xi^{h_{0}}_{3}(t,x)\big ) \one \;,
	\end{equ}
	where furthermore $F^2_{\ft_0} = G$, $F^2_{\ft_1} = \d_x h_{1} \tilde F_{2}(\d_x h_{1}) - \bar C_{1}  + \tilde F_{3}(\d_x h_1) \xi^{h_{0}}_{3}$, $\BF^1_{\ft_2} = \sigma \hat\Xi_{\CB, \fl_2}$, $\BF^1_{\ft_3} = 0$, 
	and the $\BF^1_{\ft_i}$ only depend on $H_i$ with $i \le 2$.
This decomposition into a function that is $(t,x)$-independent and a function that only depends on the polynomial components of $H$ allows to identify the $\bbUpsilon^{F}$ operator for $F = (F_{\ft_0}, F_{\ft_1}, F_{\ft_2})$ with the operator $\bbUpsilon^{\tilde F}$ for $\tilde F = (F^1_{\ft_0}, F^1_{\ft_1}, F^1_{\ft_2})$ as in Remark~\ref{rem:F_z}. Thus, Theorem~\ref{thm:renormalised_PDE} still applies for such functions and the functions $F^2_{\ft_i}$ have no effect on the counterterms in~\eqref{eq:reconstructed_PDE}. We shall see that one has
	\begin{equ}
		(t,x) \mapsto F^2_{\ft_i} \big(\partial_x h_1(t,x), \partial_x h_2(t,x), \partial^2_x h_3(t,x), \xi^{h_{0}}_{3}(t,x)\big) \one\; \in\; \CD^{\frac 16 - \delta, \eta}\;,
	\end{equ}
	for some small $\delta > 0$ and some $\eta > 0$ which will allow us to close a Picard iteration. Note that the nonlinearity \eqref{eq:abstract-functions} conforms to the rule presented in Section~\ref{sec:first}.
\end{remark}

The above nonlinearities then define the system of equations
\begin{equ}[eq:main4]
		H_{k} = \CP_{\ft_k} \bigl(\one_+ \BF_{\ft_k}(H)\bigr) + G h^0_{k}\;,\qquad k \le 3\;,
\end{equ}
where $G$ is the heat kernel and the integration maps $\CP_{\ft_k}$ are defined 
by~\eqref{eq:P_t} for the Green's function $G$ and the model $Z$. For $k = 3$, $\CP_{\ft_3}$ can be simply identified with the convolution with the heat kernel as $H_3$ has only polynomial components. We 
now prove that this system of equations has a (unique) solution. 
The following properties of the integration map $\CP_\ft$, which follow 
from~\cite[Prop.~6.16, Thm.~7.1]{Regularity}, are required to close our
Picard iteration.

\begin{lemma}\label{lem:abstract_integration}
	Let  $\ft \in \fL_+$ and $\alpha_\ft, \gamma_\ft, \eta_\ft \in \R$ such that $\gamma_\ft, \eta_\ft \notin \N$, 
	$\gamma_\ft > \deg \ft$ and $\alpha_\ft \wedge \eta_\ft > -\s_0 + \deg \ft$. Let $\kappa \geq 0$ and $\bar V_\kappa$ denote some sector in $\bar \CT^\ex_\ft$ of degree $(\alpha_\ft - \deg \ft + \kappa) \wedge 0$. Then the map~\eqref{eq:P_t} is locally Lipschitz from 
	$\CD^{\partial\gamma_\ft, \partial\eta_\ft}(\bar V_0)$ to $\CD^{\gamma_\ft, \eta_\ft \wedge \alpha_\ft \wedge \deg\ft}(V)$, where $\partial\eta_\ft \eqdef \eta_\ft - \deg \ft$, and $V$ is a sector in $\CT^\ex_\ft$ of degree $\alpha_\ft \wedge 0$. Moreover, for every $T \in (0,1]$ one has
	\begin{equ}
		\$ \CP_\ft (\one_+  f) \$_{\gamma_\ft, \eta_\ft \wedge \alpha_\ft; T} \lesssim T^{\kappa / \s_0} \$ f \$_{\partial\gamma_\ft, \d\eta_\ft + \kappa; T}\;,
	\end{equ}
	for all $f \in \CD^{\partial\gamma_\ft, \partial\eta_\ft+\kappa}(\bar V_\kappa)$, where the proportionality constant is affine in 
	$\$Z\$_{[-1,2] \times \T}$. The norms on spaces of modelled distributions are defined in Section~\ref{sec:models}.
\end{lemma}

Now we can prove existence and uniqueness of the solution to \eqref{eq:main4}. For this, let us define the space $\CC^{\init} \eqdef \bigoplus_{i = 0, \ldots, 3} \CC^{(1 + 2i)/3 + \kappa}(\T)$ to contain the tuples $(h_{i})_{i = 0, \ldots, 3}$ such that $h_{i} \in \CC^{(1 + 2i)/3 + \kappa}$ for $i = 0, \ldots, 3$ for some $\kappa \in (0, \frac{1}{6})$. We denote the norm on this space by $\| \bigcdot \|_{\CC^{\init}}$. 

\begin{proposition}\label{prop:system_soln}
	Let Assumption~\ref{assum:all-functions} be satisfied. There exists $\gamma = (\gamma_i)_{i = 0,\ldots,3},\eta = (\eta_i)_{i = 0,\ldots,3} \in \R^4_+$ with $\gamma_i > 0$ such that the following holds true. For any fixed $M, L > 0$, assume that the quantities involved in \eqref{eq:abstract-functions}
	satisfy the following bounds:
	\begin{enumerate}
		\item One has $|\tilde F_1(0)| + |\tilde F_1'(0)| + \| \tilde F_1''\|_{L^\infty(\R)}  + \| \tilde F_1'''\|_{L^\infty(\R)}\leq M$.
		\item One has $\|F_i\|_{\CC^\func_{m, 7}} \leq L$ and $\|\tilde F_j\|_{\CC^\func_{m, 7}} \leq L$ for $i = 2,3,\, j = 1,2,3$, and similarly for $\|G\|_{\CC^\func_{\bar m, 5}}, \|\bar G\|_{\CC^\func_{\bar m, 5}} \leq L$ for $m, \bar m > 0$ as in Assumption~\ref{assum:all-functions}.
		\item One has $|\bar C_i| \le L$ for $i=1,2,3$.
		\item For $\ell = 0, 1, 2$ the function $\xi_{3}$ satisfies $\sup_{h \in \R}(1+|h|)^{-1}\|\d^\ell_h \xi^{h}_{3}\|_{\CC^{1/2 - \kappa_\star}_{\s, 1}} \leq L$.
		\item The model $Z \in \M_0(\CT_{\CB, \max\{\gamma_i\}})$ is admissible and satisfies $\$ Z \$_{[-1,2] \times \T} \leq L$.
	\end{enumerate}
	Then there exist $\bar \eps = \bar\eps(M) \in (0,1]$ and $T = T \bigl(L, \| h^0 \|_{\CC^{\init}}\bigr) \in (0,1]$, such that given $\eps \in [0, \bar \eps]$ and an initial condition $h^0 = (h_i^0)_{i = 0,\dots, 3} \in \CC^\init$ satisfying 
	\begin{equ}
		\| h^0 \|_{\CC^{\init}} \le \eps^{-\frac 14}\;,
	\end{equ}
the system~\eqref{eq:main4} has a unique local solution $H \in \SU^{\gamma, \eta}_{T}$, satisfying furthermore
	\begin{equ}[e:normh]
		\sup_{t \in [0,T]} \|(\CR H)(t,\cdot) \|_{\CC^\init} \leq 2 \| h^0 \|_{\CC^{\init}}  \vee 1\;.
	\end{equ}
	Setting $\SF \eqdef \bigl(F_2, F_3, G, \bar G, (\tilde F_i, \bar C_i)_{i=1}^3\bigr)$, the solution map $\CS \colon \bigl(\eps, h^0, \xi_{3}, Z, \SF \bigr) \mapsto (T, \CR H)$ is continuous in the topology defined by the various norms appearing in \eqref{e:normh} and the assumptions. 
Finally, $\bar \eps$ can be taken to be continuous decreasing in $M$, and $T$ can be taken continuous decreasing in $L$ and $\| h^0 \|_{\CC^{\init}}$.
\end{proposition}

\begin{remark}
In order to compare solutions with different existence times $T$, we extend them to the whole
time interval $[0,1]$ by setting them to be constant on $[T,1]$.
\end{remark}

\begin{proof}
	For $T > 0$, let us consider the system~\eqref{eq:main4} on the time interval $[0, T]$. Let then 
	$\CM_{k, \eps}(T, H)$ be the right-hand side of the $k^{\text{th}}$ equation of~\eqref{eq:main4}, and set $\CM_{\eps}(T, H) = (\CM_{k, \eps}(T, H))_{k = 0, \ldots, 3}$. 
	It is then sufficient to show that for some $T < 1$ and for all $\eps > 0$ small enough as in the statement of the proposition the map $H \mapsto \CM_{\eps}(T, H)$ is a contraction on a ball in $\SU^{\gamma, \eta}_{T}$, with suitable $\gamma$ and $\eta$. While this is a standard 
	procedure in the framework of regularity structures, which can be found in~\cite{Regularity}, our 
	system~\eqref{eq:main4} is non-standard, because the equation for $h_{3}$ should be solved in the 
	classical sense. Moreover, the $\eps$-dependence requires some care.
	
	Let's start with finding the image spaces of the nonlinearities~\eqref{eq:abstract-functions}. Let $\alpha_k = \deg \Xi_k = - \frac 32 + \frac{2k}{3} - \kappa_\star$, for $k = 0, 1,2$, be the minimal degree of the elements on the 
	right-hand side of the $k^{\text{th}}$ nonlinearity. Then Lemma~\ref{lem:abstract_integration} implies that 
	$H_k$, $\SD H_k$ and $\SD^2 H_k$ take values in sectors with minimal degrees $(\alpha_k + 2) \wedge 0$, 
	$(\alpha_k + 1) \wedge 0$ and $\alpha_k \wedge 0$ respectively. This implies that the term with the 
	smallest regularity in the last function in~\eqref{eq:abstract-functions} is $\partial_x h_1$.
	
	These minimal degrees of the modelled distributions suggest that it is sufficient to view the function $\tilde \BF_1$ in~\eqref{eq:abstract-functions} as its Taylor expansion of order $6$.
	This is because we would like to guarantee that $\tilde \BF_1(\SD H_1) \in \CD^\gamma$ for some
	$\gamma$ such that $\gamma + \alpha_1 > 0$ in order for the first term in $\BF_{\ft_0}$ to
	have a well-defined reconstruction.
Analogously, the functions $\tilde \BF_2$, $\tilde \BF_3$ and $\bar \BG$ can be 
expanded to orders $4$, $2$ and $4$ respectively.
	
	If we take $H \in \SU^{\gamma, \eta}_{T}$ for all $\gamma_i > 0$, then $\SD H_1 \in \CD^{\gamma_1 - 1, \eta_1 - 1}_{T}$. Moreover, the restriction $\kappa_\star \leq \frac{1}{42}$ in~\eqref{eq:deg_main} yields 
	$\alpha_1 + 1 \geq 0$, so that $\SD H_1$ takes values in a function-like sector. By~\cite[Prop.~6.13]{Regularity},
	we conclude that $\tilde \BF_{i}( \SD H_{1}), \bar \BG(\SD H_{1}) \in \CD^{\gamma_1 - 1, \eta_1 - 1}_{T}$ and these modelled distributions are function-like. 
	
	We also need to bound the elements $\hat\Xi_{\CB, \fl_k}(H)$. The definition of the rule in Section~\ref{sec:first} yields $\Eps_+(\fl_k) = \{\ft_0\}$, for $k = 0, 1, 2$. 
	Then, using the notation of Lemma~\ref{lem:Xi}, we have $\hat \gamma_{\fl_k} = \gamma_{0}$ and $\hat \eta_{\fl_k} = \eta_{0}$, which are both positive. Then for $\eta_0 \leq \gamma_0$ Lemma~\ref{lem:Xi} yields $\Xi_{\CB, \fl_k}(H) \in \CD_T^{\gamma_0 + \alpha_k, \eta_0 + \alpha_k}$ with the minimal degree $\alpha_k$.
	
	Now, we make precise values of $\gamma$ and $\eta$. The values of $\eta$ are dictated by the regularities of the initial states in $\CC^{\init}$ (see \cite[Lem.~7.5]{Regularity}) and are taken to be $\eta_0 = \frac{1}{3} +\kappa$ and $\eta_k = \eta_0 + \frac{2k}{3}$, for $k = 1,2$, where $\kappa > 0$ is the same as in the definition of $\CC^{\init}$. The values of $\gamma$ should be such that the right-hand sides in \eqref{eq:abstract-functions} are modelled distributions of strictly positive regularities. Applying \cite[Prop.~6.12]{Regularity} we can find out which spaces
 the terms in \eqref{eq:abstract-functions} belong to. We provide these spaces in Table~\ref{tab:spaces}, where `Regularity' and `Order of singularity' denote the exponents 
 $\tilde\gamma$ and $\tilde\eta$ respectively such that the term
 belongs to $\CD_T^{\tilde\gamma,\tilde\eta}$ whenever $H \in \SU^{\gamma, \eta}_{T}$.
	\begin{table}[ht]
	\centering
	\begin{tabular}{ccc} \toprule
		Term & Regularity & Order of singularity \\
		\midrule
		$\tilde \BF_{1}( \SD H_{1})\, \SD^2 H_{1}$ & $\gamma_1 - 2$ & $\eta_1 - 2$ \\[0.3em]
		$(\SD H_0)^2$ & $\gamma_0 + \alpha_0$ & $(\eta_0 + \alpha_0) \wedge (2 \eta_0 - 2)$ \\[0.3em]
		$\SD H_{0} \tilde \BF_{2}(\SD H_{1})$ & $(\gamma_0 - 1) \wedge (\gamma_1 + \alpha_0)$ & $\eta_0 - 1$ \\[0.3em]
		$\SD H_{0} \bar \BG(\SD H_{1})$ & $(\gamma_0 - 1) \wedge (\gamma_1 + \alpha_0)$ & $\eta_0 - 1$ \\[0.3em]
		$\tilde \BF_{1}( \SD H_{1})\, \SD^2 H_{2}$ & $(\gamma_1 - 1 + \alpha_2) \wedge (\gamma_2 - 2)$ & $\eta_2 - 2$ \\[0.3em]
		$\SD H_0 \SD H_1$ & $(\gamma_0 - 1) \wedge (\gamma_1 + \alpha_0)$ & $\eta_0 - 1$ \\[0.3em]
		$(\SD H_1)^2 \hat\Xi_{\CB, \fl_1}(H)$ & $(\gamma_1 - 1 + \alpha_1) \wedge (\gamma_0 + \alpha_1)$ & $\eta_1 - 1 + \alpha_1$ \\[0.3em]
		$\tilde \BF_{3}(\SD H_{1}) \hat\Xi_{\CB, \fl_2}(H)$ & $(\gamma_1 - 1 + \alpha_2) \wedge (\gamma_0 + \alpha_2)$ & $\eta_1 - 1 + \alpha_2$ \\[0.3em]
		$(\SD H_1)^2 \hat\Xi_{\CB, \fl_2}(H)$ & $(\gamma_1 - 1 + \alpha_2) \wedge (\gamma_0 + \alpha_2)$ & $\eta_1 - 1 + \alpha_2$ \\
		\bottomrule
	\end{tabular}
	\caption{For each term $f$ appearing in \eqref{eq:abstract-functions}, we 
have $f \in \CD^{\tilde \gamma, \tilde \eta}_{T}$ with ``Regularity'' $\tilde \gamma$ and 
``Order of singularity'' $\tilde \eta$.}
	\label{tab:spaces}
\end{table} 
	
	To have all the regularities in Table~\ref{tab:spaces} strictly positive, we need to take $\gamma_0 > - \alpha_0$ and $\gamma_1 > 2$. To be precise, we take $\gamma_0 = - \alpha_0 + \bar \kappa$ and $\gamma_1 = 2 + \bar \kappa$, for $\bar \kappa > 0$ whose value is given below. Moreover, all the singularities in Table~\ref{tab:spaces} are of orders strictly greater than $-2$, which allows to apply Lemma~\ref{lem:abstract_integration}.
	
	The values of $\gamma_2$ and $\gamma_3$ should be such that $\d_x^2 h_2$ and $\d_x^2 h_{3}$ are well-defined continuous functions, i.e.\ $\gamma_2, \gamma_3 > 2$. Again we can take the precise values $\gamma_2 = 2 + 2 \bar  \kappa$ and $\gamma_3 = 2 + 4 \bar  \kappa$ (as we show below, we need to have $\gamma_1 < \gamma_2 < \gamma_3$ to be able to close our Picard iteration). 
	Moreover, we fix $\eta_3 = \eta_1 + 1 - \bar \kappa$.  This value of $\eta_3$ will be convenient in the bounds \eqref{eq:F3-bounds} below. 
	
	Let us fix $L, M > 0$ and let the assumptions of this proposition be satisfied for these values. In what follows we consider modelled distributions $H, \bar H \in \SU^{\gamma, \eta}_{T}$ satisfying $\$H\$_{\gamma, \eta} < \bar L$ and $\$\bar H\$_{\gamma, \eta} < \bar L$, for $\bar L = 2 \| h^0 \|_{\CC^{\init}} \vee 1$.
	
	As was stated above, Lemma~\ref{lem:Xi} can be applied and we get
	\begin{equs}[eq:Xi_bounds-fixedpoint]
		\$ \hat\Xi_{\CB, \fl_k}(H)\$_{\gamma_0 + \alpha_k, \eta_0 + \alpha_k} &\leq C \bar L\;,\\
		\$ \hat\Xi_{\CB, \fl_k}(H) - \hat\Xi_{\CB, \fl_k} (\bar H)\$_{\gamma_0 + \alpha_k, \eta_0 + \alpha_k} &\leq C \$H - \bar H\$_{\gamma, \eta}\;,
	\end{equs}
	for a constant $C \geq 0$. We note that our assumptions imply $\eta_0 + \alpha_k > - 2$, which allows to apply the 
	integration map to $\hat\Xi_{\CB, \fl_k}(H)$ and to use Lemma~\ref{lem:abstract_integration} to bound it. 
	
	We define $\bxi_3(H_0, t,x)$ to be the composition of $H_0$ with the $\CC^2$ function $h \mapsto \xi^{h}_{3}(t,x)$, defined as in \cite[Sec.~4.2]{Regularity}. This gives a modelled distribution $(t,x) \mapsto \bxi_3(H_0(t,x), t,x)$, for which the bounds \eqref{eq:Xi_bounds-fixedpoint} hold with $\alpha_3 = 0$ (see \cite[Prop.~6.13]{Regularity}). We will simply write $\bxi_3(H_0)$ for this modelled distribution. 
	
 Observe that $\gamma_0 + \alpha_k + 2 > \gamma_k$, for $k = 0,1,2$, which will allow us to show that the solution map leaves the space $\SU^{\gamma, \eta}_{T}$ invariant.
	
	Now, we will bound the first three functions in \eqref{eq:abstract-functions}. For this, we write $\BF_{\ft_i} (H, t,x) = \BF_{\ft_i}^1 (H) + \BF_{\ft_i}^2 (H, t, x)$ for $i = 0,1,2$ as in Remark~\ref{rem:local_functions}, where $\BF_{\ft_i}^2 (H, t, x)$ is a multiple of $\1$. Using Table~\ref{tab:spaces}, the bounds~\eqref{eq:Xi_bounds-fixedpoint} and our choice of $\gamma$ and $\eta$, we get 
	\begin{equs}[eq:F_bounds]
		\$ \BF_{\ft_i}^1 (H)\$_{\bar{\gamma}_i, \bar \eta_i} &\leq C \bigl(L + \bar L + e^{m \bar L} + e^{\bar m \bar L}\bigr)^7\;,\\
		\$ \BF^1_{\ft_i} (H) - \BF^{\eps, 1}_{\ft_i} (\bar H)\$_{\bar{\gamma}_i, \bar \eta_i} &\leq C \bigl(L + \bar L + e^{m \bar L} + e^{\bar m \bar L}\bigr)^6 \$H - \bar H\$_{\gamma, \eta}\;,
	\end{equs}
	where $\bar \eta_0 = (\eta_1 - 2) \wedge (\eta_0 + \alpha_0) \wedge (2 \eta_0 - 2)$, $\bar \eta_1 = \eta_0 - 1$, $\bar \eta_2 = \eta_1 - 1 - 2 \bar \kappa$ and $\bar{\gamma}_i = (i+1) \bar \kappa$ for $i = 0,1,2$. We note that according to Table~\ref{tab:spaces}, we could take $\bar \gamma_2 = \gamma_0 + \alpha_2$. However, we prefer to take a smaller value $\bar{\gamma}_2 = 3 \bar \kappa$ (for this we need to assume $\bar \kappa < \frac{4}{9}$) which does not play any role for the Picard iteration but will be convenient when analysing the functions \eqref{eq:abstract-functions-extra} below. We similarly could have taken $\bar \eta_2 = \eta_0 + \alpha_2$, but we prefer to take a smaller value $\bar \eta_2 = \eta_1 - 1 - 2 \bar \kappa$. The right-hand sides of \eqref{eq:F_bounds} can be made smaller by using sharper bounds on each term in \eqref{eq:abstract-functions}, but we prefer to write them in this form to have shorter formulas. In order to perform a Picard iteration by applying Lemma~\ref{lem:abstract_integration} to these functions, we need to have $\bar{\gamma}_i + 2 > \gamma_i$ and $\bar{\eta}_i + 2 > \eta_i$ for $i = 0, 1, 2$. These inequalities hold for our choice of $\gamma$ and $\eta$ if we take $\bar \kappa < \frac{1}{6}$.
	
	To bound the functions $\BF_{\ft_i}^2$, we note that, provided that 
	the constant $\bar \kappa > 0$ appearing in the definition of $\gamma$ is small enough, 
	they can be written as
	\begin{equs}[eq:abstract-functions-extra]
		\BF^2_{\ft_0}(H) &= \CQ_{< \bar \gamma_0} \BG (\SD H_1)\;, \\[0.5em]
		\BF^2_{\ft_1}(H) &= \CQ_{< \bar \gamma_1} \Bigl(\SD H_1 \tilde \BF_{2}(\SD H_1) - \bar C_{1} + \tilde \BF_{3}(\SD H_1)\, \bxi_3(H_0) \Bigr)\;, \\[0.5em]
		\BF^2_{\ft_2} (H) &= \CQ_{< \bar \gamma_2} \Bigl( \tilde \BF_{1}(\SD H_1)\, \SD^2 H_3 + \lambda\, (\SD H_1)^2 \\
		&\qquad + \tilde \BF_{2}(\SD H_1)\, \SD H_2 - \bar C_{2} + \bigl( \sigma_1 (\SD H_1)^2 + \eps^{\frac 23} \tilde \BF_{3}(\SD H_1)\bigr)\, \bxi_3(H_0) \Bigr)\;.
	\end{equs}
This is because, as $\bar \kappa$ decreases to $0$, the same happens to the values $\bar \gamma_i$, so that for $\bar \kappa$ small enough the projections on the right simply yield the correct multiples of $\1$. As we did in Table~\ref{tab:spaces}, we can use \cite[Props.~6.12, 6.13]{Regularity} and compute the spaces to which the terms in \eqref{eq:abstract-functions-extra} belong. All of the terms except $\tilde \BF_{1}(\SD H_1)\, \SD^2 H_3$ belong to $\CD^{\gamma_1 - 1, \eta_1 - 1}_T$, while the latter belongs to $\CD^{\gamma_3 - 2, \eta_3 - 2}_T$. Because of our choice of the values $\gamma$ we see that the regularities of the functions in \eqref{eq:abstract-functions-extra} are strictly higher than the applied projections. Moreover, $\bar \eta_i < \eta_1 - 1$ for $i = 0,1$, and $\bar \eta_2 < \eta_3 - 2$. Then the bounds \eqref{eq:F_bounds} on the functions \eqref{eq:abstract-functions-extra} readily follow. 
	
	The operator $\CP_{\ft_k}$ ``improves'' regularity by $2$ (see~\cite[Thm.~7.1, Lem.~7.3]{Regularity}), and 
	a bound on the term $G h^0_{k,\eps}$ is provided in~\cite[Lem.~7.5]{Regularity}. Since we have $\alpha_0 > -2$, then $\bar \gamma_k + 2 > \gamma_k$ and $\bar \eta_k + 2 > \eta_k$. Combining the just cited results 
	with~\eqref{eq:F_bounds},we get for $k = 0,1,2$
	\begin{equs}[eq:M_bounds]
		\$ \CM_{k, \eps}(T, H)\$_{\gamma_k, \eta_k} \leq \| h^0 \|_{\CC^{\init}} &+ \bar C T^\delta \bigl(L + \bar L + e^{m \bar L} + e^{\bar m \bar L}\bigr)^7\;,\\
		\$ \CM_{k, \eps} (T, H) - \CM_{k, \eps} (T, \bar H)\$_{\gamma_k, \eta_k} &\leq \bar C T^\delta \bigl(L + \bar L + e^{m \bar L} + e^{\bar m \bar L}\bigr)^6 \\
		&\hspace{4cm}\times \$H - \bar H\$_{\gamma, \eta}\;,
	\end{equs}
	for some $\delta > 0$ and for a new constant $\bar C \geq 0$.

In order to bound the right-hand side of the last equation 
in~\eqref{eq:main4}, we proceed similarly to above. More precisely, we 
set $\BF_{\ft_3}(H) = \BF^{(1)}_{\ft_3}(H) + \BF^{(2)}_{\ft_3}(H)$, where
	\begin{equs}
		\BF^{(1)}_{\ft_3}(H) &= \eps^{\frac{2}{3}} \tilde \BF_{1}(\SD H_1)\, \SD^2 H_{3}\;,\\
		\BF^{(2)}_{\ft_3}(H) &= \BF_{2}\bigl( \eps^{\frac{1}{3}} \SD H_1 \bigr) \SD H_1 \SD H_2 - \bar C_{3} + \BF_{3}\bigl(\eps^{\frac{1}{3}}\SD H_1\bigr) \bxi_3(H_0)\,,
	\end{equs}
	whose range lies again in a function-like sector. As before, we then have
	\begin{equ}
		\CM_{3, \eps}(T, H) = \CR \Bigl( \CP_{\ft_3} \bigl[\one_+ \CQ_{< \bar \gamma_{3}} \BF_{\ft_3}(H)\bigr] + G h^0_{3} \Bigr)\;.
	\end{equ}
	Using the assumptions on boundedness of second and third derivative of $\tilde F_{1}$ and the fact that $F_2,F_3 \in \CC^\func_{m,7}$ we obtain $\$ \BF^{(1)}_{\ft_3} (H)\$_{\bar{\gamma}_3, \bar \eta_3} \leq \eps^{\frac{2}{3}} C M \bar L^3$, $\$ \BF^{(2)}_{\ft_3} (H)\$_{\tilde{\gamma}_3, \tilde \eta_3} \leq C \bigl(L + \bar L + e^{m \bar L}\bigr)^7$, and
	\begin{equs}[eq:F3-bounds]
		\$ \BF^{(1)}_{\ft_3} (H) - \BF^{(1)}_{\ft_3} (\bar H)\$_{\bar{\gamma}_3, \bar \eta_3} &\leq \eps^{\frac{2}{3}} C M \bar L^2 \$H - \bar H\$_{\gamma, \eta}\;,\\
		\$ \BF^{(2)}_{\ft_3} (H) - \BF^{(2)}_{\ft_3} (\bar H)\$_{\tilde{\gamma}_3, \tilde \eta_3} &\leq C \bigl(L + \bar L + e^{m \bar L}\bigr)^6 \$H - \bar H\$_{\gamma, \eta}\;,
	\end{equs}
	for $\bar{\gamma}_3 = \gamma_3 - 2$, $\bar \eta_3 = \eta_3 - 2$, $\tilde \gamma_3 = \gamma_1 - 1$ and 
	$\tilde \eta_3 = \eta_1 - 1$. Furthermore, we have $\tilde \gamma_3 + 2 > \gamma_3$ and $\tilde \eta_3 + 2 > \eta_3$. Similarly to~\eqref{eq:M_bounds}, we then use the properties of $\CP_{\ft_3}$ and the 
	reconstruction operator $\CR$ to get
	\begin{equs}[eq:M_bounds2]
		\| \CM_{3, \eps}(T, H)\|_{\gamma_3, \eta_3} &\leq \| h^0 \|_{\CC^{\init}} + \bar C \Bigl(\eps^{\frac{2}{3}} M \bar L^3 + T^\delta \bigl(L + \bar L + e^{m \bar L}\bigr)^7 \Bigr)\;,\\
		\| \CM_{3, \eps} (T, H) - \CM_{3, \eps} &(T, \bar H)\|_{\gamma_3, \eta_3} \\
		&\leq \bar C \Bigl(\eps^{\frac{2}{3}} M \bar L^2 + T^\delta \bigl(L + \bar L + e^{m \bar L}\bigr)^6 \Bigr) \$H - \bar H\$_{\gamma, \eta}\;,
	\end{equs}
	for some $\delta > 0$.
	
	Since $\bar L \leq 2 \eps^{-\frac{1}{4}} \vee 1 = 2 \eps^{-\frac{1}{4}}$, one has $\bar C \eps^{\frac 23} M \bar L^2 \leq 4 \bar C M \eps^{\frac 16}$ thus we can take $\bar\eps(M)$ small enough in order to make the term $\bar C \eps^{\frac 23} M \bar L^2$ sufficiently small for all $\eps \leq \bar \eps(M)$. Similarly, we can choose $T < 1$ and $\eps < 1$ sufficiently small and as in the statement of this proposition, such that the bounds~\eqref{eq:M_bounds} and~\eqref{eq:M_bounds2} imply that the map $\CM_{\eps}(T, H)$ is a contraction on the ball in $\SU^{\gamma, \eta}_{T}$, containing the modelled distributions satisfying $\$H\$_{\gamma, \eta} \leq \bar L$. Hence, its fixed point yields a solution map $\CS \colon \bigl(\eps, (h^0_{i})_{i = 0, \ldots, 3}, \xi_{3}, Z, \SF\bigr) \mapsto (T,H)$. The continuity of this solution map with respect to all data can be proved analogously (see \cite{Regularity}).
\end{proof}

\begin{remark}\label{rem:general_fixed}
	Note that in the above fixed point argument a requirement of a small enough $\bar \eps$ is needed to  avoid having to view the equation for $h_3$ as a quasilinear equation. The requirement $\|h^0\|_{\CC^\init} < \eps^{- \frac 14}$ and at most quadratic growth of $\tilde F_1$ ensures that this term is bounded by $M \eps^{\frac 23} (1 + \eps^{- \frac 12}) \lesssim M\bar\eps^{\frac 16}$. This allows in particular to take $\bar\eps$ independent of the initial condition and the model and thus deterministic even when applied to a random model.
\end{remark}
	
Now, we will construct a maximal (in time) solution. To consider functions, which can explode to infinity, we set $\widehat{\CC^{\init}} \eqdef \CC^{\init} \sqcup \{\infty\}$ and equip it with the topology generated by open balls in $\CC^{\init}$ and the sets $\{h \in \CC^{\init} : \| h \|_{\CC^{\init}} > N\} \sqcup \{\infty\}$ for any $N \geq 0$. We also set $\| \infty \|_{\CC^{\init}} = +\infty$ by convention.

We then recall the definition of $\CC^\sol$ given in \cite[Sec.~2.7.2]{BCCH} (the arXiv version). Given $h \in \CC(\R_+, \widehat{\CC^{\init}})$, we set
\begin{equ}
T^L[h] \eqdef \inf\{t \in \R_+ : \| h(t) \|_{\CC^{\init}} \geq L\}\;, \qquad T[h] \eqdef T^\infty[h]\;,
\end{equ}
and we define a space of solutions with potential blow-up by 
\begin{equ}
\CC^{\sol} \eqdef \Bigl\{h \in \CC(\R_+, \widehat{\CC^{\init}})\; :\; \begin{aligned} &h(t) = \infty ~~ \forall\, t > T[h] \\[-0.1cm] &h \restr_{[0, T]} \in \CC([0, T], \CC^{\init}) ~~ \forall\, T < T[h] \end{aligned} \Bigr\}\;.
\end{equ}
Unfortunately, we cannot simply stop the solution when it reaches certain level, because such stopped solution is not continuous with respect to the initial state. Instead we need to make a smoother stopping. We fix a smooth decreasing function $\chi : \R \to \R$ which is identical to $1$ on $(-\infty, 0]$ and identical to $0$ on $[1, \infty)$. For any $L \in \N$ we define the map $\Theta_L : \CC^{\sol} \to \CC(\R_+, \CC^{\init})$ by
\begin{equ}
\Theta_L(h)(t) \eqdef \chi \Bigl(\frac{t - T^{L/2}[h]}{T^{L}[h] - T^{L/2}[h]}\Bigr) h(t)\;.
\end{equ}
Then we equip $\CC^{\sol}$ with a metric $d (\bigcdot, \bigcdot) \eqdef \sum_{L = 1}^\infty 2^{-L} d_L (\bigcdot, \bigcdot)$, where for $h, \bar h \in \CC^{\sol}$
\begin{equ}
d_L (h, \bar h) \eqdef 1 \wedge \Bigl[ \sup_{t \in [0, L]} \| \Theta_L(h)(t) - \Theta_L(\bar h)(t) \|_{\CC^{\init}} \Bigr]\;.
\end{equ}
Finally, given $\bar\eps \in (0,1]$, we equip 
$[0,\bar\eps]\times \CC^\sol$ with the topology induced by the pseudo-metric
\begin{equ}[eq:pseudo_metric]
	\bar d \big( (\eps_1, h_1), (\eps_2, h_2) \big) = d \bigl(\Theta_{\eps_1^{-1 / 4}} (h_1), \Theta_{\eps_2^{-1 / 4}} (h_2)\bigr) + |\eps_1 - \eps_2|\;,
\end{equ}
where $\Theta_\infty$ is interpreted as being the identity.
(In particular, we identify elements $(\eps,h)$ and $(\eps,\bar h)$ if $h(s) = \bar h(s)$ for all
$s \le T_\star = \inf\{t>0\,:\, \|h(t)\| \le \eps^{-1/4}\}$.)  

\begin{remark}
For a sequence $(\eps_n, h_n)$ with $\eps_n \to 0$, convergence to $(0,h)$ with 
respect to $\bar d$ is equivalent
to simply having convergence $h_n \to h$ in $\CC^\sol$. This is because
$\Theta_L \circ \Theta_M = \Theta_L$ for $L \le M/2$, so that 
$|\bar d((\eps,\bar h), (0,h)) - d(\bar h,h)| \lesssim \eps + 2^{-\eps^{-1/4}/2}$.
The reason for introducing the pseudo-metric $\bar d$ is that our local solution
theory breaks down when solutions get too large. Simply stopping solutions when they
exit some large ball is unfortunately not a continuous operation, which is why we use
$\Theta_{\eps^{-1 / 4}}$ instead.
\end{remark}

Using these definitions, we can construct the maximal solution in the space $\CC^{\sol}$. 

\begin{proposition}\label{prop:maximal_sol}
Let $Z \in \M_0(\CT_{\CB, \max\{\gamma_i\}})$ be admissible model and $M, L > 0$ and $\gamma, \eta, \bar\eps$ be as in Proposition~\ref{prop:system_soln}. Let Assumption~\ref{assum:all-functions} be satisfied, and let $F_i, \tilde F_{j}, G, \bar G$, satisfy bounds in $M$ and $L$ as in Proposition~\ref{prop:system_soln}. Let $\eps \in [0,\bar\eps]$ and $h^0 \in \CC^\init$, then there exists a unique $T_\star \geq 0$ such that following holds
\begin{enumerate}
	\item If $\|h^0\|_{\CC^\init} \geq \eps^{-\frac 14}$ then $T_\star = 0$.
	\item If $\|h^0\|_{\CC^\init} < \eps^{-\frac 14}$ then $T_\star > 0$ and there exists a unique $H : [0, T_\star) \times \T \to \CT_{\CB, \max \gamma_i}$ such that $H\restr_{[0, T]}  \in \SU^{\gamma, \eta}_{T}$ and solves~\eqref{eq:main4} on $[0,T]$ for all $T < T_\star$. Furthermore, either $T_\star = \infty$ or $\lim_{t \to T_\star} \|h_t\|_{\CC^\init} = \eps^{-\frac 14}$ for $h = \CR H \in \CC^\sol$. 
\end{enumerate}
Moreover, $(\eps, h) \in [0, \bar\eps] \times \CC^\sol$ depends continuously on $h^0, \eps, Z, \xi_3, F_i, \tilde F_i, G, \bar G$ and $\bar C_i$ with respect to the topology induced by~\eqref{eq:pseudo_metric}.
\end{proposition}

\begin{proof}
	Let $\eps \in [0,\bar\eps]$ and assume that $\|h^0\|_{\CC^\init} < \eps^{- \frac 14}$. Patching together the local solutions, constructed in Proposition~\ref{prop:system_soln}, yields a sequence of times $0 < T_1 < T_2 < \cdots$ where solution $H$ of~\eqref{eq:main4} exists. Each of these times $T_i$ is lower-semicontinuous with respect to $\|h^0\|_{\CC^\init} $ and corresponding norms on $\xi_3$ and $Z$ restricted on $[0,T_{i-1}+1]$. Now let $h = \CR H$. By continuity of the reconstruction map $\CR$ we have that either $\|h_t\|_{\CC^\init} < \eps^{-\frac 14}$ for all $t \geq 0$ and thus we take $T_\star = \infty$ or $T_\star$ is defined to be the first time such that $\|h_{T_\star}\|_{\CC^\init} = \eps^{-\frac 14}$. The fact that we can restart the solution from $h_{T_i}$ and obtain continuity of the resulting $h$ follows from the continuity of $\Theta_{\eps^{-1/4}}$ with respect to $\eps$, continuity of the solution map from Proposition~\ref{prop:system_soln} as well as~\cite[Prop.~7.11, Cor.~7.12]{Regularity}.
\end{proof}

\section{Application to the mean curvature flow}
\label{sec:application}

In this section we apply the algebraic\slash analytic  framework developed above to the system of SPDEs~\eqref{eq:main2}, 
which we remember is an equivalent way to write~\eqref{eq:main} (provided of course that initial conditions are 
chosen in a consistent way). The rule and nonlinearity for the 
system~\eqref{eq:main2} were described in Section~\ref{sec:first} and Example~\ref{ex:Examplenonlin} with the 
scaling $\s = (2,1)$. In particular, this system falls into the framework developed in the previous sections
to solve general systems of SPDEs of the form~\eqref{eq:system}.

\subsection{Convergence of the noises}
\label{sec:noises}

Recall that in our directed mean curvature equation~\eqref{eq:main_before_scaling} the noisy environment was given by $\eta_\eps = \rho * \zeta_\eps$ and $\zeta_\eps$ is a space-time white noise on $\R\times (\eps^{-1} \T)$ (so it only depends on 
$\eps$ through the size of its domain).  If the noise $\xi_\eps$ is given by~\eqref{eq:noise}, then thanks to the scaling properties of the 
white noise one actually has
\begin{equ}[eq:noise2]
	\xi_\eps(t,x) = (\rho_\eps*\xi)(t, x)\;,
\end{equ}
where $\xi$ is a space-time white noise on $\R \times \T$. Therefore, we shall see that in order to be able to 
consider equation~\eqref{eq:main2}, we need to define the BPHZ model whose approximation is constructed 
via the canonical lift of noises $\xi_{i,\eps}$, translated by the $2 \times 2$ diagonal matrix $c = \text{diag} (\eps^2 C_\eps , 0)$ as in~\eqref{eq:canonical}. At a first glance such definition seems to be circular since renormalisation constant $C_\eps$ needs to be computed using the noises $\xi_{i,\eps}$, which themselves 
are translated by $\eps^2C_\eps$ in the time variable. To resolve this we shall be more general and consider a 
shift $c = \text{diag} (c_\eps , 0)$ with arbitrary $c_\eps$ such that $\lim_{\eps\to0}c_\eps = 0$. This leads to a translated noise
\begin{equ}[eq:noise_translated]
	\xi_\eps^{c_\eps}(t,x) = \xi_\eps(t + c_\eps t, x)\;.
\end{equ}
Obviously, such translated noise converges to a space-time white noise $\xi$. We will also show 
 some other properties of these translated noises that will be needed for the construction of the BPHZ lift.

Moreover, we will see in the last part of the proof of Proposition~\ref{prop:abstract_system} that it will  
always be possible to find for every small enough $\eps$ such $c_\eps  = \CO(\eps)$ that the renormalisation 
constant $C_\eps$, computed using~\eqref{eq:ren_const_def} with this choice of the translation $c_\eps$, satisfies $\eps^2C_\eps = c_\eps$.

We start by studying limits of the random field~\eqref{eq:noise_translated}. Since our noise is actually smooth, multiplication of $\xi^{c_\eps}_\eps$ by 
$\eps^{\frac{2 k}{3}}$ ``increases its regularity'' by 
$\frac{2 k}{3}$ in the sense that while $\xi^{c_\eps}_\eps$ is uniformly (as $\eps \to 0$) bounded in $\CC^\alpha_\s$
only if $\alpha < - \frac{3}{2}$, $\eps^\beta \xi^{c_\eps}_\eps$ is uniformly bounded in $\CC^\alpha_\s$ as soon as $\alpha < \beta - \frac{3}{2}$ for $\beta \ge 0$.

\begin{lemma}\label{lem:noise_convergence}
Let the random field $\xi_\eps^{c_\eps}$ be given by~\eqref{eq:noise_translated} for some constants 
$c_\eps$ with $\lim_{\eps\to0}c_\eps = 0$, and let $\xi^{c_\eps}_{k, \eps} = \eps^{\frac{2 k}{3}} \xi^{c_\eps}_\eps$. Then for any $\alpha < -\frac{3}{2}$ the following results hold:
\begin{enumerate}
	\item\label{it:first_limit} one has the limit $\xi^{c_\eps}_{0, \eps} \to \xi$ in probability in $\CC^\alpha_{\s}$,
	\item\label{it:second_limit} for any $k, m \in \N$ such that $k + m \geq 1$ and $\alpha + \frac{2k}{3} \notin \N$ one has $\eps^{2 m} \partial_t^{m} \xi^{c_\eps}_{k, \eps} \to 0$ in probability in $\CC^{\alpha + 2k / 3}_\s$.
\end{enumerate}
\end{lemma}

\begin{proof}
Since $\xi$ is a space-time white noise, from~\eqref{eq:noise2} and~\eqref{eq:noise_translated} we get 
\begin{equ}[eq:xi_L2]
	\E | \xi^{c_\eps}_{0, \eps}(\phi_z^\lambda)|^2 = \int_{\R^2} \biggl( \int_{\R^2} \phi_z^\lambda(z_1) \rho_\eps(z^{c_\eps}_1 - z_2) d z_1 \biggr)^2 d z_2\;,
\end{equ}
where $(t,x)^{c_\eps} = (t + c_\eps t, x)$. Here, the function $\phi$ is smooth and supported in the ball $B(0, 1)$ with respect to the parabolic scaling $\s$. We consider only $z \in \fK$ for a fixed compact subset $\fK$ of $\R^2$. The function inside the brackets in~\eqref{eq:xi_L2} is supported in $\| z_2 - z^{c_\eps}\|_{\s} \lesssim \lambda + \eps$ and is bounded by a constant times $(\lambda + \eps)^{-3}$, which yields the bound
\begin{equ}[eq:xi_simple1]
	\sup_{z \in \fK} \E | \xi^{c_\eps}_{0, \eps}(\phi_z^\lambda)|^2 \lesssim (\lambda + \eps)^{- 3}\;,
\end{equ}
where the proportionality constant depends linearly on the diameter of $\fK$. 

Similarly, the definitions~\eqref{eq:noise2} and~\eqref{eq:noise_translated} yield
\begin{equ}
	\E | (\xi^{c_\eps}_{0, \eps} - \xi)(\phi_z^\lambda)|^2 = \int_{\R^2} \biggl( \int_{\R^2} \Bigl( \phi_z^\lambda(z_1) - \frac{1}{1 + c_\eps} \phi_z^\lambda(z_2) \Bigr) \rho_\eps(z^{c_\eps}_1 - z_2) d z_1 \biggr)^2 d z_2\;,
\end{equ}
where we used the property that $\rho$ integrates to $1$. As before, we consider only $z \in \fK$. Since $c_\eps = \CO(\eps)$, we can bound 
\begin{equs}
	 \Bigl| \phi_z^\lambda(z_1) &- \frac{1}{1 + c_\eps} \phi_z^\lambda(z_2) \Bigr| \leq \bigl| \phi_z^\lambda(z_1) - \phi_z^\lambda(z_2) \bigr| + \Bigl| \frac{c_\eps}{1 + c_\eps} \phi_z^\lambda(z_2) \Bigr|\\
	 &\quad \lesssim \lambda^{-3 - \delta} \|z_2 - z_1\|^\delta_\s \1_{\|z_1 - z\|_\s \wedge \|z_2 - z\|_\s \leq \lambda} + \eps \lambda^{-3} \1_{\|z_2 - z\|_\s \leq \lambda}\;,
\end{equs}
for any $\delta \in (0,1]$. From this, similarly to~\eqref{eq:xi_simple1}, we readily obtain the bound
\begin{equ}[eq:xi_simple2]
	\sup_{z \in \fK} \E | (\xi^{c_\eps}_{0, \eps} - \xi)(\phi_z^\lambda)|^2 \lesssim \eps^{2\delta} (\lambda + \eps)^{- 3 - 2 \delta} \;, \qquad 0 < \delta \leq 1\;.
\end{equ}
From the bounds \eqref{eq:xi_simple1} and \eqref{eq:xi_simple2} we get the convergence of $\xi^{c_\eps}_{0, \eps}$, described in the statement of this lemma, as in \cite[Prop.~9.5]{Regularity}. Similarly we get $\partial_t^{m} \xi^{c_\eps}_{0, \eps} \in \CC^{\alpha - 2 m}_\s$. Then from \cite[Lem.~A.5]{PAMPreprint} we conclude that $\eps^{2 m - \kappa} \partial_t^{m} \xi^{c_\eps}_{k, \eps} \in \CC^{\alpha + \frac{2k}{3}}_\s$ for all $\kappa \geq 0$ small enough. From this the statement \ref{it:second_limit} of this lemma follows.
\end{proof}

\subsection{Convergence of smooth models}
\label{sec:smooth_models}

In this section we construct the BPHZ model for the system of locally subcritical 
equations~\eqref{eq:main2}. The relevant sets and the rule for the system~\eqref{eq:main2} were defined in 
Section~\ref{sec:first}, which we are going to use throughout the rest of this section. For simplicity, we 
prefer to use the shorthand $\Xi_k$ in place of $\J_{\fl_k}[\one]$. Recall from Section~\ref{sec:first} that 
$\deg \fl_k = - \frac 32 + \frac{2k}{3} - \kappa_\star$ and $\deg \ft_k = 2$ for $k = 0, 1, 2$ and for a fixed 
value $0  < \kappa_\star < \frac{1}{42}$. Our choice of $\kappa_\star$ is governed by an intent to minimise the number 
of trees of negative degrees. In particular, we want the trees  
$\J_{\fl_0}[\J_{\ft_0}[\Xi_0]^4], \J_{(\ft_0,1)}[\Xi_0]\J_{(\ft_1,1)}[\Xi_1]^4 ,\J_{(\ft_1,2)}[\Xi_1]\J_{(\ft_1,1)}[\Xi_1]^6$, with respective degrees $\frac 12 - 5\kappa_\star, \frac 16 - 5\kappa_\star, \frac 16 - 7\kappa_\star$, to be of positive 
degree. Let $\gamma > 0$ to be determined later and let $\CT_{\CB, \gamma}$ be the corresponding 
truncated vector-valued regularity structure introduced in Sections~\ref{sec:VectorRS}, \ref{sec:spaces} with a choice of 
$k_\star$ for the spaces $\CB_\fl$ to be given by Lemma~\ref{lem:k_star}. 

The aim of this section is to construct smooth admissible models $Z^{\eps, \BPHZ}$ for $\CT_{\CB, \gamma}$, 
which are lifts of the noises $\xi^{c_\eps}_{k,\eps}$ with $c_\eps = \CO(\eps)$\footnote{At the 
moment we will show convergence of the models for any translation $c_\eps = \CO(\eps)$ and make a 
specific choice of $c_\eps$ later in the proof of Proposition~\ref{prop:abstract_system}.} and such that the $Z^{\eps,\BPHZ}$ converge to an admissible model as $\eps \to 0$. Unfortunately, we cannot use directly the lift 
defined in~\cite{HC,Rhys} because the spaces of distributions $\CB$ are infinite-dimensional. We will first define the 
lifts on certain finite-dimensional subspaces of $\CB^\fl$, and then construct their extension to the whole 
space. 

Another problem when trying to apply the results from \cite{HC,Rhys} is that
our rule allows for trees with noise edges 
that aren't ``terminal'', i.e.\ trees that contain subtrees of the form $\J_{(\fl,p)}[\tau]$ for $\tau \neq \one$ 
and $\fl \in \fL_-$.
Nevertheless, one can use the results from~\cite{HC} thanks to the
multiplicative structure of the underlying admissible maps $\PPi$. 
To see this, consider the rule $\Rule'$ given by setting $\Rule'(\fl) = \{()\}$ for every $\fl \in \fL_-$ and, for $\fl \in \fL_+$, we set
\begin{equ}[eq:down_rule]
\Rule'(\fl) = \bigcup_{\CN \in \Rule(\fl)} \downop\CN\;,
\end{equ}
where the collection of node types $\downop\CN$ is defined as follows. If there
exists $\fl_- \in \fL_-$ such that $\fl_- \in \CN$ (if such a type exists it is necessarily unique by the definition of the rule $\Rule$), then we set
\begin{equ}
\downop\CN = \{\CN \sqcup \bar \CN \,:\, \bar \CN \in \Rule(\fl_-)\}\;,
\end{equ}
otherwise we simply set $\downop \CN = \{\CN\}$. 
 It is not hard to see that if $\Rule$ is the rule defined in Section~\ref{sec:first}, 
then $\Rule'$ is normal, subcritical and complete and satisfies the assumptions of~\cite{HC}.

Given any vector space assignment $V$, we recursively define an operator $\downop \colon \CT_V(\Rule) \to \CT_V(\Rule')$ for $\sigma$ of the form~\eqref{eq:recursive_tree_V} by
\begin{equ}[eq:down]
	\downop[\sigma] = \X^m \prod_{\substack{1 \leq i \leq n \\ o_i \in \CO}} \J_{o_i} \bigl[v_i \otimes \downop[\sigma_i]\bigr] \prod_{\substack{1 \leq i \leq n \\ o_i \notin \CO}} \J_{o_i} [v_i \otimes \one] \downop[\sigma_i]\;.
\end{equ}
(Recall the definition of $\CO$ on p.~\pageref{p:defCO}.) In words, $\downop$ takes any 
tree located atop of a noise type edge and moves it down to the root of that edge.

Since the convergence results of~\cite{HC} can be applied to models of $\CT_V(\Rule')$,
we aim to show that these can be transferred to models on $\CT_V(\Rule)$ using $\downop$. 
This is the case provided that we can show that $\downop$
preserves both the renormalisation and recentering procedures.
Using the notations of \cite[Sec.~6]{BHZ}
, we have the following.

\begin{lemma}\label{lem:down}
	Let $\Rule$ be a rule satisfying Assumption~\ref{ass:rule} and assume that for every $\tau = T^\fm_\ff \in \fT(\Rule)$ one has 
	\begin{equ}[eq:sub_trees]
		\deg(\tau) - \max_{e \in E_T : \ff(e) \in \fL_{-}} \deg(\ff(e)) \geq 0\,.
	\end{equ}
	For a vector space assignment $V = (V_\ft)_{\ft \in \fL}$ we recall that $\tilde\CA^-_V, \tilde\CA^+_V, \Deltam_V, \Deltap_V$ denote the negative and positive twisted antipodes and coproducts on $\CT_V$ (see Section~\ref{sec:VectorRS}). Then one has
	\begin{equ}[e:commuteDown]
		\tilde\CA^-_V \downop \,=\, \downop \tilde\CA^-_V \qquad\text{and}\qquad (\downop \otimes \downop) \Deltam_V = \Deltam_V \downop\;,
	\end{equ}
	and the same holds true for $\tilde\CA^+_V$ and $\Deltap_V$.
\end{lemma}

\begin{proof}
It is immediate from the definition that $\downop$ is multiplicative. Furthermore, since
$\Deltap_V \J_{o} \bigl[v \otimes \sigma\bigr]
= (\J_{o}[v \otimes \bigcdot] \otimes \id) \Deltap_V \sigma$ for $o \not \in \CO$ (see Remark~\ref{rem:quotient}),
it follows from the multiplicativity of $\Deltap_V$ and $\downop$ that the required commutation
relation holds for $\Deltap_V$. This in turn implies that $\tilde\CA^+_V$ and $\downop$ commute, since 
$\downop$ commutes with all the operations appearing in the definition of $\tilde\CA^+_V$.

Regarding commutation with $\Deltam_V$, we note that there is a natural type-preserving
bijection $\iota_\sigma$ between the edges of any given tree $\sigma$ and those of $\downop\sigma$.
In particular, for every occurrence of a tree $\tau$ as
a subtree of $\sigma$, there is a corresponding occurrence of $\downop\tau$ (the image of $\tau$ under
$\iota_\sigma$) as a 
subtree of $\downop\sigma$. Since furthermore $\downop$ preserves degrees and $\Deltam_V$ acts
by contracting \slash extracting subtrees of negative degree, the only obstruction to 
\eqref{e:commuteDown} is the possibility of having a negative subtree $\hat\tau$ of $\downop\sigma$ 
which does not come from a corresponding subtree of $\sigma$, i.e.\ such that $\iota_\sigma^{-1}\hat\tau$
is disconnected. This in turn can only happen if there is a subtree $ \tau$ of $\sigma$
containing $\iota_\sigma^{-1}\hat\tau$ and such that $\tau \setminus \iota_\sigma^{-1}\hat\tau$ 
only consists of noise-type edges. 

Note now that \eqref{eq:sub_trees} states that, given any tree $\tau$ conforming to the rule
$\Rule$, it is not possible to delete any noise edges from $\tau$ in such a way that the
remainder is of negative degree.
However, the tree $\tau$ we have just constructed would precisely be of that kind
since $\deg(\hat \tau) < 0$, thus yielding the desired contradiction.
Again, this implies that $\tilde\CA^-_V$ and $\downop$ commute in the same way as above, thus
completing the proof.
\end{proof}

Fix now a finite-dimensional vector space assignment $V$ as well as some stationary 
random smooth noise assignment $\xi^\eps$. We then denote by $\PPi^\eps$ the canonical lift
to $\CT_V(\Rule)$ and by  $\PPi^{\eps\prime}$ the canonical lift
to $\CT_V(\Rule')$. It follows immediately from \eqref{eq:admissible2} that we have 
the identity $\PPi^\eps  = \PPi^{\eps\prime} \circ \downop$. As a consequence of Lemma~\ref{lem:down}
and the definition of BPHZ renormalisation in Section~\ref{sec:renorm-models}, we deduce the following.

\begin{corollary}\label{cor:idemMod}
Let $\PPi^{\eps,\BPHZ}$ and $\PPi^{\eps,\BPHZ\prime}$ denote the BPHZ renormalisations of the 
models described above. Then, one has the identity
$\Pi^{\eps,\BPHZ}_x = \Pi_x^{\eps,\BPHZ\prime} \circ \downop$.\qed
\end{corollary}

We rely on the following construction. Given $\ell \ge 1$, we consider the 
space assignment $\CD^{(\ell)}$ given by
\begin{equ}
\CD^{(\ell)}_{\ft} = 
\left\{\begin{array}{cl}
\big(\CD_{\ft}\big)^\ell	 & \text{if $\ft \in \fL_-$\;,} \\
	\R & \text{otherwise\;.}
\end{array}\right.
\end{equ}
Given $\bu \in \R^\ell$ and $\ft \in \fL_-$, 
and writing $\fd_i^{\alpha}$
for the basis vectors of the $i^{\text{th}}$ copy of $\CD_{\ft}$ in $\CD_{\ft}^{(\ell)}$, 
we then define $\Ev_\bu^\ell \colon \CD^{(\ell)}_{\ft} \to \CB_\ft$ by
\begin{equ}[eq:Evalell]
	\Ev_\bu^\ell[\fd_i^{\alpha}] = \delta^{(\alpha)}_{\bu_i} \;,
\end{equ}
similarly to \eqref{eq:Eval}. As before, we also set $\Ev_\bu^\ell$ to be the identity
on $\CD^{(\ell)}_{\ft}$ with $\ft \in \fL_+$, and we write again $\Ev_\bu^\ell$ for
the map $\CT_{\CD^{(\ell)}} \to \CT_{\CB}$ built as in \cite[Rem.~5.18]{CCHS}.

If $\PPi^\eps$ denotes the (untranslated) canonical lift as in Definition~\ref{def:canonical} 
of the noise $\xi^\eps$ given by
\begin{equ}
\xi^\eps_{\fl_k}(t,x) = \xi^{c_\eps}_{k, \eps}(t,x)
\end{equ}
(the translated noises $\xi^{c_\eps}_{k, \eps}$ are defined in Lemma~\ref{lem:noise_convergence}),
to an admissible model on $\CT_\CB$, 
we then obtain an admissible model $\PPi^{(\eps,\ell,\bu)}$ on $\CT_{\CD^{(\ell)}}$ by setting
$\PPi^{(\eps,\ell,\bu)} = \PPi^\eps \circ \Ev_\bu^\ell$. 
Also, setting $\bar \fL = \fL_+ \cup \{(\fl,i,\alpha)\,:\, \fl \in \fL_-, i \le \ell, |\alpha| \le k_\star\}$
we note that we have a canonical identification
\begin{equ}
\CD_\ft^{(\ell)} \simeq \R^{\bar \fL}\;,
\end{equ}
so that \cite[Sec.~5.6]{CCHS} allows to identify $\CT_{\CD^{(\ell)}}$ with the regularity structure $\CT^{(\ell)}$
generated by the rule $\Rule^{(\ell)}$ defined just like $\Rule$, with each of the 
$(\fl,i,\alpha)$ playing the same role in $\Rule^{(\ell)}$ as $\fl$ does in $\Rule$.

With this identification, it follows immediately from Definition~\ref{def:canonical} that
$\PPi^{(\eps,\ell,\bu)}$ is nothing but the canonical lift to $\CT_{\CD^{(\ell)}} \simeq \CT^{(\ell)}$
of the noise $\xi^{\eps,\bu}$ given by
\begin{equ}[e:shiftedNoise]
\xi^{\eps,\bu}_{(\fl_k,i,\alpha)}(t,x) = 
\bigl(\eps^{2 \alpha + \f{2k}3}\d_t^\alpha \xi_{\eps}\bigr)\bigl(t + c_\eps t + \eps^2 \bu_i, x\bigr)\;.
\end{equ}
(It follows from \eqref{eq:non-hom_noises} that the constants $\ba_{\fl, o}$ from Definition~\ref{def:canonical} are given by $\ba_{\fl, o} = \eps^2$ for $o = (\ft_0, 0)$ and $\ba_{\fl, o} = 0$ otherwise.) Similarly, we set 
$\hPPi^{(\eps,\ell,\bu)} = \PPi^{\eps,\BPHZ} \circ \Ev_\bu^\ell$ and recall that, given any 
admissible (smooth) stationary random model $\PPi$ in the sense of 
\cite[Def.~6.17]{BHZ}, the associated 
recentering and BPHZ renormalisation characters $g_x(\PPi)$ and $g_-(\PPi)$ are given by
\begin{equ}
g_x(\PPi)(\sigma_+) = (\PPi \sigma_+)(x)\;,\qquad 
g_-(\PPi)(\sigma_-) = \E (\PPi \sigma_-)(0)\;,
\end{equ}
where $\sigma_\pm \in \scal{\tau_\pm}$ for $\tau_\pm \in \fT_\pm(\Rule)$, see also \cite[Sec.~5.7.2]{CCHS}. These are then extended multiplicatively and we recall that
\begin{equ}
\hat \Pi_x = (g_-(\PPi^\ex)\tilde \CA_- \otimes \PPi^\ex \otimes g_x(\PPi^\ex)\tilde \CA_+)(\Deltam_\ex\otimes \id)\Deltap_\ex\;,
\end{equ} 
where $\PPi^\ex$ denotes the lift of $\PPi$ to the extended regularity structure as in
\cite[Thm.~6.33]{BHZ}, and where $\tilde \CA_-$ and $\tilde \CA_+$ are the twisted antipodes for the new vector space assignment (see \cite{BHZ} for the definitions).
Since
$\Ev_\bu^\ell$ is a natural transformation, we conclude that $g_x(\hPPi^{(\eps,\ell,\bu)}) = g_x(\PPi^{\eps,\BPHZ}) \circ \Ev_\bu^\ell$, and similarly for the renormalisation characters.
In particular, we conclude that $\hPPi^{(\eps,\ell,\bu)}$ coincides with the BPHZ lift of \eqref{e:shiftedNoise}.
We can therefore apply \cite[Thm.~2.15]{HC} to conclude that $\hPPi^{(\eps,\ell,\bu)}$ converges
in probability in the space of admissible models to a limit $\hPPi^{(\ell,\bu)}$,
which is the BPHZ lift of the limiting noise assignment $\bar \xi$ given by
\begin{equ}[e:defxibar]
\bar \xi_{\fl_k,i,\alpha} = 
\left\{\begin{array}{cl}
	\xi & \text{if $k=0$ and $\alpha = 0$\;,} \\
	0 & \text{otherwise\;.}
\end{array}\right.
\end{equ}
Actually, this is true for a much larger class of approximations to white noise by the continuity 
of the BPHZ lift under weak convergence of the driving noises, provided that one has uniform bounds.

We do however need the following more quantitative bound, which will be crucial for
showing that one has convergence of models on $\CT_\CB$.

\begin{proposition}\label{prop:boundModels}
For every $\gamma$ there exist $M >0$ and $\theta > 0$ such that, for every compact set $\fK$, the bounds
\begin{equ}[e:boundPPieps]
\E \$\hPPi^{(\eps,\ell,\bu)}; \hPPi^{(\ell,\bu)}\$_\fK^q \lesssim \eps^{q\theta} |\bu|^{qM}\;,\qquad
\E \$\hPPi^{(\ell,\bu)}\$_\fK^q \lesssim |\bu|^{qM}\;,
\end{equ}
hold for every $\ell$ and every $q$, uniformly over $\eps \in (0,1]$ and $\bu \in \R^\ell$.
Here, we restrict the relevant models to $\CT_{\CD^{(\ell)},\le \gamma}$.
\end{proposition}

\begin{proof}
We aim to apply \cite[Thm.~2.31]{HC} with noises given by \eqref{e:shiftedNoise}. 
These are far from being independent in our case, but have covariances given by
\begin{equ}
\E \bigl[\zeta^{u,\alpha}_{k,\eps}(0,0)\zeta^{v,\beta}_{\ell,\eps}(t,x)\bigr]
= \eps^{2(k+\ell) \over 3} \rho^{\alpha,\beta}_\eps(t+\eps^2(v-u),x)\eqdef C^{k+\ell,v-u,\alpha,\beta}_\eps(t,x)\;,
\end{equ} 
for some smooth compactly supported functions $\rho^{\alpha,\beta}$.
These functions satisfy
the bound
\begin{equ}
|C^{k,u,\alpha,\beta}_\eps(z)| \lesssim \eps^{{2k \over 3}-|\s|} \one\{|z|_\s \le \eps \sqrt{1+|u|}\}
\lesssim \eps^\theta \Big({1+\sqrt{|u|} \over |z|_\s}\Big)^{3+\theta-{2k \over 3}}\;,
\end{equ}
for every $\theta \ge {2k \over 3}-3$. Furthermore, one has
$\int C^{k,u,\alpha,\beta}_\eps(z)\,dz = 0$
whenever $\alpha + \beta \neq 0$.

A similar bound holds for 
$\E [\zeta^{u,\alpha}_{k,\eps}(0)\xi(z)]$
as well as for
\begin{equ}
\E \bigl[\zeta^{u,\alpha}_{k,\eps}(0)\big(\xi(z) - \zeta^{v,\beta}_{0,\eps}(z)\big)\bigr]
\eqdef \bar C^{k,u,\alpha,\beta}_\eps(z)\;,
\end{equ}
the difference being that the integral of this function vanishes even when
$\alpha + \beta = 0$.

It follows in particular that, in the notations of \cite[Def.~2.21 \&\ 2.23]{HC} (and recalling that
all cumulants except for the covariances vanish for jointly Gaussian processes), one has
for every $N > 0$ the bounds 
\begin{equ}[e:boundxiCov]
\|\xi^{\eps,\bu}\|_{2N,|\cdot|_\s} \lesssim (1+|\bu|)^{\f{3N}2+\theta N}\;, \qquad 
\|\bar \xi,\xi^{\eps,\bu}\|_{2N,|\cdot|_\s} \lesssim \eps^{\theta N} (1+|\bu|)^{\f{3N}2+\theta N}\;,
\end{equ}
uniformly over $\eps \in (0,1]$ and $\bu \in \R^\ell$, and with $\bar \xi$ given in 
\eqref{e:defxibar}. Here, the quantity $\|\cdot\|_{2N,|\cdot|_\s}$ is computed for the whole collection of noises \eqref{e:shiftedNoise}. The homogeneity assignment $|\cdot|_\s$ used for the norms in \eqref{e:boundxiCov}
is given by assigning homogeneity $\f{2k}3 - \f32 - \f\theta2$ to the noises of type $\fl_k$.
The claim then follows at once by combining \cite[Thm.~2.31]{HC} with 
Corollary~\ref{cor:idemMod} and \cite[Thm.~10.7]{Regularity}.
\end{proof}

\begin{remark}
In fact, the law of the left-hand side of \eqref{e:boundPPieps} only depends on the differences 
$\bu_i-\bu_j$, but this is not something we exploit.
\end{remark}

In order to formulate our convergence result, we introduce the following notation.
Given  $\tau \in \fT(\Rule)$, we write $\delta_0 \tau \in \scal{\tau}_\CB$ for the
element assigning the distribution $\delta_0 \in \CB$ to each edge of noise type.

\begin{proposition}\label{prop:BPHZlift}
Fix $\gamma > 0$ and a sequence of translations $c_\eps$ with $\lim_{\eps \to 0}c_\eps = 0$, 
and let $\PPi^{\eps,\BPHZ}$ be as in Corollary~\ref{cor:idemMod}, but restricted to $\CT_{\CB,\gamma}$.

Then, for any $p \geq 2$ and any compact set $\fK$,
\begin{equ}
	\E \$\PPi^{\eps,\BPHZ} \$_\fK^p \lesssim 1\;,\label{eq:lift_bound}
\end{equ}
uniformly in $\eps \in (0,1]$ and, as $\eps \to 0$, $\PPi^{\eps,\BPHZ}$ converges in probability 
in $\SM_0(\CT_{\CB,\gamma})$ to a model $\PPi^{\BPHZ}$ with the following properties:
\begin{enumerate}
	\item $\PPi^{\BPHZ}$ is independent of the choice of translations $c_\eps$.
	\item\label{item:vanish_noise1} $\PPi^{\BPHZ} \sigma = 0$ for any  element $\sigma \in \langle \tau \rangle_{\CB}$ such that the tree $\tau \in \fT(\Rule)$ contains at least one edge 
	of type $\J_{\fl_k}$ for some $k = 1, 2$.
	\item\label{item:vanish_noise2} For any $z \in \R^2$, any $\tau \in \fT(\Rule)$ and any $\sigma \in \langle \tau \rangle_{\CB}$, one has
	\begin{equ}[eq:models_limit]
		\Pi^{\BPHZ}_z \sigma = \langle \sigma, 1 \rangle\, \Pi_z^{\BPHZ} \delta_0\tau\;,
	\end{equ}
	where $\scal{\sigma,1}$ denotes the duality pairing between $\sigma$, viewed as
	a distribution on $\R^{E_-(\tau)}$, where $E_-(\tau)$ denotes the set of 
	edges of noise type (i.e.\ with a type in $\fL_-$), and the constant function $1$.
\end{enumerate}
\end{proposition}

\begin{remark}
In particular, $\Pi^\BPHZ_z \J_{\fl_{k}}[\mu \otimes \one] = \scal{\mu,1} \xi_{k}$, where $\xi_k = 0$ for $k > 0$ and $\xi_0$ is a space-time white noise. 
\end{remark}

\begin{proof}
We fix a tree $\tau = (T,\ff,\fm) \in \fT(\Rule)$. Regarding the first statement,
it is sufficient to exhibit a candidate limit model $\PPi^\BPHZ$ and to show that, for
every $z \in \R^2$, one has the bound
\begin{equ}
\bigl|\bigl(\bigl(\Pi_z^{\BPHZ} - \Pi_z^{\eps,\BPHZ}\bigr)\sigma\bigr)(\phi_z^\lambda)\bigr|
\lesssim \eps^\theta \lambda^{\deg \tau} \|\sigma\|_{\CB^{\otimes \tau}}\;,
\end{equ}
uniformly over $\sigma \in \scal{\tau}_\CB$, $\eps,\lambda \in (0,1]$, and test functions 
$\phi \in \SB^r_\s$ for some fixed large enough $r > 0$.

Let now $\ell$ be the number  of edges of $\tau$ of noise type. 
Writing as before $E_-$ for the set of these edges,
$\scal{\tau}_\CB$ is canonically isomorphic to the subspace of $\CB^{\otimes E_-}$
invariant under the permutations of $E_-$ arising from the symmetries of $\tau$.
Any ordering $E_- = \{e_1,\ldots,e_\ell\}$ furthermore yields an isomorphism 
$\CB^{\otimes E_-} \approx \CB^{\otimes \ell}$. Let then 
$\hat \tau \in \scal{\tau}_{\CD^{(\ell)}}$ be the element given by 
$\hat \tau = \pi(\one_{\ff(e_1)} \otimes \ldots \otimes \one_{\ff(e_\ell)})$, 
where we use the identification
of $\scal{\tau}_{\CD^{(\ell)}}$ with the subspace of $\CD_{\ff(e_1)}\otimes\ldots\otimes \CD_{\ff(e_\ell)}$ invariant under the symmetries of $\tau$ and write $\pi$ for the projection onto
that subspace given by symmetrisation. 
As a consequence of
the definition of the map $\Ev_\bu^\ell$, one has for every smooth element
$\mu \in \scal{\tau}_\CB$, viewed on one hand as an element of $\CT_\CB$ and
on the other hand as a function on $\R^\ell$, the identity
\begin{equ}
\big(\PPi^{\eps,\BPHZ} \mu\big)(\phi)
= \int \big(\PPi^{\eps,\BPHZ}  \pi(\delta_{\bu})\big)(\phi)\,\mu(d\bu)\
= \int \big(\hPPi^{(\eps,\ell,\bu)}  \hat \tau\big)(\phi)\,\mu(d\bu)\;,
\end{equ}
with all expressions clearly independent of the choice of ordering.
Similarly to above, the same identity also holds if $\PPi^{\eps,\BPHZ}$ and $\hPPi^{(\eps,\ell,\bu)}$
are replaced by $\Pi^{\eps,\BPHZ}_x$ and $\hat\Pi_x^{(\eps,\ell,\bu)}$ respectively.

Given any multiindex $\alpha$ of $\R^\ell$ we furthermore define
$\d^\alpha \hat \tau = \pi(\fd_1^{\alpha_1} \otimes \ldots \otimes \fd_\ell^{\alpha_\ell})$
and we note that, by \eqref{e:shiftedNoise},
\begin{equ}[eq:usetauhat]
\d_\bu^\alpha\big(\hat\Pi_x^{(\eps,\ell,\bu)}  \hat \tau\big)(\phi)
= \big(\hat\Pi_x^{(\eps,\ell,\bu)}  \d^\alpha \hat \tau\big)(\phi)\;.
\end{equ}
This is true for any fixed value of $k_\star$ and of the exponent $\eta$ in the definition of $\CB$.

This formula allows us to apply Corollary~\ref{cor:useful} in the following way. Consider first a wavelet basis
on $\R^2$ compatible with the parabolic scaling and such that its wavelets are of class $\CC^2$.
We write $\phi$ for the ``father wavelet'' and $\Psi$ for the finite collection of ``mother wavelets'' as in \cite[Sec.~3.1]{Regularity}. Fix a bounded open set
$U \subset \R^2$ such that $\bigcup_{z \in \fK} B(z,C) \subset U$ for a large enough constant 
$C$ (depending on the size of the supports of the wavelets) and 
write $\Lambda_n = \{(2^{-2n}k,2^{-n}m)\,:\, (k,m) \in \Z^2\}\cap U$. 
Given a map $x \mapsto \Pi_x$ with
$\Pi_x \colon \CT_{\CD^{(\ell)},<\gamma} \to \CC^{-2}(\R^2)$, we then define
the seminorm 
\begin{equ}
\|\Pi\|_E =  \sup_{\alpha < \gamma} \sup_{\deg\tau = \alpha \atop \|\tau\| = 1} \Big(\sup_{n \ge 0} \sup_{x \in \Lambda_n}\sup_{\psi \in \Psi} 2^{\alpha n} |(\Pi_x \tau)(\psi_x^{2^{-n}})|
+ \sup_{x \in \Lambda_0} |(\Pi_x \tau)(\phi_x^{1})|\Big)\;.
\end{equ}
The space $E$ then consists of (equivalence classes of) functions $\Pi$ such that $\|\Pi\|_E$
is finite.
(Note that this \textit{is} a linear space since we do not impose any algebraic relations.
The actual space of admissible models is then naturally identified with a closed subset 
of $E$.)

The point of this construction is that there exists $M>0$ (depending on $\gamma$)
such that if $Z = (\Pi,\Gamma)$ is an admissible model, then as a consequence of
\cite[Prop.~3.32]{Regularity} (and the admissibility of the model) one has
\begin{equ}
\$Z\$_\fK \lesssim \|\Pi\|_E \bigl(1+ \|\Pi\|_E\bigr)^M\;,
\end{equ}
for some proportionality constant independent of $Z$. Similarly, given two such models,
one has  
\begin{equ}
\$Z; \bar Z\$_\fK \lesssim \|\Pi-\bar \Pi\|_E\bigl(1+ \|\Pi\|_E+\|\bar \Pi\|_E\bigr)^M\;.
\end{equ}
Conversely, we have
\begin{equ}
\|\Pi\|_E \lesssim \$Z\$_{\bar\fK}\;,\qquad
\|\Pi-\bar \Pi\|_E \lesssim \$Z; \bar Z\$_{\bar\fK}\;,
\end{equ}
for $\bar \fK$ given by the closure of $U$.

Combining this with Proposition~\ref{prop:boundModels} and Corollary~\ref{cor:useful} together with~\eqref{eq:usetauhat},
we immediately conclude that $\PPi^{\eps,\BPHZ}$ converges
to a limit $\PPi^{\BPHZ}$ provided that the exponent $\eta$ in the definition of $\CB$ is
chosen sufficiently large.
\end{proof}

\subsection{Renormalisation constants}
\label{sec:renorm_constants}

In this section we consider the renormalisation constants for the BPHZ model constructed above that appear in 
the renormalised mean curvature system. 

For noises that satisfy the assumptions of Proposition~\ref{prop:stationarity}, the renormalisation 
of the nonlinearity is given by $\hat{M}^\eps F = F + \sum_{\tau \in \fT_-(\Rule)} \ren^\eps \bbUpsilon^F[\tau]$ 
(see Lemma~\ref{lem:MF_u} and Proposition~\ref{prop:translation}),
where $\ren^\eps = (g^-(\PPi^\eps)\Ev_\0 \tilde\CA^-_\CD \otimes \id) \Delta^-_\CD$ is the BPHZ character on $\CT_\CD$ defined 
via Definition~\ref{def:BPHZ} and Remark~\ref{rem:bphz_abuse}.

In our case the noises trivially satisfy the assumptions of Proposition~\ref{prop:stationarity}
since $\Eps_+(\fl) = \{\ft_0\}$ for all $\fl \in \fL_{-}$. We now turn to the analysis of the counterterms $\ren^\eps \bbUpsilon^F[\tau]$. 

For this we will first describe a graphical representation of the trees in $\fT_-(\Rule)$. Once 
again we recall that the rule $\Rule$ is given in Section~\ref{sec:first} and that $\kappa_\star < \frac{1}{42}$. One can generalise the 
$\downop$ operator from~\eqref{eq:down} to the trees in $\fT$ in an obvious way by just taking $V_\ft = \R$ for every $\ft \in \fL$ 
in~\eqref{eq:down}. We will first describe the graphical representation of the trees that do not contain polynomials. For 
simplicity of presentation and because it is easier to see the true subdivergencies of the trees, we will represent each $\tau \in \fT_-(\Rule)$ that does not contain polynomials by a graphical representation of $\downop[\tau]$. This is harmless since 
Lemma~\ref{lem:down} implies that $\ren^\eps\bbUpsilon^F[\tau] = \ren^\eps\downop[\bbUpsilon^F[\tau]]$.\footnote{Note though that $\bbUpsilon$ and $\downop$ do not commute, which is the whole 
point of introducing the trees that can be drawn above noises.} Noises are represented by edges $\J_{\fl_0} = \<Jl>$, $\J_{\fl_1} = \<Jl1>$, $\J_{\fl_2} = 
\<Jl2>$ when these noises are coming from $\J_{\fl}[\tau]$ for $\tau \neq \one$, but in this case we represent graphically $\J_{\fl}[\tau]$ by 
gluing $\tau$ to the \textit{root} of $\J_{\fl}$ if $\tau$ does not contain polynomials. Noise symbols without anything drawn above them are denoted by $\Xi_0 = \<Xi>$, $\Xi_1 = \<Xi1>$, 
$\Xi_2 = \<Xi2>$. The tree product is represented by gluing trees at their roots (which is always the node at the bottom).  Edges that represent kernels are $\J_{\ft_0} = \<Jt0>, \J_{(\ft_0, 1)} = \<J1t0>, \J_{(\ft_1, 1)} = 
\<J1t1>,  \J_{(\ft_1, 2)} = \<J2t1>, \J_{(\ft_2, 2)} = \<J2t2>$ and in this case $\J_{(\ft,p)}[\tau]$ is represented by drawing an edge $\J_{(\ft,p)}$
below $\tau$ and joining it at the root of $\tau$. When the multiplication by the noise is meant, i.e.\ we have $\tau\Xi_k$, then such a tree is represented by drawing a large node of the colour corresponding to the noise $\Xi_k$ at the root of $\tau$. For instance  $\J_{(\ft_1, 1)}[\Xi_1]\Xi_2$ is drawn as $\<t1Xi2>$ while $\J_{\fl_0}[\J_{\ft_0}[\J_{\fl_0}[\J_{\ft_0}[\Xi_0]]]]$ is drawn as $\<J[HXi]Xi>$.

Now the multiplication by polynomial $\X_1$ will simply be drawn with an index $\tiny{1}$ at the root. Subtrees of the form $\J_{\fl}[\X_1\tau]$ (with $\tau$ not having any polynomials) will be graphically represented by a graph of $\J_{\fl}[\X_1]\tau$. This is because there are trees $\btau \in \fT_-(\Rule)$ with polynomials inside such that $\downop[\btau] \in \fT_-(\Rule)$. Thus, graphically $\<Jl>$, $\<Jl1>$, $\<Jl2>$ can only join edges $\J_{\ft_0} =  \<Jt0>$ or can have polynomial label at the root as well as have a polynomial label above them. For instance,
\begin{equ}
	\X_1 \J_{\fl_0}[\J_{\ft_0}[\Xi_0]] = \<XHXi>\qquad\text{and}\quad \J_{\fl_0}[\X_1 \J_{\ft_0}[\Xi_0]] = \<XHXia>.
\end{equ}

It turns out that for every tree in $\tau \in \fT_-(\Rule)$ for the rule $\Rule$ from Section~\ref{sec:first} there exists a unique $\alpha \in \N^\ell$, where $\ell$ is the number of noise edges in $\tau$, such that
\begin{equ}[eq:unique_multiindex]
	\Ev_0 \bbUpsilon_\ft^F[\tau] = a \Ev^\ell_\0 \partial^\alpha \hat \tau \otimes F \;\in\; \scal{\tau}_\CB \otimes \SP\;,
\end{equ}
for some $a \in \R$, $F \in \SP$ and $\partial^\alpha \hat \tau = \pi (\fd_1^{\alpha_1} \otimes \dots \otimes \fd_\ell^{\alpha_\ell})$ as in \eqref{eq:usetauhat}. Table~\ref{table:trees} lists all the trees in $\fT_-(\Rule)$ for the rule from Section~\ref{sec:first}. 

First consider $\tau$ which do not contain any polynomials inside. It is clear that $\one_{\fl_k}$ will be 
attached on a noise edge $\Xi_k$, with nothing on top (i.e.\ noises that are graphically represented by $\<Xi>$, $\<Xi1>$,  $\<Xi2>$). Other noise edges $\<Jl>$, $\<Jl1>$, $\<Jl2>$ can only be joined at the root with 
$n$ edges $\<Jt0>$ of type $(\ft_0,0)$ for some $n \in \N$ (in our case $n\leq 3$). In this case $\fd^n_{\ft_0}$ will be attached on $\<Jl>$, $\<Jl1>$ or $\<Jl2>$ viewing $n$ as an element of $\N^{\Eps_+(\fl_k)} = \N^{\{\ft_0\}}$.

Inspecting Table~\ref{table:trees}, we see that $\tau \in \fT_-(\Rule)$ can contain only at most one polynomial $\X_1$. Thanks to~\eqref{eq:LeibnizFhat} and because $\Eps_+(\fl_k) = \{\ft_0\}$ for $k = 0,1,2$ we have
\begin{equ}
	\partial_1 \one_{\fl_k} = \fd_{\ft_0} \otimes \fX_{(\ft_0, 1)}.
\end{equ}
Therefore, whenever $\X_1$ is above the noise, i.e.\ if we have $\<XiX>$, $\<bXiX>$ or $\<bbXiX>$, then $\fd_{\ft_0}$ will be attached to this noise edge.

The fact that each tree in $\fT_-(\Rule)$ gives rise to a unique attachment of the $\fd^n_{\ft_0}$ at each noise edge (for some $n \in \N^{\{\ft_0\}}$) allows us to define a renormalisation constant as
\begin{equ}[eq:ren_const_def]
	C_\eps[\tau] \eqdef \ren^\eps \partial^\alpha \hat \tau\;,
\end{equ}
where $\alpha \in \N^\ell$ is as in \eqref{eq:unique_multiindex}.
We also introduce the notion of homogeneity of a tree, which should be thought of
as its ``real'' degree while
$\deg\tau$ is used for analytical purposes to give ourselves a bit of ``wiggle room''.

\begin{definition}\label{def:homo}
	For $\tau = (T,\rho,\ff,\fm) \in \fT$ the \emph{homogeneity} $|\tau|$ is defined in a similar way to its degree by setting
$
		|\tau| = \sum_{v \in N_T} | \fm(v) |_\s + \sum_{e \in E_T} |\ff(e)|
$
	and $|\ff(e)| = \deg(\ff(e))$ if $\ff(e) \in \CO$. We also set $|\fl_k| = -\frac 32 + \frac{2k}{3}$ for $k = 0,1,2$.
\end{definition}

One heuristically then has $\deg \tau = |\tau|-\kappa$ for some ``small'' $\kappa$, in any case we always have 
$\deg \tau < |\tau|$. 
The next lemma concerns the rate of divergence of the renormalisation constants. 
Its proof is the subject of Appendix~\ref{app:constants}.

\begin{lemma}\label{lem:constants}
	Let $c_\eps = \CO(\eps)$ be a choice of translations of the noises from Section~\ref{sec:noises}. For each $\tau \in \fT_-(\Rule)$ we have the following limits as $\eps \to 0$:
	\minilab{eq:constants}
	\begin{equs}[2]
		\eps^{- |\tau|} C_\eps[\tau] &\quad \to\quad c[\tau]\;,\qquad\qquad &&\text{if ~~$|\tau| \leq -\frac{1}{3}$}\;, \label{eq:constants1}\\
		\eps^{\frac 13} C_\eps[\tau] &\quad\to\quad 0\;,\qquad\qquad &&\text{if ~~$|\tau| > -\frac{1}{3}$}\;, \label{eq:constants2}\\
		C_\eps\bigl[\,\<XiJ[H']+>\,\bigr] &\quad\to\quad C\bigl[\,\<XiJ[H']+>\,\bigr]\;,\label{eq:constants3} \\
		C_\eps\bigl[\,\<H'J1[XiX]>\!\bigr] &\quad\to\quad C\bigl[\,\<H'J1[XiX]>\!\bigr]\;,\label{eq:constants4}
	\end{equs}
	where $c[\tau]$ are real numbers independent of $c_\eps$, and
	\minilab{eq:special-constants}
	\begin{equs}
		C\bigl[\,\<XiJ[H']+>\,\bigr] &= - \int \int  \d_t \rho(z_1)\, (G * G')(z_1 - z_2)\, \rho(z_2)\, d z_1 d z_2\;, \label{eq:special-constants1}\\
		C\bigl[\,\<H'J1[XiX]>\!\bigr] &= -\int \int G'(-z_1) x_1 \d_t \rho(z_1 - z_2)\, \bigl(G' * \rho\bigr)(-z_2)\, dz_1 dz_2\;, \label{eq:special-constants2}
	\end{equs}
	with $G$ the heat kernel, $G' = \d_x G$, and $\rho$ the mollifier used to define the noise in~\eqref{eq:noise2}.
\end{lemma}

\subsection{A renormalised equation}\label{sec:abstract_system}

In this section we apply the machinery of renormalisation of singular SPDEs developed in the
previous sections in order to obtain 
system~\eqref{eq:main2}. At this point the main question is what function $\BF_\CB$ on the right-hand side of~\eqref{eq:AbstractProblemB} 
will give us the system~\eqref{eq:main2} after the reconstruction with respect to the BPHZ model. Unfortunately the trivial choice of $\BF$ 
given by taking the right-hand side of~\eqref{eq:main2} (with the constants $C_{k,\eps}$ set to $0$) and applying~\eqref{eq:FofU_B} is not going to 
work. This is because such a system when renormalised and reconstructed 
produces renormalisation functions rather than 
renormalisation constants and therefore cannot give rise to~\eqref{eq:main2}. Nevertheless, we will show in this section that it is possible to 
start from a different system with coefficients of order $1$ 
that does renormalise to~\eqref{eq:main2}.

We consider the setting of Section~\ref{sec:intro} and assume that the 
functions $F_i$ satisfy Assumption~\ref{assum:functions}, i.e. they belong to $\CC^\func_{m,7}$ for some $m > 0$. We then define the constants 
\begin{equ}[eq:F_constants]
\lambda = F_2(0)\;, \qquad \sigma = F_3(0)\;, \qquad \sigma_1 = \frac 12 F''_3(0)\;,
\end{equ}
which will be used throughout this section, and let the functions $F_{i,\eps}$ be defined 
by~\eqref{eq:F_eps} for $i = 1,2,3$.

We view the $F_{i,\eps}$ as elements of $\SP$ by taking their argument to be $\fX_{(\ft_1, 1)}$ and
we lift them to functions $\BF_{i,\eps} \colon \SH_\CB \to \bar \SH_\CB$ via~\eqref{eq:FofU_B}. Assume that we are also given a collection of constants 
$\bar C_{i,\eps}$ as well as $\CC^\func_{\bar m, 5}$ functions $G_\eps$, $\bar G_{\eps} : \R \to \R$, for some $\bar m > 0$.
We will view these as functions in $\SP$ 
 as well and lift $\bar G_\eps$ to functions $\bar \BG_{\eps}\colon \SH_\CB \to \bar\SH_\CB$ as in~\eqref{eq:abstract-functions}. 

Let $Z^{\eps, \BPHZ} = \ZZ(\PPi^{\eps, \BPHZ})$ for $\eps \in [0,1]$ be the BPHZ model obtained in Section~\ref{sec:smooth_models} for a translation $c_\eps = \CO(\eps)$ and let $\SU^{\gamma, \eta}_{\eps, T} = \SU^{\gamma, \eta}_{T}(Z^{\eps, \BPHZ})$ be a space defined in Section~\ref{sec:fixed_point} for some $\gamma_i,\eta_i$ with $\gamma_i > 2$ for $i = 0, \ldots, 3$.

We also use the shorthand $\SG_\eps = (G_\eps,\bar G_\eps, \bar{C}_{1,\eps}, \bar C_{2,\eps}, \bar C_{3,\eps})$ for constants $\bar{C}_{i,\eps}$, $i = 1, 2, 3$, and for functions $G_\eps$, $\bar G_\eps \colon \R \to \R$. We make the following assumption on these objects. 

\begin{assumption}\label{assum:SG}
The tuple $\SG_\eps$ given above is such that 
	$(G_\eps, \bar G_{\eps})$ converge in $\CC^\func_{\bar m, 5}$ to some functions $(G_0, \bar G_{0})$ for some $\bar m > 0$ and such that $\lim_{\eps \to 0} \max_i |\bar{C}_{i,\eps}|=0$.
\end{assumption}

\begin{remark}\label{rem:bound_in_ML}
	We can readily see that Assumption~\ref{assum:all-functions} is satisfied for our choice of functions. Indeed, the fact that $F_{1,\eps}(u) = \eps^{-\frac 23} F_1 (\eps^{\frac 13} u)$ and $F_1(0) = F'_1(0) = 0$ implies
	\begin{equ}
		|F_{1,\eps}(0)| + |F_{1,\eps}'(0)| + \| F_{1,\eps}''\|_{L^\infty(\R)} + \| F_{1,\eps}'''\|_{L^\infty(\R)} \leq \| F_{1}''\|_{L^\infty(\R)} + \| F_{1}'''\|_{L^\infty(\R)} =: M\;,
	\end{equ}
	if $\eps\leq 1$.
	It is easy to see that $F_i \in \CC^\func_{m,7}$ for some $m > 0$ implies that $\tilde F_i = F_{i,\eps} \in \CC^\func_{m,7}$. In order to make the $\CC^\func_{m,7}$ norm of $F_{i,\eps}$ uniformly bounded in $\eps$ we just need to restrict it to $\eps \in [0,\tilde \eps]$ for some $\tilde \eps < 1$. For instance, $F_{2,\eps}$ has the form
	\begin{equ}
			F_{2,\eps} (u) = \eps^{-1} R(\eps^{\frac 13}u)\,,\quad\text{for}\quad R(u) = u(F_2(u) - u)\;.
	\end{equ}
	By Assumption~\ref{assum:functions} function $R$ satisfies $R(0) = R'(0) = R''(0) = 0$ thus implying $|R^{(\ell)}(u)|  \lesssim e^{m|u|} \|F_2\|_{\CC^{\func}_{m,7}} (|u|^{(3-\ell)\wedge 0} \vee 1)$. Hence, using $\eps \leq \tilde\eps < 1$
	\begin{equ}
		e^{-m|u|} |F^{(\ell)}_{2,\eps} (u)| \lesssim e^{-m(1 - \tilde \eps^{\frac 13})|u|} (|u|^{(3-\ell)\wedge 0} \vee 1)\|F_2\|_{\CC^{\func}_{m,7}}\;.
	\end{equ}
	Thus, since $1 - \tilde \eps^{\frac 13} > 0$ we see that $\|F_{2,\eps}\|_{\CC^{\func}_{m,7}}$ is uniformly bounded in $\eps \in [0,\tilde \eps]$. 
	Moreover, from Lemma~\ref{lem:noise_convergence} we readily conclude that the function $(t,x,h) \mapsto \xi^{h}_{3, \eps}(t,x) \eqdef \xi_{3, \eps} (t + c_\eps t + \eps^2 h, x)$ satisfies Assumption~\ref{assum:all-functions}(\ref{it:all-functions-xi}) on every time interval, uniformly in $\eps \in [0,1]$. For this, we use the fact that each derivative with respect to $h$ 
	gives rise to a factor $\eps^2$. 
	Thus, for every $\tilde \eps < 1$ there exist $M,L > 0$ such that bounds on $F_i, F_{i,\eps}, \xi_{3, \eps}$ from Proposition~\ref{prop:maximal_sol} hold uniformly in $\eps \in [0,\tilde\eps]$. We will address the functions $G_\eps$ and $\bar G_\eps$ in Remark~\ref{rem:Gs}.
\end{remark}

 Using the functions $F_i, F_{j,\eps}, G, \bar G$ and the constants $\bar{C}_{i, \eps}$, we define the nonlinearities $\BF^\eps_{\ft_i}$ for $i \le 3$ by \eqref{eq:abstract-functions}.
Given a tuple $\SG_\eps$, we write $\bbUpsilon^{\SG_\eps}$ and $\bUpsilon^{\SG_\eps}$ 
for the maps defined as in \eqref{eq:maps_bold_Upsilon}, but using the functions 
$\BF^\eps_{\ft_i} \in \SQ(\Rule)$. In the case $\SG_\eps = (0,0,0,0,0)$ we simply write $\bbUpsilon$ and $\bUpsilon$. 

Let $\mathbf{h} = (h_{(i,p)}) \in \R^{\{0,\ldots,3\}\times \N}$, we will now heuristically describe how a counterterm that depends on a 
specific variable $h_{(i,p)}$ can be generated by $\bbUpsilon^{\SG_\eps}$. Given a tree $\tau$, we
view each node of $\tau$ as a differential operator on $\SP$, with 
each incoming edge of type $(j,q)$ yielding a differentiation with respect to
$\SD^q H_j$. A necessary condition is then that at least one of these differential operators,
when applied to the right-hand side of \eqref{eq:abstract-functions}, yields a function that still
depends on $\SD^p H_i$.

For instance, in order to get an $h_{(1,2)}$-dependent counterterm from $\bbUpsilon^{\SG_\eps}_\ft[\tau]$, the tree $\tau$ must contain a node of type $(\partial \ft_1)^n$ for some $n \geq 1$ (using the notation of Section~\ref{sec:first}). This is because $h_{(1,2)}$ correspond to $\SD^2H_1$ in~\eqref{eq:abstract-functions} and $\SD^2H_1$ is multiplied by $\BF_{1,\eps}(\SD H_1)$. On the other hand, in order to get an $h_{(0,1)}$-dependent counterterm, $\tau$ should contain either a node of type $(\partial \ft_0)$ or a node of type $(\partial \ft_1)^n$ for $n \geq 1$ because $\SD H_0$ is only multiplied by $\SD H_0$ itself or by a function of $\SD H_1$.

The next lemma uses these considerations to describe restrictions on the form of 
the counterterms arising in renormalisation.

\begin{lemma}\label{lem:counterterms}
	Let $\gamma = (\gamma_i)_{i=0,\dots,3}, \eta = (\eta_i)_{i=0,\dots,3} \in \R_+^4$ with all $\gamma_i > 2$ 
	and let $F_{i}$ satisfy Assumption~\ref{assum:functions} for $i=1,2,3$. Let $\SG_\eps$ satisfy Assumption~\ref{assum:SG} and $Z^{\eps, \BPHZ} = \ZZ(\PPi^{\eps, \BPHZ})$ be the BPHZ model constructed in Section~\ref{sec:smooth_models}.
	Then for every $\ft \in \fL_{+}$ and $\tau \in \fT_-(\Rule)\backslash\{\<XiJ[H']+s>, \<H'J1[XiX]s>\!\}$ with 
	$\J_\ft[\tau] \in \fT_\ft(\Rule)$ and $C_\eps[\tau] \neq 0$ there exists a function $F^\tau_{\ft, \eps} : \R \to \R$ independent of $G_\eps$ and the constants $\bar C_{i,\eps}$ such that 
	\begin{equ}[eq:F^tau]
		\ren^\eps \bbUpsilon^{\SG_\eps}_{\ft}[\tau](\mathbf{h}) = C_\eps[\tau] F^\tau_{\ft, \eps} (h_{(1,1)})\;,
	\end{equ}
	where $\mathbf{h} = (h_{(i,p)}) \in \R^{\{0,\ldots,3\}\times \N}$.
	In addition, there exists $F^{\<XiJ[H']+>}_{\ft_0,\eps}$ independent of $G_\eps$ and the constants $\bar C_{i,\eps}$ such that
\begin{equs}
	\ren^\eps \bbUpsilon^{\SG_\eps}_{\ft_0} [\<XiJ[H']+>](\mathbf{h}) &= C_\eps[\<XiJ[H']+>]  \big(2 \lambda \sigma^2 h_{(0,1)} +  F^{\<XiJ[H']+>}_{\ft_0,\eps} (h_{(1,1)})\big)\;, \\
	\ren^\eps \bbUpsilon^{\SG_\eps}_{\ft_0} [\<H'J1[XiX]>\!](\mathbf{h}) &= C_\eps[\<H'J1[XiX]>\!] 2 \lambda \sigma^2 h_{(0,1)}\;,
\end{equs}
	where $\lambda$ is defined in~\eqref{eq:F_constants}. 
	
	Finally, the $F^\tau_{\ft,\eps}$ are constant functions independent of $\eps$ for every $\tau \in \fT_-(\Rule)$ with $|\tau| < -\frac{1}{3}$ and $C_\eps[\tau] \neq 0$.
\end{lemma}

\begin{proof}
	Potentially $\bbUpsilon^{\SG_\eps}_\ft[\tau]$ can produce counterterms that are functions of $h_{(i,1)}$ ($i = 0,1,2$)  or $h_{(j,2)}$ ($j = 1,2,3$). We now show that if $\deg \tau < 0$ and $C_\eps[\tau] \neq 0$ then 
	$\bbUpsilon^{\SG_\eps}_{\ft}[\tau](\mathbf{h})$ cannot depend on either of $h_{(0,1)}, h_{(2,1)}$ or $h_{(j,2)}$ ($j = 1,2,3$) unless $\tau \in  \{\<XiJ[H']+s>, \<H'J1[XiX]s>\!\}$ where it will depend on $h_{(0,1)}$. We shall 
	see in Appendix~\ref{app:constants} that $\tau = \<H'J1[XiX]s>$ is the only tree $\tau \in \fT_-(\Rule)$ that contains polynomials $\X^k$ (with $k \neq 0$) and such that $C_\eps[\tau] \neq 0$. Using Example~\ref{ex:derivatives} one finds that
	\begin{equ}[eq:nasty_tree2]
		\ren^\eps \bbUpsilon^{\SG_\eps}_{\ft}[\,\<H'J1[XiX]>\!](\mathbf{h}) = 
		\begin{cases}
			C_\eps[\,\<H'J1[XiX]>\!] 2\lambda \sigma^2 h_{(0,1)}\;&\text{if $\ft = \ft_0$,}\\
			0&\text{if $\ft \neq \ft_0$,}
		\end{cases} 
	\end{equ}
	therefore showing that all the trees containing polynomials agree with the statement of the 
	lemma. In the rest of the proof we will therefore assume that $\tau$ does not contain any polynomials.
	
	If $\bbUpsilon^{\SG_\eps}_{\ft}[\tau](\mathbf{h})$ depends on $h_{(1,2)}$, it means that $\tau$ 
	contains a subtree $\bar\tau$ with a root of node of type $(\partial\ft_1)^n$ for $n\geq 1$ and therefore $\deg\btau \geq \frac 16-\kappa_\star$. Moreover, since $\deg \tau < 0$ we can't 
	have $\tau = \btau$ and therefore $\tau$ must have a subtree either of the form $\tilde{\tau} = \J_{(\ft_i, \ell)}[\btau]$ or $\tilde{\tau} = \J_{(\ft_i, \ell)}[\J_{(\ft_j, 2)}[\btau]]$ for $i \in \{0,1,2\}$, $j \in \{1,2\}$ and 
	$\ell \in \{0,1\}$. In either case $\deg\tilde{\tau} \geq \frac 76 - \kappa_\star$ and as $\deg \tau < 0$ we see 
	that the tree obtained from deleting $\tilde{\tau}$ from $\tau$ is of degree smaller than $-\frac 76 + \kappa_\star$. There is only one tree of degree smaller than $-\frac 76 + \kappa_\star$ namely $\Xi_0$, 
	since $\kappa_\star < \frac {1}{42}$. Therefore, the only option left is if $\tau = \J_{\fl_0}[\J_{\ft_0}[\btau]]$ 
	but in this case $\tau \geq \frac 23 - \kappa_\star$ which is a contradiction with $\tau \in \fT_-(\Rule)$. 
	Similar considerations allow us to rule out the possibility of $\bbUpsilon^{\SG_\eps}_{\ft}[\tau](\mathbf{h})$ being dependent 
	on $h_{(2,1)}, h_{(2,2)}$ and $h_{(3,2)}$.
	We now consider the case when the counterterm $\bbUpsilon^{\SG_\eps}_{\ft}[\tau](\mathbf{h})$ depends on 
	$h_{(0,1)}$. Since we already know that $\tau$ cannot contain a subtree 
	$\btau$ with a root of node of type $(\partial\ft_1)^n$ for $n\geq 1$, it must contain a 
	subtree $\tilde\tau = \J_{(\ft_0, 1)}[\btau]$ with $\deg \btau < 0$. (One can 
	show that $\deg \btau > 0$ implies $\deg\tau > 0$ by a reasoning similar to above.) The right-hand side of~\eqref{eq:abstract-functions} is at most quadratic in $\SD H_0$ thus nodes types $(\partial \ft_0)^2$ do not produce $h_{(0,1)}$ counterterms.  
	Moreover, $\tau \neq \tilde\tau$ since $\tilde\tau$ is 
	planted, so $\tau$ also must contain a subtree of the form $\J_{(\ft_i, \ell)}[\tilde \tau]$ for $i,\ell \in \{0,1\}$. Cases when $i = 1$ can be ruled out easily since these will produce trees which are too 
	positive for our rule. Now the only tree of negative degree that contains a subtree 
	of the form $\J_{(\ft_0, 1)}[\J_{(\ft_0, 1)}[\btau]]$ is $\<H'J1[H']+s>$ but it is not difficult to see that $C_\eps[\<H'J1[H']+s>] = 0$. 
	The only tree of negative degree that contains a subtree of the form $\J_{(\ft_0, 0)}[\J_{(\ft_0, 1)}[\btau]]$ is 
	$\<XiJ[H']+s>$. Note that $\<XiJ[H']+s> \in \fT_\ft$ only if $\ft = \ft_0$ and we have
	\begin{equ}[eq:nasty_tree]
		\ren^\eps \bbUpsilon^{\SG_\eps}_{\ft_0}[\<XiJ[H']+>](\mathbf{h}) = C_\eps[\<XiJ[H']+>] \big(2 \lambda \sigma^2  h_{(0,1)} + F_{2,\eps}(\partial_x h_{(1,1)}) + \bar G_\eps(h_{(1,1)})\big)\;.
	\end{equ}
	Therefore, $F^{\<XiJ[H']+>}_{\ft_0, \eps} = F_{2,\eps} + \bar G_\eps$ which converges in $\CC^\func_{\tilde m, 5}$ for $\tilde m > 0$.
	
	The fact that counterterms $F^\tau_{\ft, \eps}$ do not depend on $G_{\eps}$ or constants $\bar C_{i,\eps}$ follows simply from Remark~\eqref{rem:F_z}.
	
	We now turn to the counterterms produced by trees $\tau \in \fT_-$ with $|\tau|< - \frac{1}{3}$. There are $15$ such trees (recall that these must be unplanted, i.e.\ cannot be of the form 
	$\J_{(\ft_i, p)}[\tau]$ for $i=0,\ldots,3$ and $p \in \N$):
	\begin{equs}
		\<Xi>,\;\<Ch>,\;\<HXi>,\;\<Xi1>,\; \<t2t1>,\; \<t1Xi1>\;, \<HXiH>\;, \<J[Ch]Xi> ,\;\<J[HXi]Xi>, \<J[Ch]H>,\; \<J[HXi]H>,\;  \<t1t2t1>\;, \<bt1Xi1bt1>, \<XiX>, \<Xi+1>.
	\end{equs}
	Only two out of these $15$ trees have non-zero renormalisation constant (see 
	Appendix~\ref{app:constants}) namely $\<Ch>$and $\<HXi>$ and one has 
	$F^{\<Ch>}_{\ft,\eps} = 2\lambda\sigma^2 \one_{\ft = \ft_0}$ and $F^{\<HXi>}_{\ft,\eps} = \sigma^2\one_{\ft = \ft_0}$.
\end{proof}

\begin{proposition}\label{prop:abstract_system}
	Let $\gamma, \eta \in \R_+^4$ be as in Lemma~\ref{lem:U=U_D} and $F_i$ satisfy Assumption~\ref{assum:functions} for $i=1,2,3$ and let $h^0_{k,\eps} \in \CC(\T)$ such that 
	$h^0_{k,\eps} = \eps^{\frac{2k}{3}} h^0_{0,\eps}$ for $k=1,2,3$. Let $Z^{\eps, \BPHZ} = \ZZ(\PPi^{\eps, \BPHZ})$ be the BPHZ model constructed in Section~\ref{sec:smooth_models}. There exist $\hat \eps > 0$ and constants $C_\eps \sim \eps^{-1}$ as 
	well as a tuple $\SG_\eps$ satisfying Assumption~\ref{assum:SG} such that taking $c_\eps = \eps^2 C_\eps$ and $C_{k,\eps} = 
	\eps^{\frac{2k}{3}} C_\eps$ for $k = 0,\ldots,3$ the following holds: for $\eps \in 
	[0,\hat\eps]$ if $H_\eps \in \SU^{\gamma, \eta}_\eps$ solves~\eqref{eq:main4} then $h_{i,\eps} = \CR^{\eps, \BPHZ} H_{i,\eps}$ 
	solve~\eqref{eq:main2} and $h_{i,\eps} = \eps^{\frac{2i}{3}}h_{0,\eps}$. 
	
	One furthermore has $\lim_{\eps \to 0} \bar G_\eps(u) = -2\lambda\sigma^2 (C[\<XiJ[H']+s>] + C[\<H'J1[XiX]s>\!])$.
\end{proposition}
\begin{proof}
	In this proof we combine algebraic relations between the counterterms and renormalisation constants derived in the above lemma together with Theorem~\ref{thm:renormalised_PDE} and Remark~\ref{rem:reduced_reg_str}. Moreover, as explained at the beginning of Section~\ref{sec:renorm_constants} the BPHZ character produced by the model constructed in Proposition~\ref{prop:BPHZlift} is translation invariant and thus Theorem~\ref{thm:renormalised_PDE} can be applied. In the reconstructed system of equations for $h_{i,\eps}$ all counterterms will be evaluated at $(\mathbf{h}_\eps, \bxi^{\mathbf{h}_\eps, c_\eps})$ which is defined via~\eqref{eq:uxi}. The fact that the functions in~\eqref{eq:abstract-functions} are affine in the noise variables and that noises are multiplied by functions from $\SQ$ which only depend on the $\Eps_+$ variables implies in particular that counterterms which appear in $\bbUpsilon^{\SG_\eps}[\tau]$ are independent of $\bxi^{\mathbf{h}_\eps, c_\eps}$,
	so we write $\bbUpsilon^{\SG_\eps}[\tau](\mathbf{h}_\eps)$ instead of $\bbUpsilon^{\SG_\eps}[\tau](\mathbf{h}_\eps,  \bxi^{\mathbf{h}_\eps, c_\eps})$. Thus, applying Lemma~\ref{lem:counterterms} in combination with Theorem~\ref{thm:renormalised_PDE} we will see $\d_x^p h_{i,\eps}$ in the reconstructed equation in place of $h_{(i,p)}$ variables.
	
	First observe that there is only one tree $\tau \in \fT_-(\Rule) \cap \bar\fT_{\ft_2}(\Rule)$ namely $\tau = \<Xi2>$. The 
	counterterm produced by this tree is $\ren^\eps \bbUpsilon^{\SG_\eps}_{\ft_2}[\<Xi2>] = 0$ for every choice of $\SG_\eps$. Therefore, both 
	equations for $H_{2,\eps}$ and $h_{3,\eps}$ do not need to be renormalised, and we can set 
	$\bar C_{2,\eps}$ and $\bar C_{3,\eps}$ to any asymptotically vanishing sequence
	of constants to be determined later.
	
	There are $7$ trees that belong to $\fT_-(\Rule) \cap \bar\fT_{\ft_1}(\Rule)$:
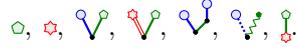
	\begin{equs}
		\<Xi1>,\;\<Xi2>,\; \<Ch1>,\;\<bt2t1>,\; \<H'J1[H']>,\; \<HXi1>,\; \<t1Xi2>.
	\end{equs}
	Only two of these trees have a non-zero renormalisation constant namely $\<Ch1>$ and $\<HXi1>$. It is 
	easy to see thanks to $\PPi^\eps \J_{\fl_1}[\mu \otimes \one] = \eps^{\frac 23} \PPi^\eps \J_{\fl_0}[\mu \otimes \one]$ that
	\begin{equs}[eq:12components]
		\ren^\eps \bbUpsilon^{\SG_\eps}_{\ft_1}[\<Ch1>] &=  C_\eps[\<Ch1>] 2\lambda \sigma^2 = \eps^{\frac 23} C_\eps[\<Ch>] 2\lambda \sigma^2 = \ren^\eps \bbUpsilon^{\SG_\eps}_{\ft_0}[\<Ch>]\;,\\
		 \ren^\eps \bbUpsilon^{\SG_\eps}_{\ft_1}[\<HXi1>] &= C_\eps[\<HXi1>] \sigma^2 =\eps^{\frac 23} C_\eps[\<HXi>] \sigma^2 = \ren^\eps \bbUpsilon^{\SG_\eps}_{\ft_0}[\<HXi>]\;,
	\end{equs}
	 for every choice of $\SG_\eps$.
	
	The above computations also mean that the only non constant counterterms may appear in the $\ft_{0}$ 
	component. For the rest of the proof we will only consider $\tau\in \fT_-(\Rule)$ such that $C_\eps[\tau] \neq 0$ because otherwise the counterterm produced by $\ren^\eps \bbUpsilon^{\SG_\eps}_{\ft_0}[\tau]$ is zero 
	independently of the choice of $\SG_\eps$. 
	
	For each tree $\tau \in \fT_-(\Rule)  \cap \bar\fT_{\ft_0}(\Rule)$ we want to 
	find a set of constants and functions $C^\tau_\eps$, $g^\tau_\eps$, $\bar g^{\tau}_{\eps}$ such 
	that using the fact that $h_{1,\eps} = \eps^{\frac{2}{3}}h_{0,\eps}$
	\begin{equ}[eq:constant_counter]
		\ren^\eps \bbUpsilon^{\SG_\eps}_{\ft_0}[\tau](\partial_x h_{1,\eps}) +  \partial_x h_{0,\eps}\, \bar g^\tau_\eps(\partial_x h_{1,\eps}) + g^{\tau}_{\eps}(\partial_x h_{1,\eps}) = C^\tau_\eps\;.
	\end{equ} 
	The point is that we allow $C^\tau_\eps$ to diverge, but the functions $g^\tau_\eps$ and $\bar g^{\tau}_{\eps}$
	should converge to finite limits.
	If for a given choice of $\SG_\eps$ we have~\eqref{eq:constant_counter} for a tree $\tau$ we say 
	that $\tau$ contributes a constant counterterm. The aim is then to show that we can find such a choice of $\SG_\eps$ so that simultaneously all $\tau \in \fT_-(\Rule) \cap \bar \fT_{\ft_0}(\Rule)$ contribute a constant 
	counterterm, while still having $h_{i,\eps} = \eps^{\frac{2i}{3}}h_{0,\eps}$. 
	
	By Lemma~\ref{lem:counterterms} we know that the counterterms $\ren^\eps \bbUpsilon^{\SG_\eps}_{\ft_0}[\tau](\partial_x h_{1,\eps})$ from~\eqref{eq:F^tau} appearing in the renormalisation
	do not depend on $G_\eps$, i.e.\ they only depend on $F_{i,\eps}$ for $i = 1,2,3$ and on 
	$\bar G_\eps$. Therefore, we shall first find an appropriate function $\bar G_\eps$ to obtain 
	necessary cancellations for divergent non constant counterterms appearing in $\ren^\eps \bbUpsilon^{\SG_\eps}_{\ft_i}[\tau](\partial_x h_{1,\eps})$ and then correct with a suitable choice 
	of $G_\eps$ to cancel the remaining counterterms. Moreover, we shall see that it is possible to choose the constants $\bar C_{i,\eps}$ in such a way that the relation $h_{k,\eps} = \eps^{\frac{2k}{3}}h_{0,\eps}$ holds true. Therefore, without loss of 
	generality we will assume such a relation throughout the proof and show that it is indeed the case in the 
	end.	
	
	Recall from Lemma~\ref{lem:counterterms} that for every choice of $\SG_\eps$, trees $\tau$ with $|\tau| < -\frac{1}{3}$ can produce only constant counterterms. In order to choose suitable
	functions $G_\eps$ and $\bar G_\eps$ we therefore first focus on those $\tau$ with 
	$|\tau| \geq -\frac{1}{3}$. We start with $\SG^0_\eps = (0,0,0,0,0)$ and will gradually build 
	the functions $G_\eps$ and $\bar G_\eps$ by adding more
	terms. First we consider $\tau = \<XiJ[H']+s>$ and $\tau = \<H'J1[XiX]s>$ and by~\eqref{eq:nasty_tree} and \eqref{eq:nasty_tree2} we explicitly 
	have
	\begin{equs}
		\ren^\eps \bbUpsilon_{\ft_0}[\<XiJ[H']+>](\mathbf{h}_\eps) &= 2 \lambda \sigma^2 C_\eps[\<XiJ[H']+>]  \partial_x h_{0,\eps} + C_\eps[\<XiJ[H']+>] F_{2,\eps}(\partial_x h_{1,\eps})\;, \\
		\ren^\eps \bbUpsilon_{\ft_0}[\<H'J1[XiX]>\!](\mathbf{h}_\eps) &= 2 \lambda \sigma^2 C_\eps[\<H'J1[XiX]>\!]  \partial_x h_{0,\eps}\;,
	\end{equs}
	which prompts us to set $\SG^1_\eps = (-C_\eps[\,\<XiJ[H']+s>\,] F_{2,\eps}, -2\lambda \sigma^2 (C_\eps[\<XiJ[H']+s>] + C_\eps[\<H'J1[XiX]s>\!]),0,0,0)$ in order to cancel both of these counterterms. It is important to note that $\bbUpsilon_{\ft_0}^{\SG}[\<XiJ[H']+s>]$ and $\bbUpsilon_{\ft_0}^{\SG}[\<H'J1[XiX]s>\!]$ are independent of $\SG$, so this choice remains valid once we start modifying $\SG$.
	The choice $\SG^1_\eps$
	satisfies Assumption~\ref{assum:SG} by Lemma~\ref{lem:constants} and Remark~\ref{rem:bound_in_ML}. Now we note that for every $\tau \in \check{\fT} \eqdef \fT_-(\Rule)\cap \bar \fT_{\ft_0}(\Rule) \setminus \{\<XiJ[H']+s>, \<H'J1[XiX]s>\!\}$, we obtain a counterterm of the form
	\begin{equs}
		\ren^\eps \bbUpsilon^{\SG^1_\eps}_{\ft_0}[\tau](\partial_x h_{1,\eps}) &= C_\eps[\tau] F^{\tau}_{\ft_0,\eps}(\partial_x h_{1,\eps}) \label{e:counterterm1}\\
		&= C_\eps[\tau] F^{\tau}_{\ft_0,\eps} (0) + C_\eps[\tau] \partial_x h_{1,\eps} \bar F^{\tau}_{\ft_0,\eps} (\partial_x h_{1,\eps})\\
		&= C_\eps[\tau] F^{\tau}_{\ft_0,\eps} (0) + C_\eps[\tau] \eps^{2/3} \partial_x h_{0,\eps} \bar F^{\tau}_{\ft_0,\eps} (\partial_x h_{1,\eps})\;,
	\end{equs}
	where $\bar F^{\tau}_{\ft_0,\eps}(x) = x^{-1} (F^{\tau}_{\ft_0,\eps}(x) - F(0))$, which is smooth and Cauchy in $\eps$. 
	This suggests that we should further add $-C_\eps[\tau] \eps^{\frac 23} \bar F^{\tau}_{\ft_0,\eps}$ to $\bar G_\eps$ for each such $\tau$ with $|\tau| \geq -\frac{1}{3}$, so we set
 $\SG^2_\eps = (-C_\eps[\<XiJ[H']+s>] F_{2,\eps}, \bar G_\eps,0,0,0)$ for
	\begin{equ}[eq:choice_G]
	 \bar G_\eps =	-2\lambda \sigma^2 \bigl(C_\eps[\<XiJ[H']+>] + C_\eps[\<H'J1[XiX]>\!]\bigr) - \eps^{\f13} \sum \Big\{\eps^{\f13}C_\eps[\tau]  \bar F^{\tau}_{\ft_0,\eps} \,:\, \tau \in \check\fT \;\text{with}\; |\tau| \geq -\frac{1}{3} \Big\} \;.
	\end{equ}
	(The reason for writing it in this way is that the summands themselves already converge
	in $\CC^\func_{\bar m, 5}$ for some $\bar m > 0$.)
	 This is going to be our final choice for the function $\bar G_\eps$. Note that by Lemma~\ref{lem:constants} $C_\eps[\tau] \eps^{\frac 23} \to 0$ since $\eps^{\frac 13} C_\eps[\tau]$ is a convergent sequence for trees $\tau$ with 
	$|\tau| = -\frac{1}{3}$  and is a vanishing sequence if $|\tau| > -\frac{1}{3}$. This in particular
	ensures that the very last claim of the proposition holds. For an explanation why $\bar F^{\tau}_{\ft_0,\eps}$ converge in $\CC^\func_{\bar m, 5}$ see Remark~\ref{rem:Gs}.

	One problem one encounters at this stage is that modifying the function $\bar G_\eps$ 
	itself changes the values of $\ren^\eps \bbUpsilon^{\SG^2_\eps}_{\ft_0}[\tau]$, but we now
	argue that this change can be taken care of by a suitable choice of $G_\eps$ (which itself only affects constant counterterms and is therefore ``harmless'') and doesn't require
	modifying $\bar G_\eps$. 
	With $\bar G_\eps$ given by~\eqref{eq:choice_G}, we write $\ren^\eps \bbUpsilon^{\SG^2_\eps}_{\ft_0}[\tau](\partial_x h_{1,\eps}) = C_\eps[\tau] \tilde F^{\tau}_{\ft_0,\eps}(\partial_x h_{1,\eps})$ for $\tau \notin \{\<XiJ[H']+s>, \<H'J1[XiX]s>\!\}$. It is easy to see from the recursive definition of $\bbUpsilon$ that 
	for such $\tau$ the function $x \mapsto F^{\tau}_{\ft_0,\eps}(x)$ (which was given in \eqref{e:counterterm1} and corresponds to $\SG^1_\eps$) is 
	a polynomial of the functions $F_{i,\eps}$ and their derivatives, as well as possibly $x$ (which comes from the term $\sigma_1 (\SD H_1)^2 \hat\Xi_{\CB,\fl_1}$), i.e.\ there exist coefficients $a^\tau_{k,m}$, only finitely many of which are non-vanishing for any given tree $\tau$, 
such that
	\begin{equ}[eq:Ftua]
		F^{\tau}_{\ft_0,\eps} (x) = \sum_{k\in \N^{\N_3}}\sum_{m \ge 0}  a_{k,m}^\tau(\lambda,\sigma,\sigma_1) x^{m} \prod_{(i,j) \in \N_3} \bigl(F^{(j)}_{i,\eps}(x) \bigr)^{k_{i,j}}\;,
	\end{equ}
	where $\N_3 = \{1,2,3\} \times \N$
	and such that each $a_{k,m}^\tau$ is some monomial of $\lambda,\sigma,\sigma_1$.
	$\tilde F^{\tau}_{\ft_0,\eps}$ is then explicitly given by 
	\begin{equ}[eq:tFtua]
		\tilde F^{\tau}_{\ft_0,\eps}(x)  = 
		\sum_{k\in \N^{\N_3}}\sum_{m \ge 0}  a_{k,m}^\tau(\lambda,\sigma,\sigma_1) x^{m} \prod_{(i,j) \in \N_3} \bigl(F^{(j)}_{i,\eps}(x)+ \1_{i=2} \bar G_\eps^{(j)}(x) \bigr)^{k_{i,j}}\;.
	\end{equ}
An important remark is that if $\tau \neq \<XiJ[H']+s>$ and $a_{k,m}^\tau \neq 0$, then $k_{2,0} = 0$.
This is because there is no $\tau \in \fT_-(\Rule)$ with $C_\eps[\tau] \neq 0$ containing a subtree of the form 
	$\J_{(\ft_0,p)}[\J_{(\ft_0,1)}[\btau]]$ apart from $\<XiJ[H']+s>$.
As a consequence, setting $G^{\tau}_\eps = \tilde F^{\tau}_{\ft_0,\eps} - F^{\tau}_{\ft_0,\eps}$,
this is of the form $G^{\tau}_\eps = \sum_{j \ge 1} \bar G^{(j)}_\eps \bar G^{\tau,j}_\eps$ for 
	finitely many $\bar G^{\tau,j}_\eps$ that converges in $\CC^\func_{\bar m, 5}$ for $\bar m > 0$.
	
This in particular implies that the constant term proportional to $C_\eps[\<XiJ[H']+s>] + C_\eps[\<H'J1[XiX]s>\!]$ in \eqref{eq:choice_G} 
does not contribute, so that by Lemma~\ref{lem:constants} $\eps^{-\frac 13} 
	\bar G^{(j)}_\eps$ converges to a finite limit in $\CC^\func_{\bar m, 5}$. This in turn implies that $C_{\eps}[\tau] G^{\tau}_{\eps} = \eps^{\frac 13 }C_\eps[\tau] \sum_{j \ge 1} \eps^{-\frac 13 }\bar G^{(j)}_\eps \bar G^{\tau,j}_\eps$ converges in $\CC^\func_{\bar m, 5}$ as $\eps \to 0$. We therefore define 
	\begin{equ}[eq:G_eps]
		G_\eps = - C_\eps[\<XiJ[H']+>] F_{2,\eps}
		-  \sum \Big\{C_\eps[\tau] G^{\tau}_\eps \,:\, \tau \in \check\fT \;\text{with}\; |\tau| \geq -\frac{1}{3} \Big\}\;,
	\end{equ}
	and $\SG^3_\eps$ with this choice of $G_\eps$ satisfies Assumption~\ref{assum:SG}.
	Finally by~\eqref{eq:nasty_tree} we have 
	\begin{equ}
		\ren^\eps \bbUpsilon^{\SG^3_\eps}_{\ft_0}[\<XiJ[H']+>](\mathbf{h}_\eps) = 2 \lambda \sigma^2 C_\eps[\<XiJ[H']+>]  \partial_x h_{0,\eps} + C_\eps[\<XiJ[H']+>] ( F_{2,\eps}(\partial_x h_{1,\eps}) + \bar G_\eps(\partial_x h_{1,\eps}))\;,
	\end{equ}
	and we add $- C_\eps[\<XiJ[H']+s>]\bar G_\eps$ to $G_\eps$ from~\eqref{eq:G_eps} to obtain our final choice of the function $G_\eps$. This final choice of $G_\eps$ also satisfies Assumption~\ref{assum:SG} because the constant $C_\eps[\<XiJ[H']+s>]$ converges.
	 Thus, we showed that with such choice of $G_\eps$ and $\bar G_\eps$ each $\tau \in \fT_-(\Rule)\cup \bar\fT_{\ft_0}(\Rule)$ contributes a constant counterterm $C^\tau_\eps$ from~\eqref{eq:constant_counter}.\footnote{With notation from~\eqref{eq:constant_counter} we in fact have $C^\tau_\eps = C_\eps[\tau]F^\tau(0)$, $\bar g^\tau_\eps = - C_\eps[\tau] \eps^{\frac 23} \bar F^\tau_{\ft_0,\eps}$ and $g^\tau_\eps = - C_\eps[\tau] G^\tau_\eps$.}
	
	It remains to show that we can guarantee that there exists $C_\eps = C_{0,\eps}$ such that 
	the constants $\bar C_{i,\eps}$ appearing in our final choice of $\SG$ can 
	be chosen such that the renormalised equation is indeed given by \eqref{eq:main2} with $C_{i,\eps} = \eps^{\frac{2i}{3}} C_{\eps}$. We also need to show that the translation in~\eqref{eq:noise_translated} can be chosen as $c_\eps = \eps^2 C_\eps$. This then implies that if $(h_{i,\eps})_{i = 0,\ldots,3}$ is a solution to the 
	system~\eqref{eq:main2} then so is $(\eps^{\frac{2i}{3}}h_{0,\eps})_{i = 0,\ldots,3}$, thus
	guaranteeing that $h_{i,\eps} = \eps^{\frac{2i}{3}}h_{0,\eps}$ for $i = 1,2,3$. With the above choice of $G_\eps$ and $\bar G_\eps$ 
	we have $C_{0,\eps} = -\sum_{\tau \in \fT_-(\Rule)\cap \bar\fT_{\ft_0}(\Rule)} C^\tau_\eps$ and therefore $C_{i,\eps} = 
	\eps^{\frac{2i}{3}} C_{0,\eps} \to 0$ as $\eps \to 0$ for $i = 2,3$. Because the $i^{\text{th}}$ component won't be renormalised for $i = 2,3$, this leads us to choosing $\bar C_{i,\eps} = C_{i,\eps} = \eps^{\frac{2i}{3}}C_{0,\eps}$ for $i = 2,3$.
	Regarding $i=1$, we note that the only two trees $\tau$ such that $\eps^{\frac 23} C^\tau_\eps$ diverges are $\<Ch>$ and $\<HXi>$. Nevertheless, thanks to~\eqref{eq:12components} we can set
	\begin{equ}
		\bar C_{1,\eps} = - \eps^{\frac 23} \sum\big\{C^\tau_\eps\,:\,\tau \in \fT_-(\Rule) \cap \bar\fT_{\ft_0}(\Rule)\setminus\{\<Ch>, \<HXi>\}\big\}  
	\end{equ}
	in order to guarantee that $C_{1,\eps} = - \ren^\eps \bar \bUpsilon^{\SG_\eps}_{\ft_1} [\<Ch1>] - \ren^\eps \bar 
	\bUpsilon^{\SG_\eps}_{\ft_1} [\<HXi1>] + \bar C_{1,\eps} = \eps^{\frac 23} C_{0,\eps}$. 
	
	Regarding the translation 
	constant $c_\eps$ in~\eqref{eq:noise_translated} we want $c_\eps = \eps^2C_\eps$. Note that 
	thanks to the computations in Appendix~\ref{app:constants} we can see that each renormalisation 
	constant $C_\eps[\tau]$ and therefore also $C_\eps$ itself are functions of $c_\eps$. Moreover, the noise $\xi^{c_\eps}_\eps$, defined in \eqref{eq:noise_translated}, depends continuously on $c_\eps$ as a $\CC^{\alpha}_\s$-values function for any $\alpha < -\frac{3}{2}$. Then \cite[Thm.~2.15]{HC} implies that the BPHZ lift depends continuously on $c_\eps$ and thanks to the definition of the renormalisation constants, $C_\eps(c_\eps)$ also depends continuously on $c_\eps$. We thus want to find a fixed 
	point $c_\eps$ of the function $c_\eps \mapsto f(c_\eps) = \eps^2C_\eps(c_\eps)$. By 
	the bounds of Appendix~\ref{app:constants}, the bound $\eps^2|C_\eps(\eps)| \lesssim \eps$ holds 
	uniformly in $c_\eps \in [-1,1]$. Therefore, for $\eps$ small enough, $f$ is a continuous function $[-1, 1] \mapsto [-1,1]$ and thus must have a fixed 
	point. For such a fixed point, one has $c_\eps = \eps^2C_\eps = \CO(\eps)$, thus 
	concluding the proof.
\end{proof}

\begin{remark}
	Note that in the above proof we have a choice of constants $C_\eps$ such that $C_\eps$ is a 
	renormalisation constant computed using the noise $\xi_\eps(t+\eps^2C_\eps t, x)$. This does not rule 
	out a possibility of having several such constants but because limits from 
	Lemma~\ref{lem:constants} and the limiting model from Proposition~\ref{prop:BPHZlift} do not depend 
	on the choice of translation $c_\eps$ from~\eqref{eq:noise_translated} we would have a convergence 
	to the same equation as $\eps \to 0$ for any such choice of $C_\eps$.
\end{remark}

\begin{remark}\label{rem:Gs}
	The above definition of $G_\eps$ and $\bar G_\eps$ imply that functions $G_\eps$ and $\bar G_\eps$ are polynomials of the functions $F_{i,\eps}$ and their derivatives, as well as possibly $x$. Moreover, Table~\ref{table:trees} implies that this derivative of $F_{i,\eps}$ in $G_\eps$ and $\bar G_\eps$ cannot be greater than $2$. Indeed, there are no trees with non-vanishing renormalisation constant that have a node of type $(\partial\ft_1)^n$ for $n > 2$. This implies that for every $\tilde \eps < 1$ there exist $L>0$ such that $\|G_\eps\|_{\CC_{\bar m,5}^\func}, \|\bar G_\eps\|_{\CC_{\bar m,5}^\func} \leq L$ uniformly in $\eps \in [0,\tilde \eps]$ and where $\bar m = k m$ for $k$ being maximal number of products of $F_{i,\eps}$ and their derivatives in $G_\eps, \bar G_\eps$. For simplicity, we will take $\tilde\eps = \frac 12$ in the next section. 
\end{remark}

\subsection{Convergence to the KPZ equation}\label{sec:KPZ}

For $\eps \geq 0$, let $H_\eps$ be the solution of the system \eqref{eq:main4}, defined for a model $Z^\eps$ and for the functions from Assumption~\ref{assum:SG} and below it. From now on we will call $H \eqdef H_\eps \big|_{\eps = 0}$ to avoid confusion with the $0^{\text{th}}$ component and write $H = (H_0, H_1, H_2, H_3)$. Next lemma is about the convergence of the solution of 
the system~\eqref{eq:main4}.

\begin{lemma}\label{lem:convergence}
	Let $C_\eps, \hat \eps, \SG_\eps$ and $Z^{\eps, \BPHZ}$ be as in Proposition~\ref{prop:abstract_system} and $\bar \eps$ as in Proposition~\ref{prop:system_soln} with $M,L > 0$ given by Remarks~\ref{rem:bound_in_ML} and~\ref{rem:Gs} for $\tilde \eps = \frac 12$. Consider 
	$h^0_{\eps} \in \CC^{7 / 3 + \kappa}(\T)$, for some $\kappa \in (0, \frac{1}{6})$, such that $h^0_{\eps} \to h^0_{0}$ in $\CC^{1/3 + \kappa}(\T)$ and 
	$\eps^{\frac{2i}{3}}\|h^0_{\eps}\|_{\CC^{(1 + 2i)/3 + \kappa}}$ are uniformly bounded in $\eps \in [0,1]$ for $i = 1,2,3$. Set $h^0_{i,\eps} = \eps^{\frac{2i}{3}} h^0_{\eps}$ for $i = 0, \ldots,3$. Let $H_{\eps} \in \SU^{\gamma, \eta}_\eps$ be the maximal solution in the sense of Proposition~\ref{prop:maximal_sol} to the system~\eqref{eq:main4} with this choice of $\SG_\eps$ for 
	$\eps \in [0, \hat \eps \wedge \bar \eps \wedge \frac 12]$, with the initial states $h^0_{i,\eps}$ and with respect to the model $Z^{\eps, \BPHZ}$.
	We extend $\CR^{\eps,\BPHZ} H_\eps$ to all of $\R_+$ by setting it equal to $\CR^{\eps,\BPHZ} H_\eps(T_\star,\cdot)$
	for times after $T_\star$.

	Then $\lim_{\eps \to 0} \CR^{\eps,\BPHZ} H_\eps = \CR^{\BPHZ} H$ in probability in $\CC^\sol$, where $H$ solves~\eqref{eq:main4} with the initial states $(h^0_{0}, 0, 0, 0)$, with respect to the model $Z^\BPHZ$ and with right-hand side determined by the functions
	\begin{equs}[eq:main5]
	\BF_{\ft_0} (H) &= \frac{1}{2} F_1''(0) ( \SD H_{1})^2 \SD^2 H_{1} + \lambda\, (\SD H_0)^2 + \nu \SD H_0 + \frac 12 F''_2(0) \SD H_{0} (\SD H_{1})^3 \\
	&\qquad + G_0 (0) + \sigma \hat\Xi_{0}(H_0) + \sigma_1 (\SD H_1)^2 \hat\Xi_{1}(H_0)\\
	&\qquad  + \frac{1}{24} F_3''(0) (\SD H_{1})^4 \hat\Xi_{2}(H_0)\;, \\
	\BF_{\ft_1} (H) &= \frac{1}{2} F_1''(0) ( \SD H_{1})^2 \SD^2 H_{2} + \lambda\, \SD H_0 \SD H_1\\ 
	&\qquad+ \sigma \hat\Xi_{1}(H_0) + \sigma_1 (\SD H_1)^2 \hat\Xi_{2}(H_0)\;, \\
	\BF_{\ft_2} (H) &= \sigma \hat\Xi_{2}(H_0)\;, \qquad \BF_{\ft_3}(H) = 0 \;,
	\end{equs}	
	where $\nu = - 2\lambda \sigma^2 (C[\<XiJ[H']+s>] + C[\<H'J1[XiX]s>\!]) =  - 2\lambda \sigma^2 \lim_{\eps \to 0} (C_\eps[\<XiJ[H']+s>] + C_\eps[\<H'J1[XiX]s>\!])$
	$\CR^{\BPHZ} H_i = 0$ for $i = 1,2,3$.
\end{lemma}

\begin{proof}
	This result follows from Proposition~\ref{prop:maximal_sol} by continuity of the solution map with respect to the input data. Remarks~\ref{rem:bound_in_ML} and~\ref{rem:Gs} show that $F_i, F_{i,\eps} , G, \bar G$ satisfy uniform bounds of Proposition~\ref{prop:maximal_sol}. Convergence in probability comes from the corresponding convergence of the models $Z^{\eps, \BPHZ}$ to $Z^\BPHZ$. $\CR^\BPHZ H_i = 0$ for $i = 1,2,3$ simply follows from the fact that $\CR^{\eps,\BPHZ} H_i = \eps^{\frac{2i}{3}} \CR^{\eps,\BPHZ} H_0$ by Proposition~\ref{prop:system_soln}. This together with Lemma~\ref{lem:noise_convergence} implies that $\BF_{\ft_3}(H) = 0$. The definition of the functions~\eqref{eq:F_eps} and Assumption~\ref{assum:functions} imply that 
	\begin{equ}
		F_{1,\eps}(u) \to \frac 12 F''_1(0) u^2\,,\quad F_{2,\eps}(u) \to \frac 12 F''_2(0) u^3\,,\quad F_{3,\eps}(u) \to \frac {1}{4!} F''_3(0) u^4\,,
	\end{equ}
	in $\CC^\func_{m, 7}$ as $\eps \to 0$. Moreover, by Propostion~\ref{prop:abstract_system} $\bar G_\eps \to \nu$ in $\CC^\func_{m,7}$. This together with $\d_x h_1 = 0$ shows convergence of other terms, thus completing the proof.
\end{proof}

We are finally ready to state the main result of this paper, whose particular cases are 
Theorems~\ref{thm:main_intro} and~\ref{thm:qKPZ_intro}. 

\begin{theorem}\label{thm:main}
Let the random field $\eta_\eps : \R^2 \to \R$ be as in Theorem~\ref{thm:main_intro} and let functions 
$F_i$, $i = 1,2,3$, satisfy Assumption~\ref{assum:functions}. Let furthermore $h_\eps^0 \in \CC^{7/3 + \kappa}$, for some $\kappa \in (0, \frac{1}{6})$, be 
such that $h^0_{\eps} \to h^0$ in $\CC^{1/3 + \kappa}(\T)$ and $\eps^{\frac{2i}{3}}\|h^0_{\eps}\|_{\CC^{(1 + 2i)/3 + \kappa}}$ are 
uniformly bounded in $\eps \in [0,1]$ for $i = 1,2,3$. Let $h_\eps$ be the solution of~\eqref{eq:main} with the 
initial state $h^0_\eps$ and with the driving noise defined via $\eta_\eps$ in~\eqref{eq:rescaled_noise}. Let 
$\nu = -2\lambda \sigma^2 (C[\<XiJ[H']+s>] + C[\<H'J1[XiX]s>\!])$. Then there is a sequence of renormalisation constants $C_\eps \sim \eps^{-1}$, such that for any $\alpha \leq \frac{1}{3} + \kappa$ the random functions $h_\eps$ 
converge in probability as $\eps \to 0$ in the space $\CC(\R_+, \CC^\alpha(\T))$ to $h$ where $(t,x) \to h(t,x - \nu t)$ is the Cole--Hopf solution of the KPZ equation~\eqref{eq:KPZ} with the initial state $h^0$.
\end{theorem}

Before we turn to the proof of this theorem, we state a few of its immediate consequences.
First, it allows us to prove Theorems~\ref{thm:main_intro} and~\ref{thm:qKPZ_intro}, which are the main 
results of this article.

\begin{proof}[of Theorem~\ref{thm:main_intro}]
We claim that this is a particular case of Theorem~\ref{thm:main} for $\lambda = \frac{1}{2}$, $\sigma = 1$ and for the functions~\eqref{eq:functions_main}. 

	Verifying that functions from~\eqref{eq:functions_main} satisfy Assumption~\ref{assum:functions} and belong to $\CC_{m,7}^\func$ is a trivial task. Now note that a convergence of $(t,x) \mapsto h_\eps(t, x+\nu t)$ for $h_\eps$ 
	from~\eqref{eq:rescaled_solution} to a function $h$ is equivalent to convergence of $h_\eps$ to a 
	function $(t,x) \mapsto h(t, x-\nu t)$. The rest of the proof is immediate once we notice that the functions 
	from~\eqref{eq:functions_main} satisfy Assumption~\ref{assum:functions}. To see that $\nu$ 
	from~\eqref{eq:speed} is indeed equal to $-2\lambda \sigma^2 (C[\<XiJ[H']+s>] + C[\<H'J1[XiX]s>\!])$ for $\lambda = \frac{1}{2}$ and for the constant \eqref{eq:special-constants}, we apply integration by parts in \eqref{eq:special-constants2} and change the integration variables.
\end{proof}

\begin{proof}[of Theorem~\ref{thm:qKPZ_intro}]\label{cor:qKPZ}
	The proof is again just a simple application of Theorem~\ref{thm:main}, this time with nonlinearities in~\eqref{eq:main} 
	given by $F_1 \equiv 0$, $F_2 \equiv \lambda$ and $F_3 \equiv \sigma$. The reason why one does not need 
	to assume in Theorem~\ref{thm:qKPZ_intro} the uniform $\CC^{(1+2k) / 3 + \kappa}$ bounds of $h^0_\eps$ in 
	comparison to Theorem~\ref{thm:main}, is because for Theorem~\ref{thm:qKPZ_intro} one does not really 
	need to split equation~\eqref{eq:qKPZ} into system of four equations $h_{k,\eps}$ which resulted in 
	assumption $h^0_{k,\eps} \in \CC^{(1+2k)/3 + \kappa}$ in Lemma~\ref{prop:system_soln}.
\end{proof}

\begin{proof}[of Theorem~\ref{thm:main}]
	First we show that for almost every realisation of the noise we have convergence of $h_\eps$ on the 
	interval of existence of $h_\eps$. From the definition of the system of equations~\eqref{eq:main2} it is 
	evident that choosing $C_\eps = C_{0,\eps}$ and $h^0_{0,\eps} = h^0_\eps$ we have $h_{0,\eps} = h_\eps$. Therefore, $h_\eps \to h\eqdef \CR^\BPHZ(H_0)$ so it is enough to show that $\CR^\BPHZ(H_0)$ is the 
	solution to the KPZ equation in a moving frame.
	
	For this let $\rho$ be a smooth compactly supported function that integrates to $1$ and denote 
	$\rho^\eps(t,x) = \eps^{-\frac{3}{2}} \rho(\eps^{-2}t, \eps^{-1} x)$. We construct a sequence of 
	smooth models $Z^\eps$ approximating the model $Z^\BPHZ$ from Proposition~\ref{prop:BPHZlift} in the 
	following way: $Z^\eps = \ZZ(\PPi^\eps \circ M^\eps)$ where $M^\eps$ is given by~\eqref{eq:Mu} and 
	$\PPi^\eps$ is admissible, multiplicative in the sense of Definitions~\ref{def:admissible} 
	and~\ref{def:canonical}, and such that 
	\begin{equ}
		\PPi^\eps \J_{\fl_0}[\mu \otimes \one] = \scal{\mu , 1} \xi^\eps\;,\qquad 	\PPi^\eps \J_{\fl_i}[\mu \otimes \one] = 0\;,\quad i = 1,2\;,
	\end{equ}
	where $\xi^\eps = \rho^\eps * \xi$. From the characterisation of the model $Z^\BPHZ$ in 
	Proposition~\ref{prop:BPHZlift} it is clear that $Z^\eps \to Z^\BPHZ$. Now let $\tilde{H}_\eps = (\tilde{H}_{0,\eps},\tilde{H}_{1,\eps},\tilde{H}_{2,\eps},\tilde{H}_{3,\eps})$ be the solution to an abstract 
	equation whose right-hand side is given by~\eqref{eq:main5} with initial conditions $(h^0_\eps,0,0,0)$. 
	
	It is easy to see that with such model we have $\tilde h_{i,\eps} \eqdef \CR(\tilde{H}_{i,\eps}) = 0$ for $i = 1,2,3$. Moreover,
	\begin{equs}
		\CR \bigl(\hat\Xi_0(\tilde{H}_{0,\eps})\bigr) (z) &= \Pi^\eps_z \big( \hat\Xi_0(\tilde{H}_{0,\eps})(z) \big) (z) = \PPi^\eps \J_{\fl_0} [\delta_{\tilde h_{0,\eps}} \otimes \one] = \scal{\delta_{\tilde h_{0,\eps}} , 1} \xi^\eps = \xi^\eps\;,\\
		\CR \bigl(\hat\Xi_i(\tilde{H}_{0,\eps})\bigr) (z) &= 0\;,\quad i = 1,2\;.
	\end{equs}
	We will show that above reconstructions and Theorem~\ref{thm:renormalised_PDE} (together with Remark~\ref{rem:reduced_reg_str}) imply that $\tilde h_{0,\eps}$ solves 
	\begin{equ}\label{eq:KPZeps}
		(\partial_t - \partial^2_x) \tilde h_{0,\eps} = \lambda(\partial_x \tilde h_{0,\eps})^2 + \nu \partial_x \tilde h_{0,\eps} + G_0(0) + \sum_{\tau \in \fT_-(\Rule^{\KPZ}) } \tilde{\ell}^\eps_{\BPHZ} \bbUpsilon^F_{\ft_0}[\tau] + \sigma \xi^\eps\;,
	\end{equ}
	where $\tilde{\ell}^\eps_\BPHZ$ is the BPHZ character for the model $Z^\eps$ as in Definition~\ref{def:BPHZ} and $-G_0(0) - \sum_{\tau \in \fT_-(\Rule^{\KPZ}) } \tilde{\ell}_\BPHZ^\eps \bbUpsilon^F_{\ft_0}[\tau]$ is the renormalisation constant that appears in the renormalisation of the 
	KPZ equation. One can observe that
	\begin{equ}
		\fT_-(\Rule^{\KPZ}) = \Big\{\<Xi>, \<Xi+1>, \<Ch>,  \<J[Ch]H>,  \<Ch+1>, \<Ch+1a>, \<H'J1[H']+>, \<LKPZ>, \<Ch2>\Big\}\;.
	\end{equ}
	It is easy to see $\tilde{\ell}^\eps_\BPHZ \sigma = \ren^\eps \sigma$ for $\sigma \in \scal{\tau}_\CD$ and $\tau \in \fT_-(\Rule^{\KPZ}) $. The fact that renormalisation constants in~\eqref{eq:KPZeps} are indeed the correct constants follows by 
	an easy computation as trees from $\fT_-(\Rule^{\KPZ})$ give constant counterterms.
	
	It remains to show that $\tilde{\ell}^\eps_\BPHZ \bbUpsilon^F_{\ft_0}[\tau] = 0$ for every tree $\tau \notin \fT_-(\Rule^{\KPZ})$. This in turn is a consequence of the fact that we have two possibilities. First is that 
	$\tau$ contains an edge $\J_{\fl_i}$ for $i = 1,2$ and therefore it is easy to see that $\PPi^\eps \sigma = 0$ for every $\sigma \in \scal{\tau}_\CB$ and thus $\tilde{\ell}^\eps_\BPHZ \bbUpsilon^F_{\ft_0}[\tau] = 0$. Second is 
	that $\tau$ contains a subtree $ \J_{\fl_0}[\btau]$ with $\btau \neq \one$. In this case $\Ev_{\tilde h_{0,\eps}} \bar \bUpsilon^F_{\ft_0} [\tau]$ contains a subtree $\J_{\fl_0}[\delta^{(k)}_{\tilde h_{0,\eps}} \otimes \bar\sigma]$ with $k > 0$, $\bar\sigma \in \scal{\btau}_\CB$ and thus $\tilde{\ell}^\eps_\BPHZ \bbUpsilon^F_{\ft_0}[\tau] = 0$ because $\scal{\delta^{(k)}_{\tilde h_{0,\eps}} , 1} = 0$. 
	
	It is clear now that $\tilde h_{0,\eps}$ is both an approximation of the solution to the KPZ equation in a 
	moving frame and that $\CR^\BPHZ(H_0) = h$. It remains just to perform a translation $(t,x) \mapsto (t,x - \nu t)$ to get rid of the $\nu\partial_x\tilde h_{0,\eps}$ term in equation~\eqref{eq:KPZeps}. Therefore, we 
	indeed have shown that $(t,x) \mapsto h(t,x-\nu t)$ solves the KPZ equation.
	
	The fact that convergence holds on every time interval $[0, T]$ follows from the fact that the solution of the KPZ equation is almost surely global in time \cite{KPZ} and because of the continuity of the solution map obtained in Proposition~\ref{prop:maximal_sol}.
\end{proof}

\section{Convergence of the qEW model}
\label{sec:qEW}

In this section we prove Theorem~\ref{thm:qEW}. For this, we use the rule and the nonlinearity defined 
in Section~\ref{sec:qEW_rule} and Example~\ref{ex:qEWnonlin}. 

If $d = 1$, then equation~\eqref{eq:QuenchedEW_rescaled} is a particular example of the 
equation~\eqref{eq:main}, with $F_1 \equiv F_2 \equiv 0$ and $F_3 \equiv 1$. In particular, the analysis of Section~\ref{sec:application} can be repeated for this 
equation, in which divergences of the renormalisation constants are different. One important difference 
is that in this case one has $\PPi^\eps \J_{\fl_0}[\delta^{(n)}_h \otimes \1] (t,x) = \eps^{(2+\beta)n} \d^n_t\xi_{0,\eps}(t + \eps^{2+\beta} C_\eps + \eps^{2+\beta} h, x)$ and thus in order for 
Lemma~\ref{lem:noise_convergence} to be true one must have $2+\beta \geq 2$ which in 
turn implies $0\leq\beta = \alpha + \frac{d-2}{2} = \alpha - \frac 12$, thus the requirement for $\alpha \geq \frac 12$ in Theorem~\ref{thm:qEW}.

In the case $d \geq 2$, the proof goes in a similar way, however now we need to consider the 
system~\eqref{eq:QuenchedEW_system} and the number of trees grows when $d$ increases. The restriction $\delta + \beta + 2$ in the statement of Theorem~\ref{thm:qEW} implies 
that $\delta+2 > -\beta = -\alpha - \frac{d-2}{2}$. Note that the spatial regularity of the solution to the 
stochastic heat equation for every fixed time is expected to be $- \frac{d-2}{2}-\kappa$ for any $\kappa > 0$. This allows to obtain a flow of the solution map to the qEW equation. Setting $h^0_{i,\eps} = \eps^{\beta(i-1)} u^0_{\eps}(\eps^{-1}\bigcdot)$ also implies that $G h^0_{1,\eps}$ converges to $0$ in a 
H\"older space $\CC^{\delta + \beta + 2, \delta + \beta}_\s$ and that the blow up in time is integrable 
since $\delta + \beta > -2 = -\s_0$. 

Note that for $d \geq 2$ we again have $\PPi^\eps \J_{\fl_i}[\delta^{(n)}_h \otimes \1] (t,x) = \eps^{2n} \d^n_t\xi_{i,\eps}(t + \eps^{2+\beta} C_\eps + \eps^{2} h, x)$ for $i=0,1$, therefore results of Section~\ref{sec:noises} can be repeated \textit{mutatis mutandis} in $d \geq 2$ dimensions, since we shall see below that $\eps^{2+\beta} C_\eps$ vanishes as $\eps \to 0$.
 
For a notational convenience we again use $\Xi_0 = \<Xi>$, $\Xi_1 = \<Xi1>$, $\J_{\fl_0} = \<Jl>$, and $\J_{\ft_1} = \<Jt1>$. Similarly to Section~\ref{sec:renorm_constants}, one can see that the renormalisation constant with the worst 
divergence is $\<H1Xi>$ with $\deg (\<H1Xi>) = \beta - d - 2\kappa_\star$ (many renormalisation 
constants from Section~\ref{sec:renorm_constants} do not show up in our case, because the 
nonlinearity only comes from the inhomogeneous noise). In particular, note that if $\beta > d$, i.e.\ if 
$\alpha > \frac{d-2}{2}$, then $\kappa_\star$ can be taken small enough so that 
$\fT_-(\Rule^{\text{qEW}})$ is empty which implies that no renormalisation is required. More precisely, 
let the rescaled and shifted noise $\xi_\eps^{c_\eps}$ be as in~\eqref{eq:noise_translated}, then 
repeating computation similar to~\eqref{eq:constant_HXi}, we conclude that 
\begin{equ}
	C_\eps\bigl[\<H1Xi>\bigr] = -\E\bigl[\eps^{2 + \beta} \d_t \xi^{c_\eps}_{\eps}(0) \bigl(K*\xi^{c_\eps}_{\eps}\bigr)(0)\bigr] \sim \eps^{\beta - d} = \eps^{\alpha - \frac{d-2}{2}}\;,
\end{equ}
as soon as the constant $c_\eps$ is bounded uniformly in $\eps$. This implies that the renormalisation 
constant $C_\eps$ in~\eqref{eq:QuenchedEW_rescaled} diverges as $\eps^{\alpha - \frac{d-2}{2}}$. Then the choice~\eqref{eq:alpha_and_beta} of $\beta$ yields $c_\eps = \eps^{2 + \beta} C_\eps \sim \eps^{2\alpha}$, which vanishes as $\eps \to 0$.

One big difference with Section~\ref{sec:KPZ} is that we cannot solve the abstract version 
of~\eqref{eq:QuenchedEW_system} in the space of modelled distributions directly. Indeed, let $Z^{\eps,\BPHZ}$ be a model similar to the one constructed in Section~\ref{sec:smooth_models} but adapted to 
the rule for qEW. Let's consider the abstract system
\begin{equ}[eq:qEW_abstract]
	H_{i,\eps} = \CP_{\ft_i} \bigl(\one_+ \CQ_{\leq \gamma_i} \hat\Xi_{\CB, \fl_i}(H_{1,\eps})\bigr) + G_{\ft_i}h^0_{i,\eps}\;,\qquad i = 0,1\,,
\end{equ}
where $h^0_{1,\eps} = \eps^{\beta} h^0_{0,\eps}$ and $\gamma_1 = \gamma_0+\beta$ for some 
$\gamma_0>2$. Then $H_0$ lives in a sector of regularity $\alpha_0 = -\frac{d-2}{2} - \kappa_\star < 0 = -\s_0 +\deg \ft_0$ for $d \geq 2$ thus conditions on integrability from 
Lemma~\eqref{lem:abstract_integration} fail to be satisfied for the $\ft_0$ component.

On the other hand these conditions are satisfied for the $\ft_1$ component $(\alpha_1 = \beta + \alpha_0 = \alpha - \kappa_\star)$ and since the equation for 
$H_{1,\eps}$ is completely independent of $H_{0,\eps}$ we are able to apply Picard's 
iteration~\cite[Thm.~2.21]{BCCH} to construct the maximal solution to $H_{1,\eps}$ first. We then 
simply apply~\cite[Thm.~5.12]{Regularity} to define $H_{0,\eps}$ through the $\ft_0^{\text{th}}$ component 
of~\eqref{eq:qEW_abstract}.\footnote{\cite[Thm.~5.12]{Regularity} simply allows us to integrate the 
modelled distribution $\one_+ \CQ_{\leq \gamma_i} \hat\Xi_{\CB, \fl_0}(H_{1,\eps})$. Integrability 
assumptions from Lemma~\eqref{lem:abstract_integration} are needed for the $\ft_1$ component to 
patch local solutions together to obtain the maximal solution.} Showing that reconstructed solutions 
$h_{i,\eps} \eqdef \hat\CR^\eps H_{i,\eps}$ still satisfy $h_{1,\eps} = \eps^{\beta}h_{0,\eps}$ can be done 
similarly to Proposition~\ref{prop:system_soln} since Theorem~\ref{thm:renormalised_PDE} 
can still be applied because if $(H_{i,\eps})_{i = 0,1}$ solves the system~\eqref{eq:qEW_abstract} then 
by definition of $H_{0,\eps}$, $(H_{i,\eps})_{i = 0,1}$ is coherent to $(\hat\Xi_{\CB, \fl_i}(H_{1,\eps}))_{i = 0,1}$ according to the Definition~\ref{def:coherence}.

Finally, to show that $h_{0,\eps}$ converges to the solution of the stochastic heat equation \eqref{eq:heat_qEW} we first use coherence again and Proposition~\ref{prop:MU_coherent} to observe
\begin{equ}
	\hat\CR^\eps\bigl(\one_+ \CQ_{\leq \gamma_i} \hat\Xi_{\CB, \fl_0}(H_{1,\eps})\bigr) = \xi^{h_{1,\eps}}_{0,\eps} + \sum_{\tau \in \fT_-(\Rule^{\qEW})} \ren^\eps(\bar\Upsilon^F_{\ft_0}[\tau]) = \xi^{h_{1,\eps}}_{0,\eps} - C_\eps\;.
\end{equ}
Then since $h_{1,\eps} \to 0$ as $\eps \to 0$ the arguments among the same lines as in the first part of the proof of Theorem~\ref{thm:main} allow us to conclude that $\xi^{h_{1,\eps}}_{0,\eps} - C_\eps$ converges to white noise $\xi$ in $\CC^{-\frac{2+d}{2}-\kappa_\star}$ as $\eps \to 0$. 

\begin{appendices}

\section{Estimates on the renormalisation constants}\label{app:constants}
We present all the unplanted trees of negative degrees in Table~\ref{table:trees}. The first 
column lists the homogeneity of the trees. If a tree belongs to one of the columns ``Gauss'', ``Odd'', 
``X'' or ``$\prod$'', this means that the corresponding renormalisation constant vanishes 
(more explanations for each column are provided below). The column ``$\infty$'' means that 
the renormalisation constant is in general non-zero (and might diverge).

\begin{table}[ht]
	\centering
	\begin{tabular}{c c c c c c} \toprule
		$|\tau|$ & Gauss & Odd & X & $\prod$ & $\infty$\\
		\midrule
		$ - \frac{3}{2}$ & \<Xi> & - & - & - & -
		\\ \rowcolor{symbols!5!pagebackground}
		$-1$ & - & - & - & - & \<Ch>, \<HXi>
		\\
		$ - \frac{5}{6}$ & \<Xi1> & - & - & - & -
		\\ \rowcolor{symbols!5!pagebackground}
		$ - \frac{2}{3}$ & - & \<t2t1>, \<t1Xi1> & - & - & -
		\\
		$ -\frac{1}{2}$ & \<t1t2t1>, \<J[HXi]Xi>, \<J[Ch]H>, \<bt1Xi1bt1> & - & \<XiX> & - & - 
		\\
		& \<J[HXi]H>, \<HXiH>, \<J[Ch]Xi> & - & \<Xi+1> & - & - 
		\\  \rowcolor{symbols!5!pagebackground} 
		$-\frac{1}{3}$ & \<J[t1t1]Xi> & - & - & \<4Xi1> & \<Ch1>, \<Ch1a>, \<HXi1>
		\\\rowcolor{symbols!5!pagebackground} 
		& - & - & - & - & \<t2J1[H']>, \<t1J2[H']>, \<H'J1[t2]>
		\\
		$ -\frac{1}{6}$ & \<Xi2>, \<t1J2[H']t1>, \<J1[HXi1]t2>, \<t1t2J1[H']> & - & - & - & - 
		\\
		& \<J[t1t2]Xi>, \<J1[Ch1]t2>, \<t1H't1>, \<J2[Ch1]t1>& - & - & - & - 
		\\
		& \<J2[Ch1]t1>, \<J1[t1t2]H'>, \<J2[HXi1]t1>, \<5Xi1> & - & - & - & - 
		\\
		& \<bt1Xi1H>, \<XiJ[Xi1bt1]>,  \<H'J1[Xi1bt1]> & - & - & - & - 
		\\  \rowcolor{symbols!5!pagebackground} 
		$ 0$ & - & \<Ch+1a>, \<bt2t1> & \<Ch+1> & \<3HXi> & \<Ch2>, \<J1[HXi]J1[HXi]>, \<ChHXi>, \<J[HHXi]Xi>
		\\  \rowcolor{symbols!5!pagebackground} 
		& - & \<H'J1[H']>, \<t2[Xi2]t1> & \<XHXi> & \<t2t1t1J1[H']> & \<J1[HXiH]H'>,  \<H'J1[XiX]>, \<XiJ[H']+>, \<J[t2[H']]Xi>
		\\  \rowcolor{symbols!5!pagebackground} 
		& - & \<H'J1[H']+>, \<t2t1[Xi2]> & \<XHXia> & \<H't1^3> & \<t2[HXi1]t1t1>, \<J1[HXi]t2t1>, \<J1[H't1]t2t1>, \<t2[H't1]t1t1>
		\\  \rowcolor{symbols!5!pagebackground} 
		& - &  \<t2[H']t1[H']>, \<t2J1[bt2]> & - & \<t1t1t1J2[H']> & \<J[Ch]XiH>, \<J[HXi]XiH>, \<LKPZ>, \<LHXiXiXi>
		\\  \rowcolor{symbols!5!pagebackground} 
		& - &  \<J1[t2[H']]H'>, \<t1t2[bt2]> & - & \<6Xi1> & \<LHXiXiH'>, \<LH'H'XiXi>, \<LH'H'XiH'>
		\\  \rowcolor{symbols!5!pagebackground} 
		& - &  \<XiJ[XiX]>, \<t1Xi2> & - & \<2bt1Xi1H> & \<LHXiH'Xi>, \<LH'H'H'Xi>, \<LHXiH'H'>
		\\ \rowcolor{symbols!5!pagebackground} 
		& - & - & - & - & \<J1[t1t2t1]Xi>, \<J1[2bt1Xi1]H'>, \<J1[t1t2t1]H'>
		\\
		\bottomrule
	\end{tabular}
	\caption{Trees of negative degree. \label{table:trees}}
\end{table} 

We now explain how to compute a renormalisation constant corresponding to each tree in the table. As 
explained in Section~\ref{sec:renorm_constants}, we view trees as graphical 
representations of $\downop \!\tau$ rather than of $\tau$. Moreover, we have also shown in the beginning of Section~\ref{sec:renorm_constants} that each tree $\tau \in \fT_-(\Rule)$ gives rise to a unique assignment of  $\delta^{(n)}_0$ at each noise in $\CT_\CB$. In short, $\delta_0$ will be attached to noises with nothing on top: $ \<Xi>, \<Xi1>$ and $\<Xi2>$. If noises $\<Jl>$, $\<Jl1>$, $\<Jl2>$ are joined at the root with $n$ number of $\<Jt0>$ edges and have $k$ powers of $\X_1$ above them ($k \in \{0,1\}$) then $\delta^{(n+k)}_0$ will be attached to such noises.

In what follows let $K$ be the singular part of the heat kernel (see Section~\ref{sec:admissible}) and let $K' \eqdef \d_x K$, $K'' \eqdef \d^2_x K$. 

\subsection{Zero renormalisation constants}

We now explain why the renormalisation constants assigned to the trees
in each of the columns ``Gauss'', ``Odd'', ``X'' and ``$\prod$'' vanish. 

For trees in the column ``Gauss'', the constant is zero because the Gaussianity of 
the noises implies that the expectation of a product of an odd number of noises is zero. 

The column ``Odd'' in Table~\ref{table:trees} is such that the respective renormalisation 
constant vanishes because the resulting kernel, that is used in computation of the constant, is an odd function and because the noise is stationary. We give a 
couple of examples of the renormalisation constant for the trees in the ``Odd'' column. For instance,
 \begin{equs}
 	C_\eps &\bigl[\,\<t1Xi2>\,\bigr] = - \E \bigl[ \bigl(\PPi^\eps \<t1Xi2>\bigr)(0)\bigr] = - \E \bigl[\eps^{2/3} \bigl(K'* \xi^{c_\eps}_{\eps}\bigr) (0)\, \eps^{4/3} \xi^{c_\eps}_{\eps}(0)\bigr] \\
	&= -\eps^2\int K'(-z) \E \bigl[\xi^{c_\eps}_{\eps}(z) \xi^{c_\eps}_{\eps}(0)\bigr] dz\\
 	&= \eps^2\int K'(z) \E\bigl[\xi^{c_\eps}_{\eps}(0) \xi^{c_\eps}_{\eps}(-z)\bigr] dz = \eps^2\int K'(-z) \E\bigl[ \xi^{c_\eps}_{\eps}(0) \xi^{c_\eps}_{\eps}(z)\bigr] dz = - C_\eps\bigl[\,\<t1Xi2>\,\bigr]\;,
 \end{equs}
where we used stationarity of the noise and oddness of $K'$ in the fourth equality, and change of variables $z \to -z$ in the fifth equality. Similarly,
 \begin{equs} 
	C_\eps \bigl[\;\<XiJ[XiX]>\!\bigr] &= - \E\bigl[\bigl(\PPi^\eps \<XiJ[XiX]>\!\bigr)(0)\bigr] = - \E\bigl[ \eps^2\d_t \xi^{c_\eps}_{\eps}(0) \eps^2 (K*(\d_t \xi^{c_\eps}_{\eps} y_1))(0)\bigr]\\
	& = -\eps^4\int K(-z)z_1 \E\bigl[\d_t \xi^{c_\eps}_{\eps}(z)\d_t\xi^{c_\eps}_{\eps}(0)\bigr] dz \\
	&= -\eps^4\int K(-z)z_1 \E\bigl[\d_t \xi^{c_\eps}_{\eps}(0)\d_t \xi^{c_\eps}_{\eps}(-z)\bigr] dz \\
	&= \eps^4\int K(-z)z_1 \E\bigl[\d_t\xi^{c_\eps}_{\eps}(0) \d_t \xi^{c_\eps}_{\eps}(z)\bigr] dz = - C_\eps\bigl[\;\<XiJ[XiX]>\!\bigr]\;,
\end{equs}
where $z_1$ is the spatial coordinate of $z$ and we use that $K(z)z_1$ is an odd function in the spatial variable. In both cases $C_\eps[\tau] = - C_\eps[\tau]$, which implies $C_\eps[\tau] = 0$.

The column ``X'' means that the constant is zero because the tree is multiplied by a polynomial at the root. 
Showing that the renormalisation constants vanish in this case is rather trivial, since $(\PPi^\eps \X_1) (0) = 0$. 
 
The column ``$\prod$'' means that the renormalisation constant is zero because the tree yields a product
of three or more Gaussian random variables. This can be seen in the example
 \begin{equ}
 	C_\eps\bigl[\,\<H't1^3ss>\,\bigr] = -\E\bigl[\bigl(\PPi^\eps \<H't1^3ss>\bigr) (0)\bigr] + 3 \E\bigl[\bigl(\PPi^\eps \<Ch1>\bigr)(0)\bigr] \E\bigl[\bigl(\PPi^\eps \<t1t1>\bigr)(0)\bigr] =  0\;,
 \end{equ}
where the first identity follows from the definition of $\Deltam$ (and the fact that odd products of Gaussians
have vanishing expectation) and the second identity is Wick's formula.
 
 In the following section we consider the trees from the last column in Table~\ref{table:trees}. For this we 
 recall that our models are defined via the shifted noises
\begin{equ}[eq:rescaled_noise_B]
	\xi^{c_\eps}_{\eps}(t,x) = \xi_{\eps}(t+c_\eps t,x)\;,
\end{equ}
for some constant $c_\eps = \CO(\eps)$. That's why it will be convenient to define the function 
\begin{equ}[eq:rescaled_rho_B]
	\tilde\rho_\eps(t,x) \eqdef \sqrt{1 + c_\eps} \,\rho_\eps \bigl((1 + c_\eps)t,x\bigr)\;,
\end{equ}
which is of course well-defined for all $\eps$ small enough since then $c_\eps > -1$. Then we can write 
$\xi^{c_\eps}_{\eps} = \tilde{\rho}_\eps * \tilde{\xi}$,
for a new space-time white noise. 
The noise $\tilde \xi$ depends on $\eps$ through the constant $c_\eps$, but we prefer to suppress it in our 
notation, since it will not play any role in our analysis. We will also use the function 
\begin{equ}
	\bar \rho_\eps(t,x) \eqdef \eps^3 \tilde\rho_\eps(\eps^2 t, \eps x) = \sqrt{1 + c_\eps} \,\rho \bigl((1 + c_\eps)t,x\bigr)\;.
\end{equ}

\subsection{Trees with homogeneity not exceeding $-\frac{1}{3}$.}

Here we will consider the trees $\<Ch>, \<HXi>, \<Ch1>, \<HXi1>, \<t2J1[H']>, \<t1J2[H']>, \<H'J1[t2]>$ and prove 
the convergence result~\eqref{eq:constants1} in Lemma~\ref{lem:constants}. These are the only trees with 
homogeneity smaller than$-\frac{1}{3}$, whose renormalisation constants are not equal to zero. 

The fact that $\eps C_\eps[\<Chs>]$ converges is proved in~\cite[Lem.~6.3]{KPZ}. Recall that we write the 
heat kernel as $G = K + R$, where $R$ is a smooth function and $K$ is a singular part. We will use 
extensively throughout the rest of this appendix the scaling properties of the heat kernel for $z = (t,x) \in \R^2$
\begin{equ}[eq:scaling]
	G(\eps^\s z) = \eps^{-1} G(z)\;.
\end{equ}
Then for the tree $\<HXis>$ we have
\begin{equs}[eq:constant_HXi]
	\eps C_\eps\bigl[\,\<HXi>\,\bigr] &= -\eps \E\bigl[\eps^2 \d_t \xi^{c_\eps}_{\eps}(0) \bigl(K*\xi^{c_\eps}_{\eps}\bigr)(0)\bigr] = -\eps^3 \int \d_t \tilde \rho_\eps(z)\, (K * \tilde \rho_\eps)(z)\, dz \\
	&\approx -\eps^3 \int \d_t \tilde \rho_\eps(z)\, (G * \tilde \rho_\eps)(z)\, dz =  - \int\d_t \bar \rho_\eps(z)\, (G * \bar \rho_\eps)(z)\, dz\;,
\end{equs}
where $\approx$ holds up to an error of order $\CO(\eps^3)$ and the last equality holds 
thanks to the scaling 
properties of the heat kernel~\eqref{eq:scaling}. Now it is easy to see that $\eps^{1/3}C_\eps[\<Ch1s>] = \eps C_\eps[\<Chs>]$ and $\eps^{1/3}C_\eps[\<HXi1s>] = \eps C_\eps[\<HXis>]$, so convergence of these terms is 
automatic. 

For the tree $\<t2J1[H']s>$ we have
\begin{equs}
	\eps^{\frac 13} C_\eps\bigl[\,\<t2J1[H']>\,\bigr] &= -\eps^{\frac 13} \E\bigl[\eps^{\frac 23} (K''* \xi^{c_\eps}_{\eps})(0) (K'*K'* \xi^{c_\eps}_{\eps})(0)\bigr] \\
	&= -\eps \int (K''*\tilde\rho_\eps)(z)\, (K'*K' *\tilde \rho_\eps)(z)\, dz \\
	&\approx -\eps \int\int (G''*G'*G')(z_1-z_2) \tilde\rho_\eps(z_1) \tilde\rho_\eps(z_2)\, dz_1dz_2 \\
	&= - \int\int (G''*G'*G')(z_1-z_2) \bar \rho_\eps(z_1) \bar\rho_\eps(z_2)\, dz_1dz_2\;,
\end{equs}
where the error term is of order $\CO(\eps)$ and where in the last equality we have used $G''*G'*G'(\eps^\s z) = \eps^{-1} G''*G'*G'(z)$, which is easy to show using~\eqref{eq:scaling}. 

Using similar computations, one can show that both $\eps^{\frac 13}C_\eps[\<t1J2[H']s>]$ and 
$\eps^{\frac 13}C_\eps[\<H'J1[t2]s>]$ converge to $0$.

\subsection{Trees of degree strictly larger than $-\frac{1}{3}$}

We prove now convergence~\eqref{eq:constants2}, for what we need to consider the $23$ trees of 
homogeneity $0$ in the column ``$\infty$'' of Table~\ref{table:trees} (we exclude the trees $\<XiJ[H']+s>$ and $\<H'J1[XiX]s>$, 
whose renormalisation constants we compute in Section~\ref{sec:special_tree}).

Among these $22$ trees, we will treat separately $\<J[t2[H']]Xi>$; the other $21$ trees 
can be split into $4$ groups, depending on whether their shape is that of
\<LKPZs>, \<Ch2s>, \<J1[t1t2t1]H's>, or \<J1[H't1]t2t1s>.
(Two trees are said to have the same shape if only differ by the decorations of some of their edges.)
The renormalisation constants of the trees in the same group can be estimated in a similar way. In what follows we are going to use the following bounds on the kernels 
\begin{equ}[eq:K_bounds]
	|D^k K(z)| \lesssim \|z\|^{-1-|k|_\s}_{\s}\;, \qquad |D^k (K * \bar\rho_\eps)(z)| \lesssim (\|z\| + \eps)^{-1-|k|_\s}_{\s}\;,
\end{equ}
for any $k \in \N^2$, and that the function $\bar\rho_\eps$ is smooth, whose derivatives are bounded 
uniformly in $\eps \in (0, \eps_0)$, for some $\eps_0 > 0$. Then we will use the results 
from~\cite[Sec.~10]{Regularity} to bound products and convolutions of such kernels. 

Let us first compute the renormalisation constant of the tree $\tau_1 = \<J[t2[H']]Xi>$. We have 
\begin{equs}
	C_\eps[\tau_1] &= -\E\bigl[ \bigl(\PPi^\eps \tau_1\bigr)(0)\bigr] = -\E\bigl[\eps^2 \d_t \xi^{c_\eps}_{\eps}(0)\, \bigl(K * K'' * K' * \xi^{c_\eps}_{\eps}\bigr)(0)\bigr]\\
	&= -\eps^2 \int \d_t \tilde \rho_\eps(z)\, \bigl(K * K'' * K' * \tilde \rho_\eps\bigr)(z)\, dz\;.
\end{equs}
Although the kernel $K''$ is not integrable in $\R^2$, we can use integration by parts to move all derivatives 
on the right most function $\tilde\rho_\eps$. Furthermore, we use $K = G-R$ for the heat kernel $G$ and a 
smooth function $R$, to get
\begin{equ}
	C_\eps[\tau_1] \approx -\eps^2 \int \d_t  \tilde\rho_\eps(z)\, \bigl(G * G * G * \tilde\rho'''_\eps\bigr)(z)\, dz\;,
\end{equ}
with an error term of order $\eps^2$. Changing the integration variable $z \mapsto \eps^{\s} z$ and using the 
scaling property~\eqref{eq:scaling} of the heat kernel, we get
\begin{equ}
	C_\eps[\tau_1] \approx - \int \d_t \bar\rho_\eps(z)\, \bigl(G * G * G * \bar\rho_\eps'''\bigr)(z)\, dz\;.
\end{equ}
This immediately implies convergence~\eqref{eq:constants2} for this tree.

Let us consider a tree $\tau_2$ from the first group, which has the same shape as $\<LKPZs>$. Then it 
follows from~\eqref{eq:K_bounds} and the results in~\cite[Sec.~10]{Regularity}, that the renormalisation 
constant $C_\eps[\tau_2]$ is a sum of terms of the form
\begin{equ}[eq:C3]
\int \int A^{(1)}_\eps(-z_1) A^{(2)}_\eps(z_1 - z_2) A^{(3)}_\eps(-z_2)\, dz_1 dz_2\;,
\end{equ}
with the kernels satisfying the bounds $|A^{(1)}_\eps(z)| \lesssim \|z\|^{\alpha_1}_{\s}$, $|A^{(2)}_\eps(z)| \lesssim (\|z\|_{\s} + \eps)^{\alpha_2}$ and $|A^{(3)}_\eps(z)| \lesssim \|z\|^{\alpha_3}_{\s}$, for some 
exponents $-3 \leq \alpha_i \leq 0$, where only $\alpha_2$ can equal $-3$ and where $\alpha_1 + \alpha_2 + \alpha_3 = -6$. Then~\eqref{eq:C3} can diverge with the speed at most $\log \eps$, which yields $\eps^{\frac 13} C_\eps[\tau_2] \to 0$ as required.

Let us now consider a tree $\tau_3$ in the second group of trees, which have the shape of $\<Ch2s>$. The 
bounds~\eqref{eq:K_bounds} and the results in~\cite[Sec.~10]{Regularity} imply that the renormalisation 
constant $C_\eps[\tau_3]$ is a sum of terms of the form 
\begin{equ}
 \int \int B^{(1)}(-z_1) B^{(2)}_\eps(z_1 - z_2) B^{(3)}(-z_2)\, dz_1 dz_2\;,
\end{equ}
where $|B^{(i)}_\eps(z)| \lesssim \|z\|^{\beta_i}_{\s}$ for some exponentials $-3 < \beta \leq 0$, such that $\beta_1 + \beta_2 + \beta_3 = -6$. Then the renormalisation constant $C_\eps[\tau_3]$ can diverge with the speed at most $\log \eps$, which yields $\eps^{\frac 13} C_\eps[\tau_3] \to 0$.

Let $\tau_4$ be a tree from the third group, i.e.\ with a shape similar to $\<J1[t1t2t1]H's>$. In this case the 
renormalisation constant factorizes  $C_\eps[\tau_4] = C_\eps[\bar \tau_1] C_\eps[\bar \tau_2]$, where the 
divergences of $C_\eps[\bar \tau_i]$ follow readily from~\eqref{eq:K_bounds}. This gives the 
limit~\eqref{eq:constants2}.

The renormalisation constant of a tree $\tau_5$ in the last group is a sum of terms of the form $\int E_\eps(-z) dz$, with a kernel satisfying $|E_\eps(z)| \lesssim \|z\|^{\delta}_{\s}$ for $\delta > - 3$. This readily 
yields $\eps^{\frac 13} C_\eps[\tau_5] \to 0$.

\subsection{Constants for special trees}
\label{sec:special_tree}

We treat separately the trees $\<XiJ[H']+s>$ and $\<H'J1[XiX]s>$. The renormalisation constants of these trees is 
the reason for obtaining a KPZ equation in a moving frame in the limit and it is 
important to show that $C_\eps[\<XiJ[H']+s>]$ and $C_\eps[\<H'J1[XiX]s>\!]$ converge to a finite limit. Now
\begin{equs}
	C_\eps\bigl[\,\<XiJ[H']+>\,\bigr] = -\E\bigl[\PPi^\eps \<XiJ[H']+> (0)\bigr] &= -\E\bigl[\eps^2 \d_t \xi^{c_\eps}_\eps(0)\, (K * K' * \xi^{c_\eps}_\eps)(0)\bigr]\\
	&= -\eps^2 \int \d_t  \tilde\rho_\eps(z)\, (K * K' * \tilde\rho_\eps)(z)\, dz\\
	&= -\eps^2 \int \int \d_t \tilde \rho_\eps(z_1)\, (K * K')(z_1 - z_2)\, \tilde \rho_\eps(z_2)\, d z_1 d z_2\;,
\end{equs}
where $\tilde \rho_\eps$ is a mollifier. Then using $K = G-R$ for a smooth function $R$
\begin{equ}
	C_\eps\bigl[\,\<XiJ[H']+>\,\bigr] \approx - \eps^2 \int \int  \d_t \tilde\rho_\eps(z_1)\, (G * G')(z_1 - z_2)\, \tilde\rho_\eps(z_2)\, d z_1 d z_2\;,
\end{equ}
where the error term is a multiple of $\eps^2$. From the scaling property of the heat kernel~\eqref{eq:scaling} 
we can obtain that $G*G'(\eps^\s z) = G*G'(z)$. Using such a scaling, recalling that $\tilde \rho_\eps(z) = \eps^{-3} \bar \rho_\eps(\eps^{-\s} z)$, and changing the variables $z_i \mapsto \eps^{\s} z_i$ yields
\begin{equ}[eq:mu_tree]
	C_\eps\bigl[\,\<XiJ[H']+>\,\bigr] \approx - \int \int  \d_t \bar \rho_\eps(z_1)\, (G * G')(z_1 - z_2)\, \bar \rho_\eps(z_2)\, d z_1 d z_2\;,
\end{equ}
from which convergence \eqref{eq:constants3} follows with the limiting constant \eqref{eq:special-constants1}.

Note that if $\rho$ is taken to be even in space, then since $K * K'$ is an odd function in space, we have $C_\eps[\<XiJ[H']+s>] = 0$.

Now, we turn to the tree $\<H'J1[XiX]s>$, for which we have 
\begin{equs}
	C_\eps[\<H'J1[XiX]>\!] &= -\E\bigl[ \bigl(\PPi^\eps \<H'J1[XiX]>\!\bigr)(0)\bigr] = -\E\biggl[\eps^2 \int K'(-z_1) x_1 \d_t \xi^{c_\eps}_\eps(z_1) d z_1\, \bigl(K' * \xi^{c_\eps}_\eps\bigr)(0)\biggr]\\
	&= -\eps^2 \int \int K'(-z_1) x_1 \d_t \tilde \rho_\eps(z_1 - z_2)\, \bigl(K' * \tilde \rho_\eps\bigr)(-z_2)\, dz_1 dz_2\;,
\end{equs}
where $x_1$ is the spatial component of $z_1$. As above, we replace the kernel $K$ by $G$, change the 
integration variables, and use the scaling property of the heat kernel to get
\begin{equ}
	C_\eps[\tau_2] \approx -\int \int G'(-z_1) x_1 \d_t \bar\rho_\eps(z_1 - z_2)\, \bigl(G' * \bar\rho_\eps\bigr)(-z_2)\, dz_1 dz_2\;,
\end{equ}
where the error term is or order $\eps^2$. Then convergence ~\eqref{eq:constants4} readily follows for this tree with the limiting constant \eqref{eq:special-constants2}.

We note that if $\rho$ is taken to be even in space, then we have $C_\eps[\<H'J1[XiX]s>\!] = 0$.

\section{Index of frequently used notations}
\label{app:notation}

Here, we list various constants, norms and other objects used in this article together 
with their meanings and references to definitions.

\begin{center}
\begin{longtable}{p{.12\textwidth}p{.65\textwidth}p{.12\textwidth}}
\toprule
Object & Meaning & Ref. \\
\midrule
\endhead
\bottomrule
\endfoot
$\CB_\ft$, $\CD_\ft$ & Spaces assigned to edges of trees & p.~\pageref{lab:Bell} \\
$\fd_o$ & Generators of the spaces $\CD_\ft$ & p.~\pageref{lab:fd} \\
$\deg$ & Degrees associated to the labels in $\fL$ & p.~\pageref{lab:index_set} \\
$D_o$, $\d_i$ & Differentiation operators on $\SP$ & p.~\pageref{lab:Do} \\
$\Eps$, $\Eps_{\pm}$ & The sets of labels of all components of the system & p.~\pageref{lab:CO} \\
$\Eps(\fl)$ & The labels of components presented in the $\fl^{\text{th}}$ noise & p.~\pageref{lab:Eell} \\
$\Ev_\bu$ & The evaluation map & p.~\pageref{lab:Eval} \\
$\hat F$ & Extension of the nonlinearity $F$ & Eq.~\ref{eq:hatF} \\
$\BF_{\CD}$, $\BF_{\CB}$ & Lifts of the nonlinearities $F$ & pp.~\pageref{eq:FofU}+\pageref{eq:FofU_B} \\
$g^-(\PPi)$ & A character on $\CT^-_\CB$ associated to $\PPi$ & Eq.~\ref{eq:character} \\
$\SH_\CD$, $\bar \SH_\CD$ & The structure spaces describing the system of equations & p.~\pageref{lab:H-spaces} \\
$\hXi_{\CD}$, $\hXi_{\CB}$ & The abstract noises & pp.~\pageref{eq:Xi_hat_new}+\pageref{eq:Xi_hat_B} \\
$\fL$, $\fL_{\pm}$ & The sets of labels of equations and noises in the system & p.~\pageref{lab:index_set} \\
$\Lambda$ & The time-space domain $\R_+ \times \T^d$ & p.~\pageref{lab:Lambda} \\
$\ren^{\,\bu} (\sigma)$ & The BPHZ renormalisation constant associated with $\sigma$ & p.~\pageref{lab:BPHZ} \\
$\hat M$ & A generic renormalisation map & Eq.~\ref{eq:MF} \\
$M$, $M_\bu$ & The BPHZ renormalisation maps & Eq.~\ref{eq:Mu} \\
$\M_\infty$, $\M_0$ & Spaces of models on $\CT_\CB$ & p.~\pageref{lab:models} \\
$\CO$ & The set of labels of solutions to the SPDEs and their derivatives  & p.~\pageref{lab:CO} \\
$\Poly$ & The set of abstract polynomials & p.~\pageref{lab:Poly} \\
$\bp_{\leq L}$ & A projection on $\CT_V$ & p.~\pageref{lab:projectionL} \\
$\Proj_\tau$ & A projection on $\CT_V$ to the $\tau$ component & p.~\pageref{lab:Proj_Q} \\
$\SP$ & The real algebra of smooth functions on $\R^\Eps$ & p.~\pageref{lab:SP}\\
$\CP_\ft$ & An integration map & Eq.~\ref{eq:P_t} \\
$\CQ_{\leq \gamma}$ & A projection map on $\CT$ & p.~\pageref{lab:projectionQ} \\
$\SQ_+$ & The $\fL_+$ components of the system of SPDEs & p.~\pageref{lab:SP}\\
$\SQ(\Rule) $ & The set of all functions that conform to the rule $\Rule$ & p.~\pageref{lab:QRule}\\
$\Rule$ & A rule & p.~\pageref{lab:Rule} \\
$S(\tau)$ & The symmetry factor of a tree $\tau$ & Eq.~\ref{eq:symmetry_factor} \\
$\fT$ & The set of decorated trees & p.~\pageref{lab:trees} \\
$\fT(\Rule)$ & The set of trees conforming to the rule $\Rule$ & p.~\pageref{lab:trees_rule} \\
$\fT_-(\Rule)$ & The set of unplanted trees of negative degree & Eq.~\ref{eq:unplanted} \\
$\CT_V$ & $V$-valued regularity structure & p.~\pageref{lab:T_V} \\
$\CT_\CD$, $\CT_\CB$ & Two regularity structures, which we use to renormalise and solve the equation & p.~\pageref{eq:norm_weighted} \\
$\TT_{\Poly}$ & The sector of polynomials & p.~\pageref{lab:Poly} \\
$\TT_o$, $\bar\TT_o$ & The structure spaces describing the $o$-equation & p.~\pageref{eq:jets} \\
$\TT_{\ft, \leq \gamma}$, $\bar\TT_{\ft, \leq \gamma}$ & Projected structure spaces & p.~\pageref{lab:projected-T} \\
$U^R$ & The non-polynomial part of a modelled distribution $U$ & Eq.~\ref{eq:UR} \\
$\SU^{\gamma,\eta}$ & A graded space of modelled distributions & Eq.~\ref{eq:U-space} \\
$\w$ & Weight discounting large values of $h$ & p.~\pageref{lab:TD_TB} \\
$\Upsilon^F$, $\bar\Upsilon^F$ & Coherence maps & Eq.~\ref{eq:upsilon} \\
$\bUpsilon^F$, $\bbUpsilon^F$ & Coherence maps with symmetry factor & Eq.~\ref{eq:maps_bold_Upsilon} \\
\end{longtable}
\end{center}

\end{appendices}

\bibliographystyle{Martin}\
 \bibliography{refs}

\end{document}